\documentclass[12pt]{article}
\usepackage{amssymb}
\usepackage{amsfonts}
\usepackage{amsmath}
\usepackage{eurosym}
\usepackage{makeidx}
\usepackage{amssymb}
\usepackage{times}
\usepackage{anysize}

\marginsize{2cm}{2cm}{1cm}{1cm}
\newtheorem{theorem}{Theorem}[section]

\newtheorem{condition}[theorem]{Condition}

\newtheorem{corollary}[theorem]{Corollary}

\newtheorem{definition}[theorem]{Definition}
\newtheorem{example}[theorem]{Example}

\newtheorem{lemma}[theorem]{Lemma}

\newtheorem{proposition}[theorem]{Proposition}
\newtheorem{remark}[theorem]{Remark}

\newenvironment{proof}[1][Proof]{\noindent\textbf{#1.} }{\ \rule{0.5em}{0.5em}}
\input{tcilatex}
\begin{document}

\title{Abstract Theory of Bogoliubov Linearizations with Application to
Nonlinear Thermodynamic Formalism}
\author{J.-B. Bru \and W. de Siqueira Pedra \and A.O. Lopes}
\date{\today }
\maketitle

\begin{abstract}
Bogoliubov's 1947 approximation, originally developed in the microscopic
theory of superfluidity, laid the foundation for solving previously
intractable quantum models and later became part of \textquotedblleft
quantum mathematics\textquotedblright . Regarding mathematically rigorous
results, one of its most advanced forms -- the only one that handles quantum
equilibrium states -- was published in the Memoirs of the AMS in 2013.
Building on key results from convex analysis, the present work significantly
extends it to obtain a general mathematical theory that enables nonlinear
variational problems on convex compact spaces to be fully studied via a
linearization process, referred to here as the \textquotedblleft Bogoliubov
linearization\textquotedblright . This problem is particularly timely, given
the current development of quantum algorithms and computers, which are
inherently linear machines. A deep connection with the optimal transport is
also proven. As a paradigmatic example of application, the approach proposed
here is applied to the nonlinear thermodynamic formalism -- an emerging
field that can have important impacts on various fields of mathematics, such
as ergodic transport, the fractals and multifractal formalism, discrete-time
linear dynamics, $C^{\ast }$-algebras, etc. Notably, even in the case of
finite alphabets the obtained results go beyond the scope of the existing
literature in nonlinear thermodynamic formalism.\bigskip

\noindent \textbf{Keywords: }Bogoliubov approximation, Ruelle operator,
nonlinear thermodynamic formalism, optimal transportation, non-cooperative
equilibria. \bigskip

\noindent \textbf{2020 AMS Subject Classification: }58E30, 37D35, 46N10,
46N55.
\end{abstract}

%TCIMACRO{\TeXButton{\tableofcontents}{\tableofcontents}}%
%BeginExpansion
\tableofcontents%
%EndExpansion

\section{Introduction}

\subsection{Abstract theory of Bogoliubov linearizations}

Nonlinear variational problems appear in such a wide variety of mathematical
contexts that it would be difficult to provide an exhaustive list. Their
interest lies in the wide range of applications in the real world, such as
in physics, biology, engineering and economics, for systems governed by
principles of equilibrium and optimality in the presence of nonlinear
effects. From a mathematical perspective, nonlinear variational problems
stimulate the development of sophisticated analytical techniques, such as
compactness methods, variational inequalities and critical point theory.
They also promote significant connections between functional and convex
analysis, partial differential equations, geometry, as well as other areas
of mathematics.

Although general approaches for studying them exist, such as the celebrated
catastrophe theory, nonlinear variational problems are generally very
difficult to analyze rigorously. In this paper, we propose a general method
to study variational problems of the form $\sup \mathbb{F}\left( K\right) $,
where $\mathbb{F}:K\rightarrow \mathbb{R}$ is a nonlinear function defined
on a compact convex Hausdorff space $K$. This is done in the scope of convex
analysis. Indeed, given two real normed spaces $\mathcal{X}_{\pm }$, the
nonlinear functionals $\mathbb{F}$ we consider are of the form 
\begin{equation*}
\mathbb{F}\doteq f-g_{-}\circ \tau _{-}+g_{+}\circ \tau _{+}\ ,
\end{equation*}%
where $f:K\rightarrow \mathbb{\{-\infty \}\cup R}$ is an upper
semicontinuous concave function, $\tau _{\pm }:K\rightarrow \mathcal{X}_{\pm
}$ are two continuous affine transformations and $g_{\pm }:\mathcal{X}_{\pm
}\rightarrow \mathbb{R}$ are two lower semicontinuous and convex functions
whose Legendre-Fenchel transforms, $g_{\pm }^{\ast }$, have a full domain%
\footnote{%
A slightly more general assumption for $g_{-}$ can be used. See Condition B2
of Section \ref{sect conc conv Bogo}.} and grow sufficiently fast at large
arguments. This situation is very general, as discussed in\ Remarks \ref%
{Remark-condition1} and \ref{Remark-condition2}. For example, any $C^{1}$%
-functions on a compact subset of $\mathbb{R}^{N}$ ($N\in \mathbb{N}$) can
be represented as the difference of convex and continuous functions. See
Remark \ref{Remark-condition2}. In applications, $f$, $-g_{-}\circ \tau $
and \ $g_{+}\circ \tau _{+}$ typically refer to an entropy, a nonlinear
attractive\ and a nonlinear repulsive interaction term, respectively.

Note that even if the function $f$ is a priori only concave (rather than
affine), our primary focus is on the nonlinearity introduced by the
functions $g_{\pm }$. In fact, one might wonder why we do not include $%
-g_{-}\circ \tau _{-}$ in the term $f$, given that they are both upper
semicontinuous and concave. The advantage of splitting $\mathbb{F}$ into
three parts, $f$, $-g_{-}\circ \tau _{-}$ and $g_{+}\circ \tau _{+}$ becomes
apparent when the Legendre-Fenchel transform of $g_{\pm }^{\ast }$ and
certain associated linear variational problems can be easily controlled.
This is the raison d'\^{e}tre of Bogoliubov linearizations. We demonstrate
the effectiveness of this method by applying it to the \emph{nonlinear
thermodynamic formalism}. However, our new approach has a much wider range
of applications.

In fact, it originates from a study of quantum lattice systems at
equilibrium, which was published in the \textit{Memoirs of the AMS} in 2013 
\cite{BruPedra2} and has been extensively developed here. We prove in
Theorem \ref{Proposition importante bogoluibov01 copy(2)} that 
\begin{equation}
\sup \mathbb{F}\left( K\right) =\sup_{y_{+}\in \mathcal{X}_{+}^{\ast
}}\inf_{y_{-}\in \mathcal{X}_{-}^{\ast }}\{\sup \mathcal{G}%
_{y_{+},y_{-}}\left( K\right) +g_{-}^{\ast }\left( y_{-}\right) -g_{+}^{\ast
}\left( y_{+}\right) \}\doteq \mathrm{P}^{\flat }\in \mathbb{R}\text{ },
\label{sdsdsdssdsdsds}
\end{equation}%
where $\mathcal{G}_{y_{+},y_{-}}\doteq f-y_{-}\circ \tau _{-}+y_{+}\circ
\tau _{+}$, $y_{\pm }\in \mathcal{X}_{\pm }^{\ast }$, which are named the 
\emph{Bogoliubov linearizations} of $\mathbb{F}$ here. Solutions to these
variational problems are studied in detail, see again Theorem \ref%
{Proposition importante bogoluibov01 copy(2)}. For instance, we obtain the
following unexpected result:%
\begin{equation*}
E_{\mathbb{F}}\doteq \left\{ \mu \in K:\exists (\mu _{j})_{j\in J}\subseteq K%
\mathrm{\ }\text{with }\lim_{J}\mu _{j}=\mu \text{ and\ }\lim_{J}\mathbb{F}%
(\mu _{j})=\mathrm{P}^{\flat }\right\} =\left\{ \mu \in K:\mathbb{F}\left(
\mu \right) =\mathrm{P}^{\flat }\right\} \ ,
\end{equation*}%
keeping in mind that $\mathbb{F}$ is\emph{\ generally not} upper
semicontinuous. Additionally, this set is compact and its elements are \emph{%
self-consistent} solutions to associated linear variational problems.
Indeed, as can be seen from Equation (\ref{sdsdsdssdsdsds}), the nonlinear%
\footnote{%
The nonlinearity here refers to the functions $g_{\pm }$.} problem can be
studied through the family of \emph{linear}\footnote{%
The function $f$ is not necessarily linear. It is the functions $g_{\pm }$
that are linearized, being replaced by $y_{\pm }\circ \tau _{\pm }$ in $%
\mathcal{G}_{y_{+},y_{-}}$. We nonetheless refer to $\sup \mathcal{G}%
_{y_{+},y_{-}}\left( K\right) $ as a \textquotedblleft
linear\textquotedblright\ variational problem, because nonlinearity is
represented here by $g_{\pm }$; furthermore, in statistical mechanics the
function $f$ represents an entropy functional (typically the entropy per
unit volume of space-invariant states) that is usually both concave and
convex in the thermodynamic limit, and is therefore affine. See, for
example, the application of this method to nonlinear thermodynamic
formalism, as explained in this paper.} variational problems%
\begin{equation}
P_{\mathrm{L}}\left( y_{+},y_{-}\right) \doteq \sup \mathcal{G}%
_{y_{+},y_{-}}\left( K\right) \text{ },\qquad y_{\pm }\in \mathcal{X}_{\pm
}^{\ast }\text{ }.  \label{linear case}
\end{equation}%
This is achieved by determining beforehand the solutions $x_{\pm }\in 
\mathcal{X}_{\pm }^{\ast }$ to the $\sup_{y_{+}\in \mathcal{X}_{+}^{\ast
}}\inf_{y_{-}\in \mathcal{X}_{-}^{\ast }}$ in (\ref{sdsdsdssdsdsds}). There
is in particular a \emph{canonical} two-person zero-sum game associated with
the maximization of $\mathbb{F}$, the payoff function of which is given by
the \emph{nonlinear} approximating pressure%
\begin{equation*}
P_{\mathrm{NL}}\left( y_{+},y_{-}\right) \doteq P_{\mathrm{L}}\left(
y_{+},y_{-}\right) +g_{-}^{\ast }\left( y_{-}\right) -g_{+}^{\ast }\left(
y_{+}\right) \ .
\end{equation*}%
This game is studied in detail in Section \ref{Decision} because it itself
has interesting applications, as shown in Section \ref{Section game +
measusre} in the context of the thermodynamic formalism or in \cite%
{BruPedra2,Kac} for lattice quantum systems.

In Section \ref{transport}, we explain how generalized nonlinear equilibrium
states 
\begin{equation*}
\mu \in G_{\mathbb{F}}\doteq \overline{\mathrm{co}}(E_{\mathbb{F}})\subseteq
K
\end{equation*}%
can lead to distributions (at equilibrium) of order parameters\footnote{%
The term \textquotedblleft order parameter\textquotedblright\ originates in
physics and refers to a quantity that measures the degree of order in a
system, distinguishing different phases of matter. It is expected to exhibit
different behavior at phase transitions. In the context of mean-field
theory, it refers to the $c$-number substitution that arises from the
mean-field approximation. A typical example of an order parameter is the
magnetization density in spin systems.} for the system under consideration,
which is an important concept in physics, related to phase transitions. In
this context, we show that, for any given order parameter distribution at
equilibrium, the nonlinear pressure $\sup \mathbb{F}(K)$ can be exactly
recovered from a \emph{Monge-Kantorovich (transportation) problem} \cite[%
page 10]{Vilani} associated with these distributions, the cost function of
which is nothing but the nonlinear approximating pressure $P_{\mathrm{NL}}$
defining the above thermodynamic game. This observation is very useful
because of the celebrated Kantorovich dual problem \cite{Vilani}, for which
various high-performance numerical tools are available. Moreover, in usual
applications, the distributions at equilibrium of order parameters are
simpler mathematical objects than general equilibrium states. In fact,
typically, the former are probability distributions in some compact subset
of a finite-dimensional space, whereas the latter are positive functionals
on some infinite-dimensional $C^{\ast }$-algebra.

Optimal transport is a fundamental theory that has numerous applications in
fields such as economics, machine learning, fluid dynamics, signal
processing, mechanics, etc. Therefore, in Theorem \ref{prop order copy(1)}
we prove the Kantorovich dual problem for $\sup \mathbb{F}(K)$, thereby
establishing a link between nonlinear equilibria and this theory. To our
knowledge, this is a new result. For more details we recommend Section \ref%
{transport}, which could even be a starting point for further developments
in optimal transport theory. Indeed, our aim is not to present here a
detailed analysis of the associated optimal transport problem, but rather to
build a fundamental connection between the maximization of real-valued
functions $\mathbb{F}$ on compact convex spaces and the dual Kantorovich
problem, thereby opening the door to entirely new mathematical and numerical
developments.

Moreover, as the frontiers of quantum computing and high-performance
processing continue to expand, the ability to effectively linearize complex
problems is becoming increasingly important. Our approach, based on
Bogoliubov linearizations, offers an alternative to conventional tools used
nowadays for dynamical systems, such as the Carleman or Koopman methods and
perturbation / Taylor expansions \cite{Colbrook,Sanz}. It thus unlocks a new
framework for solving nonlinear variational problems of paramount importance
at present.

\subsection{The nonlinear thermodynamic formalism of dynamical systems}

In this paper the paradigmatic example of application is the \emph{nonlinear
thermodynamic formalism}, a relatively new area of mathematical research: It
appears in a series of seminal papers \cite{TF1,TF2,TF3,TF4,TF5}, which
introduce a rigorous approach to studying several questions in mean-field
theory of classical statistical mechanics, in particular in relation to the
Curie-Weiss-Potts models.

The linear thermodynamic formalism, {in the Bowen-Ruelle-Sinai sense (see,
e.g., \cite{TF-Linear9})}, is a comparatively old but still extremely active
area of mathematics. See \cite%
{TF-Linear0,TF-Linear1,TF-Linear2,TF-Linear3,TF-Linear4,TF-Linear5,TF-Linear6,TF-Linear8,TF-Linear9,Wang,Climenhaga}%
. As explained in \cite{TF-Linear9}, it has numerous ramifications towards
other fields of mathematics, such as the ergodic transport, the fractals and
multifractal formalism, discrete-time linear dynamics, $C^{\ast }$-algebras,
Haar systems, groupoids and quasi-invariant probabilities for cocycles,
mathematical statistics, etc.

In 2023, Buzzi, Kloeckner and Leplaideur \cite{TF4} explained how the
nonlinear thermodynamic formalism can be understood from the linear one. Our
aim is the same, but we take a \emph{completely different} approach with our
general theory of Bogoliubov linearizations. Indeed, similar to Buzzi%
%TCIMACRO{\TeXButton{\-}{\-}}%
%BeginExpansion
\-%
%EndExpansion
-Kloeckner-Leplaideur's method \cite{TF4}, the theory of Bogoliubov
linearizations transforms the nonlinear problem into a family of linear
ones. As explained above, this is achieved at the expense of some
self-consistency condition, which can, however, be effectively studied in
the linear thermodynamic formalism, as already demonstrated in \cite%
{Bru-Pedra-Lopes1} for some explicit examples.

This corresponds to Theorem \ref{Theorem-main1}, which allows us to
determine equilibrium measures via the so-called thermodynamic game
discussed above. In fact, this game is shown in Theorem \ref{Theorem-main1
copy(1)} to have a direct and general interpretation in terms of (nonlinear)
equilibrium measures. As explained in Section \ref{Linear thermo}, the case
of H\"{o}lder potentials in linear thermodynamic formalism remains a very
general situation and has excellent mathematical properties. For example, in
this case (linear) equilibrium measures are always unique and ergodic (see 
\cite{LMMS,ACR}). We thus apply Theorem \ref{Theorem-main1} to a nonlinear
version of this situation and obtain a much stronger result in Corollary \ref%
{Theorem-main1 copy(2)} (see also Corollary \ref{Theorem-main1 copy(3)}).

We study the nonlinear thermodynamic formalism here via its linear version,
as Buzzi, Kloeckner and Leplaideur did in \cite{TF4}; however, in our view,
our method is simpler and more transparent. This allows us to reach the same
conclusions under much more general assumptions, as well as derive new
results, such as Theorems \ref{Theorem-main1} and \ref{Theorem-main1 copy(1)}%
. See Section \ref{Buzzy-Kloeckner-Leplaideur's Approach} which compares
Buzzi%
%TCIMACRO{\TeXButton{\-}{\-}}%
%BeginExpansion
\-%
%EndExpansion
-Kloeckner-Leplaideur's approach with ours. See also Conditions TF1--TF3 of
Section \ref{Bogoliubov linearizations TF}.

Last but not least, with regard to the mathematical scope, it should be
noted that the (symbolic) space considered in relation to the thermodynamic
formalism is the set of all (infinite) sequences in an alphabet $\Omega $,
which is not necessarily finite but is only assumed to be compact with
respect to some metric. Despite the great importance of this more general
case in applications -- such as in the classical $XY$ model -- the
mathematical literature remains sparse. This has made it necessary to
provide proofs for several fundamental results which, although anticipated,
are essential to our analysis. In particular, we define the entropy via
(transfer) Ruelle operators and make the connection with a thermodynamic
limit of finite volume entropies, thanks to \cite{ACR}. We also introduce
new concepts, such as $\Delta $-functionals. We believe that, in addition to
establishing a very general mathematical foundation for the Bogoliubov
linearization method, which stems from a long tradition in statistical
physics, the present work provides a solid basis for developing the
nonlinear version of the thermodynamic formalism for compact metric
alphabets.

\subsection{Structure of the paper}

To summarize, Theorems \ref{Theorem-main1}, \ref{Theorem-main1 copy(1)}, \ref%
{Proposition importante bogoluibov01 copy(2)} and \ref{prop order copy(1)},
as well as Corollaries \ref{Theorem-main1 copy(2)} and \ref{Theorem-main1
copy(3)}, are the main results of the paper, which is divided into two main
parts:

\begin{itemize}
\item Section \ref{Setup of the Problem} explains the thermodynamic
formalism, including Buzzi-Kloeckner%
%TCIMACRO{\TeXButton{\-}{\-}}%
%BeginExpansion
\-%
%EndExpansion
-Leplaideur's seminal\ approach \cite{TF4} to its nonlinear version, in
Section \ref{Buzzy-Kloeckner-Leplaideur's Approach}. We begin with this
example to directly illustrate a concrete application of our abstract theory
of Bogoliubov linearizations. This example is indeed pedagogical and even
paradigmatic, with a very general scope of application.

\item Section \ref{Abstract Theory} presents our general abstract theory of
Bogoliubov linearizations, which could prove very useful in many other
contexts. Note that in Section \ref{Sect Bogoliubov app tech1} we provide a
more detailed explanation of why it was named after Bogoliubov.
\end{itemize}

\noindent Each of the two main sections ends with a more technical
subsection and an appendix that compile mathematical assertions used in this
work. The results presented in the two appendices are either standard,
largely unknown (e.g., Theorem \ref{theorem sympa}) or new. This makes the
paper self-contained and accessible to a wide audience.

\section{Paradigmatic Example: Nonlinear Thermodynamic Formalism of
Dynamical Systems\label{thermo form}\label{Setup of the Problem}}

The thermodynamic formalism is an old subject in strong connection with
dynamical systems and ergodic theory. See, e.g., Ruelle's book \cite%
{TF-Linear5} which first systematized this theory in 1978. It remains an
extremely active and important area of mathematics. See, for example, the
book \cite{CIRM1} that emerged from a semester in 2019 at the Centre
International de Rencontres Math\'{e}matiques on \textquotedblleft
Thermodynamic Formalism: Applications to Probability, Geometry and
Fractals\textquotedblright . It has indeed numerous ramifications across
other areas of mathematics, including: the ergodic transport, the fractal
and multifractal formalism, discrete-time linear dynamics, $C^{\ast }$%
-algebras, Haar systems, groupoids, quasi-invariant probabilities for
cocycles, mathematical statistics, etc. For example, it has been known since
the 1980s that the Hausdorff dimension of the support of a probability
measure can, in certain cases, be determined through the derivative of a
(linear) topological pressure. See, e.g., \cite{Lopez-fractal}. For various
references and examples of works connected to the thermodynamic formalism,
see \cite%
{TF-Linear0,TF-Linear1,TF-Linear2,TF-Linear3,TF-Linear4,TF-Linear5,TF-Linear6,TF-Linear8,TF-Linear9,Wang,Climenhaga}%
.

In recent years, approaches to study questions in mean-field theory from the
ergodic viewpoint were introduced, in particular for the Curie-Weiss type
models. These results constitute the foundations of a new area called the 
\emph{nonlinear thermodynamic formalism} \cite{TF1,TF2,TF3,TF4,TF5}. As
explained in \cite{TF4}, the nonlinear version of the thermodynamic
formalism has been also considered in relation to the multifractal analysis 
\cite{Climenhaga}.

As it turns out, there is an intricate relationship between the\ nonlinear
formalism and the linear one. For compact metric alphabets, in the case of
convex and concave nonlinearities, or a combination of both, we present a
new method inspired by \textquotedblleft quantum
mathematics\textquotedblright\ to solve this problem effectively. In fact,
the link between both formalisms is made via a min-max principle which leads
to the concept of thermodynamic games, the cooperative equilibria of which
provide a complete classification of the nonlinear equilibrium measures of a
given nonlinear (topological) pressure as equilibrium measures of
self-consistent (effective) linear pressures.

The general (abstract) version of this method is postponed to Section \ref%
{Abstract Theory}. To motivate it, we first present its application to the
nonlinear version of the thermodynamic formalism as paradigmatic example.
Although the nonlinear thermodynamic formalism is a known subject, our
approach introduces new developments in the form of a novel yet natural
variational problem on probability measures (Section \ref{Section game +
measusre}). Even in the case of finite alphabets our results cover a more
general class than the current literature in nonlinear thermodynamic
formalism \cite{TF1,TF2,TF3,TF4,TF5}.

To make the paper self-contained, we begin with the mathematical framework
of the symbolic dynamical systems formalism.

\subsection{Mathematical framework\label{math framework}}

\textbf{Alphabet.} Let $(\Omega ,d)$ be any compact metric space. It
represents a general (possibly infinite) alphabet. For simplicity and
without loss of generality, we assume that 
\begin{equation}
\max d(\Omega \times \Omega )\doteq \max \left\{ d\left( \omega _{1},\omega
_{2}\right) :\omega _{1},\omega _{2}\in \Omega \right\} =1\ ,
\label{max metric}
\end{equation}%
i.e., $\Omega $ has normalized diameter. In addition, we set an arbitrary a
priori probability measure $\mathrm{m}$, which is fixed once and for all, on
the Borel $\sigma $-algebra of the alphabet $(\Omega ,d)$. Recall that such
a measure is always regular, for the alphabet is a metric space. Below, it
is important to assume that $\mathrm{m}$ has always support equal to $\Omega 
$.\medskip

\noindent \textbf{Infinite strings.} From the compact metric space $\Omega $%
\ we construct the compact\footnote{%
It is compact, by Tychonoff's theorem \cite[Section A.3]{Rudin}.}
topological space $\Sigma \doteq \Omega ^{\mathbb{N}}$ as the set of
infinite strings in the alphabet $\Omega $ endowed with the product
topology. A standard basis of the topology of $\Sigma $ is given by the
collection of \textit{cylinders}, defined for any $N\in \mathbb{N}$ and
finite sequence $O_{1},\ldots ,O_{N}$ of open sets of $(\Omega ,d)$ by 
\begin{equation*}
\left[ O_{1},\ldots ,O_{N}\right] \doteq \left\{ \sigma =\left( \omega
_{n}\right) _{n\in \mathbb{N}}\in \Sigma :\omega _{1}\in O_{1},\ldots
,\omega _{N}\in O_{N}\right\} \subseteq \Sigma \text{ }.
\end{equation*}%
Fix once and for all throughout the paper a parameter $\eta \in (0,1)$.
Then, a metric generating the topology of $\Sigma $ can be defined by 
\begin{equation}
d_{\eta }\left( \sigma ,\sigma ^{\prime }\right) \doteq \left( 1-\eta
\right) \sum_{n\in \mathbb{N}}\eta ^{n-1}d\left( \omega _{n},\omega
_{n}^{\prime }\right) \ ,\text{\qquad }\sigma =\left( \omega _{n}\right)
_{n\in \mathbb{N}},\sigma ^{\prime }=\left( \omega _{n}^{\prime }\right)
_{n\in \mathbb{N}}\in \Sigma \ .  \label{metric sigma}
\end{equation}%
Note that the metric is defined to satisfy the normalization 
\begin{equation*}
\max d_{\eta }\left( \Sigma \times \Sigma \right) \doteq \max \left\{
d_{\eta }\left( \sigma _{1},\sigma _{2}\right) :\sigma _{1},\sigma _{2}\in
\Sigma \right\} =1,
\end{equation*}%
given a fixed parameter $\eta \in (0,1)$. Observe that \cite[Proposition 1.2]%
{TF-Linear1} shows that results for equilibrium probabilities on $\Omega ^{%
\mathbb{Z}}$ can be derived from the ones obtained in $\Omega ^{\mathbb{N}}$%
, which are studied here. \medskip

\noindent \textbf{Real-valued functions on strings.} $C(\Sigma )\equiv
C(\Sigma ;\mathbb{R})$ is the Banach space of all real-valued continuous
functions endowed with the supremum norm:%
\begin{equation*}
\left\Vert \varphi \right\Vert _{\infty }\equiv \left\Vert \varphi
\right\Vert _{\infty ,\Sigma }\doteq \sup \left\vert \varphi \left( \Sigma
\right) \right\vert \doteq \sup \left\{ \left\vert \varphi \left( \sigma
\right) \right\vert :\sigma \in \Sigma \right\} \ ,\qquad \varphi \in
C\left( \Sigma \right) \ .
\end{equation*}%
$C_{\mathrm{f}}(\Sigma )\equiv C_{\mathrm{f}}(\Sigma ;\mathbb{R})$ denotes
the subspace of all continuous functions $\Sigma \rightarrow \mathbb{R}$
that are supported in some cylinder $[O_{1},\ldots ,O_{N}]$. It is a dense
subspace of $C(\Sigma )$.

For any $\alpha \in (0,1]$, 
\begin{equation*}
C^{\alpha }\left( \Sigma \right) \equiv C^{\alpha }(\Sigma ;\mathbb{R}%
)\doteq \left\{ \varphi \in C\left( \Sigma \right) :\exists C>0\ \text{so
that }\left\vert \varphi \left( \sigma \right) -\varphi \left( \sigma
^{\prime }\right) \right\vert \leq Cd_{\eta }\left( \sigma ,\sigma ^{\prime
}\right) ^{\alpha }\text{ for }\sigma ,\sigma ^{\prime }\in \Sigma \right\}
\end{equation*}%
is the space of all real-valued $\alpha $-H\"{o}lder-continuous functions
with respect to the metric $d_{\eta }$. We use in this space the so-called H%
\"{o}lder norm%
\begin{equation}
\left\Vert \varphi \right\Vert _{\alpha }\doteq \left\Vert \varphi
\right\Vert _{\infty }+\sup_{\sigma ,\sigma ^{\prime }\in \Sigma ,\text{ }%
\sigma \neq \sigma ^{\prime }}\left\{ d_{\eta }\left( \sigma ,\sigma
^{\prime }\right) ^{-\alpha }\left\vert \varphi \left( \sigma \right)
-\varphi \left( \sigma ^{\prime }\right) \right\vert \right\} \text{ },\text{%
\qquad }\varphi \in C^{\alpha }\left( \Sigma \right) \ ,  \label{norm1}
\end{equation}%
with respect to which $C^{\alpha }(\Sigma )$ is a Banach space. Note also
that $C_{\mathrm{f}}(\Sigma )\subseteq C^{\alpha }(\Sigma )$. In particular, 
$C^{\alpha }(\Sigma )$ is a dense subspace of $C(\Sigma )$.

Some of our results require the use of the unit closed ball 
\begin{equation}
S\doteq \left\{ \varphi \in C\left( \Sigma \right) :\left\Vert \varphi
\right\Vert _{\infty }\leq 1\right\} \subseteq C\left( \Sigma \right)
\label{norm1bis}
\end{equation}%
of the Banach space $C(\Sigma )$ of continuous functions. $S$ is endowed
with the uniform metric%
\begin{equation*}
d_{S}(\varphi ,\varphi ^{\prime })\doteq \left\Vert \varphi -\varphi
^{\prime }\right\Vert _{\infty }\ ,\text{\qquad }\varphi ,\varphi ^{\prime
}\in S\text{ }.
\end{equation*}%
Note that $S$ is a separable metric space, $C(\Sigma )$ being a separable
normed space. Then, $\mathcal{M}(S)\equiv \mathcal{M}(S;\mathbb{R})$ denotes
the Banach space of bounded, real-valued Borel-measurable functions $%
S\rightarrow \mathbb{R}$ with the supremum norm:%
\begin{equation*}
\left\Vert f\right\Vert _{\infty }\equiv \left\Vert f\right\Vert _{\infty
,S}\doteq \sup \left\vert f\left( S\right) \right\vert \doteq \sup \left\{
\left\vert f\left( \varphi \right) \right\vert :\varphi \in S\right\} \
,\qquad f\in \mathcal{M}(S)\ .
\end{equation*}

\noindent \textbf{Shift-invariant probability measures on strings.} The
so-called \textit{shift mapping} $T:\Sigma \rightarrow \Sigma $ is defined by%
\begin{equation}
T\left( \sigma \right) _{n}\doteq \omega _{n+1}\ ,\qquad n\in \mathbb{N},\
\sigma =\left( \omega _{n}\right) _{n\in \mathbb{N}}\in \Sigma \ .
\label{Transformation definition}
\end{equation}%
Clearly, $T$ is continuous and in particular Borel measurable. Then, a
probability measure $\mu $ on the Borel $\sigma $-algebra of $\Sigma $ is,
by definition, \textit{$T$-invariant}\ when $T_{\ast }(\mu )=\mu $, where $%
T_{\ast }(\mu )$ stands for the pushforward of the measure $\mu $ with
respect to $T$, i.e., $T_{\ast }\left( \mu \right) (B)\doteq \mu (T^{-1}(B))$
for any Borel set $B\subseteq \Sigma $.

As $\Sigma $ is a metrizable compact space, by the Riesz-Markov-Kakutani
representation theorem, one can identify the space of Borel measures of
finite variation on $\Sigma $ with the (topological) dual space $C(\Sigma
)^{\ast }$ of the Banach space $C(\Sigma )$. See, e.g., \cite[14.15 Corollary%
]{AliBor}. $\mathcal{P}\subseteq C(\Sigma )^{\ast }$ denotes the convex
space of all Borel probability measures on $\Sigma $ and the convex set of $%
T $-invariant ones is denoted 
\begin{equation}
\mathcal{P}\left( T\right) \doteq \left\{ \mu \in \mathcal{P}:T_{\ast }(\mu
)=\mu \right\} \ .  \label{set T-inv}
\end{equation}%
It is a weak$^{\ast }$-closed subset of $\mathcal{P}$ and, since $\mathcal{P}
$ is weak$^{\ast }$-compact, $\mathcal{P}\left( T\right) $ is also weak$%
^{\ast }$-compact. Notice that the weak$^{\ast }$ topology of $\mathcal{P}$
and, consequently, that of $\mathcal{P}(T)$, is metrizable, because the
Banach space $C(\Sigma )$ is separable, as $\Sigma $ is a metrizable compact
space.

Thanks to the Krein-Milman theorem \cite[Theorem 3.23]{Rudin}, $\mathcal{P}%
(T)$ is the weak$^{\ast }$-closure of the convex hull of the (nonempty) set $%
\mathcal{P}_{\mathrm{erg}}(T)$ of its extreme points, i.e., 
\begin{equation*}
\mathcal{P}\left( T\right) =\overline{\mathrm{co}}\ \mathcal{P}_{\mathrm{erg}%
}\left( T\right) \ .
\end{equation*}%
By the metrizability of $\mathcal{P}(T)$ and Lemma \ref{lemma extr gd}, $%
\mathcal{P}_{\mathrm{erg}}(T)$ is a Borel set with respect to the weak$%
^{\ast }$ topology. Extreme $T$-invariant Borel probability measures on $%
\Sigma $, i.e., the elements of $\mathcal{P}_{\mathrm{erg}}(T)$, are the 
\textit{ergodic} measures, because of Proposition \ref{prop erg extreme}.
See also \cite{Walters} for the case of a finite alphabet. In addition, by
Proposition \ref{density of ergodic measures}, the set $\mathcal{P}_{\mathrm{%
erg}}(T)$ of ergodic measures is a weak$^{\ast }$-dense subset of $\mathcal{P%
}(T)$:%
\begin{equation}
\mathcal{P}\left( T\right) =\overline{\mathcal{P}_{\mathrm{erg}}\left(
T\right) }\ .  \label{eq erg dense}
\end{equation}

\subsection{Entropy for compact metric alphabets\label{sect entropy}}

In the literature, entropy is usually only defined for finite alphabets $%
\Omega $, but, motivated by possible applications to continuous-spin
systems, following \cite{LMMS} (see also \cite{ACR}), we allow the alphabet $%
\Omega $ to be an arbitrary compact metric space. Entropy is not as well
understood in this broader context and much less literature is available.

To define it, recall that we fix in all the paper an arbitrary a priori
probability measure $\mathrm{m}$ whose support is equal to $\Omega $. It is
used to introduce the (transfer) Ruelle operator $\mathcal{L}_{0}^{\mathrm{m}%
}$, which is the linear operator on $C\left( \Sigma \right) $ defined, for
any $\varphi \in C(\Sigma )$ and $\sigma =(\omega _{n})_{n\in \mathbb{N}}\in
\Sigma $, by 
\begin{equation*}
\mathcal{L}_{0}^{\mathrm{m}}\varphi \left( \sigma \right) \doteq
\int_{\Omega }\varphi \left( \omega _{0}\sigma \right) \mathrm{m}\left( 
\mathrm{d}\omega _{0}\right) \ ,\qquad \omega _{0}\sigma \doteq \left(
\omega _{0},\omega _{1},\ldots \right) \in \Sigma \ .
\end{equation*}%
In the case the alphabet is the finite set $\{1,2,...,k\}$, it is natural to
take the a priori probability measure $\mathrm{m}$ as the normalized
counting measure, see \cite{LMMS} for details.

Let $\mathcal{C}^{+}$ be the cone of positive functions of $C(\Sigma )$,
i.e., the set of continuous functions $\Sigma \rightarrow \mathbb{R}_{0}^{+}$%
. Then, the entropy for $T$-invariant probability measures is defined as
follows:

\begin{definition}[Entropy]
\label{uod}\mbox{ }\newline
The entropy of any $T$-invariant probability measure is, by definition,
equal to 
\begin{equation}
h\left( \mu \right) \equiv h_{\mathrm{m}}\left( \mu \right) \doteq
\inf_{\varphi \in \mathcal{C}^{+}}\left\{ \int_{\Sigma }\ln \left( \frac{%
\mathcal{L}_{0}^{\mathrm{m}}\varphi \left( \sigma \right) }{\varphi \left(
\sigma \right) }\right) \mu \left( \mathrm{d}\sigma \right) \right\} \
,\qquad \mu \in \mathcal{P}\left( T\right) \ .  \label{oiu54}
\end{equation}
\end{definition}

Remark that the infimum in (\ref{oiu54}) is not necessarily attained for
some positive function $\varphi \in \mathcal{C}^{+}$. In addition, the
entropy is a mapping $\mu \mapsto h(\mu )$ from $\mathcal{P}(T)$ to $%
[-\infty ,0]$. For instance, the $T$-invariant measure whose support is a
fixed point for $T$ is equal to $-\infty $, while 
\begin{equation*}
h\left( \mu \right) \leq \int_{\Sigma }\ln \left( \frac{\mathcal{L}_{0}^{%
\mathrm{m}}\mathbf{1}\left( \sigma \right) }{\mathbf{1}\left( \sigma \right) 
}\right) \mu \left( \mathrm{d}\sigma \right) =0
\end{equation*}%
for any $T$-invariant probability measure $\mu \in \mathcal{P}\left(
T\right) $, where $\mathbf{1}\in \mathcal{C}^{+}$ is the (constant) function
equal to $1$.

A useful observation concerns an equivalent definition of the entropy
defined above. To state it, we use a generalization of the Ruelle operator.
For a given $\alpha $-H\"{o}lder-continuous potential $f\in C^{\alpha
}(\Sigma )$ ($\alpha \in (0,1)$) and an a priori measure $\mathrm{m}$ on $%
\Omega $ whose support is equal to $\Omega $, it is defined, for any $%
\varphi \in C^{\alpha }(\Sigma )$, by%
\begin{equation*}
\mathcal{L}_{f}^{\mathrm{m}}(\varphi )\left( \sigma \right) =\int_{\Omega
}\exp (f(\omega _{0}\sigma ))\varphi (\omega _{0}\sigma )\,\mathrm{m}\left( 
\mathrm{d}\omega _{0}\right) \ ,\qquad \omega _{0}\sigma \doteq \left(
\omega _{0},\omega _{1},\ldots \right) \in \Sigma \ .
\end{equation*}%
This Ruelle operator is a linear operator on $C^{\alpha }\left( \Sigma
\right) $. A normalized potential $f$ is, by definition, a potential $f\in
C^{\alpha }(\Sigma )$ ($\alpha \in (0,1)$) satisfying $\mathcal{L}_{f}^{%
\mathrm{m}}(\mathbf{1})(\sigma )=1$ for all $\sigma \in \Sigma $. Then, we
have the following fact:\ 

\begin{proposition}[Alternative definition of the entropy]
\label{kul copy(1)}\mbox{ }\newline
For the entropy of Definition \ref{uod}, the following equality holds: 
\begin{equation}
h\left( \mu \right) =\inf_{f\in C^{\alpha }(\Sigma )}\left\{ \log \lambda
_{f}-\int_{\Sigma }f\left( \sigma \right) \mu \left( \mathrm{d}\left( \sigma
\right) \right) \right\} \ ,  \label{vam1}
\end{equation}%
where $\lambda _{f}$ is the main (i.e., the largest) eigenvalue of the
Ruelle operator $\mathcal{L}_{f}^{\mathrm{m}}$.
\end{proposition}

\begin{proof}
See in \cite[Remark 7.8]{GKLM} and \cite[Definition 2]{ACR}. Note that the
arguments in \cite[Section 7]{GKLM} refer to a finite alphabet $\Omega $,
but a close analysis of the proofs shows that only the abstract properties
of the Legendre-Fenchel transform are necessary. Therefore, everything can
be done in the same way with a compact metric space $\Omega $ and, in
particular, the first equation of \cite[Section 7.1]{GKLM} applies to the
case the alphabet is a compact metric space $\Omega $.
\end{proof}

By Definition \ref{uod}, the mapping $\mu \mapsto h(\mu )$ from $\mathcal{P}%
(T)$ to $[-\infty ,0]$ is upper semicontinuous in the weak$^{\ast }$
topology, the entropy being defined as the infimum of weak$^{\ast }$%
-continuous functions (indexed by $\varphi \in \mathcal{C}^{+}$). It is
additionally affine, thanks to Proposition \ref{Affinity of the entropy
copy(1)}. That is, the entropy $h$ is both concave and convex. See also \cite%
{Walters} for the case of finite alphabets. This follows from the concept of
specific entropy given in \cite{ACR} for compact metric alphabets, along
with the characterization of the entropy in terms of a thermodynamic limit,
as given by \cite[Theorems 3.1 and 3.4]{ACR}. See Theorem \ref%
{teo-equiv-entrop}. We elaborate on that in Section \ref{Further Study of
the Entropy}.

\begin{remark}[Alternative formulations of entropy]
\mbox{ }\newline
Non-dynamical entropies, such as the Gibbs-Boltzmann (or Shannon) entropy
(density) can also be considered. Recent relevant examples include more
general forms of entropy on the probability space, as presented in \cite%
{LMO,LO}, along with their relationship to the pushforward dynamics (the
level-2 setting). There, the authors define an entropy that generalizes the
one we use for a compact metric alphabet. \cite{TF-Linear8} defines an
abstract pressure and derives a corresponding entropy via the
Legendre-Fenchel transform. In all these variants, the entropy functional
remains concave and upper semicontinuous, ensuring that the results of
Section \ref{Abstract Theory} still apply.
\end{remark}

\subsection{Linear thermodynamic formalism\label{Linear thermo}}

Given an a priori probability measure $\mathrm{m}$ whose support is equal to 
$\Omega $, the so-called \emph{topological pressure} (pressure for short) $%
\mathfrak{P}_{L}(\varphi )$ of a continuous function $\varphi \in C(\Sigma )$
is defined by%
\begin{equation}
\mathfrak{P}_{L}\left( \varphi \right) \doteq \sup_{\mu \in \mathcal{P}%
\left( T\right) }\left\{ h\left( \mu \right) +\varphi \left( \mu \right)
\right\} \ ,  \label{sdsdsdsdfssfg}
\end{equation}%
where 
\begin{equation*}
\varphi \left( \mu \right) \doteq \int_{\Sigma }\varphi \left( \sigma
\right) \mu \left( \mathrm{d}\sigma \right) \ ,\qquad \mu \in \mathcal{P}%
\left( T\right) \ .
\end{equation*}%
In this context, the continuous function $\varphi \in C(\Sigma )$ is called 
\emph{potential}. This pressure is nothing else than the Legendre-Fenchel
transform of minus the entropy and, conversely, the entropy is minus the
Legendre-Fenchel transform of the pressure (see \cite[Theorem 9.12]{Walters}
for the case of a finite alphabet):

\begin{proposition}[Pressure versus entropy]
\label{kul}\mbox{ }\newline
Considering the entropy of Definition \ref{uod}, one has 
\begin{equation*}
h\left( \mu \right) =\inf_{\varphi \in C(\Sigma )}\left\{ \mathfrak{P}%
_{L}\left( \varphi \right) -\varphi \left( \mu \right) \right\} \ ,\qquad
\mu \in \mathcal{P}\left( T\right) \ .
\end{equation*}
\end{proposition}

\begin{proof}
By Definition \ref{uod}, $h$ is concave and weak$^{\ast }$-upper
semicontinuous, being the infimum of weak$^{\ast }$ continuous functions
indexed by $\varphi \in \mathcal{C}^{+}$. So, the assertion is nothing else
than Equation (\ref{Biconjugate definition}) for the concave and weak$^{\ast
}$-upper semicontinuous function $g$ on the dual space $C(\Sigma )^{\ast }$
defined by $g=h$ on $\mathcal{P}(T)$ and $-\infty $ otherwise.
\end{proof}

In the thermodynamic formalism (and statistical mechanics of one-dimensional
lattices) one is interested in $T$-invariant probabilities $\mu _{\varphi }$
giving the exact pressure $\mathfrak{P}_{L}\left( \varphi \right) $, i.e.,
satisfying the equality 
\begin{equation}
P_{L}\left( \mu _{\varphi }\right) =\sup P_{L}\left( \mathcal{P}\left(
T\right) \right) \text{ },  \label{tryp}
\end{equation}%
where%
\begin{equation*}
P_{L}\left( \mu \right) \doteq h\left( \mu \right) +\varphi \left( \mu
\right) \ ,\text{\qquad }\mu \in \mathcal{P}\left( T\right) \text{ }.
\end{equation*}%
(Remark that the potential $\varphi $ considered here corresponds to minus
the Hamiltonian in statistical physics.) Such probabilities refer to the
following notion:

\begin{definition}[Equilibrium measures -- linear case]
\label{Equilibrium measures def}\mbox{ }\newline
Given a continuous function $\varphi \in C(\Sigma )$, a probability measure $%
\mu _{\varphi }$ satisfying (\ref{tryp}) is called a linear equilibrium
measure for $\varphi $.
\end{definition}

The set 
\begin{equation}
E_{P_{L}}\doteq \left\{ \mu \in \mathcal{P}\left( T\right) :P_{L}\left( \mu
\right) =\sup P_{L}\left( \mathcal{P}\left( T\right) \right) \right\}
\label{dfgdfgdfg}
\end{equation}%
of linear equilibrium measures is nonempty, weak$^{\ast }$-compact and
convex, because the mapping $\mu \mapsto \varphi \left( \mu \right) +h\left(
\mu \right) $ from $\mathcal{P}(T)$ to $\mathbb{R}$ is upper weak$^{\ast }$%
-semicontinuous and affine, thanks to Proposition \ref{Affinity of the
entropy copy(1)}.

When $\varphi $ is of H\"{o}lder class, i.e., $\varphi \in C^{\alpha
}(\Sigma )$ for some $\alpha \in (0,1]$, the linear equilibrium measure $\mu
_{\varphi }$ is unique and ergodic \cite{LMMS,ACR}. The set of all linear
equilibrium measures $\mu _{\varphi }$ for all possible H\"{o}lder
potentials $\varphi $ is even weak$^{\ast }$-dense in $\mathcal{P}(T)$,
i.e., 
\begin{equation*}
\overline{\left\{ \mu _{\varphi }:\exists \alpha \in (0,1]\ \text{so that}\
\varphi \in C^{\alpha }(\Sigma )\right\} }=\mathcal{P}\left( T\right) \ .
\end{equation*}%
(In particular one has that the ergodic probabilities are dense in $\mathcal{%
P}\left( T\right) $, see (\ref{eq erg dense}).) This follows from \cite[%
Corollary 7.14]{GKLM}. It is easy to see that the proof of this result given
in \cite{LMMS,ACR} is general enough to include the case where the alphabet
is a compact metric space.

\subsection{Nonlinear thermodynamic formalism}

A nonlinear version of the thermodynamic formalism is relatively recent. It
appeared in a series of seminal papers \cite{TF1,TF2,TF3,TF4,TF5}, which
contribute a rigorous approach to studying mean-field theory in classical
statistical mechanics (cf. the Curie-Weiss-Potts models).

In previous works, the nonlinearity is encoded by a continuous function $F:%
\mathbb{R}^{N}\rightarrow \mathbb{R}$ ($N\in \mathbb{N}$), the arguments of
which are the expectation values $\mu \left( \varphi _{1}\right) ,\ldots
,\mu \left( \varphi _{N}\right) $ of $N$ fixed $\alpha $-H\"{o}%
lder-continuous functions $\varphi _{1},\ldots ,\varphi _{N}\in C^{\alpha
}(\Sigma )$ with respect to a probability measure $\mu \in \mathcal{P}(T)$.
This refers to Example \ref{example F II} presented below. Here we consider
a significant generalization of what has been done in the literature on
nonlinear thermodynamic formalism by considering nonlinearities that,
instead of depending on a finite vector $\varphi _{1},\ldots ,\varphi
_{N}\in C^{\alpha }(\Sigma )$ of $\alpha $-H\"{o}lder-continuous functions,
now depend on a continuous function on the space $C^{\alpha }(\Sigma )$ of $%
\alpha $-H\"{o}lder-continuous functions. This allows us, for instance, to
include Example \ref{example F I}, which is the classical analogue of the
quantum lattice systems considered in \cite{BruPedra2}.

Intuitively, this allows us to use infinite, possibly uncountable, families $%
(\varphi _{\alpha })_{\alpha \in I}\in C^{\alpha }(\Sigma )$ of H\"{o}lder
potentials to formally define the nonlinear part as $F((\mu (\varphi
_{\alpha }))_{\alpha \in I})$ for any probability measure $\mu \in \mathcal{P%
}(T)$. We restrict ourselves to uniformly bounded families $(\varphi
_{\alpha })_{\alpha \in I}$ to avoid innocuous technical complications and
discussions. Since the function $F$ will be completely arbitrary, we can
take $(\varphi _{\alpha })_{\alpha \in I}$ in the unit ball $S\subseteq
C^{\alpha }(\Sigma )$ without loss of generality (for the uniformly bounded
case).

Therefore, for any continuous linear functional $\mu :C(\Sigma )\rightarrow 
\mathbb{R}$, define the continuous bounded function $\mu _{S}\in C(S)\equiv
C(S;\mathbb{R})$ by%
\begin{equation}
\mu _{S}(\varphi )\doteq \mu (\varphi )\ ,\text{\qquad }\varphi \in S\text{ }%
.  \label{eq def muS}
\end{equation}%
In other words, $\mu _{S}$ is a kind of projectivization of the action of $%
\mu $ on $C(\Sigma )$. Then, our aim is to study the variational problem 
\begin{equation}
\mathfrak{P}\left( F\right) \doteq \sup_{\mu \in \mathcal{P}\left( T\right)
}\left\{ h\left( \mu \right) +F\left( \mu _{S}\right) \right\} \text{ },
\label{variational problem}
\end{equation}%
where $F:C(S)\rightarrow \mathbb{R}$ is some function satisfying 
\begin{equation}
\sup_{\mu \in \mathcal{P}\left( T\right) }F\left( \mu _{S}\right) <\infty \ .
\label{sdsdsdsdsdsdsdsd}
\end{equation}%
(For instance, $F$ is bounded on the unit closed ball of $C(S)$. Note that $%
\mu _{S}$ is in this ball for all $\mu \in \mathcal{P}\left( T\right) $.)
Recall that $h(\mu )$ stands for the entropy of the $T$-invariant measure $%
\mu \in \mathcal{P}\left( T\right) $.\ The quantity $\mathfrak{P}(F)$ is
called here the nonlinear pressure of $F$.

As an example, if for some fixed $\alpha $-H\"{o}lder-continuous potential $%
A\in C^{\alpha }(\Sigma )\backslash \{0\}$ we are interested in maximizing
the quantity%
\begin{equation*}
\frac{1}{2}\mu \left( A\right) ^{2}+h\left( \mu \right)
\end{equation*}%
over $T$-invariant measures $\mu \in \mathcal{P}(T)$, then we take 
\begin{equation}
F\left( f\right) =\frac{1}{2}\left( \left\Vert A\right\Vert _{\infty
}f\left( \frac{A}{\left\Vert A\right\Vert _{\infty }}\right) \right)
^{2},\qquad f\in C\left( S\right) .  \label{ssdsdsdsklkl}
\end{equation}%
Moreover, since the metric $d$ and the function $F$ are arbitrary, our
setting also includes the case studied by \cite{TF4}, which is described in
detail in Section \ref{Buzzy-Kloeckner-Leplaideur's Approach}. The most
important cases for applications in statistical mechanics are probably the
following examples:

\begin{example}
\label{example F II}\mbox{ }\newline
Fix $N\in \mathbb{N}$, take any continuous function $F$\ on $\mathbb{R}^{N}$%
, along with (normalized) potentials $\varphi _{1},\ldots ,\varphi _{N}\in S$%
. Then%
\begin{equation*}
F^{(\varphi _{1},\ldots ,\varphi _{N})}\left( f\right) \doteq F\left(
f\left( \varphi _{1}\right) ,\ldots ,f\left( \varphi _{N}\right) \right) \
,\qquad f\in \mathcal{M}(S)\text{ },
\end{equation*}%
defines a continuous function on $\mathcal{M}(S)$ that is bounded on bounded
sets. This includes the following important particular choice: 
\begin{equation*}
F\left( \mu _{S}\right) =\mu \left( \varphi _{1}\right) +\lambda \left( \mu
\left( \varphi _{2}\right) \right) ^{2}\ ,\qquad \mu \in \mathcal{P}\left(
T\right) \ ,
\end{equation*}%
for any continuous functions\footnote{%
A priori we should have $\varphi _{1},\varphi _{2}\in S$, but a simple
change of the definition of the original function $F$ gives $\mu (\varphi
_{1})+\lambda (\mu (\varphi _{2}))^{2}$ for arbitrary $\varphi _{1},\varphi
_{2}\in C(\Sigma )$.} $\varphi _{1},\varphi _{2}\in C(\Sigma )$ and $\lambda
\in \mathbb{R}$. In particular, the linear case $\lambda =0$ is included.
\end{example}

\begin{example}
\label{example F I}\mbox{ }\newline
Take any finite positive Borel measure $\mathfrak{a}$ on $S$, $\lambda \in 
\mathbb{R}$ and some continuous function $\varphi \in S$. Then, 
\begin{equation*}
F_{\pm }^{(\varphi ,\mathfrak{a})}\left( f\right) \doteq \lambda f\left(
\varphi \right) \pm \int_{S}f\left( \varphi ^{\prime }\right) ^{2}\mathfrak{a%
}\left( \mathrm{d}\varphi ^{\prime }\right) \ ,\qquad f\in \mathcal{M}(S)\ ,
\end{equation*}%
defines a continuous real-valued function on $\mathcal{M}(S)$ that is
bounded on bounded sets.
\end{example}

Similar to\ Definition \ref{Equilibrium measures def}, we define equilibrium
measures in the nonlinear thermodynamic formalism as follows:

\begin{definition}[Equilibrium measures -- nonlinear case]
\label{Equilibrium measures def copy(1)}\mbox{ }\newline
A $T$-invariant probability measure maximizing (\ref{variational problem})
is called a nonlinear equilibrium measure.
\end{definition}

We show, among other things, that the nonlinear equilibrium measures
considered in the paper \cite{TF4} fit (under its assumptions) within the
above definition. However, our setting covers a much broader class of
models, which are in line with important examples from mathematical physics.
These are related to a convex nonlinearity $F_{+}$, representing long-range
attractive forces, and a concave one $F_{-}$, associated with long-range
repulsive forces.

Note that the bare existence of a nonlinear equilibrium measure is unclear,
depending on the choice of the function $F$ in (\ref{variational problem}).
Consequently, the existence of nonlinear equilibrium measures, and the
entire scope of our study, require certain minimum conditions on the
nonlinearity $F$. These are the following assumptions:

\begin{condition}
\label{Condition essential}\mbox{ }\newline
\emph{(i) }$F$ is the restriction to $C(S)\subseteq \mathcal{M}(S)$ of the
sum of two functions $F_{\pm }:\mathcal{M}(S)\rightarrow \mathbb{R}$ that
are $\sigma $-normal, meaning here that, for any bounded sequence $%
(f_{n})_{n\in \mathbb{N}}\subseteq \mathcal{M}(S)$ converging point-wise to $%
f\in \mathcal{M}(S)$, one has%
\begin{equation*}
\lim_{n\rightarrow \infty }F_{\pm }\left( f_{n}\right) =F_{\pm }\left(
f\right) \text{ }.
\end{equation*}%
\emph{(ii) }The functions $F_{+}$ and $F_{-}$ are respectively convex and
concave.
\end{condition}

\noindent Under these assumptions, the functions $F_{\pm }$ are necessarily
continuous with respect to the uniform convergence in $\mathcal{M}(S)$ and
they always satisfy (\ref{sdsdsdsdsdsdsdsd}): Observe that $\mathcal{P}(T)$
is weak$^{\ast }$-compact and its weak$^{\ast }$ topology is metrizable, as $%
C(\Omega )$ is separable (with respect to the supremum norm, $\Omega $ being
a compact metric space). In particular it is sequentially compact with
respect to the weak$^{\ast }$ topology. Furthermore, the weak$^{\ast }$
convergence of $\mu $ in $\mathcal{P}(T)$ implies the point-wise convergence
of $\mu _{S}$ in $\mathcal{M}(S)$. Consequently, the set $\{F(\mu _{S}):\mu
\in \mathcal{P}(T)\}\subseteq \mathbb{R}$ is (sequentially) compact and
thus, bounded.

Typically, in statistical mechanics the convex function $F_{+}$ represents
some mean-field attraction, while $F_{-}$ represents a mean-field repulsion.
This physical interpretation is similar to that given in \cite{BruPedra2}
for quantum lattice systems. Here we simply refer to the concave and convex
parts of the nonlinear pressure functional $P$, which is equal in this case
to%
\begin{equation}
P\left( \mu \right) =h\left( \mu \right) +F_{-}\left( \mu _{S}\right)
+F_{+}\left( \mu _{S}\right) \ ,\qquad \mu \in \mathcal{P}\left( T\right) \ .
\label{nonlinear pressure functional}
\end{equation}%
Examples of such a $\sigma $-normal convex (concave) function $F_{+}$ ($%
F_{-} $) for applications in statistical mechanics are given by Examples \ref%
{example F II} and \ref{example F I}\footnote{%
The $\sigma $-normality of $F_{\pm }^{(\varphi ,\mathfrak{a})}$ is a direct
consequence of Lebesgue's dominated convergence theorem.} when the
real-valued function $F_{+}$\ ($F_{-}$) is additionally convex (concave) on $%
\mathbb{R}^{N}$. These situations are included in the following more general
example:

\begin{example}
\label{example F III}\mbox{ }\newline
Let $\mathcal{X}_{\pm }$ be two real normed spaces and $g_{\pm }:\mathcal{X}%
_{\pm }\rightarrow \mathbb{R}$ two continuous convex functions. Take two
bounded linear transformations $\theta _{\pm }:\mathcal{M}(S)\rightarrow 
\mathcal{X}_{\pm }$ that are $\sigma $-normal, as explained in Condition \ref%
{Condition essential} (i). Then define $F_{\pm }\doteq \pm g_{\pm }\circ
\theta _{\pm }$, which are functions that are $\sigma $-normal. $F_{+}$ and $%
F_{-}$ are obviously convex and concave, respectively.
\end{example}

\noindent The above examples $F_{\pm }^{(\varphi _{1},\ldots ,\varphi _{N})}$
and $F_{\pm }^{(\varphi ,\mathfrak{a})}$ refer to the (real normed) spaces $%
\mathcal{X}_{\pm }=\mathbb{R}^{N}$ and $\mathcal{X}_{\pm }=\mathbb{R}\times
L^{2}(S,\mathfrak{a})$, respectively, with obvious choices for the convex
functions $g_{\pm }$ and the linear transformations $\theta _{\pm }$.

Under Condition \ref{Condition essential} the nonlinear pressure functional $%
P:\mathcal{P}\left( T\right) \rightarrow \mathbb{R}$ of a $T$-invariant
measure is weak$^{\ast }$-upper semicontinuous, observing that the mappings $%
\mu \mapsto F_{\pm }\left( \mu _{S}\right) $ from $\mathcal{P}\left(
T\right) $ to $\mathbb{R}$ are weak$^{\ast }$-continuous and the entropy $h$%
, weak$^{\ast }$-upper semicontinuous. Therefore, the variational problem (%
\ref{variational problem}) has maximizers and the set 
\begin{equation}
E_{P}\doteq \left\{ \mu \in \mathcal{P}\left( T\right) :P\left( \mu \right)
=\sup P\left( \mathcal{P}\left( T\right) \right) \right\}
\label{NL eq measure}
\end{equation}%
of nonlinear equilibrium measures forms a (nonempty) weak$^{\ast }$-compact
and metrizable set, thanks to the metrizability of the weak$^{\ast }$
topology in $\mathcal{P}(T)$. However, unlike the linear case (cf. (\ref%
{dfgdfgdfg})), this set is \textbf{not necessarily convex} if the functional 
$F$ is not at least concave. This contrasts with the linear case, for which $%
F$ is affine (i.e., both convex and concave).

\subsection{Generalized equilibrium measures\label{Generalized Equilibrium
Measures}}

The non-convexity of the set $E_{P}$ (\ref{NL eq measure}) of nonlinear
equilibrium measure can be highly problematic in physical applications,
since, in this case, the (convex) mixture of two equilibrium states would no
longer be a new equilibrium measure, as physically expected. The same issue
occurs for quantum spin systems or fermions on lattices in the presence of
mean-field interactions \cite{BruPedra2}. In other words, physically, one
should not only consider equilibrium measures as being the solutions to the
nonlinear pressure, but also any convex combination of them should be an
equilibrium measure. This refers in the present work to \emph{generalized}
nonlinear equilibrium measures. See Section \ref{Minimization} and \cite[%
Definition 2.15 and Theorem 2.21]{BruPedra2}.

\begin{definition}[Generalized nonlinear equilibrium measures]
\label{Gneralized equilibrium measures}\mbox{ }\newline
A generalized nonlinear equilibrium measure is any element of the weak$%
^{\ast }$-closed convex hull $G_{P}\doteq \overline{\mathrm{co}}(E_{P})$ of
the set $E_{P}$ of (usual) nonlinear equilibrium measures.
\end{definition}

$G_{P}$ is in particular weak$^{\ast }$-compact and convex. Thus, by the
Milman theorem \cite[Proposition 1.5]{Phe}, generalized nonlinear
equilibrium measures that are extreme in $G_{P}$ must belong to the weak$%
^{\ast }$-compact set $E_{P}$. In other words, extreme generalized nonlinear
equilibrium measures are always usual nonlinear equilibrium measures. Then,
by the Choquet theorem (Theorem \ref{th Choquet}), for any $\mu \in G_{P}$,
there is a probability measure $m_{\mu }$ on $E_{P}$ such that%
\begin{equation*}
\mu =\int_{E_{P}}\nu m_{\mu }\left( \mathrm{d}\nu \right) \ ,
\end{equation*}%
i.e., the generalized nonlinear equilibrium measures are nothing but the
barycenters of probability measures on the set of (usual) nonlinear
equilibrium measures. See Definition \ref{Barycenters of a measure} and
compare also with Corollary \ref{Choquet decomposition}.

We use the terminology \textquotedblleft generalized nonlinear equilibrium
measure\textquotedblright\ because of the property stated in Theorem \ref%
{Theorem-main1 copy(1)} (ii), which establishes a precise relationship
between these measures and the maximization of an affine functional ($%
\mathfrak{F}^{\flat }$) that naturally appears in our setting for the
nonlinear thermodynamic formalism. This approach is inspired by the study of
quantum lattice systems done in \cite{BruPedra2}, in which the affine
functional is canonical from a thermodynamic point of view, being the
function that determines the properties of such quantum systems at
equilibrium.

To explain this in a simple way, we come back to Example \ref{example F II}
above. Let $\mathrm{m}$ be any fixed a priori probability measure on the
character set $(\Omega ,d)$, whose support is equal to $\Omega $, as in
Section \ref{sect entropy}, and $\mathrm{m}_{\otimes }\in \mathcal{P}\left(
T\right) $ the corresponding product (probability) measure on the symbolic
space $\Sigma \doteq \Omega ^{\mathbb{N}}$. Notice that $\mathrm{m}_{\otimes
}$ is even ergodic. Let further 
\begin{equation}
\mathbb{E}_{n}\left[ \varphi \right] \doteq n^{-1}\left( \varphi +\varphi
\circ T+\cdots +\varphi \circ T^{n-1}\right) \ ,\qquad n\in \mathbb{N},\
\varphi \in C(\Sigma )\ ,  \label{sdsdsdssdsd}
\end{equation}%
be the normalized Birkhoff sums of continuous potentials $\varphi $. Then,
for any $n\in \mathbb{N}$, the corresponding Gibbs measure on $C(\Sigma )$
is defined for any $\varphi \in C(\Sigma )$ by 
\begin{equation*}
\mu _{\mathrm{Gibbs}}^{(n)}\left( \varphi \right) \doteq \frac{1}{Z_{\mathrm{%
Gibbs}}^{(n)}}\int_{\Sigma _{n}}\mathrm{e}^{F_{+}\left( \mathbb{E}_{n}\left[
\varphi _{1}\right] \left( \sigma \right) ,\ldots ,\mathbb{E}_{n}\left[
\varphi _{N}\right] \left( \sigma \right) \right) +F_{-}\left( \mathbb{E}_{n}%
\left[ \varphi _{1}\right] \left( \sigma \right) ,\ldots ,\mathbb{E}_{n}%
\left[ \varphi _{N}\right] \left( \sigma \right) \right) }\varphi (\sigma )%
\mathrm{m}_{\otimes }\left( \mathrm{d}\sigma \right) \ ,
\end{equation*}%
where $\varphi _{1},\ldots ,\varphi _{N}$ are the potentials of Example \ref%
{example F II} and%
\begin{equation*}
Z_{\mathrm{Gibbs}}^{(n)}\doteq \int_{\Sigma _{n}}\mathrm{e}^{F_{+}\left( 
\mathbb{E}_{n}\left[ \varphi _{1}\right] \left( \sigma \right) ,\ldots ,%
\mathbb{E}_{n}\left[ \varphi _{N}\right] \left( \sigma \right) \right)
+F_{-}\left( \mathbb{E}_{n}\left[ \varphi _{1}\right] \left( \sigma \right)
,\ldots ,\mathbb{E}_{n}\left[ \varphi _{N}\right] \left( \sigma \right)
\right) }\mathrm{m}_{\otimes }\left( \mathrm{d}\sigma \right) \ .
\end{equation*}%
Then, using Varadhan's lemma one shows \cite{Bru-Pedra-Lopes1} that the
following limit exists and is given by a variational problem on the set of $%
T $-invariant probability measures:%
\begin{equation}
\lim_{n\rightarrow \infty }n^{-1}\ln Z_{\mathrm{Gibbs}}^{(n)}=\sup \mathfrak{%
F}^{\flat }\left( \mathcal{P}\left( T\right) \right) \ ,
\label{sdsdsssdsdsdsdsdsd}
\end{equation}%
where $\mathfrak{F}^{\flat }:\mathcal{P}\left( T\right) \rightarrow \mathbb{R%
}$ is the affine, albeit \textbf{not} necessarily weak$^{\ast }$-upper
semicontinuous, functional defined below by (\ref{F^bemol}) and named\ here
the \emph{affine nonlinear pressure}. Keeping in mind that $\mathcal{P}(T)$
is sequentially weak$^{\ast }$-compact, the set of approximated maximizers
of this variational problem is then defined by 
\begin{equation}
E_{\mathfrak{F}^{\flat }}\doteq \left\{ \mu \in \mathcal{P}\left( T\right)
:\exists \left( \mu _{n}\right) _{n\in \mathbb{N}}\subseteq \mathcal{P}%
\left( T\right) \mathrm{\ }\text{with }\lim_{n\rightarrow \infty }\mu
_{n}=\mu \text{ and\ }\lim_{n\rightarrow \infty }\mathfrak{F}^{\flat }\left(
\mu _{n}\right) =\sup \mathfrak{F}^{\flat }\left( \mathcal{P}\left( T\right)
\right) \right\} \ ,  \label{sdsdsdfdgdfhfghhf}
\end{equation}%
where the limit of $T$-invariant probability measures refers to the weak$%
^{\ast }$ convergence. This set is nothing but the weak$^{\ast }$-closed
convex hull $G_{P}\doteq \overline{\mathrm{co}}(E_{P})$ of $E_{P}$:

\begin{proposition}[Generalized equilibrium measures]
\label{structure of sets of maximizers copy(1)}\mbox{ }\newline
Fix $N\in \mathbb{N}$, take any convex (concave) real-valued function $F_{+}$%
\ ($F_{-}$) on $\mathbb{R}^{N}$, as well as norm-one potentials $\varphi
_{1},\ldots ,\varphi _{N}\in S$, and define the continuous function $F_{\pm
}^{(\varphi _{1},\ldots ,\varphi _{N})}$ as in\ Example \ref{example F II}.
Then, $E_{\mathfrak{F}^{\flat }}=\overline{\mathrm{co}}(E_{P})\doteq G_{P}$.
\end{proposition}

\begin{proof}
The assertion is a direct consequence of Theorem \ref{structure of sets of
maximizers}.
\end{proof}

Observe that weak$^{\ast }$-accumulation points of probability measures $%
(\mu _{\mathrm{Gibbs}}^{(n)}\circ \mathbb{E}_{n})_{n\in \mathbb{N}}$ always
belong to $G_{P}=E_{\mathfrak{F}^{\flat }}$, but \textbf{not necessarily }to 
$E_{P}$. This last property is proven in \cite{Bru-Pedra-Lopes1} and we omit
the details here because we are solely interested in the variational
problems themselves, with no (extended) thermodynamic considerations other
than those presented above for pedagogical reasons.

In fact, Theorem \ref{structure of sets of maximizers} is a pivotal result
in this context, which can be used to generalize Proposition \ref{structure
of sets of maximizers copy(1)} beyond Example \ref{example F II}, i.e., for
any $\sigma $-normal functions $F_{\pm }:\mathcal{M}(S)\rightarrow \mathbb{R}
$, with $F_{+}$ and $F_{-}$ being respectively convex and concave. This
generalization is not considered in the present section to avoid repeating
essentially the same arguments in a case that is more complicated to explain
properly and, therefore, much less pedagogical. In fact, Proposition \ref%
{structure of sets of maximizers copy(1)} already shows us that the notion
of generalized nonlinear equilibrium states is the correct one for
describing the thermodynamic limit of nonlinear systems at equilibrium.

\begin{remark}
\mbox{ }\newline
Buzzi, Kloeckner and Leplaideur \cite{TF4} already explain how to obtain
nonlinear equilibrium measures from an affine pressure, but their approach 
\textbf{is very different} from the one that leads to the affine pressure $%
\mathfrak{F}^{\flat }$ presented above. Compare (\ref{the free-energy
density functional}) below with (\ref{maximizers z})--(\ref{var y}). In
fact, the arguments used by these authors are reminiscent of the Bogoliubov
linearization, which is explained below. However, they do not exploit it to
the fullest extent, which is what we propose to do here.
\end{remark}

\subsection{Buzzi-Kloeckner-Leplaideur's approach to nonlinear pressures 
\label{Buzzy-Kloeckner-Leplaideur's Approach}}

In 2019, Leplaideur and Watbled \cite{TF4} pioneered the study of nonlinear
topological pressure through an affine pressure. Considering a generalized
Curie-Weiss model, they investigated the special case of quadratic
non-linearity. In 2023, in the context of Example \ref{example F II} Buzzi,
Kloeckner and Leplaideur \cite{TF4} show how nonlinear equilibrium measures
can be obtained from an affine pressure. In fact, given $N\in \mathbb{N}$, $%
\varphi _{1},\ldots ,\varphi _{N}\in C(\Sigma )$ and a real-valued
continuous function $F$ on some set $U\subseteq \mathbb{R}^{N}$ (cf. Example %
\ref{example F II}) containing the compact convex set 
\begin{equation}
\left\{ \left( \mu \left( \varphi _{1}\right) ,\ldots ,\mu \left( \varphi
_{N}\right) \right) :\mu \in \mathcal{P}\left( T\right) \right\} \subseteq
U\ ,  \label{compact set}
\end{equation}%
observe that 
\begin{equation}
\sup_{\mu \in \mathcal{P}\left( T\right) }P\left( \mu \right) \doteq
\sup_{\mu \in \mathcal{P}\left( T\right) }\left\{ F\left( \mu \left( \varphi
_{1}\right) ,\ldots ,\mu \left( \varphi _{N}\right) \right) +h\left( \mu
\right) \right\} =\sup_{z\in U}\left\{ F\left( z\right) +\mathrm{h}\left(
z\right) \right\} \ ,  \label{maximizers z}
\end{equation}%
where 
\begin{equation*}
\mathrm{h}\left( z\right) \doteq \sup \left\{ h\left( \mu \right) :\mu \in 
\mathcal{P}\left( T\right) \text{ such that }z=\left( \mu \left( \varphi
_{1}\right) ,\ldots ,\mu \left( \varphi _{N}\right) \right) \right\} \ .
\end{equation*}%
The variational problem (\ref{maximizers z}) (used in \cite{TF4}) is a
particular case of (\ref{variational problem}), in accordance with Example %
\ref{example F II}.

By \cite[Proposition 3.3]{TF4}, solutions $\mu _{z}$ to this last
variational problem are linear equilibrium measures $\nu $ for some linear
combination $y_{1}\varphi _{1}+\cdots +y_{N}\varphi _{N}$, i.e.,%
\begin{equation}
p\left( \nu ,y_{1},\ldots ,y_{N}\right) =\sup_{\mu \in \mathcal{P}\left(
T\right) }p\left( \mu ,y_{1},\ldots ,y_{N}\right) \doteq \mathrm{P}\left(
y_{1},\ldots ,y_{N}\right) \mathbf{\ ,}  \label{var y}
\end{equation}%
where 
\begin{equation*}
p\left( \mu ,y_{1},\ldots ,y_{N}\right) \doteq h\left( \mu \right) +y_{1}\mu
\left( \varphi _{1}\right) +\cdots +y_{N}\mu \left( \varphi _{N}\right)
,\qquad y_{1},\ldots ,y_{N}\in \mathbb{R}^{N},\ \mu \in \mathcal{P}\left(
T\right) \ .
\end{equation*}%
By \cite[Theorem C]{TF4}, all the values of $y_{1},\ldots ,y_{N}$ to be
taken to maximize the nonlinear pressure functional $P$ in (\ref{maximizers
z}) are obtained from the maximizers in $U\subseteq \mathbb{R}^{N}$ of the
variational problem on the right-hand side of Equation (\ref{maximizers z}),
in the following way: Define the gradient $\nabla \mathrm{P}:\mathbb{R}%
^{N}\rightarrow \mathbb{R}^{N}$ by 
\begin{equation*}
\nabla \mathrm{P}\left( y_{1},\ldots ,y_{N}\right) \doteq \left( \partial
_{y_{1}}\mathrm{P}\left( y_{1},\ldots ,y_{N}\right) ,\ldots ,\partial
_{y_{N}}\mathrm{P}\left( y_{1},\ldots ,y_{N}\right) \right) \ .
\end{equation*}%
Then, we must take all possible $(y_{1},\ldots ,y_{N})\in \mathcal{Y}\doteq
\left( \nabla \mathrm{P}\right) ^{-1}\left( \mathcal{V}\right) $ in (\ref%
{var y}), where 
\begin{equation}
\mathcal{V}\doteq \left\{ c\in \mathbb{R}^{N}:F\left( c\right) +\mathrm{h}%
\left( c\right) =\sup_{z\in \mathbb{R}^{N}}\left\{ F\left( z\right) +\mathrm{%
h}\left( z\right) \right\} \right\} \ ,  \label{maximizers zbis}
\end{equation}%
in order to obtain all nonlinear equilibrium measures, i.e., maximizers of $%
P $ (see the left-hand side of the variational problem (\ref{maximizers z}%
)), as linear equilibrium measures $\nu $ of (\ref{var y}). This is done for 
$C^{1}$-functions $F$.

The authors get interesting results, but there are however several
drawbacks:\ First, the computation of the set $\mathcal{V}$ looks highly
nontrivial for general $C^{1}$-functions $F$. It requires, in particular, a
good control of the entropy function \textrm{$h$}, which is minus the
Legendre%
%TCIMACRO{\TeXButton{\-}{\-}}%
%BeginExpansion
\-%
%EndExpansion
-Fenchel transform of the pressure function $\mathrm{P}$. Then, one has to
be able to compute the preimage of the set $\mathcal{V}$ trough the function 
$\nabla \mathrm{P}$, which is again another nontrivial task, in general.

Below we present a more effective method related to the \emph{Bogoliubov
linearization}, which is explained in detail and great generality in Section %
\ref{Abstract Theory}. As applied to the setting of \cite[Theorem C]{TF4},
this method allows us to study the variational problem (\ref{maximizers z}),
but the corresponding set $\mathcal{Y}$ is obtained much more
straightforwardly than in \cite[Theorem C]{TF4}, by means of a simple
variational problem that takes advantage of the convexity or concavity of
the function $F$. Summarizing, via the Buzzi%
%TCIMACRO{\TeXButton{\-}{\-}}%
%BeginExpansion
\-%
%EndExpansion
-Kloeckner-Leplaideur's method, 
\begin{equation*}
\sup_{\mu \in \mathcal{P}\left( T\right) }P\left( \mu \right) =P\left( \mu
_{c}\right)
\end{equation*}%
for some $T$-invariant measure $\mu _{c}\in \mathcal{P}\left( T\right) $
satisfying so-called self-consistency equations $\mu _{c}(\varphi
_{j})=c_{j} $, $j\in \{1,\ldots ,N\}$, and%
\begin{equation*}
\sup_{\mu \in \mathcal{P}\left( T\right) }p\left( \mu ,y_{1},\ldots
,y_{N}\right) =p\left( \mu _{c},y_{1},\ldots ,y_{N}\right) \ ,
\end{equation*}%
for some $c\in \mathcal{V}$ and $y_{1},\ldots ,y_{N}\in \mathcal{Y}\doteq
\left( \nabla \mathrm{P}\right) ^{-1}\left( \mathcal{V}\right) $. Theorem %
\ref{Theorem-main1} below yields essentially the same result, but it does
not require explicit knowledge of the set $\mathcal{V}$, which avoids the
need to study a variational problem involving the highly non-trivial entropy
function $\mathrm{h}\left( z\right) $. Instead, the numbers $y_{1},\ldots
,y_{N}$ are computed from a variational problem involving only linear
pressures and the Legendre%
%TCIMACRO{\TeXButton{\-}{\-}}%
%BeginExpansion
\-%
%EndExpansion
-Fenchel transform of $F$.

To get a first intuitive idea, let us assume for simplicity that $F$ is
convex and satisfies 
\begin{equation}
F\left( z\right) =\sup_{x\in \mathbb{R}^{N}}\left\{ \left\langle
z,x\right\rangle _{\mathbb{R}^{d}}-F^{\ast }\left( x\right) \right\} \qquad 
\text{with}\qquad F^{\ast }\left( x\right) \doteq \sup_{z\in \mathbb{R}%
^{N}}\left\{ \left\langle x,z\right\rangle _{\mathbb{R}^{d}}-F\left(
z\right) \right\} \ .  \label{legendre transforms of F}
\end{equation}%
Then%
\begin{eqnarray*}
\sup_{\mu \in \mathcal{P}\left( T\right) }P\left( \mu \right) &=&\sup_{z\in 
\mathbb{R}^{N}}\left\{ F\left( z\right) +\mathrm{h}\left( z\right) \right\}
=\sup_{x\in \mathbb{R}^{N}}\sup_{z\in \mathbb{R}^{N}}\left\{ \left\langle
z,x\right\rangle _{\mathbb{R}^{d}}+\mathrm{h}\left( z\right) -F^{\ast
}\left( x\right) \right\} \\
&=&\sup_{x\in \mathbb{R}^{N}}\left\{ \mathrm{P}\left( x\right) -F^{\ast
}\left( x\right) \right\} =\sup_{x\in \mathbb{R}^{N}}\inf_{z\in \mathbb{R}%
^{N}}\left\{ \mathrm{P}\left( x\right) -\left\langle x,z\right\rangle _{%
\mathbb{R}^{d}}+F\left( z\right) \right\} \ .
\end{eqnarray*}%
In particular, if one can interchange the last $\sup $ over $x\in \mathbb{R}%
^{N}$ and $\inf $ over $z\in \mathbb{R}^{N}$, then%
\begin{eqnarray*}
\sup_{\mu \in \mathcal{P}\left( T\right) }P\left( \mu \right) &=&\inf_{z\in 
\mathbb{R}^{N}}\sup_{x\in \mathbb{R}^{N}}\left\{ \mathrm{P}\left( x\right)
-\left\langle x,z\right\rangle _{\mathbb{R}^{d}}+F\left( z\right) \right\}
=\inf_{z\in \mathbb{R}^{N}}\left\{ \mathrm{P}\left( y\right) -\left\langle
y,z\right\rangle _{\mathbb{R}^{d}}+F\left( z\right) \right\} \\
&=&\mathrm{P}\left( y\right) -\left\langle y,c\right\rangle _{\mathbb{R}%
^{d}}+F\left( c\right)
\end{eqnarray*}%
with $y$ being chosen so that 
\begin{equation*}
\mathrm{P}\left( y\right) -\left\langle y,z\right\rangle _{\mathbb{R}%
^{d}}=\sup_{x\in \mathbb{R}^{N}}\left\{ \mathrm{P}\left( x\right)
-\left\langle x,z\right\rangle _{\mathbb{R}^{d}}\right\} \ .
\end{equation*}%
In the same way, $c$ is chosen so that 
\begin{equation*}
\mathrm{P}\left( y\right) -\left\langle y,c\right\rangle _{\mathbb{R}%
^{d}}+F\left( c\right) =\inf_{z\in \mathbb{R}^{N}}\left\{ \mathrm{P}\left(
y\right) -\left\langle y,z\right\rangle _{\mathbb{R}^{d}}+F\left( z\right)
\right\} \ .
\end{equation*}%
In particular, $\nabla \mathrm{P}(y)=c$, i.e., $y=(y_{1},\ldots ,y_{N})\in
\left( \nabla \mathrm{P}\right) ^{-1}(\{c\})$. Compare with \cite[Theorem C]%
{TF4} and $y_{1},\ldots ,y_{N}\in \mathcal{Y}\doteq \left( \nabla \mathrm{P}%
\right) ^{-1}\left( \mathcal{V}\right) $ above. Equation (\ref{legendre
transforms of F}) refers to the Legendre-Fenchel transform of the convex
functions $F$ and $F^{\ast }$, see Section \ref{Legendre-Fenchel transform
section} below.

Our method makes this heuristics rigorous. In fact, the problematic entropy
function $\mathrm{h}\left( z\right) $ is \textbf{not} used directly in 
\textbf{any }of our arguments. Instead, we use its Legendre-Fenchel
transform, $\mathrm{P}\left( x\right) $, which is much more tractable in
many important applications. Our method also avoids calculating any
gradients, but requires the convex and concave parts $F_{\pm }$ of $F$ to be
separated and written as Legendre-Fenchel transforms. Since \cite{TF4} only
considers the compact convex set (\ref{compact set}) and any $C^{1}$%
-functions on a compact subset of $\mathbb{R}^{N}$ ($N\in \mathbb{N}$) can
be represented as the difference of convex and continuous functions (see
Remark \ref{Remark-condition2}), our assumptions are much weaker than those
of \cite{TF4}. Moreover, we provide new results. See Conditions TF1--TF3 of
Section \ref{Bogoliubov linearizations TF}, Theorems \ref{Theorem-main1}, %
\ref{Theorem-main1 copy(1)} and Corollary \ref{Theorem-main1 copy(2)}. For
example, we can handle not only the case of finitely many potentials $%
\varphi _{1},\ldots ,\varphi _{N}$, as in \cite{TF4} or Example \ref{example
F II}, but also the very general framework of Example \ref{example F III},
while making the approach transparent with regard to the main arguments used
in proofs.

\subsection{New approach to nonlinear thermodynamic formalism\label%
{Bogoliubov Approximation}}

The sets $E_{P}$ and $G_{P}\doteq \overline{\mathrm{co}}(E_{P})$ of
equilibrium measures are pivotal; see Equation (\ref{NL eq measure}),
Definitions \ref{Equilibrium measures def copy(1)} and \ref{Gneralized
equilibrium measures}. However, it is a priori not clear how useful the
corresponding variational principles defining these sets are to study phase
transitions. In fact, it turns out to be more convenient to reduce the
nonlinear thermodynamic formalism to the linear one, for which concrete
calculations can often be performed, similar to Buzzi-Kloeckner-Leplaideur's
approach (Section \ref{Buzzy-Kloeckner-Leplaideur's Approach}).

Our abstract theory of \emph{Bogoliubov linearizations}, which is explained
in detail in Section \ref{Abstract Theory}, provides a systematic approach
to this. It generalizes the study conducted in \cite{BruPedra2} on quantum
lattice systems, which, in this special case, resolves an old open problem
addressed by Ginibre in 1968 \cite{Ginibre} concerning the exactness of
Bogoliubov approximations of equilibrium states in the thermodynamic limit.
Indeed, in the quantum framework previously studied \cite{BruPedra2} only
the specific case of \textquotedblleft Example \ref{example F I}\footnote{%
Adapted appropriately to the fermionic/quantum spin situation.}%
\textquotedblright\ is studied. Our framework here is much more general,
encompassing Example \ref{example F III} and thus going far beyond the
quadratic case of Example \ref{example F I}.

\subsubsection{Bogoliubov linearizations\label{Bogoliubov linearizations TF}}

\noindent \textbf{General assumptions.} We study the nonlinear pressure 
\begin{equation*}
P\left( \mu \right) \doteq h\left( \mu \right) +F_{-}\left( \mu _{S}\right)
+F_{+}\left( \mu _{S}\right) \ ,\qquad \mu \in \mathcal{P}\left( T\right) \ ,
\end{equation*}%
for the general case given by Example \ref{example F III}. In fact, our
abstract theory of Bogoliubov linearizations would allow us to extend this
example even further to functions that are not necessarily of the $F_{\pm
}=\pm g_{\pm }\circ \theta _{\pm }$, and even \textbf{non}-$\sigma $-normal
ones. For simplicity, we refrain from making this further generalization
here, bearing in mind that the objective of this section is to illustrate
our approach using the nonlinear thermodynamic formalism as a paradigmatic
example.

To present our general assumptions, we use the Legendre-Fenchel transform $%
g^{\ast }:\mathcal{X}^{\ast }\rightarrow (-\infty ,\infty ]$ of a function $%
g:\mathcal{X}\rightarrow \mathbb{R}$ defined for a given dual pair $(%
\mathcal{X},\mathcal{X}^{\ast })$ by 
\begin{equation}
g^{\ast }\left( x\right) \doteq \sup_{y\in \mathcal{X}}\left\{ x\left(
y\right) -g\left( y\right) \right\} \ ,\qquad x\in \mathcal{X}^{\ast }\ .
\label{conjugate definitionbis}
\end{equation}%
We say that this Legendre-Fenchel transform has \emph{minimal linear growth} 
$\lambda \in \mathbb{R}^{+}$ if 
\begin{equation}
\lim_{R\rightarrow \infty }\sup_{y\in \mathcal{X}^{\ast }\backslash
B(0,R)}\left\{ \lambda \left\Vert y\right\Vert _{\mathrm{op}}-g^{\ast
}\left( y\right) \right\} =-\infty \ ,  \label{linear growbis}
\end{equation}%
where $\Vert \cdot \Vert _{\mathrm{op}}$ denotes the operator norm (see (\ref%
{ball2})) and $B(0,R)$ is the closed ball of radius $R\in \mathbb{R}^{+}$
and center $0\in \mathcal{X}^{\ast }$ (see (\ref{ball1})). We recommend
referring to Section \ref{Legendre-Fenchel transform section} for more
details. We consider the following assumptions on the nonlinear energies:

\begin{itemize}
\item[TF1] $\mathcal{X}_{\pm }$ are two real normed spaces and $\theta _{\pm
}:\mathcal{M}(S)\rightarrow \mathcal{X}_{\pm }$ are two linear
transformations that are $\sigma $-normal, i.e., for any bounded sequence $%
(f_{n})_{n\in \mathbb{N}}\subseteq \mathcal{M}(S)$ converging point-wise to $%
f\in \mathcal{M}(S)$, $(\theta _{\pm }(f_{n}))_{n\in \mathbb{N}}$ converges
to $\theta _{\pm }(f)$ in $\mathcal{X}_{\pm }$.

\item[TF2] $g_{-}:\mathcal{X}_{-}\rightarrow \mathbb{R}$ is a lower
semicontinuous and convex function for which there is a positive radius $%
R_{0}\in \mathbb{R}^{+}$ such that $\{x\in \mathcal{X}_{-}:\Vert x\Vert _{%
\mathrm{op}}<R,\ g_{-}^{\ast }(x)<\infty \}$ is nonempty and weak$^{\ast }$%
-closed for all $R\in \lbrack R_{0},\infty )$. Moreover, $g_{-}^{\ast }$ has
minimal\emph{\ }linear growth $\Vert \theta _{-}\Vert _{\mathrm{op}}$ in the
sense of Equation (\ref{linear growbis}).

\item[TF3] $g_{+}:\mathcal{X}_{+}\rightarrow \mathbb{R}$ is a lower
semicontinuous, bounded and convex function for which $g_{+}^{\ast }(%
\mathcal{X}_{+})\subseteq \mathbb{R}$ and $g_{+}^{\ast }$ has minimal\emph{\ 
}linear growth $\Vert \theta _{+}\Vert _{\mathrm{op}}$ in the sense of
Equation (\ref{linear growbis}).
\end{itemize}

\noindent Note that a $\sigma $-normal linear mapping on $\mathcal{M}(S)$ is
always bounded. In particular, $\Vert \theta _{\pm }\Vert _{\mathrm{op}%
}<\infty $. Observe also that Conditions TF2--TF3 are trivially satisfied by
quadratic functions $g_{\pm }:\mathcal{X}_{\pm }\rightarrow \mathbb{R}$, as $%
g_{+}^{\ast }$ and $g_{-}^{\ast }$ are also quadratic in this case. This is
the most relevant case in statistical physics. These assumptions are in fact
very general. They are an adapted version of (the much more general)
Conditions B1--B3 of Section \ref{sect conc conv Bogo}. Their general nature
is discussed in Remarks \ref{Remark-condition1} and \ref{Remark-condition2}.
\bigskip

\noindent \textbf{Bogoliubov linearizations.} They refer to a $\mu $%
-linearization method for the nonlinear pressure, which is now defined as
follows: 
\begin{equation}
P\left( \mu \right) \doteq h\left( \mu \right) -g_{-}\circ \theta _{-}\left(
\mu _{S}\right) +g_{+}\circ \theta _{+}\left( \mu _{S}\right) \ ,\qquad \mu
\in \mathcal{P}\left( T\right) \ ,  \label{eq 21}
\end{equation}%
under Conditions TF1--TF3. Since $g^{\ast \ast }=g$ for any lower
semicontinuous convex function (see (\ref{Biconjugate definition})), we
observe from Equation (\ref{eq 21}) that 
\begin{equation}
P\left( \mu \right) =\sup_{y_{+}\in \mathcal{X}_{+}^{\ast }}\inf_{y_{-}\in 
\mathcal{X}_{-}^{\ast }}\left\{ \mathrm{P}\left( y_{+},y_{-},\mu \right)
+g_{-}^{\ast }\left( y_{-}\right) -g_{+}^{\ast }\left( y_{+}\right) \right\}
\ ,\qquad \mu \in \mathcal{P}\left( T\right) \ ,  \label{pressure legendre}
\end{equation}%
where, for any $y_{\pm }\in \mathcal{X}_{\pm }^{\ast }$ and $\mu \in 
\mathcal{P}\left( T\right) $,%
\begin{equation}
\mathrm{P}\left( y_{+},y_{-},\mu \right) \doteq h\left( \mu \right)
-y_{-}\circ \theta _{-}\left( \mu _{S}\right) +y_{+}\circ \theta _{+}\left(
\mu _{S}\right) \ ,  \label{approximating pressure}
\end{equation}%
is called here an approximating pressure of the $T$-invariant probability
measure $\mu $.

We call the new functionals $\mu \mapsto \mathrm{P}(y_{+},y_{-},\mu )$, $%
y_{\pm }\in \mathcal{X}_{\pm }^{\ast }$, \emph{Bogoliubov linearizations}\
of $P$. Then, the method we present refers to the study the $\mu $%
-linearized variational problem%
\begin{equation}
\mathrm{P}^{\flat }\doteq \sup_{y_{+}\in \mathcal{X}_{+}^{\ast
}}\inf_{y_{-}\in \mathcal{X}_{-}^{\ast }}P_{\mathrm{NL}}\left(
y_{+},y_{-}\right) \ ,  \label{bogo approx}
\end{equation}%
instead of $\sup P(\mathcal{P}(T))$, where, for any continuous linear
functionals $y_{\pm }\in \mathcal{X}_{\pm }^{\ast }$,%
\begin{eqnarray}
P_{\mathrm{NL}}\left( y_{+},y_{-}\right) &\doteq &P_{\mathrm{L}}\left(
y_{+},y_{-}\right) +g_{-}^{\ast }\left( y_{-}\right) -g_{+}^{\ast }\left(
y_{+}\right) \ ,  \label{nonlinear approximating pressure} \\
P_{\mathrm{L}}\left( y_{+},y_{-}\right) &\doteq &\sup \mathrm{P}\left(
y_{+},y_{-},\mathcal{P}(T)\right) \doteq \sup_{\mu \in \mathcal{P}(T)}%
\mathrm{P}\left( y_{+},y_{-},\mu \right) \ .  \label{pression approx}
\end{eqnarray}%
$P_{\mathrm{NL}}$ and $P_{\mathrm{L}}$ are called here the \emph{nonlinear
and linear approximating pressures}, respectively. Thus, proceeding in this
manner, we first analyze the linear problem for $T$-invariant probability
measures and then study the nonlinear part via the other two variational
problems over $\mathcal{X}_{\pm }^{\ast }$.

The variational problem (\ref{bogo approx}) is supposed to be much easier
than $\sup P(\mathcal{P}(T))$ in many important cases because the linear
thermodynamic formalism (Section \ref{Linear thermo}), which is now very
well developed, should give us good control over the linear approximating
pressures $P_{\mathrm{L}}\left( y_{+},y_{-}\right) $. We explain below how (%
\ref{bogo approx}) can be used rigorously to study the original variational
problem $\sup P(\mathcal{P}(T))$. When $g_{-}=0$, note from (\ref{pressure
legendre}) that (\ref{bogo approx}) and $\sup P(\mathcal{P}(T))$ are
trivially equivalent variational problems, because two suprema always
commute with each other. The main difficulty is therefore to handle the case
where $g_{-}\neq 0$. This is done in Sections \ref{Sect Bogoliubov app tech1}%
--\ref{sect conc conv Bogo} using the celebrated von Neumann minimax theorem
(Theorem \ref{theorem minmax von Neumann}).\bigskip

\noindent \textbf{Bogoliubov linearizations -- H\"{o}lder case.} Remark
that, for any continuous linear functional $y_{\pm }\in \mathcal{X}_{\pm
}^{\ast }$, the new functional 
\begin{equation*}
\mu \mapsto y_{+}\circ \theta _{+}\left( \mu _{S}\right) -y_{-}\circ \theta
_{-}\left( \mu _{S}\right)
\end{equation*}%
from $\mathcal{P}$ to $\mathbb{R}$ in Equation (\ref{approximating pressure}%
)\ is affine and weak$^{\ast }$ continuous. Similar to \cite[Proposition 3.9]%
{Bru-pedra-MF-I}, for any $y_{\pm }\in \mathcal{X}_{\pm }^{\ast }$, there is
a unique continuous function $\Theta _{y_{-},y_{+}}\in C(\Sigma )$ such that 
\begin{equation}
y_{+}\circ \theta _{+}\left( \mu _{S}\right) -y_{-}\circ \theta _{-}\left(
\mu _{S}\right) =\mu \left( \Theta _{y_{-},y_{+}}\right) \ ,\qquad \mu \in 
\mathcal{P}\ .  \label{def potential}
\end{equation}%
We name the continuous function $\Theta _{y_{-},y_{+}}$ the \emph{%
approximating potential}\ associated with $y_{-},y_{+},\theta _{-},\theta
_{+}$.

Note that the approximating potentials are not necessarily H\"{o}lder
continuous. In particular, the linear equilibrium measure, i.e., the $T$%
-invariant probability measure realizing the supremum in (\ref{pression
approx}), is not necessarily unique. This property can nevertheless always
be ensured for specific linear transformations $\theta _{\pm }$:

\begin{definition}[H\"{o}lder-type linear functions]
\label{sect Holder type theta}\mbox{ }\newline
The pair $(\theta _{-},\theta _{+})$ of linear functions is H\"{o}lder-type\
if, for all $y_{\pm }\in \mathcal{X}_{\pm }^{\ast }$, their associated
approximating potential (see (\ref{def potential})) is H\"{o}lder
continuous, i.e., $\Theta _{y_{-},y_{+}}\in C^{\alpha }(\Sigma )$ for some $%
\alpha =\alpha _{y_{-},y_{+}}\in (0,1]$.
\end{definition}

It corresponds in Example \ref{example F II} to fix functions $\varphi
_{1},\ldots ,\varphi _{N}$ only in $C^{\alpha }(\Sigma )$, instead of the
whole space $C(\Sigma )$. Another general example of this situation can be
given in the scope of Example \ref{example F I} as follows: Fix $\alpha \in
(0,1]$. Define the unit closed ball 
\begin{equation*}
S_{\alpha }\doteq \left\{ \varphi \in C^{\alpha }\left( \Sigma \right)
:\left\Vert \varphi \right\Vert _{\alpha }\leq 1\right\} \subseteq C^{\alpha
}\left( \Sigma \right)
\end{equation*}%
of the Banach space $C^{\alpha }(\Sigma )$ of $\alpha $-H\"{o}lder
continuous functions, endowed with the H\"{o}lder metric 
\begin{equation*}
d_{S_{\alpha }}\left( \varphi ,\psi \right) \doteq \left\Vert \varphi -\psi
\right\Vert _{\alpha }\ ,\text{\qquad }\varphi ,\psi \in S_{\alpha }\text{ }.
\end{equation*}%
See also Equation (\ref{norm1}). Take any finite positive Borel measure $%
\mathfrak{a}_{\alpha }$ on $S_{\alpha }$, $\lambda \in \mathbb{R}$ and some $%
\alpha $-H\"{o}lder continuous function $\varphi \in S_{\alpha }$. Observe
that the identity mapping $S_{\alpha }\rightarrow S$ is continuous. Thus, we
can define a measure $\mathfrak{a}$ on $S$ as the pushforward of $\mathfrak{a%
}_{\alpha }$ through the identity mapping. Then, with such a measure $%
\mathfrak{a}$ in Example \ref{example F I} we again obtain H\"{o}lder-type
functions $(\theta _{-},\theta _{+})$, where $\mathcal{X}_{\pm }=\mathbb{R}%
\times L^{2}(S,\mathfrak{a})$ with obvious choices for the linear
transformations $\theta _{\pm }$.

H\"{o}lder-type linear transformations $\theta _{\pm }$ are particularly
useful since the linear approximating pressure (\ref{pression approx}) leads
in this case to a unique and ergodic linear equilibrium measure for any $%
y_{\pm }\in \mathcal{X}_{\pm }^{\ast }$, as already explained in Section \ref%
{Linear thermo}. See also \cite{LMMS,ACR}. This is a very interesting
situation for studying the nonlinear problem, as we shall see.

\subsubsection{Thermodynamic game}

The $\mu $-linearized variational problem $\mathrm{P}^{\flat }$ refers to
Equations (\ref{bogo approx})--(\ref{pression approx}). Observe in
particular that the concave and convex parts of the nonlinear pressure do
not have symmetric roles, since an infimum and a supremum do not commute in
general. In special situations, this could be the case, but certainly not in
the general case, as is explicitly shown in \cite[Section 2.7]{BruPedra2}
for lattices quantum systems. A sufficient condition for $\sup $ and $\inf $
to commute is given by Sion's minimax theorem \cite{HIDETOSHI KOMIYA}.

The switching of $\sup $ and $\inf $ as it appears in $\mathrm{P}^{\flat }$ (%
\ref{bogo approx}) leads to the min-max variational problem 
\begin{equation}
\mathrm{P}^{\sharp }\doteq \inf_{y_{-}\in \mathcal{X}_{-}^{\ast
}}\sup_{y_{+}\in \mathcal{X}_{+}^{\ast }}\left\{ P_{\mathrm{L}}\left(
y_{+},y_{-}\right) +g_{-}^{\ast }\left( y_{-}\right) -g_{+}^{\ast }\left(
y_{+}\right) \right\} \doteq \inf_{y_{-}\in \mathcal{X}_{-}^{\ast
}}\sup_{y_{+}\in \mathcal{X}_{+}^{\ast }}P_{\mathrm{NL}}\left(
y_{+},y_{-}\right) \ .  \label{bogo approxdiese}
\end{equation}%
The max-min variational problem $\mathrm{P}^{\flat }$ and min-max
variational problem $\mathrm{P}^{\sharp }$ have both a meaning in terms of
equilibrium measures, as shown in Section \ref{Section game + measusre}. In
the case of lattice quantum systems, this has already been observed in \cite[%
Theorem 2.36]{BruPedra2}. The max-min variational problem also appears in
the so-called Kac\ or van der Waals limit of lattice quantum models \cite%
{Kac}.

Similar to what is done in \cite[Definition 2.35]{BruPedra2} for quantum
lattice systems, the variational problems $\mathrm{P}^{\flat }$ (\ref{bogo
approx}) and $\mathrm{P}^{\sharp }$ (\ref{bogo approxdiese}) can be
interpreted as the conservative values of a two-person zero-sum game whose
payoff function is nothing but the nonlinear approximating pressure $P_{%
\mathrm{NL}}$ (\ref{nonlinear approximating pressure}), called here \emph{%
the thermodynamic game}. Note that we always consider here \emph{non-zero}
continuous functions $g_{\pm }:\mathcal{X}_{\pm }\rightarrow \mathbb{R}$, as
otherwise there is no proper two-person game. The thermodynamic game is
characterized by the following objects:

\begin{itemize}
\item For any continuous linear functionals $y_{\pm }\in \mathcal{X}_{\pm
}^{\ast }$, we define the variational problems%
\begin{eqnarray}
P^{\flat }\left( y_{+}\right) &\doteq &\inf P_{\mathrm{NL}}\left( y_{+},%
\mathcal{X}_{-}^{\ast }\right) \ ,\qquad y_{+}\in \mathcal{X}_{+}^{\ast }\ ,
\label{Pbemol} \\
P^{\sharp }\left( y_{-}\right) &\doteq &\sup P_{\mathrm{NL}}\left( \mathcal{X%
}_{+}^{\ast },y_{-}\right) \ ,\qquad y_{-}\in \mathcal{X}_{-}^{\ast }\ ,
\label{Pdiese}
\end{eqnarray}%
\ as well as their set of minimizers: 
\begin{eqnarray}
M^{\flat }\left( y_{+}\right) &\doteq &\left\{ x_{-}\in \mathcal{X}%
_{-}^{\ast }:P^{\flat }\left( y_{+}\right) =P_{\mathrm{NL}}\left(
y_{+},x_{-}\right) \right\} \subseteq \mathcal{X}_{-}^{\ast }\ ,
\label{Mbemol(y)} \\
M^{\sharp }\left( y_{-}\right) &\doteq &\left\{ x_{+}\in \mathcal{X}%
_{+}^{\ast }:P^{\sharp }\left( y_{-}\right) =P_{\mathrm{NL}}\left(
x_{+},y_{-}\right) \right\} \subseteq \mathcal{X}_{+}^{\ast }\ .
\label{Mdiese(y)}
\end{eqnarray}

\item For the max-min variational problem $\mathrm{P}^{\flat }$ and min-max
variational problem $\mathrm{P}^{\sharp }$ we consider the sets of
optimizers defined by%
\begin{eqnarray}
\mathrm{M}^{\flat } &\doteq &\left\{ x_{+}\in \mathcal{X}_{+}^{\ast
}:P^{\flat }\left( x_{+}\right) =\sup P^{\flat }\left( \mathcal{X}_{+}^{\ast
}\right) \doteq \mathrm{P}^{\flat }\right\} \subseteq \mathcal{X}_{+}^{\ast
}\ ,  \label{Mbemol} \\
\mathrm{M}^{\sharp } &\doteq &\left\{ x_{-}\in \mathcal{X}_{-}^{\ast
}:P^{\sharp }\left( x_{-}\right) =\sup P^{\sharp }\left( \mathcal{X}%
_{-}^{\ast }\right) \doteq \mathrm{P}^{\sharp }\right\} \subseteq \mathcal{X}%
_{-}^{\ast }\ .  \label{Mdiese}
\end{eqnarray}
\end{itemize}

\noindent As proven in Section \ref{Nonlinear pressure}, under Conditions
TF1--TF3 these objects have the properties gathered in the proposition
below. Let $\mathrm{dom}(g_{-}^{\ast })$ be the domain of the
Legendre-Fenchel transform $g_{-}^{\ast }$ of the function $g_{-}:\mathcal{X}%
_{-}\rightarrow \mathbb{R}$, i.e., the set of all $x\in \mathcal{X}%
_{-}^{\ast }$ such that $g_{-}^{\ast }\left( x\right) <\infty $. Since the
function $g_{-}$ never takes the value $-\infty $ (Condition TF2), observe
that $\mathrm{dom}(g_{-}^{\ast })\neq \emptyset $ (see (\ref%
{sdfsdfsdfsdfsdfs})) and $\mathrm{dom}(g_{-}^{\ast })$ is a convex subset of 
$\mathcal{X}_{-}^{\ast }$, by convexity of $g_{-}^{\ast }$.

\begin{proposition}[Properties of variational problems]
\label{Proposition importante bogoluibov01 copy(5)}\mbox{ }\newline
Assume Conditions TF1--TF3. Then, the following assertions hold: \newline
\emph{(i)} For all $y_{+}\in \mathcal{X}_{+}^{\ast }$ and $y_{-}\in \mathrm{%
dom}(g_{-}^{\ast })\subseteq \mathcal{X}_{-}^{\ast }$, $P^{\flat
}(y_{+}),P^{\sharp }(y_{-})\in \mathbb{R}$ and $M^{\flat }(y_{+}),M^{\sharp
}(y_{-})$ are nonempty weak$^{\ast }$-compact sets. For all $y_{+}\in 
\mathcal{X}_{+}^{\ast }$, $M^{\flat }(y_{+})$ is additionally convex.\newline
\emph{(ii)} $\mathrm{P}^{\flat },\mathrm{P}^{\sharp }\in \mathbb{R}$ and $%
\mathrm{M}^{\flat },\mathrm{M}^{\sharp }$ are nonempty, norm-bounded and weak%
$^{\ast }$-compact. $\mathrm{M}^{\sharp }$ is additionally a convex subset
of $\mathrm{dom}(g_{-}^{\ast })$.
\end{proposition}

\begin{proof}
Conditions TF1--TF3 imply Conditions B1--B3 of Section \ref{sect conc conv
Bogo} for $K=\mathcal{P}(T)$, $f=h$ and $\tau _{\pm }:\mathcal{P}%
(T)\rightarrow \mathcal{X}_{\pm }$ defined by $\tau _{\pm }(\mu )\doteq
\theta _{\pm }(\mu _{S})$ for any $\mu \in \mathcal{P}(T)$. Indeed, $%
\mathcal{P}(T)$ is a compact convex Hausdorff space, $h$ is affine and upper
semicontinuous (Definition \ref{uod} and Proposition \ref{Affinity of the
entropy copy(1)}), and $\tau _{\pm }$ is affine and continuous in this case,
because the weak$^{\ast }$ topology of $\mathcal{P}(T)$ is metrizable and
the weak$^{\ast }$ convergence of $\mu $ in $\mathcal{P}(T)$ implies the
point-wise convergence of $\mu _{S}$ in $\mathcal{M}(S)$. Note also that 
\begin{equation*}
\left\Vert \tau _{\pm }\right\Vert _{\infty }\doteq \sup \left\{ \left\Vert
\tau _{\pm }\left( \mu \right) \right\Vert _{\mathcal{X}}:\mu \in \mathcal{P}%
\left( T\right) \right\} \leq \left\Vert \theta _{\pm }\right\Vert _{\mathrm{%
op}}\sup_{\mu \in \mathcal{P}\left( T\right) }\left\Vert \mu _{S}\right\Vert
_{\mathrm{op}}=\left\Vert \theta _{\pm }\right\Vert _{\mathrm{op}}\ .
\end{equation*}%
So, it suffices to invoke Propositions \ref{Proposition importante
bogoluibov01 copy(3)} and \ref{Proposition importante bogoluibov01 copy(4)}
to get the assertion.
\end{proof}

Much more features of the thermodynamic game can be considered. This is done
in Section \ref{Decision}. For example, we can define decision rules for the
thermodynamic game (Definition \ref{decision rules}) and find their
particular properties. E.g., if we assume Conditions TF1--TF3 but also that $%
g_{+}^{\ast }$ is continuous, $\mathrm{dom}(g_{-}^{\ast })=\mathcal{X}%
_{-}^{\ast }$ and $g_{-}^{\ast }:\mathcal{X}_{-}^{\ast }\rightarrow \mathbb{R%
}$ is strictly convex, then we can deduce from Proposition \ref{decision
rules prop} that, for all $x_{+}\in \mathrm{M}^{\flat }$, $M^{\flat }(x_{+})$
contains exactly one element $x_{-}(x_{+})\in \mathcal{X}_{+}^{\ast }$. In
this case, the function $x_{+}\mapsto x_{-}(y_{+})$ from $\mathrm{M}^{\flat
} $ to $\mathcal{X}_{-}^{\ast }$ is continuous, both sets being endowed with
the weak$^{\ast }$ topology. This means that there exists a unique $\flat $%
-decision rule, see Definition \ref{decision rules}.

\begin{remark}
\label{Remark1}\mbox{ }\newline
The simpler case where $g_{+}=0$ or $g_{-}=0$ can be studied in the same way
(mutatis mutandis). For example, if $g_{+}=0$ and $g_{-}\neq 0$, we consider 
$P^{\flat }\doteq \inf P_{\mathrm{NL}}\left( \mathcal{X}_{-}^{\ast }\right) $
and $M^{\flat }\left( 0\right) $, since the set $\mathrm{M}^{\flat }$ does
not make sense in this case. Similarly, if $g_{+}\neq 0$ and $g_{-}=0$, we
consider $\mathrm{P}^{\flat }\doteq \sup P^{\flat }\left( \mathcal{X}%
_{+}^{\ast }\right) $ and $\mathrm{M}^{\flat }$ with the new function $%
P^{\flat }\doteq P_{\mathrm{L}}\left( \cdot ,0\right) -g_{+}^{\ast }$. Then
Proposition \ref{Proposition importante bogoluibov01 copy(5)} holds true,
mutatis mutandis. We refrain from going into further detail, as this special
situation is already discussed in Section \ref{Abstract Theory}.
\end{remark}

\subsubsection{Validity of Bogoliubov linearizations}

Under Condition \ref{Condition essential} nonlinear equilibrium measures
(Definition \ref{Equilibrium measures def copy(1)}) form the weak$^{\ast }$%
-compact set (\ref{NL eq measure}), that is,%
\begin{equation*}
E_{P}\doteq \left\{ \mu \in \mathcal{P}\left( T\right) :P\left( \mu \right)
=\sup P\left( \mathcal{P}\left( T\right) \right) \right\} \ .
\end{equation*}%
However, if we look at the same variational problem with the pressure $P$
defined by (\ref{eq 21}) but under the \emph{more general} Conditions
TF1--TF3, it is not a priori clear whether the above set is nonempty, since $%
P$ is not necessarily weak$^{\ast }$-upper semicontinuous. So, in this case,
it is instead convenient to use%
\begin{equation}
E_{P}\doteq \left\{ \mu \in \mathcal{P}\left( T\right) :\exists (\mu
_{n})_{n\in \mathbb{N}}\subseteq \mathcal{P}\left( T\right) \mathrm{\ }\text{%
with }\lim_{n\rightarrow \infty }\mu _{n}=\mu \text{ and\ }%
\lim_{n\rightarrow \infty }P(\mu _{n})=\sup P\left( \mathcal{P}\left(
T\right) \right) \right\} \ ,  \label{Ep}
\end{equation}%
as the basic definition for the (nonempty) set of nonlinear equilibrium
measures, similar to (\ref{sdsdsdfdgdfhfghhf}). Above, the limit $\mu
_{n}\rightarrow \mu $ of $T$-invariant probability measures is taken with
respect to the weak$^{\ast }$ topology. Surprisingly, this set is exactly
what one would like from a conceptual point of view, the set 
\begin{equation*}
M_{P}\doteq \left\{ \mu \in \mathcal{P}\left( T\right) :P\left( \mu \right)
=\sup P\left( \mathcal{P}\left( T\right) \right) \right\} \ 
\end{equation*}%
of the usual maximizers of $P$. This is proven in the next theorem.

In fact, the sets $\mathrm{M}^{\flat }\subseteq \mathcal{X}_{+}^{\ast }$\
and $M^{\flat }(y_{+})\subseteq \mathcal{X}_{-}^{\ast }$, $y_{+}\in \mathcal{%
X}_{+}^{\ast }$, of optimizers can be used to obtain all elements of $E_{P}$
via the weak$^{\ast }$-compact convex sets 
\begin{equation*}
E_{\mathrm{L}}\left( y_{+},y_{-}\right) \doteq \left\{ \nu \in \mathcal{P}%
\left( T\right) :\mathrm{P}\left( y_{+},y_{-},\nu \right) =\sup \mathrm{P}%
\left( y_{+},y_{-},\mathcal{P}\left( T\right) \right) \doteq P_{\mathrm{L}%
}\left( y_{+},y_{-}\right) \right\} \ ,\qquad y_{\pm }\in \mathcal{X}_{\pm
}^{\ast }\ ,
\end{equation*}%
of approximating linear equilibrium measures. Note that the weak$^{\ast }$
compactness and convexity of $E_{\mathrm{L}}(y_{+},y_{-})$ for all $y_{\pm
}\in \mathcal{X}_{\pm }^{\ast }$ is a direct consequence of the affine
property and weak$^{\ast }$-upper semicontinuity of the approximating
pressure $\mu \mapsto \mathrm{P}(y_{+},y_{-},\mu )$ defined by (\ref%
{approximating pressure}), thanks to Proposition \ref{Affinity of the
entropy copy(1)} and the weak$^{\ast }$-upper semicontinuity of the entropy
(cf. Definition \ref{uod}).

As is usual, here the so-called \emph{subdifferential} of a convex function $%
g$ at $z\in \mathrm{dom}\left( g\right) $ is denoted $\partial g(z)$, see
Section \ref{Legendre-Fenchel transform section} and in particular Equation (%
\ref{subdifferential}) for more details. Recall that if $g$ is
Gateaux-differentiable at $z$ then $\partial g(z)$ must be a singleton.

We are now in a position to show how the $\mu $-linearized\ variational
problem $\mathrm{P}^{\flat }$ (\ref{bogo approx})--(\ref{pression approx})
can be used to completely solve the original problem, including not only the
calculation of the nonlinear pressure but also the nonlinear equilibrium
measures:

\begin{theorem}[Nonlinear pressure and nonlinear equilibrium measures]
\label{Theorem-main1}\mbox{ }\newline
Assume Conditions TF1--TF3. Let $\tau _{\pm }(\mu )\doteq \theta _{\pm }(\mu
_{S})$ for any $T$-invariant probability measure $\mu \in \mathcal{P}(T)$. 
\newline
\emph{(i)} Nonlinear pressure:%
\begin{equation*}
\sup P\left( \mathcal{P}\left( T\right) \right) =\mathrm{P}^{\flat }\doteq
\sup_{y_{+}\in \mathcal{X}_{+}^{\ast }}\inf_{y_{-}\in \mathcal{X}_{-}^{\ast
}}\left\{ P_{\mathrm{L}}\left( y_{+},y_{-}\right) +g_{-}^{\ast }\left(
y_{-}\right) -g_{+}^{\ast }\left( y_{+}\right) \right\} \ .
\end{equation*}%
\emph{(ii)} Self-consistency conditions: For any $x_{+}\in \mathrm{M}^{\flat
}$ and $x_{-}\in M^{\flat }\left( x_{+}\right) $, the set 
\begin{equation*}
E_{\mathrm{L}}^{\mathrm{sc}}\left( x_{+},x_{-}\right) \doteq \left\{ \mu \in
E_{\mathrm{L}}\left( x_{+},x_{-}\right) :x_{-}\in \partial g_{-}\left( \tau
_{-}\left( \mu \right) \right) \right\} \equiv E_{\mathrm{L}}^{\mathrm{sc}%
}\left( x_{+}\right)
\end{equation*}%
of self-consistent equilibrium measures of Bogoliubov linearizations is
nonempty, convex and weak$^{\ast }$-compact, and does not depend upon the
choice of $x_{-}\in M^{\flat }\left( x_{+}\right) $. Furthermore, for all $%
x_{+}\in \mathrm{M}^{\flat }$,%
\begin{equation*}
E_{\mathrm{L}}^{\mathrm{sc}}\left( x_{+}\right) \subseteq \left\{ \mu \in 
\mathcal{P}\left( T\right) :x_{+}\in \partial g_{+}\left( \tau _{+}\left(
\mu \right) \right) \right\} \ .
\end{equation*}%
\emph{(iii)} Nonlinear equilibrium measures: $E_{P}=M_{P}$ is weak$^{\ast }$%
-compact and corresponds to the union of all above sets of self-consistent
equilibrium measures, that is, 
\begin{equation*}
E_{P}=M_{P}=\bigcup_{x_{+}\in \mathrm{M}^{\flat }}E_{\mathrm{L}}^{\mathrm{sc}%
}\left( x_{+}\right) \ .
\end{equation*}%
If the function $g_{+}:\mathcal{X}_{+}\rightarrow \mathbb{R}$ is
additionally Gateaux%
%TCIMACRO{\TeXButton{\-}{\-}}%
%BeginExpansion
\-%
%EndExpansion
-differentiable, then the above union is disjoint.
\end{theorem}

\begin{proof}
As explained in Proposition \ref{Proposition importante bogoluibov01 copy(5)}%
, Conditions TF1--TF3 imply Conditions B1--B3 of Section \ref{sect conc conv
Bogo} for $K=\mathcal{P}(T)$, $f=h$ and $\tau _{\pm }:\mathcal{P}%
(T)\rightarrow \mathcal{X}_{\pm }$ defined by $\tau _{\pm }(\mu )\doteq
\theta _{\pm }(\mu _{S})$ for any $\mu \in \mathcal{P}(T)$. The assertions
are therefore direct consequences of Theorem \ref{Proposition importante
bogoluibov01 copy(2)} (i)--(iii).
\end{proof}

Compared to \cite[Theorem C]{TF4}, in particular Equations (\ref{maximizers
z})--(\ref{maximizers zbis}), the computation of the sets $\mathrm{M}^{\flat
}$, $M^{\flat }\left( x_{+}\right) $, $x_{+}\in \mathrm{M}^{\flat }$, and $%
E_{\mathrm{L}}\left( x_{+},x_{-}\right) $ of optimizers is much simpler,
while our setup allows for parameter sets in \textbf{infinite-dimensional}
spaces $\mathcal{X}_{\pm }$. The cost is that we have to separate the convex
and concave parts of the nonlinearity. This cannot be avoided and is not
just a technical artefact. On the other hand, the $C^{1}$-character of the
nonlinearity is not needed\footnote{%
In contrast to the corresponding results of \cite{TF4}.}, although it can be
very useful for solving the corresponding variational problems.

Note that the (nonempty)\ compact convex set $E_{\mathrm{L}}^{\mathrm{sc}%
}\left( x_{+}\right) $ is naturally called the set of self-consistent
equilibrium measures\ of Bogoliubov linearizations. Indeed, take for
instance $\mathcal{X}_{\pm }=\mathbb{R}$, $g_{-}(x)=g_{+}(x)=x^{2}/2$ and $%
\tau _{\pm }(\mu )\doteq \theta _{\pm }(\mu _{S})=\mu (\varphi )$ for fixed H%
\"{o}lder potentials $\varphi _{\pm }\in C^{\alpha }(\Sigma )$ ($\alpha \in
(0,1]$). In this case we canonically identify the dual spaces $\mathcal{X}%
_{\pm }^{\ast }$ with $\mathbb{R}$. Then, for all $T$-invariant probability
measures $\mu \in \mathcal{P}\left( T\right) $,%
\begin{equation*}
\partial g_{\pm }\left( \tau \left( \mu \right) \right) =\left\{ \mu \left(
\varphi _{\pm }\right) \right\} \ .
\end{equation*}%
Thus, $\mu \in E_{\mathrm{L}}^{\mathrm{sc}}\left( x_{+}\right) $ means that $%
\mu $ is the unique linear equilibrium probability measure of the H\"{o}lder
potential $(x_{+}\varphi _{+}-x_{-}\varphi _{-})$ and satisfies $\mu
(\varphi _{\pm })=x_{\pm }$. By Theorem \ref{Theorem-main1}, it implies in
particular that, for any fixed $\mu \in E_{P}=M_{P}$, $\mu (\varphi _{\pm
})=x_{\pm }$ for all $x_{+}\in \mathrm{M}^{\flat }$ and $x_{-}\in M^{\flat
}\left( x_{+}\right) $, i.e., $\mathrm{M}^{\flat }=\{x_{+}\}$ and $M^{\flat
}(x_{+})=\{x_{-}\}$. Compare this with Corollary \ref{Theorem-main1 copy(2)}
below. More generally, similar results are true if the functions $g_{\pm }$
are Gateaux-differentiable and their gradient mapping is injective.

This last example refers to the H\"{o}lder case, which is very important in
the linear thermodynamic formalism. Indeed, when a potential $\varphi $ is
of H\"{o}lder class, i.e. $\varphi \in C^{\alpha }(\Sigma )$ for a certain $%
\alpha \in (0,1]$, recall that the linear equilibrium measure $\mu _{\varphi
}$ is unique and ergodic (see \cite{LMMS,ACR}). In the nonlinear framework,
this leads us to define the H\"{o}lder-type for pairs $(\theta _{-},\theta
_{+})$ of linear functions, in Definition \ref{sect Holder type theta}. In
this particular case the nonlinear equilibrium measures have strong
properties. Among other things, they are always ergodic.

\begin{corollary}[Nonlinear equilibrium measures -- H\"{o}lder case]
\label{Theorem-main1 copy(2)}\mbox{ }\newline
Assume Conditions TF1--TF3 and that the pair $(\theta _{-},\theta _{+})$ of
linear functions is H\"{o}lder-type. \newline
\emph{(i)} Ergodic nonlinear equilibrium measures: For any $x_{+}\in \mathrm{%
M}^{\flat }$ and $x_{-}\in M^{\flat }\left( x_{+}\right) $, 
\begin{equation*}
E_{\mathrm{L}}^{\mathrm{sc}}\left( x_{+}\right) =E_{\mathrm{L}}^{\mathrm{sc}%
}\left( x_{+},x_{-}\right) =E_{\mathrm{L}}\left( x_{+},x_{-}\right) =\left\{
\mu _{x_{+}}\right\} \subseteq \mathcal{P}_{\mathrm{erg}}\left( T\right)
\end{equation*}%
and $E_{P}=M_{P}$ is a nonempty weak$^{\ast }$-compact set of ergodic
probability measures:%
\begin{equation*}
E_{P}=M_{P}=\left\{ \mu _{x_{+}}:x_{+}\in \mathrm{M}^{\flat }\right\}
\subseteq \mathcal{P}_{\mathrm{erg}}\left( T\right) \ .
\end{equation*}%
\emph{(ii)} If $g_{+}:\mathcal{X}_{+}^{\ast }\rightarrow \mathbb{R}$ is
Gateaux%
%TCIMACRO{\TeXButton{\-}{\-}}%
%BeginExpansion
\-%
%EndExpansion
-differentiable then the mapping $x_{+}\mapsto \mu _{x_{+}}$ from the weak$%
^{\ast }$-compact $\mathrm{M}^{\flat }\subseteq \mathcal{X}_{+}^{\ast }$ to $%
E_{P}\subseteq \mathcal{P}_{\mathrm{erg}}\left( T\right) $ is a
homeomorphism with respect to the weak$^{\ast }$ topology of $\mathrm{M}%
^{\flat }$ and $E_{P}$.
\end{corollary}

\begin{proof}
Use Theorem \ref{Theorem-main1}, Proposition \ref{equilibrium state
selections prop} (iii) and the fact that the nonlinear equilibrium measure
for H\"{o}lder potentials is unique and ergodic (see \cite{LMMS,ACR}).
\end{proof}

The conditions of Corollary \ref{Theorem-main1 copy(2)} are satisfied in
most cases of interest. For instance, consider Example \ref{example F II}
with $\varphi _{1},\ldots ,\varphi _{N}\in C^{\alpha }(\Sigma )$ for some $%
\alpha \in (0,1]$ and a lower semicontinuous convex function $F=F_{+}$\ on $%
\mathbb{R}^{N}$ for which the Legendre-Fenchel transform is strictly convex
(e.g., $N=1$ and $F(x)=x^{2}/2$). We refrain to discuss in more detail such
explicit examples here, since this is already done in \cite{Bru-Pedra-Lopes1}%
.

Recall that $G_{P}\doteq \overline{\mathrm{co}}(E_{P})$ is the set of
so-called generalized nonlinear equilibrium measures, which can be
equivalently defined in a more fundamental way by a property stated below in
Theorem \ref{Theorem-main1 copy(1)} (ii). See Definition \ref{Gneralized
equilibrium measures} and discussions thereafter. Under the conditions of
Corollary \ref{Theorem-main1 copy(2)}, by the Milman theorem \cite[%
Proposition 1.5]{Phe}\ and observing that ergodic measures are extreme in
the convex set $\mathcal{P}\left( T\right) $ of $T$-invariant measures, the
set $G_{P}$ of generalized nonlinear equilibrium measures is a \emph{face}%
\footnote{%
A face $F$ of a convex set $K$ is a subset of $K$ with the property that, if 
$\rho =\lambda _{1}\rho _{1}+\cdots +\lambda _{n}\rho _{n}\in F$ with $\rho
_{1},\ldots ,\rho _{n}\in K$, $\lambda _{1},\ldots ,\lambda _{n}\in (0,1)$
and $\lambda _{1}+\cdots +\lambda _{n}=1$, then $\rho _{1},\ldots ,\rho
_{n}\in F$.} of $\mathcal{P}\left( T\right) $, the extreme boundary of which
is precisely the set $E_{P}$ of (simple) nonlinear equilibrium measures. In
particular, as $\mathcal{P}\left( T\right) $ is a Choquet simplex
(Proposition \ref{prop Choquet inv meas}), $G_{P}$ is in this case a \emph{%
Bauer simplex}, i.e., a Choquet simplex whose extreme boundary is compact.

\begin{remark}
\mbox{ }\newline
Theorem \ref{Theorem-main1} and Corollary \ref{Theorem-main1 copy(2)} also
hold when $g_{+}=0$ or $g_{-}=0$ with obvious modifications. In fact, this
case is even simpler. For more details, see Remark \ref{Remark1} as well as
the discussions following Theorem \ref{Proposition importante bogoluibov01
copy(2)}.
\end{remark}

\begin{remark}
\mbox{ }\newline
Theorem \ref{Theorem-main1} and Corollary \ref{Theorem-main1 copy(2)} are,
of course, only relevant when the functions $g_{\pm }$ are not affine,
because otherwise we are dealing with the linear thermodynamic formalism.
\end{remark}

\subsubsection{Validity of the thermodynamic game\label{Section game +
measusre}}

The conservative values $\mathrm{P}^{\flat }$ (\ref{bogo approx}) and $%
\mathrm{P}^{\sharp }$ (\ref{bogo approxdiese}) of the thermodynamic game are
related to two variational problems on $T$-invariant measures respectively
setup from the functionals $\mathfrak{F}^{\sharp }:\mathcal{P}(T)\rightarrow 
\mathbb{R}$ and $\mathfrak{F}^{\flat }:\mathcal{P}(T)\rightarrow \mathbb{R}$
defined, for any $T$-invariant measure $\mu \in \mathcal{P}\left( T\right) $%
, by%
\begin{eqnarray}
\mathfrak{F}^{\sharp }\left( \mu \right) &\doteq &\Delta ^{g_{+}\circ \theta
_{+}}\left( \mu \right) -g_{-}\circ \theta _{-}\left( \mu _{S}\right)
+h\left( \mu \right) \ ,  \label{F^diese2} \\
\mathfrak{F}^{\flat }\left( \mu \right) &\doteq &\Delta ^{g_{+}\circ \theta
_{+}}\left( \mu \right) -\Delta ^{g_{-}\circ \theta _{-}}\left( \mu \right)
+h\left( \mu \right) \ ,  \label{F^bemol}
\end{eqnarray}%
where $g_{\pm }:\mathcal{X}_{\pm }\rightarrow \mathbb{R}$ are \emph{%
continuous} convex functions and $\theta _{\pm }:\mathcal{M}(S)\rightarrow 
\mathcal{X}_{\pm }$ are two linear transformations that are $\sigma $%
-normal. Here, 
\begin{equation}
\Delta ^{F}\left( \mu \right) \doteq \int_{\mathcal{P}_{\mathrm{erg}}\left(
T\right) }F\left( \nu _{S}\right) \xi _{\mu }\left( \mathrm{d}\nu \right) \
,\qquad \mu \in \mathcal{P}\left( T\right) \ ,  \label{sdsssdsdsd}
\end{equation}%
for any convex and $\sigma $-normal function $F:\mathcal{M}(S)\rightarrow 
\mathbb{R}$, where $\xi _{\mu }$ is the (unique) Choquet measure associated
with $\mu $. To understand why the functional $\Delta ^{F}$ in both (\ref%
{F^diese2}) and (\ref{F^bemol}) is important and natural, we recommend
Proposition \ref{prop Delta F} below. See meanwhile Definition \ref%
{definition de Delta}.

Recall that the entropy $h$ is affine weak$^{\ast }$-upper semicontinuous
(Definition \ref{uod} and Proposition \ref{Affinity of the entropy copy(1)}%
). We can thus apply \cite[Lemma 10.17]{BruPedra2} to $h$ in order to get --
via Equations (\ref{F^bemol})--(\ref{sdsssdsdsd}) -- the ergodic
decomposition of the functional $\mathfrak{F}^{\flat }$: 
\begin{equation}
\mathfrak{F}^{\flat }\left( \mu \right) =\int_{\mathcal{P}_{\mathrm{erg}%
}\left( T\right) }P\left( \nu \right) \xi _{\mu }\left( \mathrm{d}\nu
\right) \text{ },\qquad \mu \in \mathcal{P}\left( T\right) \ ,
\label{assertion sympa}
\end{equation}%
where $\xi _{\mu }$\ is the unique Choquet measure associated to $\mu $ (see
Proposition \ref{prop Choquet inv meas}). It is an elementary result which,
whilst certainly not surprising given that the functional $\mathfrak{F}%
^{\flat }$ is affine, is highlighted here because it is a very useful
observation, enabling the functional $\mathfrak{F}^{\flat }$ to be
calculated in practice via the nonlinear pressure functional $P$.

In this subsection, we only consider \emph{continuous} functions $g_{+}$ and 
$g_{-}$ in order to highlight the singular nature of $\Delta $-functionals,
which even in this simpler context are discontinuous on a dense set (cf.
Theorem \ref{properties de Delta}). The aim is not to obtain the most
general framework possible, but rather to explain the associated variational
problems in an accessible manner, given that they are new in the
thermodynamic formalism.

The functional $\mathfrak{F}^{\flat }$ is exactly the one obtained from the
limit (\ref{sdsdsssdsdsdsdsdsd}). It is therefore quite intuitive. The
functional $\mathfrak{F}^{\sharp }$ turns out to be very useful in the proof
of Proposition \ref{Gamma-regularisation of pressure}, in which \textbf{no}
Bogoliubov linearization is considered. See Equation (\ref{F^diese}). In
fact, these seemingly incidental functionals are very natural and the
theorem below shows their importance for the study of nonlinear equilibrium
measures.

By Definition \ref{uod}, Theorem \ref{properties de Delta} and Proposition %
\ref{Affinity of the entropy copy(1)}, for any (weak$^{\ast }$-)lower
semicontinuous and convex function $g_{-}$ and $\sigma $-normal linear
transformations $\theta _{\pm }:\mathcal{M}(S)\rightarrow \mathcal{X}_{\pm }$%
, the functional $\mathfrak{F}^{\sharp }$ is concave and weak$^{\ast }$%
-upper semicontinuous. Notice also that Definition \ref{uod} implies the weak%
$^{\ast }$-upper semicontinuity of the entropy. Thus, the set 
\begin{equation*}
E_{\mathfrak{F}^{\sharp }}=\left\{ \mu \in \mathcal{P}\left( T\right) :%
\mathfrak{F}^{\sharp }\left( \mu \right) =\sup \mathfrak{F}^{\sharp }\left( 
\mathcal{P}\left( T\right) \right) \right\}
\end{equation*}%
of maximizers of $\mathfrak{F}^{\sharp }$ is convex and weak$^{\ast }$%
-compact. By contrast, the functional $\mathfrak{F}^{\flat }$ is affine but
(generally) \textbf{not} weak$^{\ast }$-upper semicontinuous, which
therefore forces us to take sequences of approximating maximizers (similar
to (\ref{sdsdsdfdgdfhfghhf}) and (\ref{Ep})): 
\begin{equation}
E_{\mathfrak{F}^{\flat }}\doteq \left\{ \mu \in \mathcal{P}\left( T\right)
:\exists \left( \mu _{n}\right) _{n\in \mathbb{N}}\subseteq \mathcal{P}%
\left( T\right) \mathrm{\ }\text{with }\lim_{n\rightarrow \infty }\mu
_{n}=\mu \text{ and\ }\lim_{n\rightarrow \infty }\mathfrak{F}^{\flat }\left(
\mu _{n}\right) =\sup \mathfrak{F}^{\flat }\left( \mathcal{P}\left( T\right)
\right) \right\} ,  \label{equilibriun statebis}
\end{equation}%
where the first limit $\mu _{n}\rightarrow \mu $ refers to the weak$^{\ast }$
topology. See again Theorem \ref{properties de Delta} for the continuity
properties of $\Delta $-functionals.

Now we are in a position to associate the conservative values $\mathrm{P}%
^{\flat }$ (\ref{bogo approx}) and $\mathrm{P}^{\sharp }$ (\ref{bogo
approxdiese}) of the thermodynamic game with two different variational
problems on $T$-invariant measures:

\begin{theorem}[Thermodynamic game and nonlinear equilibrium measures]
\label{Theorem-main1 copy(1)}\mbox{ }\newline
Assume Conditions TF1--TF3 with continuous functions $g_{\pm }$. Let $\tau
_{\pm }(\mu )\doteq \theta _{\pm }(\mu _{S})$ for any $T$-invariant
probability measure $\mu \in \mathcal{P}(T)$.\newline
\emph{(i)} Nonlinear pressures:%
\begin{eqnarray*}
\sup \mathfrak{F}^{\flat }\left( \mathcal{P}\left( T\right) \right) &=&%
\mathrm{P}^{\flat }\doteq \sup_{y_{+}\in \mathcal{X}_{+}^{\ast
}}\inf_{y_{-}\in \mathcal{X}_{-}^{\ast }}\left\{ P_{\mathrm{L}}\left(
y_{+},y_{-}\right) +g_{-}^{\ast }\left( y_{-}\right) -g_{+}^{\ast }\left(
y_{+}\right) \right\} \ , \\
\sup \mathfrak{F}^{\sharp }\left( \mathcal{P}\left( T\right) \right) &=&%
\mathrm{P}^{\sharp }\doteq \inf_{y_{-}\in \mathcal{X}_{-}^{\ast
}}\sup_{y_{+}\in \mathcal{X}_{+}^{\ast }}\left\{ P_{\mathrm{L}}\left(
y_{+},y_{-}\right) +g_{-}^{\ast }\left( y_{-}\right) -g_{+}^{\ast }\left(
y_{+}\right) \right\} \ .
\end{eqnarray*}%
\emph{(ii)} Generalized nonlinear equilibrium measures:%
\begin{eqnarray*}
E_{\mathfrak{F}^{\flat }} &=&\overline{\mathrm{co}}(E_{P})\doteq G_{P}\ , \\
E_{\mathfrak{F}^{\sharp }} &=&\left\{ \mu \in \overline{\mathrm{co}}\left(
E_{\mathrm{NL}}\left( x_{-}\right) \right) :x_{-}\in \mathrm{M}^{\sharp
}\cap \partial g_{-}\left( \tau _{-}\left( \mu \right) \right) \right\} \ ,
\end{eqnarray*}%
where $E_{P}=M_{P}$ is given by Theorem \ref{Theorem-main1} (ii)--(iii),
while 
\begin{equation*}
E_{\mathrm{NL}}\left( x_{-}\right) \doteq \bigcup_{x_{+}\in M^{\sharp
}\left( x_{-}\right) }E_{\mathrm{L}}\left( x_{+},x_{-}\right) \ .
\end{equation*}
\end{theorem}

\begin{proof}
Assertions (i)--(ii) for the case ($\flat $) are direct consequences of
Theorems \ref{Theorem-main1} and \ref{structure of sets of maximizers}. See
also Definition \ref{Gneralized equilibrium measures}. Assertions (i)--(ii)
for the case ($\sharp $) are proven as follows:\medskip

\noindent \underline{Step 1:} The functional $f:\mathcal{P}(T)\rightarrow 
\mathbb{R}$, as defined by%
\begin{equation}
f=\Delta ^{g_{+}\circ \theta _{+}}+h\ ,  \label{f}
\end{equation}%
is affine (in particular concave) and weak$^{\ast }$ upper semicontinuous,
thanks to Definition \ref{uod}, Proposition \ref{Affinity of the entropy
copy(1)} and Theorem \ref{properties de Delta} (i). Thus, using Proposition %
\ref{Proposition importante bogoluibov01} for $K=\mathcal{P}\left( T\right) $%
, $\mathcal{X}=\mathcal{X}_{-}$, $g=g_{-}$, $\tau (\mu )=\theta _{-}(\mu
_{S})$ and $f$ defined by (\ref{f}), we obtain that%
\begin{equation}
\sup \mathfrak{F}^{\sharp }\left( \mathcal{P}\left( T\right) \right)
=\inf_{y_{-}\in \mathcal{X}_{-}^{\ast }}\left\{ \sup \mathcal{G}%
_{y_{-}}\left( \mathcal{P}\left( T\right) \right) +g_{-}^{\ast }\left(
y_{-}\right) \right\}  \label{Eq diese1}
\end{equation}%
with 
\begin{equation*}
\mathcal{G}_{y_{-}}\left( \mu \right) \doteq \Delta ^{g_{+}\circ \theta
_{+}}\left( \mu \right) +h\left( \mu \right) -y_{-}\circ \theta _{-}(\mu
_{S})\ ,\qquad \mu \in \mathcal{P}\left( T\right) \ .
\end{equation*}%
In addition, Corollary \ref{corollary importante bogoluibov01 copy(1)} also
yields 
\begin{equation}
E_{\mathfrak{F}^{\sharp }}=\left\{ \mu \in E_{\mathcal{G}_{x_{-}}}:x_{-}\in 
\mathrm{M}^{\sharp }\cap \partial g_{-}\left( \tau _{-}\left( \mu \right)
\right) \right\} \ ,  \label{Eq diese2}
\end{equation}%
where%
\begin{equation*}
E_{\mathcal{G}_{y_{-}}}\doteq \left\{ \mu \in \mathcal{P}\left( T\right) :%
\mathcal{G}_{y_{-}}\left( \mu \right) =\sup \mathcal{G}_{y_{-}}\left( 
\mathcal{P}\left( T\right) \right) \right\}
\end{equation*}%
for any continuous linear functional $y_{-}\in \mathcal{X}_{-}^{\ast }$.
\medskip

\noindent \underline{Step 2:} Note that, with obvious adaptations in the
proofs, Theorem \ref{structure of sets of maximizers} and Corollary \ref%
{Choquet decomposition} also hold if we add any arbitrary affine and weak$%
^{\ast }$ continuous function to the entropy $h$. It follows that 
\begin{equation}
\sup \mathcal{G}_{y_{-}}\left( \mathcal{P}\left( T\right) \right) =\sup
P_{y_{-}}\left( \mathcal{P}\left( T\right) \right)  \label{Eq diese4}
\end{equation}%
with 
\begin{equation*}
P_{y_{-}}\left( \mu \right) \doteq g_{+}\circ \theta _{+}\left( \mu
_{S}\right) +h\left( \mu \right) -y_{-}\circ \theta _{-}(\mu _{S})\ ,\qquad
\mu \in \mathcal{P}\left( T\right) \ ,
\end{equation*}%
while (the extended version of) Corollary \ref{Choquet decomposition} yields 
\begin{equation}
E_{\mathcal{G}_{y_{-}}}=\overline{\mathrm{co}}(E_{P_{y_{-}}})\ ,\qquad
y_{-}\in \mathcal{X}_{-}^{\ast }\ ,  \label{Eq diese5}
\end{equation}%
with 
\begin{equation*}
E_{P_{y_{-}}}=\left\{ \mu \in \mathcal{P}\left( T\right) :P_{y_{-}}\left(
\mu \right) =\sup P_{y_{-}}\left( \mathcal{P}\left( T\right) \right) \right\}
\end{equation*}%
for any continuous linear functional $y_{-}\in \mathcal{X}_{-}^{\ast }$.
\medskip

\noindent \underline{Step 3:} Now, we apply the simplified version of
Theorem \ref{Proposition importante bogoluibov01 copy(2)} (see the
discussions following Theorem \ref{Proposition importante bogoluibov01
copy(2)}) with $K=\mathcal{P}\left( T\right) $, $g_{-}=0$, and 
\begin{equation*}
\tau _{+}(\mu )=\theta _{+}(\mu _{S})\ ,\quad f\left( \mu \right) =h\left(
\mu \right) -y_{-}\circ \theta _{-}(\mu _{S})\ ,\qquad \mu \in \mathcal{P}%
\left( T\right) \ ,
\end{equation*}%
observing that $g_{+}$ is by assumption a continuous and convex function,
that the functional $f$ defined as above is affine and weak$^{\ast }$-upper
semicontinuous (cf. Definition \ref{uod} and Proposition \ref{Affinity of
the entropy copy(1)}) and that $\tau _{+}$ is a continuous linear
transformation, $\theta _{+}$ being a $\sigma $-normal linear
transformation. Doing so, for any continuous linear functional $y_{-}\in 
\mathcal{X}_{-}^{\ast }$, we get 
\begin{equation}
\sup P_{y_{-}}\left( \mathcal{P}\left( T\right) \right) =\sup_{y_{+}\in 
\mathcal{X}_{+}^{\ast }}\left\{ P_{\mathrm{L}}\left( y_{+},y_{-}\right)
-g_{+}^{\ast }\left( y_{+}\right) \right\}  \label{Eq diese7}
\end{equation}%
and 
\begin{equation}
E_{P_{y_{-}}}=\bigcup_{y_{+}\in M^{\sharp }\left( y_{-}\right) }E_{\mathrm{L}%
}\left( y_{+},y_{-}\right) \ .  \label{Eq diese8}
\end{equation}%
It remains now to combine the three steps. \medskip

\noindent \underline{Step 4:} We combine Equations (\ref{Eq diese1}), (\ref%
{Eq diese4}) and (\ref{Eq diese7}) to get Assertion (i) for the case ($%
\sharp $). Assertion (ii) for the case ($\sharp $) is a consequence of
Equations (\ref{Eq diese2}), (\ref{Eq diese5}) and (\ref{Eq diese8}).
\end{proof}

By Theorem \ref{Theorem-main1 copy(1)}, the thermodynamic game has a direct
interpretation in terms of equilibrium measures. This is a very useful
information that can be employed, for example, to show that generally $%
\mathrm{P}^{\sharp }>\mathrm{P}^{\flat }$, that is, the $\sup $ and $\inf $
in $\mathrm{P}^{\flat }$ (\ref{bogo approx}) or $\mathrm{P}^{\sharp }$ (\ref%
{bogo approxdiese}) do not commute, as is done in the quantum case in \cite[%
Discussions after Theorem 2.36]{BruPedra2}. The interpretation of the
conservative values $\mathrm{P}^{\flat }$ (\ref{bogo approx}) and $\mathrm{P}%
^{\sharp }$ (\ref{bogo approxdiese}) as variational problems on $T$%
-invariant measures given by Theorem \ref{Theorem-main1 copy(1)} also turns
out to be crucial for the Kac,\ or van der Waals, limit, very well-known in
statistical mechanics. This is exploited in \cite{Kac} in the context of
quantum lattice models, but the arguments used, in the light of the results
presented here, can clearly be adapted to produce a version of \cite{Kac}
for the nonlinear thermodynamic formalism. We therefore believe that the
above results on the thermodynamic game can be useful for further
developments of the nonlinear version of the thermodynamic formalism. Notice
additionally that Theorem \ref{Theorem-main1 copy(1)} follows relatively
easily from our general approach, which is presented in\ Section \ref%
{Abstract Theory}.

We conclude this subsection by examining whether generalized equilibrium
states can be true maximizers of the affine pressure $\mathfrak{F}^{\flat }$%
. So, consider the set 
\begin{equation*}
M_{\mathfrak{F}^{\flat }}\doteq \left\{ \mu \in \mathcal{P}\left( T\right) :%
\mathfrak{F}^{\flat }\left( \mu \right) =\sup \mathfrak{F}^{\flat }\left( 
\mathcal{P}\left( T\right) \right) \right\} \subseteq E_{\mathfrak{F}^{\flat
}}
\end{equation*}%
of strict maximizers of $\mathfrak{F}^{\flat }$. Compare this definition
with (\ref{equilibriun statebis}). By the affine property of $\mathfrak{F}%
^{\flat }$, $M_{\mathfrak{F}^{\flat }}$ is convex. One might wonder whether
the equality $M_{\mathfrak{F}^{\flat }}=E_{\mathfrak{F}^{\flat }}$ holds
true even if $\mathfrak{F}^{\flat }$ is not necessarily upper
semicontinuous. In fact, in general, one only has the strict inclusion $M_{%
\mathfrak{F}^{\flat }}\varsubsetneq E_{\mathfrak{F}^{\flat }}$. The reason
is that, because of the affine property of $\mathfrak{F}^{\flat }$, $M_{%
\mathfrak{F}^{\flat }}$ must be a (convex) face of $\mathcal{P}\left(
T\right) $, but this is generally not the case for $E_{\mathfrak{F}^{\flat
}} $. We proved this in the context of quantum lattice systems, but our
arguments can easily be adapted to the nonlinear thermodynamic formalism,
see \cite[Lemma 9.8]{BruPedra2}. However, if the linear transformations $%
\theta _{\pm }$ are H\"{o}lder-type then the equality $M_{\mathfrak{F}%
^{\flat }}=E_{\mathfrak{F}^{\flat }}$ does indeed hold. This is proven in
the next corollary, deduced from Theorem \ref{Theorem-main1 copy(1)}
(i)--(ii) and Equation (\ref{assertion sympa}).

\begin{corollary}[Generalized equilibrium measures and maximizers of the
affine pressure]
\label{Theorem-main1 copy(3)}\mbox{ }\newline
Assume Conditions TF1--TF3 with continuous functions $g_{\pm }$. Then,%
\begin{equation}
\mathrm{co}\left( E_{P}\cap \mathcal{P}_{\mathrm{erg}}(T)\right) \subseteq
M_{\mathfrak{F}^{\flat }}\subseteq \overline{\mathrm{co}}\left( E_{P}\right)
=E_{\mathfrak{F}^{\flat }}  \label{eq iii}
\end{equation}%
and if $\theta _{\pm }$ are additionally H\"{o}lder-type then $E_{\mathfrak{F%
}^{\flat }}=M_{\mathfrak{F}^{\flat }}$.
\end{corollary}

\begin{proof}
Clearly, $M_{\mathfrak{F}^{\flat }}\subseteq E_{\mathfrak{F}^{\flat }}$ and
thus, by Theorem \ref{Theorem-main1 copy(1)} (ii) for the case ($\flat $), $%
M_{\mathfrak{F}^{\flat }}\subseteq \overline{\mathrm{co}}(E_{P})$. Note that 
$\mathfrak{F}^{\flat }=P$ on ergodic measures, by the definition of $\Delta $%
-functionals. Hence, by combining Theorem \ref{Theorem-main1} (i) with
Theorem \ref{Theorem-main1 copy(1)} (i) for the case ($\flat $), we deduce
that $E_{P}\cap \mathcal{P}_{\mathrm{erg}}(T)\subseteq M_{\mathfrak{F}%
^{\flat }}$ and (\ref{eq iii}) follows from the affine property of $%
\mathfrak{F}^{\flat }$. If $\theta _{\pm }$ are additionally H\"{o}lder-type
then, by Corollary \ref{Theorem-main1 copy(2)} (i), $E_{P}$ is a weak$^{\ast
}$-compact set of ergodic measures only. Hence, the Choquet measure in $%
\overline{\mathrm{co}}(E_{P})$ of any $\mu \in \overline{\mathrm{co}}(E_{P})$%
, as given by Corollary \ref{Choquet decomposition} ($E_{\mathfrak{F}^{\flat
}}=\overline{\mathrm{co}}(E_{P})$), coincides with its unique Choquet
measure in $\mathcal{P}(T)$, as given by Proposition \ref{prop Choquet inv
meas}. As explained in Corollary \ref{Choquet decomposition} (cf. the Milman
theorem \cite[Proposition 1.5]{Phe}), this measure must be supported in $%
E_{P}$. Therefore, we can combine Theorem \ref{Theorem-main1 copy(1)} (ii)
with Equation (\ref{assertion sympa}) to get $\overline{\mathrm{co}}%
(E_{P})=M_{\mathfrak{F}^{\flat }}$ when $\theta _{\pm }$ are additionally H%
\"{o}lder-type.
\end{proof}

As a consequence, in the H\"{o}lder case, generalized nonlinear equilibrium
measures are precisely all (strict) maximizers of the affine pressure
functional $\mathfrak{F}^{\flat }$, which is a very satisfactory
characterization of such $T$-invariant probability measures.

\subsection{Technical Results: $\Delta $-functionals and nonlinear pressures 
\label{Non-linear and Affine Pressures}\label{Technical Results}}

We address the maximization of the \emph{nonlinear} pressure functional $P:%
\mathcal{P}\left( T\right) \rightarrow \mathbb{R}$ defined on $T$-invariant
measures by (\ref{nonlinear pressure functional}). Although $P$ is generally
neither convex nor concave, this variational problem can be studied via
another real-valued functional on $T$-invariant measures, which is \emph{%
affine}, i.e., both concave and convex. This fact may seem curious, or even
strange, at a first sight. To properly understand this, we must first
precisely define the affine functional, which requires a few preliminary
definitions and observations.

Recall that $\mathbb{E}_{n}[\varphi ]$ stands for Birkhoff sums (\ref%
{sdsdsdssdsd}) on $C(\Sigma )$. It defines a linear contraction $\mathbb{E}%
_{n}$ mapping $S$ to itself. Now, we use the following observation:

\begin{lemma}[Properties of ergodic measures]
\label{lemma In conv pw}\mbox{ }\newline
For any ergodic measure $\mu \in \mathcal{P}_{\mathrm{erg}}(T)$ and for $\mu 
$-a.s. $\sigma \in \Sigma $, 
\begin{equation*}
\lim_{n\rightarrow \infty }\mathbb{E}_{n}\left[ \varphi \right] \left(
\sigma \right) =\mu \left( \varphi \right) \doteq \mu _{S}\left( \varphi
\right) \ ,\qquad \varphi \in S\ .
\end{equation*}
\end{lemma}

\begin{proof}
Recall that any countable intersection of measurable sets of full measure
has full measure. Therefore, we infer from Proposition\textit{\ }\ref{prop
erg disp free} that, for any countable subset $\tilde{S}\subseteq S$ and for 
$\mu $-a.s. $\sigma \in \Sigma $,%
\begin{equation*}
\lim_{n\rightarrow \infty }\mathbb{E}_{n}\left[ \varphi \right] =\mu \left(
\varphi \right) \ ,\qquad \varphi \in \tilde{S}\ .
\end{equation*}%
Since $S\subseteq C(\Sigma )$ is separable (with respect to the supremum
norm) and for any $n\in \mathbb{N}$, $\sigma \in \Sigma $ and $\varphi
_{1},\varphi _{2}\in C(\Sigma )$, 
\begin{equation*}
\max \left\{ |\mathbb{E}_{n}\left[ \varphi _{1}\right] \left( \sigma \right)
-\mathbb{E}_{n}\left[ \varphi _{2}\right] \left( \sigma \right) |,|\mu
\left( \varphi _{1}\right) -\mu \left( \varphi _{2}\right) |\right\} \leq
\left\Vert \varphi _{1}-\varphi _{2}\right\Vert _{\infty }\ ,
\end{equation*}%
the lemma directly follows.
\end{proof}

\begin{remark}
\mbox{ }\newline
Ergodic probability measures are mutually singular. See \cite[Theorem 6.10
(iv)]{Walters}. Note also that we restrict Lemma \ref{lemma In conv pw} to
functions $\varphi \in S$ in view of our applications. By linearity, the
same assertion holds true for all $\varphi \in C(\Sigma )$.
\end{remark}

Recall that a function $F:\mathcal{M}(S)\rightarrow \mathbb{R}$ is $\sigma $%
-normal when, for any bounded sequence $(f_{n})_{n\in \mathbb{N}}\subseteq 
\mathcal{M}(S)$ converging point-wise to $f\in \mathcal{M}(S)$, one has%
\begin{equation*}
\lim_{n\rightarrow \infty }F\left( f_{n}\right) =F\left( f\right) \text{ }.
\end{equation*}%
We also remind that $\mathcal{P}_{\mathrm{erg}}(T)$ denotes the set of
ergodic (or extreme) measures of the space $\mathcal{P}(T)$ of $T$-invariant
probability measures. Note also that, given $\sigma \in \Sigma $ and $n\in 
\mathbb{N}$, Equation (\ref{sdsdsdssdsd}) defines a continuous and bounded
function $\mathbb{E}_{n,\sigma }:S\rightarrow \mathbb{R}$ by 
\begin{equation}
\mathbb{E}_{n,\sigma }\left[ \varphi \right] \doteq \mathbb{E}_{n}\left[
\varphi \right] (\sigma )\ ,\qquad \varphi \in S\ .  \label{eq def Ensigbis}
\end{equation}%
In particular, $\mathbb{E}_{n,\sigma }\in \mathcal{M}(S)$. Clearly, for any $%
\sigma \in \Sigma $ and $n\in \mathbb{N}$, $\mathbb{E}_{n,\mathbb{\sigma }%
^{\prime }}$ converges point-wise to $\mathbb{E}_{n,\sigma }$, as $\sigma
^{\prime }\rightarrow \sigma $. Therefore, for any $\sigma $-normal function 
$F:\mathcal{M}(S)\rightarrow \mathbb{R}$, the mapping $\sigma \mapsto F(%
\mathbb{E}_{n,\sigma })$ is a continuous function from $\Sigma $ to $\mathbb{%
R}$. By a slight abuse of notation, we denote this function by
\textquotedblleft $F\circ \mathbb{E}_{n}$\textquotedblright .

We can now show an important property that leads to a convenient definition
of the affine pressure functional on $T$-invariant measures.

\begin{proposition}
\label{prop Delta F}\mbox{ }\newline
Let $F:\mathcal{M}(S)\rightarrow \mathbb{R}$ be any convex and $\sigma $%
-normal function. Then, for any $T$-invariant measure $\mu \in \mathcal{P}%
(T) $, 
\begin{equation*}
\lim_{n\rightarrow \infty }\mu \left( F\circ \mathbb{E}_{n}\right)
=\inf_{n\in \mathbb{N}}\mu \left( F\circ \mathbb{E}_{n}\right) =\int_{%
\mathcal{P}_{\mathrm{erg}}\left( T\right) }F\left( \nu _{S}\right) \xi _{\mu
}\left( \mathrm{d}\nu \right) \ ,
\end{equation*}%
where $\xi _{\mu }$ is the Choquet measure associated with $\mu $
(Proposition \ref{prop Choquet inv meas}).
\end{proposition}

\begin{proof}
Fix\textit{\ }$\mu \in \mathcal{P}(T)$ and denote by $\xi _{\mu }$\ the
associated Choquet measure. Since $F$ is $\sigma $-normal, remark that the
mapping $\nu \mapsto F\left( \nu _{S}\right) $ from $\mathcal{P}(T)$ to $%
\mathbb{R}$ is weak$^{\ast }$-continuous and, thus, integrable with respect
to the Choquet measure $\xi _{\mu }$. We divide the proof into two steps:
\medskip

\noindent \underline{Step 1:} Since $F$ is a convex continuous function on
the Banach space $\mathcal{M}(S)$, it has continuous tangents at any $f\in 
\mathcal{M}(S)$. That is, for any fixed $f\in \mathcal{M}(S)$, there is a
continuous linear functional $\vartheta \in \mathcal{M}(S)^{\ast }$ such
that, for all $u\in \mathcal{M}(S)$, 
\begin{equation}
F\left( u\right) -F\left( f\right) \geq \vartheta \left( u\right) -\vartheta
\left( f\right) \text{ }.  \label{jensen1}
\end{equation}%
Such a linear functional $\vartheta $ is also $\sigma $-normal. Indeed, take
any bounded sequence $(f_{n})_{n\in \mathbb{N}}\subseteq \mathcal{M}(S)$
converging pointwise to $f$. By Inequality (\ref{jensen1}) for $u=f_{n}$ and
the $\sigma $-normality of $F$, we deduce that 
\begin{equation}
\vartheta \left( f\right) \geq \limsup_{n\rightarrow \infty }\vartheta
\left( f_{n}\right) \ .  \label{jkl1}
\end{equation}%
Similarly, from Inequality (\ref{jensen1}) for $u=2f-f_{n}$,%
\begin{equation}
\liminf_{n\rightarrow \infty }\vartheta \left( f_{n}\right) \geq \vartheta
\left( f\right) \text{ }.  \label{jkl2}
\end{equation}%
Altogether, Inequalities (\ref{jkl1}) and (\ref{jkl2})\ yield%
\begin{equation}
\lim_{n\rightarrow \infty }\vartheta \left( f_{n}\right) =\vartheta \left(
f\right)  \label{jkl3}
\end{equation}%
for any bounded sequence $(f_{n})_{n\in \mathbb{N}}\subseteq \mathcal{M}(S)$
converging pointwise to $f$. By linearity, it follows that $\vartheta $ is $%
\sigma $-normal. Since $(\Omega ,d)$ is a compact metric space, $\Sigma $ is
also a compact metric space (see, e.g., (\ref{metric sigma})). In
particular, $C(\Sigma )$ is a separable Banach space, which implies that the
weak$^{\ast }$ topology of $\mathcal{P}$, the space of all (i.e., not
necessarily $T$-invariant) probability measures on $\Sigma $, is metrizable.
From this fact and the $\sigma $-normality of $\vartheta $ it follows that
the mapping $\nu \mapsto \vartheta (\nu _{S})$ from $\mathcal{P}$ to $%
\mathbb{R}$ is weak$^{\ast }$-continuous. Note that, for any $\sigma \in
\Sigma $,%
\begin{equation*}
\delta _{\sigma }(\vartheta \circ \mathbb{E}_{n})=\vartheta \circ \mathbb{E}%
_{n}(\sigma )=\vartheta ((\delta _{\sigma }\circ \mathbb{E}_{n})_{S})\text{ }%
,
\end{equation*}%
where $\vartheta \circ \mathbb{E}_{n}$ stands for the continuous function $%
\sigma \mapsto \vartheta \left( \mathbb{E}_{n,\sigma }\right) $ and $\delta
_{\sigma }$ is the $\delta $-probability measure at $\sigma \in \Sigma $,
that is, $\delta _{\sigma }(\{\sigma \})=1$. Since $\Sigma $ is a separable
(being compact) metric space, any probability measure on $\Sigma $ is the
weak$^{\ast }$ limit of a sequence of convex combinations of $\delta $%
-probability measures. See, for instance, \cite[Chapter II, Theorem 7.1]%
{Parthasarathy}. Using this fact, along with the weak$^{\ast }$ continuity
of the mapping $\nu \mapsto \vartheta (\nu _{S})$ and the linearity of $%
\vartheta $, one arrives at 
\begin{equation*}
\nu (\vartheta \circ \mathbb{E}_{n})=\vartheta \left( (\nu \circ \mathbb{E}%
_{n})_{S}\right)
\end{equation*}%
for all $\nu \in \mathcal{P}$ and $n\in \mathbb{N}$. In particular, if $\nu $
is $T$-invariant then%
\begin{equation*}
\nu (\vartheta \circ \mathbb{E}_{n})=\vartheta \left( \nu _{S}\right) \ .
\end{equation*}%
We then conclude from Inequality (\ref{jensen1}) applied to $f=\nu _{S}$ and 
$u=\mathbb{E}_{n,\sigma }$ that%
\begin{equation}
\nu \left( F\circ \mathbb{E}_{n}\right) -F\left( \nu _{S}\right) \geq 0\text{
},  \label{Jensen ineq}
\end{equation}%
for any $T$-invariant probability measure $\nu \in \mathcal{P}(T)$, which is
nothing but a version of Jensen's inequality. For a more general statement,
see \cite[Lemma 10.33]{BruPedra2}. Hence, by convexity of $F$ and (the above
version of) Jensen's inequality, 
\begin{equation}
\mu \left( F\circ \mathbb{E}_{n}\right) =\int_{\mathcal{P}_{\mathrm{erg}%
}\left( T\right) }\nu \left( F\circ \mathbb{E}_{n}\right) \xi _{\mu }\left( 
\mathrm{d}\nu \right) \geq \int_{\mathcal{P}_{\mathrm{erg}}\left( T\right)
}F(\nu _{S})\xi _{\mu }\left( \mathrm{d}\nu \right) \text{ }.
\label{dsdssddsd}
\end{equation}

\noindent \underline{Step 2:} Keeping in mind (\ref{dsdssddsd}), it suffices
now to show that%
\begin{equation}
\lim_{n\rightarrow \infty }\int_{\mathcal{P}_{\mathrm{erg}}\left( T\right)
}\nu \left( F\circ \mathbb{E}_{n}\right) \xi _{\mu }\left( \mathrm{d}\nu
\right) =\int_{\mathcal{P}_{\mathrm{erg}}\left( T\right) }F(\nu _{S})\xi
_{\mu }\left( \mathrm{d}\nu \right)  \label{dssdsdsdssd}
\end{equation}%
to get the assertion. Using the $\sigma $-normality of $F$ and Lemma \ref%
{lemma In conv pw}, for any ergodic measure $\nu \in \mathcal{P}_{\mathrm{erg%
}}\left( T\right) $ and for $\nu $-a.s. $\sigma \in \Sigma $, 
\begin{equation}
\lim_{n\rightarrow \infty }F\left( \mathbb{E}_{n,\sigma }\right) =F\left(
\nu _{S}\right) \text{ }.  \label{sdadadadsad}
\end{equation}%
The continuous functions 
\begin{equation*}
\sigma \mapsto F\circ \mathbb{E}_{n}\left( \sigma \right) \doteq F\left( 
\mathbb{E}_{n,\sigma }\right)
\end{equation*}%
are uniformly bounded for all $n\in \mathbb{N}$, i.e., 
\begin{equation*}
\sup_{n\in \mathbb{N}}\left\Vert F\circ \mathbb{E}_{n}\right\Vert _{\infty
}=\sup_{\sigma \in \Sigma }\sup_{n\in \mathbb{N}}\left\vert F\left( \mathbb{E%
}_{n,\sigma }\right) \right\vert \leq \sup \left\{ \left\vert F\left(
f\right) \right\vert :f\in \mathcal{M}(S)\text{ with }\left\Vert
f\right\Vert _{\infty }<1\right\} <\infty \ ,
\end{equation*}%
because 
\begin{equation*}
\sup_{n\in \mathbb{N}}\left\Vert \mathbb{E}_{n,\sigma }\right\Vert _{\infty
}=\sup \left\{ \left\vert \mathbb{E}_{n,\sigma }\left[ \varphi \right]
\right\vert :n\in \mathbb{N},\ \varphi \in S\text{ with }\left\Vert \varphi
\right\Vert _{\infty }=1\right\} \leq 1
\end{equation*}%
and any $\sigma $-normal function $F$ satisfies (\ref{sdsdsdsdsdsdsdsd}), as
explained after Condition \ref{Condition essential}. So, we can use
Lebesgue's dominated convergence to infer from (\ref{sdadadadsad}) that, for
any ergodic measure $\nu \in \mathcal{P}_{\mathrm{erg}}\left( T\right) $,%
\begin{equation*}
\lim_{n\rightarrow \infty }\nu \left( F\circ \mathbb{E}_{n}\right) =F\left(
\nu _{S}\right) \ .
\end{equation*}%
Applying again Lebesgue's dominated convergence, we obtain the limit (\ref%
{dssdsdsdssd}).
\end{proof}

Proposition \ref{prop Delta F} motivates the following definition of a
real-valued functional on $T$-invariant probability measures:

\begin{definition}[$\Delta $-Functionals]
\label{definition de Delta}\mbox{ }\newline
Given a convex and $\sigma $-normal function $F:\mathcal{M}(S)\rightarrow 
\mathbb{R}$, $\Delta ^{F}$ denotes the mapping $\mathcal{P}(T)\rightarrow 
\mathbb{R}$ defined by 
\begin{equation*}
\Delta ^{F}\left( \mu \right) \doteq \int_{\mathcal{P}_{\mathrm{erg}}\left(
T\right) }F\left( \nu _{S}\right) \xi _{\mu }\left( \mathrm{d}\nu \right) \
,\qquad \mu \in \mathcal{P}\left( T\right) \ ,
\end{equation*}%
where $\xi _{\mu }$ is the Choquet measure associated with $\mu $
(Proposition \ref{prop Choquet inv meas}).
\end{definition}

\noindent As far as we know, the use of such a functional is entirely new to
thermodynamic formalism. Additionally, mutatis mutandis, it is a broad
extension of the space-averaging functional introduced in \cite[Section 1.3]%
{BruPedra2} for quantum lattice systems. Important properties of this
functional are gathered in the next theorem. But before stating them, we
recall that the $\Gamma $-regularization of a functional $f:\mathcal{P}%
(T)\rightarrow (-\infty ,\infty ]$ is the convex functional $\Gamma (f):%
\mathcal{P}(T)\rightarrow (-\infty ,\infty ]$ defined by the supremum over
all affine and continuous minorants of $f$, i.e., for all $\mu \in \mathcal{P%
}(T)$, 
\begin{equation}
\Gamma \left( f\right) \left( \mu \right) \doteq \sup \left\{ m\left( \mu
\right) :m\in \mathrm{A}\;\text{and }m|_{\mathcal{P}(T)}\leq f\right\} \ ,
\label{Gamma-regularisation def}
\end{equation}%
where $\mathrm{A}$ denotes the set of all affine and weak$^{\ast }$%
-continuous functionals on the dual space $C(\Sigma )^{\ast }$. It is the
largest weak$^{\ast }$-lower semicontinuous convex function below $f$, see 
\cite[Corollary 3.2]{BruPedraconvex}. Cf. Section \ref{Minimization}, in
particular Equation (\ref{gamm regu general}). It is also convenient to
define the concave\ counterpart of $\Gamma $-regularizations: For any
functional $f:\mathcal{P}(T)\rightarrow \lbrack -\infty ,\infty )$ and all $%
\mu \in \mathcal{P}(T)$,%
\begin{equation}
\Gamma _{-}\left( f\right) \left( \mu \right) \doteq -\Gamma \left(
-f\right) \left( \mu \right) =\inf \left\{ -m\left( \mu \right) :m\in 
\mathrm{A}\;\text{and }-m|_{\mathcal{P}(T)}\geq f\right\} \ ,
\label{Gamma-regularisation defbis}
\end{equation}%
which is the smallest weak$^{\ast }$-upper semicontinuous concave function
above $f$. We call $\Gamma _{-}\left( f\right) $ the upper $\Gamma $%
-regularization of $f$.

Note that taking the $\Gamma $-regularization is equivalent to performing
the Legendre-Fenchel transform twice, i.e., it is equal to the biconjugate
of a function. See Sections \ref{Legendre-Fenchel transform section} and \ref%
{Minimization}. However, the concept of $\Gamma $-regularization is more
convenient, in the sense that we do not need to talk about dual pairs $(%
\mathcal{X},\mathcal{X}^{\ast })$ to introduce it. Furthermore, its
definition via affine weak$^{\ast }$-continuous functionals is also
technically very useful. In particular, as shown in \cite[Lemma 3.4]%
{BruPedraconvex}, it leads to an extension of the Bauer maximum principle.

\begin{theorem}[Properties of the functional $\Delta ^{F}$]
\label{properties de Delta}\mbox{ }\mbox{ }\newline
Let $F:\mathcal{M}(S)\rightarrow \mathbb{R}$ be a convex and $\sigma $%
-normal function. \newline
\emph{(i) }$\Delta ^{F}$ is a weak$^{\ast }$-upper semicontinuous affine
functional. \newline
\emph{(ii) }$\Delta ^{F}$ is weak$^{\ast }$-continuous iff the mapping $\mu
\mapsto F(\mu _{S})$ is affine on $\mathcal{P}(T)$. In this case, there is $%
f\in C(\Sigma )$ such that 
\begin{equation}
\Delta ^{F}(\mu )=\mu (f)\ ,\text{\qquad }\mu \in \mathcal{P}\left( T\right) 
\text{ }.  \label{sdsdsd}
\end{equation}%
\emph{(iii)} If the mapping $\mu \mapsto F(\mu _{S})$ is strictly convex
then $\Delta ^{F}$ is weak$^{\ast }$-discontinuous on a weak$^{\ast }$-dense
subset of $\mathcal{P}(T)$.\newline
\emph{(iv)} $\Delta ^{F}$ is continuous on the $G_{\delta }$ weak$^{\ast }$%
-dense subset $\mathcal{P}_{\mathrm{erg}}\left( T\right) $ of ergodic
measures. In particular, the set of all $T$-invariant measures on which $%
\Delta ^{F}$ is weak$^{\ast }$-discontinuous is weak$^{\ast }$-meager.%
\newline
\emph{(v)} Its $\Gamma $-regularization $\Gamma (\Delta ^{F})$ is the weak$%
^{\ast }$-continuous convex mapping $\mu \mapsto F\left( \mu _{S}\right) $
on $\mathcal{P}(T)$.
\end{theorem}

\begin{proof}
We take a function $F$ as specified in the theorem and divide the proof into
several steps. Note that this proof strongly relies on the one of \cite[%
Theorems 1.18-1.19]{BruPedra2}, which regards quantum lattice systems.
\medskip

\noindent \underline{Step 1:} By\ Definition \ref{definition de Delta}, $%
\Delta ^{F}$ is clearly affine. By Proposition \ref{prop Delta F}, 
\begin{equation}
\Delta ^{F}\left( \mu \right) =\lim_{n\rightarrow \infty }\mu \left( F\circ 
\mathbb{E}_{n}\right) =\inf_{n\in \mathbb{N}}\mu \left( F\circ \mathbb{E}%
_{n}\right) \ ,\qquad \mu \in \mathcal{P}\left( T\right) \ .
\label{sdsdsdsdsd}
\end{equation}%
As a consequence, $\Delta ^{F}$ is an infimum over weak$^{\ast }$-continuous
functionals $\mu \mapsto \mu \left( F\circ \mathbb{E}_{n}\right) $ and is
therefore weak$^{\ast }$-upper semicontinuous. (i) is thus proven. From
Jensen's inequality (see (\ref{Jensen ineq})), for any $n\in \mathbb{N}$, 
\begin{equation*}
\mu \left( F\circ \mathbb{E}_{n}\right) \geq F\left( \mu _{S}\right) \ .
\end{equation*}%
This last inequality combined with Definition \ref{definition de Delta} and (%
\ref{sdsdsdsdsd}) leads to 
\begin{equation}
\Delta ^{F}\left( \mu \right) \geq F\left( \mu _{S}\right) \text{ },\qquad
\mu \in \mathcal{P}\left( T\right) \ .  \label{inequality sympa}
\end{equation}%
Note that $\Delta ^{F}(\mu )=F(\mu _{S})$ for all $\mu \in \mathcal{P}_{%
\mathrm{erg}}\left( T\right) $, by Definition \ref{definition de Delta}%
.\medskip

\noindent \underline{Step 2:} Assume that the mapping $\mu \mapsto F(\mu
_{S})$ from $\mathcal{P}(T)$ to $\mathbb{C}$ is affine. We infer from
Definition \ref{definition de Delta} and the $\sigma $-normality of $F$ that 
\begin{equation*}
\Delta ^{F}\left( \mu \right) =F\left( \mu _{S}\right) \ ,\qquad \mu \in 
\mathcal{P}\left( T\right) \ .
\end{equation*}%
By the $\sigma $-normality of $F$, the mapping $\mu \mapsto F\left( \mu
_{S}\right) $ is affine and weak$^{\ast }$-continuous on the whole
(topological) dual space $C(\Sigma )^{\ast }$. In particular, $\Delta ^{F}$
is continuous. Similar to \cite[Proposition 3.9]{Bru-pedra-MF-I}, there is $%
f\in C(\Sigma )$ such that 
\begin{equation*}
F(\mu _{S})=\mu \left( f\right) \ ,\qquad \mu \in \mathcal{P}\left( T\right)
\ ,
\end{equation*}%
which in turn implies (\ref{sdsdsd}). \medskip

\noindent \underline{Step 3:} Assume now that $\mu \mapsto F(\mu _{S})$ is
not affine. Then, there are at least two $T$-invariant probability measures $%
\nu ,\zeta \in \mathcal{P}(T)$ and $\lambda \in (0,1)$ such that, for $\mu
=\lambda \nu +\left( 1-\lambda \right) \zeta $, 
\begin{equation}
F\left( \mu _{S}\right) <\lambda F\left( \nu _{S}\right) +\left( 1-\lambda
\right) F\left( \zeta _{S}\right) \ ,  \label{sdfssdf}
\end{equation}%
the function $F$ being convex. By Proposition \ref{density of ergodic
measures}, the set $\mathcal{P}_{\mathrm{erg}}(T)$ of ergodic probability
measures is weak$^{\ast }$-dense in $\mathcal{P}\left( T\right) $. So, there
is a sequence $(\nu ^{(n)})_{n\in \mathbb{N}}\subseteq \mathcal{P}_{\mathrm{%
erg}}(T)$ of ergodic measures converging with respect to the weak$^{\ast }$
topology to the above probability measure $\mu $. Then, by Definition \ref%
{definition de Delta} as well as\ Equations (\ref{inequality sympa}) and (%
\ref{sdfssdf}), 
\begin{equation}
\lim_{n\rightarrow \infty }\Delta ^{F}(\nu ^{(n)})=\lim_{n\rightarrow \infty
}F(\nu _{S}^{(n)})=F\left( \mu _{S}\right) <\lambda F\left( \nu _{S}\right)
+\left( 1-\lambda \right) F\left( \zeta _{S}\right) \leq \Delta ^{F}\left(
\mu \right) \ ,  \label{yui}
\end{equation}%
because $F$ is $\sigma $-normal, $\lambda \in (0,1)$ and $\Delta ^{F}$ is
affine. Consequently, $\Delta ^{F}$ is discontinuous on the probability
measure $\mu =\lambda \nu +\left( 1-\lambda \right) \zeta $. Steps 2 and 3
yield Assertion (ii).\medskip

\noindent \underline{Step 4:} Assume that the mapping $\mu \mapsto F(\mu
_{S})$ is strictly convex. In particular, it is not the constant mapping.
Therefore, for any $\nu \in \mathcal{P}(T)$, there is at least one $T$%
-invariant probability measure $\zeta ^{(\nu )}\in \mathcal{P}(T)$ such that 
$F(\nu _{S})\neq F(\zeta _{S}^{(\nu )})$. For all $\nu \in \mathcal{P}(T)$,
we define the subset%
\begin{equation*}
\mathrm{I}(\nu )\doteq \left\{ \lambda \nu +\left( 1-\lambda \right) \zeta
^{(\nu )}\text{ for\ any\ }\lambda \in \left( 0,1\right) \right\} \subseteq 
\mathcal{P}\left( T\right) \ .
\end{equation*}%
Finally, let us consider the subset 
\begin{equation*}
\mathrm{D}\doteq \bigcup\limits_{\nu \in \mathcal{P}\left( T\right) }\mathrm{%
I}\left( \nu \right) \subseteq \mathcal{P}\left( T\right) \backslash 
\mathcal{P}_{\mathrm{erg}}\left( T\right) \ .
\end{equation*}%
By continuity of the mapping $\lambda \mapsto \lambda \nu $ for $\lambda \in 
\mathbb{C}$ and $\nu \in \mathcal{P}(T)$, the set $\mathrm{D}$ is weak$%
^{\ast }$-dense in $\mathcal{P}\left( T\right) $. Any $\mu \in \mathrm{D}$
is clearly of the form 
\begin{equation}
\mu =\lambda \nu +\left( 1-\lambda \right) \zeta  \label{form}
\end{equation}%
for some $\lambda \in (0,1)$ and $\nu ,\zeta \in \mathcal{P}(T)$ with $F(\nu
)\neq F(\zeta )$. By Proposition \ref{density of ergodic measures}, for any $%
\mu \in \mathrm{D}$, there is a sequence $(\nu ^{(n)})_{n\in \mathbb{N}%
}\subseteq \mathcal{P}_{\mathrm{erg}}(T)$ of ergodic measures converging
with respect to the weak$^{\ast }$ topology to $\mu $. Then, by Definition %
\ref{definition de Delta} and Equation (\ref{inequality sympa}), for any $%
\mu \in \mathrm{D}$ of the form (\ref{form}), Equation (\ref{yui}) holds
true because $F$ is $\sigma $-normal and strictly convex, $\lambda \in (0,1)$
and $\Delta ^{F}$ is affine. As a consequence, the mapping $\Delta ^{F}$ is
weak$^{\ast }$-discontinuous at any $\mu \in \mathrm{D}$ and hence on a weak$%
^{\ast }$-dense subset of $\mathcal{P}\left( T\right) $. In other words,
(iii) is proven.\medskip

\noindent \underline{Step 5:} We show now that $\Delta ^{F}$ is weak$^{\ast
} $-continuous for any ergodic probability measure, which yields (iv),
thanks to Proposition \ref{density of ergodic measures}. Take $\nu \in 
\mathcal{P}_{\mathrm{erg}}(T)$ and consider any sequence $(\nu ^{(n)})_{n\in 
\mathbb{N}}\subseteq \mathcal{P}(T)$ of probability measures converging with
respect to the weak$^{\ast }$ topology to $\nu $. By Definition \ref%
{definition de Delta}, Assertion (i) and Equation (\ref{inequality sympa}),%
\begin{equation*}
F(\nu _{S})=\Delta ^{F}\left( \nu \right) \geq \limsup_{n\rightarrow \infty
}\Delta ^{F}(\nu ^{(n)})\geq \liminf_{n\rightarrow \infty }\Delta ^{F}(\nu
^{(n)})\geq \lim_{n\rightarrow \infty }F(\nu _{S}^{(n)})=F(\nu _{S})
\end{equation*}%
because $F$ is $\sigma $-normal. In other words, the functional $\Delta ^{F}$
is weak$^{\ast }$-continuous on $\mathcal{P}_{\mathrm{erg}}(T)$. Notice that
the weak$^{\ast }$ topology of $\mathcal{P}(T)$\ is metrizable. (So, there
is no need to use nets here to check continuity on $\mathcal{P}(T)$.)\medskip

\noindent \underline{Step 6:} Recall that the $\Gamma $-regularization $%
\Gamma (f)$ of functionals $f$, as defined by Equation (\ref%
{Gamma-regularisation def}), is the largest weak$^{\ast }$-lower
semicontinuous convex function below $f$. See, e.g., \cite[Corollary 3.2]%
{BruPedraconvex}. By Inequality (\ref{inequality sympa}), it follows that 
\begin{equation}
F\left( \mu _{S}\right) \leq \Gamma \left( \Delta ^{F}\right) \left( \mu
\right) \leq \Delta ^{F}\left( \mu \right) \ ,\qquad \mu \in \mathcal{P}%
\left( T\right) \ ,  \label{sdsds1}
\end{equation}%
because the mapping $\mu \mapsto F(\mu _{S})$ is weak$^{\ast }$-continuous
and convex on $\mathcal{P}(T)$, thanks the convexity and $\sigma $-normality
of $F$. In addition, for any ergodic probability measure $\nu \in \mathcal{P}%
_{\mathrm{erg}}(T)$, 
\begin{equation}
F\left( \nu _{S}\right) =\Gamma \left( \Delta ^{F}\right) \left( \nu \right)
=\Delta ^{F}\left( \nu \right) \ .  \label{sdsds2}
\end{equation}%
Therefore, using the weak$^{\ast }$-density of $\mathcal{P}_{\mathrm{erg}%
}(T) $ in $\mathcal{P}(T)$ (Proposition \ref{density of ergodic measures}),
the weak$^{\ast }$-continuity of $\mu \mapsto F(\mu _{S})$ and the weak$%
^{\ast }$-lower semicontinuity of $\Gamma (\Delta ^{F})$, we deduce from (%
\ref{sdsds1})--(\ref{sdsds2}) that 
\begin{equation*}
\Gamma \left( \Delta ^{F}\right) \left( \mu \right) =F\left( \mu _{S}\right)
\ ,\qquad \mu \in \mathcal{P}\left( T\right) \ ,
\end{equation*}%
that is, (v) follows. This concludes the proof of the theorem.
\end{proof}

Thanks to Theorem \ref{properties de Delta}, under Condition \ref{Condition
essential} the \emph{nonlinear} pressure functional 
\begin{equation}
P\left( \mu \right) \doteq F\left( \mu _{S}\right) +h\left( \mu \right)
=F_{+}\left( \mu _{S}\right) +F_{-}\left( \mu _{S}\right) +h\left( \mu
\right) \ ,\qquad \mu \in \mathcal{P}\left( T\right) \ ,  \label{def P mu}
\end{equation}%
on $T$-invariant measures has a direct relation to the functional $\mathfrak{%
F}^{\flat }:\mathcal{P}(T)\rightarrow \mathbb{R}$ defined by (\ref{F^bemol}%
), that is, 
\begin{equation}
\mathfrak{F}^{\flat }\doteq \Delta ^{F_{+}}-\Delta ^{-F_{-}}+h\ .
\label{the free-energy density functional}
\end{equation}%
In contrast to the nonlinear pressure functional $P:\mathcal{P}%
(T)\rightarrow \mathbb{R}$ (\ref{def P mu}) which is not necessarily affine
but weak$^{\ast }$-upper semicontinuous (for $\sigma $-normal $F_{\pm }$)
(see Proposition \ref{Affinity of the entropy copy(1)}), the functional $%
\mathfrak{F}^{\flat }$ is affine but not necessarily weak$^{\ast }$-upper
semicontinuous, because of Theorem \ref{properties de Delta} and Proposition %
\ref{Affinity of the entropy copy(1)}. However, the variational problems $%
\sup P\left( \mathcal{P}\left( T\right) \right) $ and $\sup \mathfrak{F}%
^{\flat }\left( \mathcal{P}\left( T\right) \right) $ can be related, as well
as can their corresponding maximizers. This is a consequence of the next
observation. Recall that $\mathcal{P}_{\mathrm{erg}}(T)$ stands for the
(dense in $\mathcal{P}\left( T\right) $) set of ergodic probability measures
on $\Sigma $. \ 

\begin{proposition}
\label{Gamma-regularisation of pressure}\mbox{ }\newline
Let $F_{\pm }:\mathcal{M}(S)\rightarrow \mathbb{R}$ be two $\sigma $-normal
functions and assume that $F_{+}$ and $F_{-}$ are respectively convex and
concave. \newline
\emph{(i)} We have 
\begin{equation}
\sup_{\mu \in \mathcal{P}\left( T\right) }\,\mathfrak{F}^{\flat }\left( \mu
\right) =\sup\limits_{\nu \in \mathcal{P}_{\mathrm{erg}}\left( T\right) }\,%
\mathfrak{F}^{\flat }\left( \nu \right) =\sup\limits_{\nu \in \mathcal{P}_{%
\mathrm{erg}}\left( T\right) }\,P\left( \nu \right) =\sup_{\mu \in \mathcal{P%
}\left( T\right) }P\left( \mu \right) <\infty \ .  \label{sdsdsssd1}
\end{equation}%
\emph{(ii)} $\Gamma _{-}(\mathfrak{F}^{\flat })=\Gamma _{-}(P)$ on the whole
space $\mathcal{P}\left( T\right) $ and $\Gamma _{-}(P)=P$ on $\mathcal{P}_{%
\mathrm{erg}}(T)$. Furthermore, 
\begin{equation}
\sup_{\mu \in \mathcal{P}\left( T\right) }\,\mathfrak{F}^{\flat }\left( \mu
\right) =\sup_{\mu \in \mathcal{P}\left( T\right) }\Gamma _{-}\left( 
\mathfrak{F}^{\flat }\right) \left( \mu \right) =\sup_{\mu \in \mathcal{P}%
\left( T\right) }\,\Gamma _{-}\left( P\right) \left( \mu \right) =\sup_{\mu
\in \mathcal{P}\left( T\right) }P\left( \mu \right) <\infty \ .
\label{sdsdsssd2}
\end{equation}
\end{proposition}

\begin{proof}
Note that all suprema in (i) are finite, because any $\sigma $-normal
function $F$ such as $F_{\pm }$ satisfies (\ref{sdsdsdsdsdsdsdsd}), as
explained after Condition \ref{Condition essential}, and the entropy
functional is uniformly bounded from above. Assertion (i) is proven in the
same way as \cite[Lemma 2.9]{BruPedra2}, while the proof of (ii) uses the
same arguments as the first part of the proof of \cite[Theorem 2.21]%
{BruPedra2}. We reproduce these proofs below, adapting them to the present
situation.

\noindent \underline{Step 1:} The Bauer maximum principle \cite[Theorem I.5.3%
]{Alfsen} says that an upper semicontinuous convex real-valued function $f$
over a compact convex subset $K$, such as the weak$^{\ast }$-compact convex
set $\mathcal{P}(T)$, attains its maximum at an extreme point of $K$.
However, by Theorem \ref{properties de Delta}, $\mathfrak{F}^{\flat }$ is
the sum of a convex weak$^{\ast }$-lower semicontinuous functional and a
convex weak$^{\ast }$-upper semicontinuous functional, both on $\mathcal{P}%
(T)$. We thus need the extension of the Bauer maximum principle given by 
\cite[Lemma 3.4]{BruPedraconvex}, which says that the supremum of the values
of a sum $f\doteq f_{-}+f_{+}$ of a lower and upper semicontinuous affine
functionals $f_{-}$ and $f_{+}$ on a compact convex subset $K$ can be
restricted to the set $\mathcal{E}(K)$ of extreme points of $K$ (see Section %
\ref{The Choquet theorem}), i.e.,%
\begin{equation*}
\sup \,f\left( K\right) =\sup \,f\left( \mathcal{E}(K)\right) \ .
\end{equation*}%
Applying this result to $f=\mathfrak{F}^{\flat }$ and $K=\mathcal{P}(T)$\
while using additionally that $\mathfrak{F}^{\flat }=P$ on $\mathcal{P}_{%
\mathrm{erg}}(T)$ (see Definition \ref{definition de Delta}), we thus obtain
the equality%
\begin{equation}
\sup_{\mu \in \mathcal{P}\left( T\right) }\,\mathfrak{F}^{\flat }(\mu
)=\sup\limits_{\nu \in \mathcal{P}_{\mathrm{erg}}\left( T\right) }\,%
\mathfrak{F}^{\flat }\left( \nu \right) =\sup\limits_{\nu \in \mathcal{P}_{%
\mathrm{erg}}\left( T\right) }\,P\left( \nu \right) \ .  \label{hjk}
\end{equation}%
Since $P$ is weak$^{\ast }$-upper semicontinuous, it has a maximizer $\mu
_{0}$ over $\mathcal{P}(T)$. Thus, by Corollary \ref{ergodic abundance},
there is a sequence $(\nu _{n})_{n\in \mathbb{N}}\subseteq \mathcal{P}_{%
\mathrm{erg}}(T)$ of ergodic probability measures converging in the weak$%
^{\ast }$ topology to $\mu _{0}$ with the property that $P(\nu _{n})$
converges to $P(\mu _{0})$ as $n\rightarrow \infty $. In other words,
maximizing $P$ over $\mathcal{P}(T)$ results the same as maximizing $P$ over 
$\mathcal{P}_{\mathrm{erg}}(T)$. Equation (\ref{sdsdsssd1}) then follows.
Observe that Equation (\ref{sdsdsssd2}) is a direct consequence of Equation (%
\ref{sdsdsssd1})\ and Theorem \ref{theorem sympa}. We now prove the equality 
$\Gamma _{-}(\mathfrak{F}^{\flat })=\Gamma _{-}(P)$ on the whole space $%
\mathcal{P}\left( T\right) $ and $\Gamma _{-}(P)=P$ on the dense subset $%
\mathcal{P}_{\mathrm{erg}}(T)$ of ergodic measures. \medskip

\noindent \underline{Step 2:} We start with a partial upper $\Gamma $%
-regularization of $\mathfrak{F}^{\flat }$ by defining the new functional $%
\mathfrak{F}^{\sharp }:\mathcal{P}(T)\rightarrow \mathbb{R}$ as follows:%
\begin{equation}
\mathfrak{F}^{\sharp }\left( \mu \right) \doteq \Delta ^{F_{+}}\left( \mu
\right) +F_{-}\left( \mu _{S}\right) +h\left( \mu \right) \ ,\qquad \mu \in 
\mathcal{P}\left( T\right) \ .  \label{F^diese}
\end{equation}%
By Theorem \ref{properties de Delta}, $\mathfrak{F}^{\sharp }$ is concave,
weak$^{\ast }$-upper semicontinuous and bounded from below by $\mathfrak{F}%
^{\flat }$. Recall that the $\Gamma $-regularization of a functional $f$ is
the largest weak$^{\ast }$-lower semicontinuous convex function below $f$,
see \cite[Corollary 3.2]{BruPedraconvex} or Section \ref{Minimization}.
I.e., (the upper $\Gamma $-regularization) $\Gamma _{-}\left( f\right) $ is
the smallest weak$^{\ast }$-upper semicontinuous concave function above $f$.
Using this property and Equation (\ref{Gamma-regularisation defbis}),%
\begin{equation*}
\mathfrak{F}^{\flat }\leq \Gamma _{-}\left( \mathfrak{F}^{\flat }\right)
\leq \mathfrak{F}^{\sharp }\ .
\end{equation*}%
Hence, for any ergodic probability measure $\nu \in \mathcal{P}_{\mathrm{erg}%
}(T)$, 
\begin{equation}
P\left( \nu \right) =\mathfrak{F}^{\sharp }\left( \nu \right) =\Gamma
_{-}\left( \mathfrak{F}^{\flat }\right) \left( \nu \right) =\mathfrak{F}%
^{\flat }\left( \nu \right) \ .  \label{equality idiote}
\end{equation}

\noindent \underline{Step 3:} We show now that $\Gamma _{-}\left( \mathfrak{F%
}^{\flat }\right) $ is an upper bound for $\Gamma _{-}(P)$. By Corollary \ref%
{ergodic abundance} and the $\sigma $-normality of $F_{\pm }$, any $T$%
-invariant probability measure $\mu \in \mathcal{P}(T)$ is the limit of some
sequence $(\nu _{n})_{n\in \mathbb{N}}\subseteq \mathcal{P}_{\mathrm{erg}%
}(T) $ of ergodic probability measures converging in the weak$^{\ast }$
topology to $\mu $ and such that $P(\nu _{n})$ converges to $P(\mu )$. By (%
\ref{equality idiote}), $\Gamma _{-}(\mathfrak{F}^{\flat })(\nu _{n})$ also
converges to $P(\mu )$. Since $\Gamma _{-}(\mathfrak{F}^{\flat })$ is, by
construction, weak$^{\ast }$-upper semicontinuous on $\mathcal{P}(T)$, we
thus conclude that, for any $T$-invariant probability measure $\mu \in 
\mathcal{P}(T)$, 
\begin{equation}
P\left( \mu \right) =\lim_{n\rightarrow \infty }\Gamma _{-}\left( \mathfrak{F%
}^{\flat }\right) \left( \nu _{n}\right) \leq \Gamma _{-}\left( \mathfrak{F}%
^{\flat }\right) \left( \mu \right) \ ,  \label{equality idiotebis}
\end{equation}%
which, combined with the fact that $\Gamma _{-}(P)$ is the smallest weak$%
^{\ast }$-upper semicontinuous concave function above $P$ (see \cite[%
Corollary 3.2]{BruPedraconvex}), yields 
\begin{equation}
\Gamma _{-}\left( P\right) \leq \Gamma _{-}\left( \mathfrak{F}^{\flat
}\right) \ .  \label{equality idiotebisbis}
\end{equation}%
We show next the converse inequality.\medskip

\noindent \underline{Step 4:} We use now the concavity of $\Gamma _{-}(P)$,
the (unique) Choquet measure $\xi _{\mu }$ associated with an arbitrary $T$%
-invariant probability measure $\mu \in \mathcal{P}(T)$, along with Jensen's
inequality applied to $\xi _{\mu }$ and $\Gamma _{-}(P)$ (see \cite[Lemma
10.33]{BruPedra2} for its proof in this context) in order to bound $\Gamma
_{-}(P)$ from below by%
\begin{equation}
\Gamma _{-}\left( P\right) \left( \mu \right) \geq \int_{\mathcal{P}_{%
\mathrm{erg}}\left( T\right) }\Gamma _{-}\left( P\right) \left( \nu \right)
\xi _{\mu }\left( \mathrm{d}\nu \right) \geq \int_{\mathcal{P}_{\mathrm{erg}%
}\left( T\right) }P\left( \nu \right) \xi _{\mu }\left( \mathrm{d}\nu \right)
\label{sdsdsdsdsds}
\end{equation}%
for all $\mu \in \mathcal{P}(T)$. Now, by applying \cite[Lemma 10.17]%
{BruPedra2} to the entropy $h$ (which is an affine weak$^{\ast }$-upper
semicontinuous functional) and using Definition \ref{definition de Delta},
recall that one gets an ergodic decomposition of $\mathfrak{F}^{\flat }$ in
the following sense: 
\begin{equation*}
\int_{\mathcal{P}_{\mathrm{erg}}\left( T\right) }P\left( \nu \right) \xi
_{\mu }\left( \mathrm{d}\nu \right) =\mathfrak{F}^{\flat }\left( \mu \right) 
\text{ },\qquad \mu \in \mathcal{P}(T)\ ,
\end{equation*}%
see Equation (\ref{assertion sympa}). By Inequality (\ref{sdsdsdsdsds}), It
follows that%
\begin{equation}
\Gamma _{-}\left( P\right) \geq \Gamma _{-}\left( \mathfrak{F}^{\flat
}\right) \ ,  \label{equality idiote2}
\end{equation}%
thanks to Equation (\ref{Gamma-regularisation defbis}) and the fact that $%
\Gamma _{-}\left( \mathfrak{F}^{\flat }\right) $ is the smallest weak$^{\ast
}$-upper semicontinuous concave function above $\mathfrak{F}^{\flat }$ (see 
\cite[Corollary 3.2]{BruPedraconvex}). Consequently, $\Gamma _{-}(P)=\Gamma
_{-}(\mathfrak{F}^{\flat })$, which, combined with (\ref{equality idiote})
also yield $\Gamma _{-}(P)=\Gamma _{-}(\mathfrak{F}^{\flat })=P$ on $%
\mathcal{P}_{\mathrm{erg}}(T)$.
\end{proof}

Proposition \ref{Gamma-regularisation of pressure} combined with Theorem \ref%
{theorem sympa} not only implies that the corresponding variational problems
for all functionals $\mathfrak{F}^{\flat }$, $\Gamma _{-}(\mathfrak{F}%
^{\flat })$, $P$ and $\Gamma _{-}(P)$ give the same numerical value but also
that the corresponding (possibly generalized) maximizers are directly
related to each other. To explain this, similar to Equations (\ref%
{sdsdsdfdgdfhfghhf}) and (\ref{equilibriun statebis}), given a bounded
functional $f:\mathcal{P}(T)\rightarrow \mathbb{R}$, we define the
(nonempty) set 
\begin{equation}
E_{f}\doteq \left\{ \mu \in \mathcal{P}\left( T\right) :\exists \left( \mu
_{n}\right) _{n\in \mathbb{N}}\subseteq \mathcal{P}\left( T\right) \mathrm{\ 
}\text{with }\lim_{n\rightarrow \infty }\mu _{n}=\mu \text{ and\ }%
\lim_{n\rightarrow \infty }f\left( \mu _{n}\right) =\sup f\left( \mathcal{P}%
\left( T\right) \right) \right\} ,  \label{equilibriun state}
\end{equation}%
of\ all weak$^{\ast }$-limits of approximating maximizers, i.e., generalized
maximizers. $E_{f}$ differs a priori from the set 
\begin{equation*}
M_{f}\doteq \left\{ \mu \in \mathcal{P}\left( T\right) :f\left( \mu \right)
=\sup f\left( \mathcal{P}\left( T\right) \right) \right\}
\end{equation*}%
of (strict) maximizers of $f$. One has a priori only the inclusion $%
M_{f}\subseteq E_{f}$. If the functional $f$ is weak$^{\ast }$-upper
semicontinuous then $E_{f}=M_{f}$, that is, $E_{f}$ is nothing but the (weak$%
^{\ast }$-compact) set of (strict) maximizers of $f$ over $\mathcal{P}(T)$.

\begin{theorem}[Sets of generalized maximizers]
\label{structure of sets of maximizers}\mbox{ }\newline
Let $F_{\pm }:\mathcal{M}(S)\rightarrow \mathbb{R}$ be two $\sigma $-normal
functions, with $F_{+}$ and $F_{-}$ being respectively convex and concave.
Then, 
\begin{equation*}
E_{\mathfrak{F}^{\flat }}=E_{\Gamma _{-}(\mathfrak{F}^{\flat })}=M_{\Gamma
_{-}(\mathfrak{F}^{\flat })}=M_{\Gamma _{-}(P)}=E_{\Gamma _{-}(P)}=\overline{%
\mathrm{co}}(E_{P})\text{ },
\end{equation*}%
that is, all sets of (possibly generalized) maximizers are nothing but the
weak$^{\ast }$-closed convex hull of the weak$^{\ast }$-compact set $%
E_{P}=M_{P}$ of strict maximizers of the (weak$^{\ast }$-upper
semicontinuous) functional $P$.
\end{theorem}

\begin{proof}
As upper $\Gamma $-regularizations are always upper semicontinuous, note
that $E_{\Gamma _{-}(\mathfrak{F}^{\flat })}=M_{\Gamma _{-}(\mathfrak{F}%
^{\flat })}$ and $E_{\Gamma _{-}(P)}=M_{\Gamma _{-}(P)}$. The equality $%
E_{\Gamma _{-}(\mathfrak{F}^{\flat })}=E_{\Gamma _{-}(P)}$ is an obvious
consequence of Proposition \ref{Gamma-regularisation of pressure} (ii) and
we infer from Theorem \ref{theorem sympa} and (\ref{Gamma-regularisation
defbis}) that $E_{\Gamma _{-}(\mathfrak{F}^{\flat })}=\overline{\mathrm{co}}%
(E_{\mathfrak{F}^{\flat }})$. On the other hand, $E_{\mathfrak{F}^{\flat }}$
is a convex set, by the affine property of $\mathfrak{F}^{\flat }$ (cf.
Theorem \ref{properties de Delta} and Proposition \ref{Affinity of the
entropy copy(1)}). By \cite[Lemma 10.36]{BruPedra2}, it is also weak$^{\ast
} $-compact because $\mathcal{P}(T)$ is weak$^{\ast }$-compact and the weak$%
^{\ast }$-topology is metrizable on $\mathcal{P}(T)$. As a consequence, $%
E_{\Gamma _{-}(\mathfrak{F}^{\flat })}=\overline{\mathrm{co}}(E_{\mathfrak{F}%
^{\flat }})=E_{\mathfrak{F}^{\flat }}$. Applying again Theorem \ref{theorem
sympa} and (\ref{Gamma-regularisation defbis}), we obtain meanwhile $%
E_{\Gamma _{-}(P)}=\overline{\mathrm{co}}(E_{P})$, which leads to $E_{%
\mathfrak{F}^{\flat }}=E_{\Gamma _{-}(\mathfrak{F}^{\flat })}=E_{\Gamma
_{-}(P)}=\overline{\mathrm{co}}(E_{P})$. Finally, $E_{P}$ is the weak$^{\ast
}$-compact set of strict maximizers of the functional $P$ as a consequence
of the weak$^{\ast }$-upper semicontinuity of this functional together with
the weak$^{\ast }$-compactness of $\mathcal{P}(T)$.
\end{proof}

\begin{corollary}[Choquet decomposition]
\label{Choquet decomposition}\mbox{ }\newline
Let $F_{\pm }:\mathcal{M}(S)\rightarrow \mathbb{R}$ be two $\sigma $-normal
functions with $F_{+}$ and $F_{-}$ being respectively convex and concave.
Then, equilibrium measures $\nu \in E_{\mathfrak{F}^{\flat }}$ that are
extreme in $E_{\mathfrak{F}^{\flat }}$ belong to $E_{P}$ and\ for any $\mu
\in E_{\mathfrak{F}^{\flat }}$, there is a probability measure $m_{\mu }$ on 
$E_{P}$ such that%
\begin{equation}
m_{\mu }(E_{P})=1\mathrm{\quad }\text{and}\mathrm{\quad }\mu
=\int_{E_{P}}\nu m_{\mu }\left( \mathrm{d}\nu \right) \ .
\label{choquet equilibirum}
\end{equation}
\end{corollary}

\begin{proof}
By the Milman theorem \cite[Proposition 1.5]{Phe} and Theorem \ref{structure
of sets of maximizers}, equilibrium measures that are extreme in $E_{%
\mathfrak{F}^{\flat }}$ must belong to $E_{P}$. Then, using this property
and the Choquet theorem (Theorem \ref{th Choquet}), we derive (\ref{choquet
equilibirum}), because\ the set $E_{\mathfrak{F}^{\flat }}$ is convex, weak$%
^{\ast }$-compact and metrizable, by Theorem \ref{structure of sets of
maximizers} and the metrizability of the weak$^{\ast }$ topology in $%
\mathcal{P}(T)$.
\end{proof}

\subsection{Appendix}

This appendix contains a selection of useful results that are difficult to
find within the mathematical framework we use. Indeed, in most previous
works the alphabet $\Omega $ is assumed to be a finite set, whereas we only
assume it to be a compact metric space, which, of course, may be infinite.
Therefore, in Sections \ref{Ergodic measures}--\ref{Further Study of the
Entropy} we re-examine known facts within this broader context.

One thing that is not very well-known is the properties of the entropy
functional for infinite alphabets. Our definition (Definition \ref{uod}) is
based on a variational problem involving the (transfer) Ruelle operator.
This definition is simple and elegant, but it poses certain difficulties in
proving the convexity of this functional. To this end, we use recent results 
\cite{ACR} to write the entropy functional for compact alphabets as a
thermodynamic limit of finite-volume entropies (Theorem \ref%
{teo-equiv-entrop}). Based on this observation, we are able to prove the
affine property\footnote{%
I.e. convex and concave.\ The concavity of the entropy functional is a
direct consequence of its definition.} of the entropy functional of
Definition \ref{uod}. See Proposition \ref{Affinity of the entropy copy(1)}.
Another consequence of this result is the so-called ergodic abundance of the
entropy functional, as given by Corollary \ref{ergodic abundance}, which is
also a fundamental property for our proofs, along with the weak$^{\ast }$
density of ergodic measures (Proposition \ref{density of ergodic measures}).

\subsubsection{Ergodic measures\label{Ergodic measures}}

Let $(X,\mathcal{A})$ be any measurable space. I.e., $X$ is any nonempty set
and $\mathcal{A}$ some $\sigma $-algebra on $X$. In ergodic theory, the
existence of limits of the form 
\begin{equation}
\tilde{\varphi}\left( x\right) =\lim_{n\rightarrow \infty }\frac{1}{n}%
\sum_{m=1}^{n}\varphi \circ f^{m}\left( x\right) \ ,\qquad x\in X\ ,
\label{limit ergodic}
\end{equation}%
is studied for any measurable function $f:X\rightarrow X$ and appropriate
functions $\varphi :X\rightarrow \mathbb{R}$. Such a question led to two
important (types of) ergodic theorems, which were proven for the first time
by von Neumann and Birkhoff, respectively. See \cite[Theorems 3.1.6 and 3.2.3%
]{Viana}. The most interesting one here is Birkhoff's ergodic theorem \cite[%
Theorem 3.2.3]{Viana}, which ensures the existence of the limit (\ref{limit
ergodic}) almost surely for integrable functions $\varphi $:

\begin{theorem}[Birkhoff]
\label{th Birkhoff}\mbox{ }\newline
Let $(X,\mathcal{A})$ be any measurable space, $f:X\rightarrow X$ a
measurable mapping and $\mu $ a probability measure that is invariant%
\footnote{%
I.e., $f_{\ast }(\mu )=\mu $, where $f_{\ast }(\mu )(A)=\mu (f(A))$ for all
Borel sets $A\in \mathcal{A}$.} with respect to $f$. For any $\mu $%
-integrable function $\varphi :X\rightarrow \mathbb{R}$, the limit (\ref%
{limit ergodic}) exists $\mu $-almost surely and defines another $\mu $%
-integrable function $\tilde{\varphi}:X\rightarrow \mathbb{R}$ ($\mu $%
-almost everywhere) satisfying 
\begin{equation*}
\int_{X}\tilde{\varphi}\left( x\right) \mu \left( \mathrm{d}x\right)
=\int_{X}\varphi (x)\mu \left( \mathrm{d}x\right) \text{ }.
\end{equation*}
\end{theorem}

\noindent The Birkhoff ergodic theorem stated above is proven, for instance,
in \cite[Section 3.2.2]{Viana}, see \cite[Theorem 3.2.3]{Viana}.

The $\mu $-integrable function $\tilde{\varphi}:X\rightarrow \mathbb{R}$,
uniquely defined $\mu $-almost everywhere, is called a \emph{time average}\
of the ($\mu $-integrable) function $\varphi :X\rightarrow \mathbb{R}$. For
any given measurable set $A\subseteq X$ (i.e., $A\in \mathcal{A}$), $\tau
_{A}$ denotes the time average of the corresponding characteristic function $%
\chi _{A}$. It is called the \emph{mean sojourn time}\ associated with $A\in 
\mathcal{A}$. These objects allow us to explain the notion of ergodicity:

\begin{definition}[Ergodicity]
\mbox{ }\newline
Under the conditions of Theorem \ref{th Birkhoff}, the pair $(f,\mu )$ is
ergodic\footnote{%
I.e., the transformation $f$ is ergodic with respect to the measure $\mu $\
or the measure $\mu $ is ergodic with respect to the transformation $f$.}
if, for all $A\in \mathcal{A}$, $\tau _{A}=\mu (A)$, i.e., the constant
function $\mu (A)$ on $X$ is a mean sojourn time for $A$.
\end{definition}

By definition, ergodicity of a probability measure $\mu $ thus refers to the
fact that, for all Borel sets $A\in \mathcal{A}$, 
\begin{equation*}
\lim_{n\rightarrow \infty }\frac{1}{n}\sum_{m=1}^{n}\chi _{A}\circ
f^{m}=\int_{X}\chi _{A}\left( x\right) \mu \left( \mathrm{d}x\right) \text{ }%
,\text{\qquad }\mu \text{-a.s.}
\end{equation*}%
See Theorem \ref{th Birkhoff}. This property is equivalent to the
corresponding one with the characteristic functions $\chi _{A}$, $A\in 
\mathcal{A}$, replaced by general $\mu $-integrable functions. See, for
instance, \cite[Proposition 4.1.3]{Viana}. The ergodicity is also equivalent
to a clustering property \cite[Proposition 4.1.4]{Viana}. All this leads to
the following proposition:

\begin{proposition}[Properties equivalent to ergodicity]
\label{prop erg disp free}\label{prop erg clust}\mbox{ }\newline
Under the assumptions of Theorem \ref{th Birkhoff}, the following conditions
are equivalent: \newline
\emph{(i)} $(f,\mu )$ is ergodic.\newline
\emph{(ii)} For any $\mu $-integrable function $\varphi :X\rightarrow 
\mathbb{R}$, 
\begin{equation*}
\lim_{n\rightarrow \infty }\frac{1}{n}\sum_{m=1}^{n}\varphi \circ
f^{m}=\int_{X}\varphi (x)\mu \left( \mathrm{d}x\right) \text{ },\qquad \mu 
\text{-a.s.}
\end{equation*}%
\emph{(iii)} For all $p,q\in (1,\infty )$ with $1/p+1/q=1$ and all functions 
$\varphi \in \mathcal{L}^{p}(\mu )$, $\psi \in \mathcal{L}^{q}(\mu )$,%
\begin{equation*}
\lim_{k\rightarrow \infty }\frac{1}{k}\sum_{n=0}^{k-1}\int_{X}\psi \left(
x\right) \varphi \circ f^{n}\left( x\right) \mu \left( \mathrm{d}x\right)
=\left( \int_{X}\psi \left( x\right) \mu \left( \mathrm{d}x\right) \right)
\left( \int_{X}\varphi \left( x\right) \mu \left( \mathrm{d}x\right) \right) 
\text{ }.
\end{equation*}
\end{proposition}

\begin{proof}
By \cite[Proposition 4.1.3, cf. (i) and (iii)]{Viana}, (i) holds true iff
(ii) is satisfied, while \cite[Proposition 4.1.4, cf. (i) and (iii)]{Viana}
proves that (i) is equivalent to the clustering properties (iii).
\end{proof}

It is well known that the ergodicity property precisely characterizes the
extreme invariant measures. In fact, under the conditions of Theorem \ref{th
Birkhoff}, let 
\begin{equation}
\mathcal{P}\left( f\right) \doteq \left\{ \mu \in \mathcal{P}:f_{\ast }(\mu
)=\mu \right\} \   \label{invariant probability measures}
\end{equation}%
be the set of all probability measures on $\mathcal{A}$ that are invariant
with respect to $f$, where $\mathcal{P}$ is the set of all (i.e., not
necessarily invariant) probability measures on $\mathcal{A}$. It is a convex
set and $\mathcal{P}_{\mathrm{erg}}(f)$ denotes the subset of its extreme
points, which turns out to be nothing but the set of ergodic measures with
respect to the measurable transformation $f$:

\begin{proposition}[Ergodicity and extremality]
\label{prop erg extreme}\mbox{ }\newline
Under the conditions of Theorem \ref{th Birkhoff}, the invariant measure $%
\mu \in \mathcal{P}(f)$ is ergodic iff $\mu \in \mathcal{P}_{\mathrm{erg}%
}(f) $.
\end{proposition}

\noindent See, for instance, \cite[Proposition 4.3.2]{Viana} or \cite[%
Theorem 6.10]{Walters} for more details, including the proof.

Of course, all these results can be applied to the case where $X$ is the
compact topological space $\Sigma \doteq \Omega ^{\mathbb{N}}$, the set of
infinite strings on the alphabet $\Omega $ endowed with the product
topology, and $f$ is the transformation $T:\Sigma \rightarrow \Sigma $
defined by (\ref{Transformation definition}). In this particular case, we
can identify $\mathcal{P}(T)$, defined above by (\ref{invariant probability
measures}), with the weak$^{\ast }$ compact and convex space of all positive
normalized linear functionals on $C(\Sigma )$, thanks to the Riesz-Markov
theorem \cite[Theorem 4.68]{BruPedra-textbook}. In particular, $\mathcal{P}%
(T)\subseteq C(\Sigma )^{\ast }$. By the Krein-Milman theorem \cite[Theorems
3.4 (b) and 3.21]{Rudin}, the subset $\mathcal{P}_{\mathrm{erg}}(T)$ of its
extreme points is nonempty and 
\begin{equation*}
\mathcal{P}\left( T\right) =\overline{\mathrm{co}}\text{ }\mathcal{P}_{%
\mathrm{erg}}\left( T\right) \ .
\end{equation*}%
Therefore, by\ Proposition \ref{prop erg extreme}, there are ergodic
measures in $\mathcal{P}(T)$ and any $T$-invariant probability measure $\mu
\in \mathcal{P}(T)$ is the weak$^{\ast }$ limit of convex combinations of
ergodic probability measures.

Since $\Sigma $ is a metrizable compact space, $C(\Sigma )$ is also
separable with respect to the topology of uniform convergence and, hence,
the weak$^{\ast }$ topology of $\mathcal{P}(T)\subseteq C(\Sigma )^{\ast }$\
is also metrizable. In particular, by Lemma \ref{lemma extr gd}, $\mathcal{P}%
(T)$ is a Borel set for the weak$^{\ast }$ topology and using the Choquet
theorem (Theorem \ref{th Choquet}), one checks that $\mathcal{P}(T)$ is a
(Choquet) simplex:

\begin{proposition}[$\mathcal{P}(T)$ as a Choquet simplex]
\label{prop Choquet inv meas}\mbox{ }\newline
For any $T$-invariant measure $\mu \in \mathcal{P}(T)$ there is a unique
(Choquet) probability measure $\xi _{\mu }$ on the Borel $\sigma $-algebra
associated with the weak$^{\ast }$ topology of $\mathcal{P}(T)$, that is
supported in $\mathcal{P}_{\mathrm{erg}}(T)$, the barycenter of which is $%
\mu $. In particular, 
\begin{equation}
\xi _{\mu }(\mathcal{P}(T)\backslash \mathcal{P}_{\mathrm{erg}}(T))=0\ .
\label{support}
\end{equation}
\end{proposition}

\begin{proof}
The existence of a (Choquet) probability measure $\xi _{\mu }$ with the
asserted property is a direct consequence of Theorem \ref{th Choquet}, since 
$\mathcal{P}(T)$ is a metrizable, weak$^{\ast }$-compact and convex set, as
already explained. It only remains to show its uniqueness. So fix\textit{\ }$%
\mu \in \mathcal{P}(T)$\textit{\ }and let $\xi _{\mu }$ be any Choquet
measure on $\mathcal{P}(T)$ whose barycenter is $\mu $, as given by Theorem %
\ref{th Choquet}. Using Proposition \ref{prop erg clust}, Equation (\ref%
{support}), Lebesgue's dominated convergence theorem as well as the notation 
$\mathbb{E}_{n}[\cdot ]$ for the Birkhoff sum (\ref{sdsdsdssdsd}), we deduce
that, for any $\varphi _{1},\varphi _{2}\in C(\Sigma )$,%
\begin{eqnarray*}
\int_{\mathcal{P}\left( T\right) }\hat{\mu}\left( \varphi _{2}\right) \hat{%
\mu}\left( \varphi _{1}\right) \xi _{\mu }\left( \mathrm{d}\hat{\mu}\right)
&=&\int_{\mathcal{P}_{\mathrm{erg}}\left( T\right) }\lim_{n\rightarrow
\infty }\hat{\mu}\left( \varphi _{2}\mathbb{E}_{n}\left[ \varphi _{1}\right]
\right) \xi _{\mu }\left( \mathrm{d}\hat{\mu}\right) \\
&=&\lim_{n\rightarrow \infty }\int_{\mathcal{P}_{\mathrm{erg}}\left(
T\right) }\hat{\mu}\left( \varphi _{2}\mathbb{E}_{n}\left[ \varphi _{1}%
\right] \right) \xi _{\mu }\left( \mathrm{d}\hat{\mu}\right) \\
&=&\lim_{n\rightarrow \infty }\mu \left( \varphi _{2}\mathbb{E}_{n}\left[
\varphi _{1}\right] \right) \ .
\end{eqnarray*}%
The second equality follows from Lebesgue's dominated convergence theorem
because 
\begin{equation*}
\left\vert \nu \left( \varphi _{2}\mathbb{E}_{n}\left[ \varphi _{1}\right]
\right) \right\vert \leq \left\Vert \varphi _{1}\right\Vert _{\infty
}\left\Vert \varphi _{2}\right\Vert _{\infty }\text{ },\text{\qquad }\nu \in 
\mathcal{P}\left( T\right) ,\text{ }n\in \mathbb{N}\text{ }.
\end{equation*}%
More generally, by the same arguments, for any $k\in \{2,\ldots ,\infty \}\ $%
and $\varphi _{1},\ldots ,\varphi _{k}\in C(\Sigma )$, 
\begin{equation}
\int_{\mathcal{P}\left( T\right) }\hat{\mu}\left( \varphi _{k}\right) \cdots 
\hat{\mu}\left( \varphi _{1}\right) \xi _{\mu }\left( \mathrm{d}\hat{\mu}%
\right) =\lim_{n_{k-1}\rightarrow \infty }\cdots \lim_{n_{1}\rightarrow
\infty }\mu \left( \varphi _{k}\mathbb{E}_{n_{k-1}}\left[ \varphi _{k-1}%
\right] \cdots \mathbb{E}_{n_{1}}\left[ \varphi _{1}\right] \right) \ .
\label{sdsdsdsffgjhghjgjkhk}
\end{equation}%
Let 
\begin{equation*}
\mathfrak{X}\doteq \mathrm{span}\left\{ \hat{\varphi}_{1}\cdots \hat{\varphi}%
_{k}:k\in \mathbb{N}_{0},\ \varphi _{1},\ldots ,\varphi _{k}\in C\left(
\Sigma \right) \right\} \subseteq C\left( \mathcal{P}\left( T\right) ;%
\mathbb{R}\right) \text{ },
\end{equation*}%
where $\hat{\varphi}_{1}\cdots \hat{\varphi}_{k}\doteq 1$ (i.e., the
constant function $1$) when $k=0$, $\hat{\varphi}:\mathcal{P}(T)\rightarrow 
\mathbb{R}$ is defined for any $\varphi \in C(\Sigma )$ by 
\begin{equation}
\hat{\varphi}\left( \nu \right) =\nu \left( \varphi \right) \ ,\qquad \nu
\in \mathcal{P}\left( T\right) \ ,  \label{adsadassffdgdfgdgfgh}
\end{equation}%
and $C(\mathcal{P}(T);\mathbb{R})$ stands for the space of weak$^{\ast }$%
-continuous functions $\mathcal{P}(T)\rightarrow \mathbb{R}$. By Equation (%
\ref{sdsdsdsffgjhghjgjkhk}), the quantities%
\begin{equation*}
\int_{\mathcal{P}\left( T\right) }\hat{\mu}\left( \varphi _{k}\right) \cdots 
\hat{\mu}\left( \varphi _{1}\right) \xi _{\mu }\left( \mathrm{d}\hat{\mu}%
\right) \ ,\qquad k\in \{2,\ldots ,\infty \},\ \varphi _{1},\ldots ,\varphi
_{k}\in C(\Sigma )\ ,
\end{equation*}%
depend only on the probability measure $\mu $, but not on the particular
choice of the Choquet measure $\xi _{\mu }$ representing $\mu $. As a
consequence, the quantities 
\begin{equation}
\xi _{\mu }\left( f\right) \equiv \int_{\mathcal{P}\left( T\right) }f\left( 
\hat{\mu}\right) \xi _{\mu }\left( \mathrm{d}\hat{\mu}\right) \ ,\qquad f\in 
\mathfrak{X}\ ,  \label{dghl}
\end{equation}%
also depend only on the probability measure $\mu $. Observe that the
elements of $C(\Sigma )$, seen as weak$^{\ast }$-continuous linear
functionals\footnote{%
I.e., for any $\varphi \in C(\Sigma ;\mathbb{R})$, $\hat{\varphi}\left( \nu
\right) =\nu \left( \varphi \right) $ for all $\nu \in C(\Sigma ;\mathbb{R}%
)^{\ast }$. Cf. Equation (\ref{adsadassffdgdfgdgfgh}).} on the dual of the
Banach space $(C(\Sigma ),\left\Vert \cdot \right\Vert _{\infty })$,
(trivially) separate the points of this dual space. As $\mathcal{P}(T)$ is
weak$^{\ast }$-compact, we can invoke the Stone-Weierstrass theorem \cite[%
Theorem 7.191]{BruPedra-textbook} to deduce that $\mathfrak{X}$ is a
uniformly dense subalgebra of $C(\mathcal{P}(T);\mathbb{R})$. Thus, the
quantities (\ref{dghl}) extended to all continuous functions $f\in C(%
\mathcal{P}(T);\mathbb{R})$ only depend on $\mu $. Therefore, the Choquet
measure $\xi _{\mu }$, which is a probability measure on $\mathcal{P}(T)$,
must be unique.
\end{proof}

Proposition \ref{prop Choquet inv meas} means that $\mathcal{P}(T)$ is a
so-called \emph{Choquet simplex}. In fact, $\mathcal{P}(T)$ is even the
so-called Poulsen simplex (which is unique \cite[Theorem 2.3.]%
{Lindenstrauss-etal} up to an affine homeomorphism), as a consequence of the
weak$^{\ast }$ density of its extreme points:

\begin{proposition}[Weak$^{\ast }$ density of ergodic measures]
\label{density of ergodic measures}\mbox{ }\newline
The set $\mathcal{P}_{\mathrm{erg}}(T)$ of ergodic measures is a $G_{\delta
} $ weak$^{\ast }$-dense subset of $\mathcal{P}(T)$.
\end{proposition}

\begin{proof}
Note that the weak$^{\ast }$ topology of $\mathcal{P}(T)$ is metrizable and
thus, the set $\mathcal{P}_{\mathrm{erg}}(T)$ of extreme points of $\mathcal{%
P}(T)$ is a $G_{\delta }$ set, thanks to Lemma \ref{lemma extr gd}. We now
prove the weak$^{\ast }$-density of $\mathcal{P}_{\mathrm{erg}}(T)$ in
several steps. We start with a general observation on the Banach space $%
C(\Sigma )$ of all continuous functions on character strings $\Sigma \doteq
\Omega ^{\mathbb{N}}$. \medskip

\noindent \underline{Step 1:} For any $m\in \mathbb{N}$, we define the
subspace%
\begin{equation*}
C_{\mathrm{cyl}}\left( \Sigma \right) \doteq \bigcup_{m\in \mathbb{N}%
}C\left( \Sigma \right) \cap \mathcal{J}_{m}
\end{equation*}%
of so-called cylinder functions, where 
\begin{equation*}
\mathcal{J}_{m}\doteq \left\{ \varphi :\Sigma \rightarrow \mathbb{R}:\exists
f:\mathbb{R}^{m}\rightarrow \mathbb{R}\text{ such that }\varphi \left(
\sigma \right) =f\left( \omega _{1},\ldots ,\omega _{m}\right) \text{ for
all }\sigma =\left( \omega _{n}\right) _{n\in \mathbb{N}}\in \Sigma \right\}
\ .
\end{equation*}%
The set $C_{\mathrm{cyl}}\left( \Sigma \right) $ is a dense subspace of $%
C(\Sigma )$, by compactness of $\Sigma $ (cf. Tychonoff's theorem \cite[%
Section A.3]{Rudin}) and the equicontinuity of well-chosen families of
continuous functions on compacts (via Ascoli's theorem \cite[Section A.5]%
{Rudin}). Indeed, fix some $s=\left( s_{n}\right) _{n\in \mathbb{N}}\in
\Sigma $. For any $\varphi \in C(\Sigma )$ and $m\in \mathbb{N}$, we define
the continuous function $\varphi _{m}\in C(\Sigma )\cap \mathcal{J}_{m}$ by 
\begin{equation*}
\varphi _{m}\left( \sigma \right) \doteq \varphi \left( \sigma _{m}\right) \
,\text{\qquad }\sigma =\left( \omega _{n}\right) _{n\in \mathbb{N}}\in
\Sigma \ ,
\end{equation*}%
where 
\begin{equation*}
\sigma _{m}=\left( \omega _{1},\ldots ,\omega _{m},s_{m+1},s_{m+2},\ldots
\right) \in \Sigma \ ,\text{\qquad }\sigma =\left( \omega _{n}\right) _{n\in 
\mathbb{N}}\in \Sigma \ .
\end{equation*}%
Recall that the metric (\ref{metric sigma}) for $\eta =1/2$, that is,%
\begin{equation*}
d_{1/2}\left( \sigma ,\sigma ^{\prime }\right) \doteq \sum_{n\in \mathbb{N}%
}2^{-n}d\left( \omega _{n},\omega _{n}^{\prime }\right) \ ,\text{\qquad }%
\sigma =\left( \omega _{n}\right) _{n\in \mathbb{N}},\sigma ^{\prime
}=\left( \omega _{n}^{\prime }\right) _{n\in \mathbb{N}}\in \Sigma \ ,
\end{equation*}%
generates the topology of $\Sigma $. From (\ref{max metric}) we observe that 
\begin{equation}
d_{1/2}\left( \sigma ,\sigma _{m}\right) =\sum_{n=m}^{\infty }2^{-n}d\left(
\omega _{n},s_{n}\right) \leq 2^{1-m}\ ,\text{\qquad }\sigma =\left( \omega
_{n}\right) _{n\in \mathbb{N}}\in \Sigma ,\ m\in \mathbb{N}\ .
\label{sdsdsdsdsdsdsdsdsd}
\end{equation}%
This implies that the family $\{\varphi _{m}\}_{n\in \mathbb{N}}$ of
functions on the compact $\Sigma $ is equicontinuous. By compactness of $%
\Sigma $ and continuity of $\varphi $, the family is also uniformly bounded.
Again by (\ref{sdsdsdsdsdsdsdsdsd}), $\varphi _{m}$ converges pointwise to $%
\varphi $, as $m\rightarrow \infty $. Now, using Ascoli's theorem \cite[%
Section A.5]{Rudin} we deduce that, at least along some subsequence, the
convergence is uniform, that is, for any $\varphi \in C(\Sigma )$,%
\begin{equation*}
\liminf_{m\rightarrow \infty }\left\Vert \varphi -\varphi _{m}\right\Vert
_{\infty }=\liminf_{m\rightarrow \infty }\sup_{\sigma \in \Sigma }\left\vert
\varphi \left( \sigma \right) -\varphi \left( \sigma _{m}\right) \right\vert
=0\ .
\end{equation*}%
In other words, $C_{\mathrm{cyl}}\left( \Sigma \right) $ is dense in $%
C(\Sigma )$. \medskip

\noindent \underline{Step 2:} For any $m\in \mathbb{N}$ and $k\in \mathbb{N}$%
, let 
\begin{equation*}
C_{m,k}\left( \Sigma \right) \doteq \left\{ \varphi \circ T^{\left(
k-1\right) m}:\varphi \in C\left( \Sigma \right) \cap \mathcal{J}%
_{m}\right\} \text{ }.
\end{equation*}%
In fact, $C_{m,k}(\Sigma )$ is nothing but the space of continuous functions 
$\left( \omega _{n}\right) _{n\in \mathbb{N}}\mapsto \varphi (\left( \omega
_{n}\right) _{n\in \mathbb{N}})$ that only depend on $\omega
_{(k-1)m+1},\ldots ,\omega _{km}$. These closed subalgebras of $C\left(
\Sigma \right) $ are all naturally isomorphic to $C(\Omega ^{m})$. Given any
positive normalized functionals (i.e., probability measures on $\Omega ^{m}$%
) $\nu _{m,k}$ on the subalgebras $C_{m,k}(\Sigma )\subseteq C\left( \Sigma
\right) $, $m,k\in \mathbb{N}$, there is a unique positive normalized
functional $\bigotimes_{k\in \mathbb{N}}\nu _{m,k}$ on $C(\Sigma )$, for
which, for all $l\in \mathbb{N}$, $k_{1},\ldots ,k_{l}\in \mathbb{N}$, $%
k_{1}<\cdots <k_{l}$, and $\varphi _{1}\in C_{m,k_{1}}(\Sigma ),\ldots
,\varphi _{l}\in C_{m,k_{l}}(\Sigma )$, 
\begin{equation*}
\bigotimes_{k\in \mathbb{N}}\nu _{m,k}\left( \varphi _{1}\cdots \varphi
_{l}\right) =\nu _{m,k_{1}}\left( \varphi _{1}\right) \cdots \nu
_{m,k_{l}}\left( \varphi _{l}\right) \text{ }.
\end{equation*}%
In other words, identifying positive normalized functionals on functions on
compacts with probability measures, $\bigotimes_{k\in \mathbb{N}}\nu _{m,k}$
is nothing but the product probability measure of the probability measures $%
\nu _{m,k}$, $m,k\in \mathbb{N}$. Fix now any $T$-invariant measure $\mu \in 
\mathcal{P}(T)\subseteq C(\Sigma )^{\ast }$ and, for all $m,k\in \mathbb{N}$%
, define the positive normalized functional $\mu _{m,k}\in C_{m,k}(\Sigma
)^{\ast }$ as being the restriction to $C_{m,k}(\Sigma )$ of $\mu $, i.e., 
\begin{equation}
\mu _{m,k}\left( \varphi \right) \doteq \mu \left( \varphi \right) \ ,\text{%
\qquad }\varphi \in C\left( \Sigma \right) \cap \mathcal{J}_{m}\ .
\label{mu_m1}
\end{equation}%
Then, for all $m\in \mathbb{N}$, let the probability measure $\mu _{m}$ on $%
\Sigma $ be defined by 
\begin{equation}
\mu _{m}^{\otimes }\doteq \bigotimes_{k\in \mathbb{N}}\mu _{m,k}\in C(\Sigma
)^{\ast }\text{ }.  \label{mu_m0}
\end{equation}%
Observe that, by construction, $T_{\ast }^{m}\left( \mu _{m}^{\otimes
}\right) =\mu _{m}^{\otimes }$, where $T_{\ast }(\nu )$ stands for the
pushforward of an arbitrary probability measure $\nu \in \mathcal{P}$ with
respect to $T$, i.e., $T_{\ast }\left( \nu \right) (A)\doteq \nu (T^{-1}(A))$
for any Borel set $A\subseteq \Sigma $. That is, $\mu _{m}^{\otimes }$ is a $%
m$-periodic probability measure. We can now define a $T$-invariant measure
by averaging out $\mu _{m}^{\otimes }$ within a period: 
\begin{equation}
\mu _{m}\doteq \frac{1}{m}\sum_{n=0}^{m-1}T_{\ast }^{n}\left( \mu
_{m}^{\otimes }\right) \in \mathcal{P}\left( T\right) \ .  \label{mu_m2}
\end{equation}%
By construction we also have 
\begin{equation*}
\lim_{m\rightarrow \infty }\mu _{m}\left( \varphi \right) =\mu \left(
\varphi \right) \ ,\text{\qquad }\varphi \in C_{\mathrm{cyl}}\left( \Sigma
\right) \text{ }.
\end{equation*}%
Since $C_{\mathrm{cyl}}\left( \Sigma \right) $ is a dense subset of $%
C(\Sigma )$ (Step 1), we deduce that the sequence $(\mu _{m})_{m\in \mathbb{N%
}}$ converges to $\mu $, in the weak$^{\ast }$ topology. It remains to show
that $\mu _{m}$ is an ergodic measure for each $m\in \mathbb{N}$.\medskip

\noindent \underline{Step 3:} Fix again a $T$-invariant measure $\mu \in 
\mathcal{P}(T)$. Let $j_{1},j_{2}\in \mathbb{N}$ and $p,q\in (1,\infty )$
with $1/p+1/q=1$. Take $\varphi \in C(\Sigma )\cap \mathcal{J}_{j_{1}}$ and $%
\psi \in C(\Sigma )\cap \mathcal{J}_{j_{2}}$. Then, for any $k\in \mathbb{N}$%
, $k\geq j_{2}+m$, using Equations (\ref{mu_m0})--(\ref{mu_m2}), we compute
that%
\begin{multline*}
\frac{1}{k}\sum_{n=0}^{k}\int_{\Sigma }\psi \left( \sigma \right) \left(
\varphi \circ T^{k}\right) \left( \sigma \right) \mu _{m}\left( \mathrm{d}%
\sigma \right) =\frac{1}{k}\sum_{n=0}^{j_{2}+m}\int_{\Sigma }\psi \left(
\sigma \right) \left( \varphi \circ T^{k}\right) \left( \sigma \right) \mu
_{m}\left( \mathrm{d}\sigma \right) \\
+\left( 1-\frac{\left( j_{2}+m\right) }{k}\right) \int_{\Sigma }\psi \left(
\sigma \right) \mu _{m}\left( \mathrm{d}\sigma \right) \int_{\Sigma }\varphi
\circ T^{k}\left( \sigma \right) \mu _{m}\left( \mathrm{d}\sigma \right) ,
\end{multline*}%
which, combined with H\"{o}lder's inequality and the $T$-invariance of $\mu
_{m}$, yields the bound%
\begin{eqnarray}
&&\left\vert \frac{1}{k}\sum_{n=0}^{k}\int_{\Sigma }\psi \left( \sigma
\right) \left( \varphi \circ T^{k}\right) \left( \sigma \right) \mu
_{m}\left( \mathrm{d}\sigma \right) -\int_{\Sigma }\psi \left( \sigma
\right) \mu _{m}\left( \mathrm{d}\sigma \right) \int_{\Sigma }\varphi \left(
\sigma \right) \mu _{m}\left( \mathrm{d}\sigma \right) \right\vert  \notag \\
&\leq &\frac{2\left( j_{2}+m\right) }{k}\left( \left( \int_{\Sigma
}\left\vert \psi \left( \sigma \right) \right\vert ^{q}\mu _{m}\left( 
\mathrm{d}\sigma \right) \right) ^{1/q}\left( \int_{\Sigma }\left\vert
\varphi \left( \sigma \right) \right\vert ^{p}\mu _{m}\left( \mathrm{d}%
\sigma \right) \right) ^{1/p}\right)  \label{ssdsdfsdfsd}
\end{eqnarray}%
It follows that 
\begin{equation*}
\lim_{k\rightarrow \infty }\frac{1}{k}\sum_{n=0}^{k}\int_{\Sigma }\psi
\left( \sigma \right) \left( \varphi \circ T^{k}\right) \left( \sigma
\right) \mu _{m}\left( \mathrm{d}\sigma \right) =\int_{\Sigma }\psi \left(
\sigma \right) \mu _{m}\left( \mathrm{d}\sigma \right) \int_{\Sigma }\varphi
\left( \sigma \right) \mu _{m}\left( \mathrm{d}\sigma \right) \ .
\end{equation*}%
By Step 1, $C_{\mathrm{cyl}}\left( \Sigma \right) $ is dense in $C(\Sigma )$%
. Therefore, we can extend the last limit to all $\varphi \in C(\Sigma )$
and $\psi \in C(\Sigma )$. In addition, since $C(\Sigma )$ is dense in $%
\mathcal{L}^{p}(\mu _{m})$ for any $p\in (1,\infty )$, using H\"{o}lder's
inequality as done in (\ref{ssdsdfsdfsd}), we can then extend the last limit
to all $\varphi \in \mathcal{L}^{p}(\mu _{m})$ and $\psi \in \mathcal{L}%
^{q}(\mu _{m})$, $p,q\in (1,\infty )$ with $1/p+1/q=1$. Now, by Proposition %
\ref{prop erg disp free}, $\mu _{m}$ is an ergodic measure for all $m\in 
\mathbb{N}$.
\end{proof}

\subsubsection{The entropy functional as a thermodynamic limit\label{Further
Study of the Entropy}}

Our initial goal in this subsection is to show that the entropy is affine in
the case where the alphabet is a general (i.e., not necessarily finite)
compact metric space (see Proposition \ref{Affinity of the entropy copy(1)}%
), what is well-known for finite alphabets \cite[Theorem 8.1]{Walters}. We
summarize the main steps of the proof of this property, which stems from 
\cite{ACR}. Then, we give in Corollary \ref{ergodic abundance} a proof of
ergodic abundance for alphabets that are general compact metric spaces; this
result was already known in the case of a finite alphabet (see \cite{L3} and
also \cite{Clime}).

First, just like in \cite[Section 2]{ACR}, we define the relative entropy
(also known as Kullback-Leibler divergence) of a probability measure $\mu
\in \mathcal{P}$ with respect to second one $\nu $, both on the measurable
space $(\Sigma ,\mathfrak{S})$: For any sub-$\sigma $-algebra $\mathcal{A}%
\subseteq \mathfrak{S}$, 
\begin{equation}
\mathcal{H}_{\mathcal{A}}\left( \mu |\nu \right) \doteq 
\begin{cases}
\int_{\Sigma }f_{\mu }\left( \sigma \right) \ln f_{\mu }\left( \sigma
\right) \ \nu \left( \mathrm{d}\sigma \right) \geq 0\ , & \ f_{\mu }=\frac{%
d\mu |_{\mathcal{A}}}{d\nu |_{\mathcal{A}}}\ \text{if }\mu |_{\mathcal{A}%
}\ll \nu |_{\mathcal{A}} \\[0.5cm] 
\infty \ , & \ \text{otherwise}\ ,%
\end{cases}
\label{def1}
\end{equation}%
where $f_{\mu }$ is the Radon-Nikodym derivative of the restrictions to $%
\mathcal{A}$, $\mu |_{\mathcal{A}}$ and $\nu |_{\mathcal{A}}$, of the
measures $\mu $ and $\nu $. Here, $\mu |_{\mathcal{A}}\ll \nu |_{\mathcal{A}%
} $ means that $\mu |_{\mathcal{A}}$ is absolutely continuous with respect
to $\nu |_{\mathcal{A}}$ and we use the convention $f_{\mu }\left( \sigma
\right) \ln f_{\mu }\left( \sigma \right) \doteq 0$ when $f_{\mu }\left(
\sigma \right) =0$, as is usual. As it is well-known, the relative entropy
is monotonically decreasing with respect to \textquotedblleft coarse
graining\textquotedblright , that is, for any $\sigma $-subalgebra $\mathcal{%
B}\subseteq \mathcal{A}$, we have%
\begin{equation}
\mathcal{H}_{\mathcal{B}}\left( \mu |\nu \right) \leq \mathcal{H}_{\mathcal{A%
}}\left( \mu |\nu \right) \text{ }.  \label{ineq coarse graining rel entr}
\end{equation}

For any finite subset $\Lambda \subseteq \mathbb{N}$, we apply this
definition to the initial $\sigma $-algebra $\mathfrak{S}_{\Lambda }$ of the
canonical projections\footnote{$(\sigma _{i})_{i\in \mathbb{N}}\mapsto
(\sigma _{i})_{i\in \Lambda }$} $\Sigma \rightarrow \Omega ^{\Lambda }$. For
simplicity we use the notation 
\begin{equation}
\mathcal{H}_{\Lambda }\left( \mu |\nu \right) \equiv \mathcal{H}_{\mathfrak{S%
}_{\Lambda }}\left( \mu |\nu \right)  \label{def2}
\end{equation}%
for any two probability measures $\mu ,\nu $ on the measurable space $%
(\Sigma ,\mathfrak{S})$ and finite set $\Lambda \subseteq \mathbb{N}$. We
call this quantity the relative entropy of $\mu $ in $\Lambda $ with respect
to $\nu $. Note from Inequality (\ref{ineq coarse graining rel entr}) that%
\begin{equation}
\mathcal{H}_{\Lambda ^{\prime }}\left( \mu |\nu \right) \leq \mathcal{H}%
_{\Lambda }\left( \mu |\nu \right) \text{ },\qquad \Lambda ^{\prime
}\subseteq \Lambda \text{ }.  \label{ineq coarse graining rel entr Lambda}
\end{equation}

We fix an \textit{a priori} probability measure $\mathrm{m}$ on $\Omega $
and, for any probability measure $\mu \in \mathcal{P}$ and finite subset $%
\Lambda \subseteq \mathbb{N}$, we define $\mathcal{H}_{\Lambda }\left( \mu
\right) $, the so-called specific entropy in $\Lambda $ of $\mu $. It is
nothing but \emph{minus} the relative entropy of $\mu $ in $\Lambda $ with
respect to the product measure on $\Sigma $ constructed from $\mathrm{m}$,
i.e., 
\begin{equation}
\mathrm{m}_{\otimes }\doteq \bigotimes_{k\in \mathbb{N}}\mathrm{m}\in 
\mathcal{P}(T)\ .  \label{product measure}
\end{equation}%
In other words, for any probability measure $\mu $ and finite subset $%
\Lambda \subseteq \mathbb{N}$, 
\begin{equation}
\mathcal{H}_{\Lambda }\left( \mu \right) \doteq -\mathcal{H}_{\Lambda
}\left( \mu |\mathrm{m}_{\otimes }\right) \ .  \label{def3}
\end{equation}%
In the limit $n\rightarrow \infty $ of finite subsets $\{1,\ldots ,n\}$ it
gives the so-called specific entropy per site of any \emph{$T$-invariant}
probability measure $\mu \in \mathcal{P}(T)$, which is well-defined in this
case by the limit%
\begin{equation}
\mathfrak{h}\left( \mu \right) \doteq \lim_{n\rightarrow \infty }n^{-1}%
\mathcal{H}_{\{1,\ldots ,n\}}\left( \mu \right) \in \lbrack -\infty ,0]\ .
\label{thermo limit entropy}
\end{equation}%
The functional $\mathfrak{h}:\mathcal{P}(T)\rightarrow \lbrack -\infty ,0]$
is both concave and weak$^{\ast }$ upper semicontinuous. See \cite{Georgii}
for more details. Note in particular that \cite[Definition 15.13, page 315]%
{Georgii} essentially correspond to this claim, but the lattice is $\mathbb{Z%
}$ there, instead of $\mathbb{N}$.

A thermodynamic limit similar to (\ref{thermo limit entropy}) can also be
obtained by using other reference measures, not just the product measure (%
\ref{product measure}), like in \cite[Theorem 15.30]{Georgii}. Observe that
the product measure is nothing but the Gibbs measure associated with a
constant potential. Indeed, more generally, we can take the (unique) Gibbs
measure associated with any H\"{o}lder potential $f$ as reference measure.
This is done as follows:\medskip

\noindent \underline{Step 1:} Given any H\"{o}lder potential $f\in C^{\alpha
}(\Sigma )$ ($\alpha \in (0,1)$), we define an appropriate normalized
potential $\bar{f}$ and consider the associated Ruelle operator $\mathcal{L}%
_{\bar{f}}^{\mathrm{m}}$. In particular, we should have $\mathcal{L}_{\bar{f}%
}^{\mathrm{m}}(\mathbf{1})(\sigma )=1$ for all $\sigma \in \Sigma $. This is
done in the following way: As already explained in \cite{ACR}, one can use
the generalization of the Ruelle-Perron-Frobenius theorem given in \cite%
{LMMS} to define a cohomologous normalized potential defined by 
\begin{equation*}
\bar{f}\doteq f+\ln h_{f}-\ln (h_{f}\circ T)-\ln \lambda _{f}\ ,
\end{equation*}%
for any $f\in C^{\alpha }(\Sigma )$, where $\lambda _{f}$ is a maximal
eigenvalue of $\mathcal{L}_{f}^{\mathrm{m}}$ having $h_{f}\in C^{\alpha
}(\Sigma )\cap \mathcal{C}^{+}$ as an eigenfunction. \medskip

\noindent \underline{Step 2:} The relation between $f$ and $\bar{f}$, or
more explicitly $\mathcal{L}_{f}^{\mathrm{m}}$ and $\mathcal{L}_{\bar{f}}^{%
\mathrm{m}}$, is explained in \cite{LMMS} and refers to the following
observations: Since $\lambda _{f}$ is a maximal eigenvalue of $\mathcal{L}%
_{f}^{\mathrm{m}}$\ there is an eigenprobability $\nu _{f}$ such that 
\begin{equation*}
(\mathcal{L}_{f}^{\mathrm{m}})^{\ast }\left( \nu _{f}\right) =\lambda
_{f}\nu _{f}
\end{equation*}%
with $(\mathcal{L}_{f}^{\mathrm{m}})^{\ast }$ being the usual dual operator
acting on probability measures. It turns out that the product $h_{f}\nu _{f}$%
, properly normalized, is an equilibrium measure $\mu _{f}\in \mathcal{P}(T)$
for $f$, in the sense of Definition \ref{Equilibrium measures def}. In other
words, $\mu _{f}$ maximizes the pressure functional $\mathfrak{P}_{L}\left(
f\right) $, as defined by (\ref{sdsdsdsdfssfg}) with $\varphi =f$. Moreover, 
$\log \lambda _{f}=P_{L}(f)$ and the set%
\begin{equation}
\left\{ \nu \in \mathcal{P}\left( T\right) :(\mathcal{L}_{\bar{f}}^{\mathrm{m%
}})^{\ast }\nu =\nu \right\}  \label{eleemnt}
\end{equation}%
has only one element, which is equal to $\mu _{f}=\mu _{\bar{f}}\in \mathcal{%
P}(T)$ (notice that the normalized potential $\bar{f}$ is of H\"{o}lder
class, as $f$ is a H\"{o}lder potential), called\footnote{%
It is sometimes named the normalized DLR probability for $\bar{f}$; for the
exact definition and equivalences, see \cite{CLS}.} the Gibbs measure
associated with the potential $f$. \medskip

Having now the Gibbs measure $\mu _{f}$ at our disposal, we study the
thermodynamic limit of the relative entropy of a $T$-invariant measure with
respect to $\mu _{f}$. This refers to \cite[Theorem 3.1]{ACR}, which states
that, for any probability measure $\mathrm{m}$ on $\Omega $ having full
support, each $\alpha $-H\"{o}lder-continuous function $f\in C^{\alpha
}(\Sigma )$ ($\alpha \in (0,1)$) and any $T$-invariant probability measure $%
\nu \in \mathcal{P}(T)$, the following limit exists: 
\begin{equation}
\mathfrak{h}\left( \nu |\mu _{f}\right) \doteq \lim_{n\rightarrow \infty
}n^{-1}\mathcal{H}_{\{1,\ldots ,n\}}\left( \nu |\mu _{f}\right) =\ln \lambda
_{f}-\int_{\Sigma }f\left( \sigma \right) \nu \left( \mathrm{d}\sigma
\right) -\mathfrak{h}\left( \nu \right) \in \left[ 0,\infty \right] \ ,
\label{sdsdsdsd}
\end{equation}%
where $\mu _{f}\in \mathcal{P}(T)$ is the Gibbs measure associated with $f$,
which is the unique element of the set (\ref{eleemnt}). From \cite[Theorem
3.1]{ACR} together with the positivity of the relative entropy it also
follows that, for any $\alpha \in (0,1)$, 
\begin{equation*}
\mathfrak{h}\left( \mu \right) \leq \ln \lambda _{f}-\int_{\Sigma }f\left(
\sigma \right) \mu \left( \mathrm{d}\sigma \right) \ ,\qquad \mu \in 
\mathcal{P}\left( T\right) ,\ f\in C^{\alpha }\left( \Sigma \right) \text{ },
\end{equation*}%
provided $\mathrm{m}$ is a probability measure on $\Omega $ having full
support. By Proposition \ref{kul copy(1)}, we thus deduce that 
\begin{equation*}
\mathfrak{h}\left( \mu \right) \leq h(\mu )=\inf_{f\in C^{\alpha }(\Sigma
)}\left\{ \ln \lambda _{f}-\int_{\Sigma }f\left( \sigma \right) \mu \left( 
\mathrm{d}\left( \sigma \right) \right) \right\} \ ,\qquad \mu \in \mathcal{P%
}\left( T\right) \ ,
\end{equation*}%
when $\mathrm{m}$ is a probability measure on $\Omega $ having full support.
Recall that $h:\mathcal{P}(T)\rightarrow \mathbb{R}$ is the entropy of
Definition \ref{uod}. In fact, the functionals $\mathfrak{h}$ and $h$ are
equal to each other. This claim refers to \cite[Theorem 3.4]{ACR}. Combined
with Proposition \ref{kul copy(1)} and Equation (\ref{thermo limit entropy}%
), we thus arrive at the Aguiar-Cioletti-Ruviaro theorem:

\begin{theorem}[The entropy as a thermodynamic limit]
\label{teo-equiv-entrop}\mbox{ }\newline
Assume that $\mathrm{m}$ is a probability measure on $\Omega $ having full
support. Then, for any $\mu \in \mathcal{P}(T)$, 
\begin{equation*}
h\left( \mu \right) =\inf_{f\in C^{\alpha }(\Sigma )}\left\{ \ln \lambda
_{f}-\int_{\Sigma }f\left( \sigma \right) \mu \left( \mathrm{d}\left( \sigma
\right) \right) \right\} =\lim_{n\rightarrow \infty }n^{-1}\mathcal{H}%
_{\{1,\ldots ,n\}}\left( \mu \right) \ .
\end{equation*}
\end{theorem}

The representation of entropy as a thermodynamic limit is very useful
because, among other things, it allows us to derive the affine property of
the entropy as a functional on $\mathcal{P}(T)$:

\begin{proposition}[Affine property of the entropy]
\label{Affinity of the entropy copy(1)}\mbox{ }\newline
Assume that $\mathrm{m}$ is a probability measure on $\Omega $ having full
support. Then, the mapping $\mu \mapsto h(\mu )$ from $\mathcal{P}(T)$ to $%
[-\infty ,0]$ is affine.
\end{proposition}

\begin{proof}
We start with the concavity of the entropy $h$, which is easy to prove
directly from its definition (Definition \ref{uod}): For any $\lambda \in
(0,1)$, consider the convex combination 
\begin{equation}
\mu _{3}=\lambda \mu _{1}+\left( 1-\lambda \right) \mu _{2}
\label{liner conb}
\end{equation}%
of two arbitrary $T$-invariant probability measures $\mu _{1},\mu _{2}\in 
\mathcal{P}(T)$. For the $T$-invariant probability measure $\mu _{3}$ and an
arbitrary strictly positive parameter $\epsilon \in \mathbb{R}^{+}$, take an
approximating minimizer of (\ref{oiu54}), i.e., a positive function $f\in 
\mathcal{C}^{+}$ such that%
\begin{equation*}
h\left( \mu _{3}\right) +\epsilon \geq \int_{\Sigma }\ln \left( \frac{%
\mathcal{L}_{0}^{\mathrm{m}}f\left( \sigma \right) }{f\left( \sigma \right) }%
\right) \mu _{3}\left( \mathrm{d}\sigma \right) \ .
\end{equation*}%
It follows from Equation (\ref{liner conb}) that%
\begin{eqnarray*}
h\left( \mu _{3}\right) +\epsilon &\geq &\lambda \int_{\Sigma }\ln \left( 
\frac{\int_{\Omega }f\left( \omega _{0}\sigma \right) \mathrm{m}\left( 
\mathrm{d}\omega _{0}\right) }{f\left( \sigma \right) }\right) \mu
_{1}\left( \mathrm{d}\sigma \right) \\
&&+\left( 1-\lambda \right) \int_{\Sigma }\ln \left( \frac{\int_{\Omega
}f\left( \omega _{0}\sigma \right) \mathrm{m}\left( \mathrm{d}\omega
_{0}\right) }{f\left( \sigma \right) }\right) \mu _{2}\left( \mathrm{d}%
\sigma \right) \\
&\geq &\lambda h\left( \mu _{1}\right) +\left( 1-\lambda \right) h\left( \mu
_{2}\right) \ .
\end{eqnarray*}%
Since $\epsilon \in \mathbb{R}^{+}$ is arbitrary, 
\begin{equation}
h\left( \mu _{3}\right) \geq \lambda h\left( \mu _{1}\right) +\left(
1-\lambda \right) h\left( \mu _{2}\right)  \label{concave}
\end{equation}%
for any $\lambda \in (0,1)$ and $\mu _{1},\mu _{2}\in \mathcal{P}(T)$.

We now prove that the convexity of the entropy $h$, which has become
relatively easy to prove, thanks to Theorem \ref{teo-equiv-entrop}: Take
again arbitrary $\lambda \in \left( 0,1\right) $ and $\mu _{1},\mu _{2}\in 
\mathcal{P}(T)$. Fix also some $n\in \mathbb{N}$. Regarding the $\sigma $%
-algebra $\mathfrak{S}_{\{1,\ldots ,n\}}$, (the convex combination (\ref%
{liner conb})) $\mu _{3}$ has Radon-Nikodym derivative equal to 
\begin{equation}
f_{\mu _{3}}=\frac{d\mu _{3}|_{\mathfrak{S}_{\{1,\ldots ,n\}}}}{d\mathrm{m}%
_{\otimes }|_{\mathfrak{S}_{\{1,\ldots ,n\}}}}=\lambda f_{\mu _{1}}+\left(
1-\lambda \right) f_{\mu _{2}}\ ,  \label{sdssdsds}
\end{equation}%
provided $f_{\mu _{1}}\ $and $f_{\mu _{2}}$, the corresponding Radon-Nikodym
derivatives of $\mu _{1}$ and $\mu _{2}$, exist. In this case, 
\begin{equation*}
\mathcal{H}_{\{1,\ldots ,n\}}\left( \mu _{3}\right) =-\mathcal{H}%
_{\{1,\ldots ,n\}}\left( \mu _{3}|\mathrm{m}_{\otimes }\right)
=-\int_{\Sigma }f_{\mu _{3}}\left( \sigma \right) \ln f_{\mu _{3}}\left(
\sigma \right) \ \mathrm{m}_{\otimes }\left( \mathrm{d}\sigma \right) \ ,
\end{equation*}%
see Equations (\ref{def1})--(\ref{def3}). If one of the two Radon-Nikodym
derivatives $f_{\mu _{1}}\ $and $f_{\mu _{2}}$ does not exist, the
Radon-Nikodym derivative $f_{\mu _{3}}$ does not exist either and, for any $%
\lambda \in (0,1)$, we trivially have 
\begin{equation}
\mathcal{H}_{\{1,\ldots ,n\}}\left( \mu _{3}\right) \leq \lambda \mathcal{H}%
_{\{1,\ldots ,n\}}\left( \mu _{1}\right) +\left( 1-\lambda \right) \mathcal{H%
}_{\{1,\ldots ,n\}}\left( \mu _{2}\right) \ ,  \label{def7bis}
\end{equation}%
both sides of the inequality being equal to $\infty $. Thus, assume that
both $f_{\mu _{1}}\ $and $f_{\mu _{2}}$ exist. Then, since the Radon-Nikodym
derivatives are $\mathrm{m}_{\otimes }$-almost everywhere positive and, by (%
\ref{sdssdsds}) and the fact that the function $\ln (\cdot )$ is
monotonically increasing, 
\begin{equation*}
\ln \left( f_{\mu _{3}}\left( \sigma \right) \right) =\ln \left( \lambda
f_{\mu _{1}}\left( \sigma \right) +\left( 1-\lambda \right) f_{\mu
_{2}}\left( \sigma \right) \right) \geq \ln \left( \lambda f_{\mu
_{1}}\left( \sigma \right) \right) \ ,
\end{equation*}%
($\mathrm{m}_{\otimes }$-almost everywhere), we arrive at 
\begin{eqnarray}
\lambda \int_{\Sigma }f_{\mu _{1}}\ln \left( f_{\mu _{3}}\right) \mathrm{m}%
_{\otimes }\left( \mathrm{d}\sigma \right) &\geq &\lambda \int_{\Sigma
}f_{\mu _{1}}\ln \left( \lambda f_{\mu _{1}}\right) \mathrm{m}_{\otimes
}\left( \mathrm{d}\sigma \right)  \notag \\
&=&\lambda \int_{\Sigma }f_{\mu _{1}}\ln \left( f_{\mu _{1}}\right) \mathrm{m%
}_{\otimes }\left( \mathrm{d}\sigma \right) +\lambda \ln \lambda
\int_{\Sigma }f_{\mu _{1}}\mathrm{m}_{\otimes }\left( \mathrm{d}\sigma
\right)  \notag \\
&=&\lambda \int_{\Sigma }f_{\mu _{1}}\ln \left( f_{\mu _{1}}\right) \mathrm{m%
}_{\otimes }\left( \mathrm{d}\sigma \right) +\lambda \ln \lambda \ .
\label{def5}
\end{eqnarray}%
In the same way, we obtain that 
\begin{equation}
\left( 1-\lambda \right) \int_{\Sigma }f_{\mu _{2}}\ln \left( f_{\mu
_{3}}\right) \mathrm{m}_{\otimes }\left( \mathrm{d}\sigma \right) \geq
\left( 1-\lambda \right) \int_{\Sigma }f_{\mu _{2}}\ln \left( f_{\mu
_{2}}\right) \mathrm{m}_{\otimes }\left( \mathrm{d}\sigma \right) +\left(
1-\lambda \right) \ln \left( 1-\lambda \right) \ .  \label{def6}
\end{equation}%
By adding (\ref{def5}) and (\ref{def6}), we get the upper bound%
\begin{equation}
\mathcal{H}_{\{1,\ldots ,n\}}\left( \mu _{3}\right) \leq \lambda \mathcal{H}%
_{\{1,\ldots ,n\}}\left( \mu _{1}\right) +\left( 1-\lambda \right) \mathcal{H%
}_{\{1,\ldots ,n\}}\left( \mu _{2}\right) -\lambda \ln \lambda -\left(
1-\lambda \right) \ln \left( 1-\lambda \right)  \label{def7}
\end{equation}%
for any $\lambda \in \left( 0,1\right) $, $\mu _{1},\mu _{2}\in \mathcal{P}%
(T)$ and $n\in \mathbb{N}$\ such that $f_{\mu _{1}}\ $and $f_{\mu _{2}}$
both exist. By Theorem \ref{teo-equiv-entrop} combined with (\ref{def7bis})
and (\ref{def7}), it follows that%
\begin{equation}
h\left( \mu _{3}\right) \leq \lambda h\left( \mu _{1}\right) +\left(
1-\lambda \right) h\left( \mu _{2}\right)  \label{convex}
\end{equation}%
for any $\lambda \in (0,1)$ and $\mu _{1},\mu _{2}\in \mathcal{P}(T)$. Thus, 
$h$ is both concave and convex, i.e., affine.
\end{proof}

\begin{corollary}[Ergodic abundance]
\label{ergodic abundance}\mbox{ }\newline
Assume that $\mathrm{m}$ is a probability measure on $\Omega $ having full
support. For any $\mu \in \mathcal{P}(T)$, there exists a sequence $(\mu
_{m})_{m\in \mathbb{N}}$ of ergodic measures converging in the weak$^{\ast }$
topology to $\mu $ such that 
\begin{equation*}
\lim_{m\rightarrow \infty }h\left( \mu _{m}\right) =h\left( \mu \right) 
\text{ }.
\end{equation*}
\end{corollary}

\begin{proof}
Let $\mu \in \mathcal{P}(T)$ and define the sequence $(\mu _{m})_{m\in 
\mathbb{N}}$ by (\ref{mu_m1})--(\ref{mu_m2}). Indeed, by the proof of
Proposition \ref{density of ergodic measures} (cf. Steps 2 and 3), $(\mu
_{m})_{m\in \mathbb{N}}$ converges to $\mu $ in the weak$^{\ast }$ topology
while $\mu _{m}$ is an ergodic measure for each $m\in \mathbb{N}$. By
Theorem \ref{teo-equiv-entrop},%
\begin{equation*}
h\left( \mu _{m}\right) =\lim_{n\rightarrow \infty }n^{-1}\mathcal{H}%
_{\{1,\ldots ,n\}}\left( \mu _{m}\right) =\lim_{n\rightarrow \infty }n^{-1}%
\mathcal{H}_{\{1,\ldots ,n\}}\left( \frac{1}{m}\sum_{k=0}^{m-1}T_{\ast
}^{k}\left( \mu _{m}^{\otimes }\right) \right) \text{ }.
\end{equation*}%
By convexity of the relative entropy, $\mathcal{H}_{\{1,\ldots ,n\}}$ is
concave. In particular, 
\begin{equation*}
\mathcal{H}_{\{1,\ldots ,n\}}\left( \frac{1}{m}\sum_{k=0}^{m-1}\mu
_{m,k}^{\otimes }\right) \geq \frac{1}{m}\sum_{k=0}^{m-1}\mathcal{H}%
_{\{1,\ldots ,n\}}\left( \mu _{m,k}^{\otimes }\right) \text{ },
\end{equation*}%
where $\mu _{m,k}^{\otimes }\doteq T_{\ast }^{k}\left( \mu _{m}^{\otimes
}\right) $. Since the Radon-Nikodym derivatives are $\mathrm{m}_{\otimes }$%
-almost everywhere positive and the function $\ln (\cdot )$ is monotonically
increasing, we have that 
\begin{eqnarray*}
\left( \frac{1}{m}\sum_{k=0}^{m-1}f_{\mu _{m,k}^{\otimes }}\right) \ln
\left( \frac{1}{m}\sum_{k=0}^{m-1}f_{\mu _{m,k}^{\otimes }}\right) &\geq &%
\frac{1}{m}\sum_{k=0}^{m-1}f_{\mu _{m,k}^{\otimes }}\ln \left( \frac{1}{m}%
f_{\mu _{m,k}^{\otimes }}\right) \\
&=&-\left( \frac{1}{m}\sum_{k=0}^{m-1}f_{\mu _{m,k}^{\otimes }}\right) \ln m+%
\frac{1}{m}\sum_{k=0}^{m-1}f_{\mu _{m,k}^{\otimes }}\ln \left( f_{\mu
_{m,k}^{\otimes }}\right)
\end{eqnarray*}%
($\mathrm{m}_{\otimes }$-almost everywhere). Therefore,%
\begin{equation*}
\frac{1}{m}\sum_{k=0}^{m-1}\mathcal{H}_{\{1,\ldots ,n\}}\left( \mu
_{m,k}^{\otimes }\right) \leq \mathcal{H}_{\{1,\ldots ,n\}}\left( \frac{1}{m}%
\sum_{k=0}^{m-1}\mu _{m,n}^{\otimes }\right) \leq \ln m+\frac{1}{m}%
\sum_{k=0}^{m-1}\mathcal{H}_{\{1,\ldots ,n\}}\left( \mu _{m,k}^{\otimes
}\right) \text{ }.
\end{equation*}%
(This is essentially the same argument used in the proof of Proposition \ref%
{Affinity of the entropy copy(1)} to prove the affine property of the
entropy.) It follows that%
\begin{align}
h\left( \mu _{m}\right) & =\lim_{n\rightarrow \infty }\frac{1}{mn}%
\sum_{k=0}^{m-1}\mathcal{H}_{\{1,\ldots ,n\}}\left( \mu _{m,k}^{\otimes
}\right) =\lim_{n\rightarrow \infty }\frac{1}{m^{2}n}\sum_{k=0}^{m-1}%
\mathcal{H}_{\{1,\ldots ,mn\}}\left( \mu _{m,k}^{\otimes }\right)  \notag \\
& =\lim_{n\rightarrow \infty }\frac{1}{m^{2}n}\sum_{k=0}^{m-1}\mathcal{H}%
_{\{1+k,\ldots ,mn+k\}}\left( \mu _{m}^{\otimes }\right)  \notag \\
& =\lim_{n\rightarrow \infty }\frac{1}{m^{2}n}\left( n\mathcal{H}%
_{\{1,\ldots ,m\}}\left( \mu \right) +\sum_{k=1}^{m-1}\mathcal{H}%
_{\{1,\ldots ,k\}}\left( \mu \right) \right.  \notag \\
& \left. \qquad \qquad \qquad \qquad \qquad +\sum_{k=1}^{m-1}\left( (n-1)%
\mathcal{H}_{\{1,\ldots ,m\}}\left( \mu \right) +\mathcal{H}_{\{k+1,\ldots
,m\}}\left( \mu \right) \right) \right)  \notag \\
& =\frac{1}{m}\mathcal{H}_{\{1,\ldots ,m\}}\left( \mu \right) \ .  \label{fj}
\end{align}%
To get the third equality we use the additivity of the relative entropy with
respect to product probability measures, along with the $T$-invariance of $%
\mu $. Notice also, that by Inequality (\ref{ineq coarse graining rel entr
Lambda}) and the $T$-invariance of $\mu $, the terms 
\begin{equation*}
\mathcal{H}_{\{1,\ldots ,k\}}\left( \mu \right) \qquad \text{and}\qquad 
\mathcal{H}_{\{k+1,\ldots ,m\}}\left( \mu \right) =\mathcal{H}_{\{1,\ldots
,m-k\}}\left( \mu \right)
\end{equation*}%
for $k\in \{1,\ldots ,m-1\}$ have to be finite (i.e., not $-\infty $)
whenever $\mathcal{H}_{\{1,\ldots ,m\}}\left( \mu \right) >-\infty $. In
turn, $\mathcal{H}_{\{1,\ldots ,m\}}\left( \mu \right) >-\infty $ for all $%
m\in \mathbb{N}$ when the entropy $h\left( \mu \right) $ is finite, i.e., $%
h\left( \mu \right) >-\infty $), by Theorem \ref{teo-equiv-entrop} and
Inequality (\ref{ineq coarse graining rel entr Lambda}). In particular $%
h\left( \mu _{m}\right) $ is finite for all\ $m\in \mathbb{N}$ when $h\left(
\mu \right) $ is finite. Now, we deduce once again from Theorem \ref%
{teo-equiv-entrop} and (\ref{fj}) that $h\left( \mu _{m}\right) $ tends to $%
h\left( \mu \right) $, as $m\rightarrow \infty $.
\end{proof}

\subsubsection{Minimization of real-valued functions via $\Gamma $%
-regularization\label{Minimization}}

In 2012 we proved \cite{BruPedraconvex} that, in great generality, the
minimization problem of any non-convex and non-lower semicontinuous function
that is bounded from below and defined on a compact convex subset of a
locally convex real topological vector space can be analyzed through an
associated convex and lower semicontinuous function, which is the so-called $%
\Gamma $-regularization of the original function. We present this result
because it is important in Sections \ref{Section game + measusre} and \ref%
{Non-linear and Affine Pressures}, but is not common knowledge (at least in
the thermodynamic formalism community).

Fix once and for all in all the subsection a (nonempty) compact convex
subset $K$ of a locally convex real (topological vector) space\footnote{%
Topological vector spaces are here Hausdorff spaces, i.e., singletons in
those spaces are closed sets.} $\mathcal{X}$. The aim of \cite%
{BruPedraconvex} is to characterize of the set of \emph{generalized}
minimizers of real-valued functions on $K$. Given an (extended) function $%
f:K\rightarrow \left( -\infty ,\infty \right] $ the set of its generalized
minimizers is, by definition, the closure of the nonempty\footnote{$\mathit{%
\Omega }\left( h,K\right) $ is non-empty because any net $(x_{i})_{i\in
I}\subseteq K$ converges along a subnet,\ $K$ being compact.} set 
\begin{equation*}
\mathit{\Omega }\left( f,K\right) \doteq \left\{ x\in K:\exists
(x_{i})_{i\in I}\subseteq K\mathrm{\ \ }\text{with}\mathrm{\ }%
x_{i}\rightarrow x\;\text{and}\mathrm{\;}\lim_{I}f(x_{i})=\inf \,f(K)\right\}
\end{equation*}%
of all limit points of approximating minimizers of $f$. Note that the
function $f$\ is here completely general. In particular, it is \textbf{not}
necessarily convex or lower semicontinuous.

It turns out that the study of the above set of generalized minimizers can
be performed via $\Gamma $-regularization, which is defined as follows \cite[%
Eq. (1.3) in Chapter I]{Alfsen}: For any extended real-valued function $%
f:K\rightarrow \lbrack \mathrm{k},\infty ]$ with $\mathrm{k}\in \mathbb{R}$,
its $\Gamma $-regularization (on $K$) is the function defined as the
supremum over all affine and continuous minorants $m:\mathcal{X}\rightarrow 
\mathbb{R}$ of $f$, i.e., for all $x\in K$, 
\begin{equation}
\Gamma \left( f\right) \left( x\right) \doteq \sup \left\{ m(x):m\in \mathrm{%
A}\left( \mathcal{X}\right) \;\text{and }m|_{K}\leq f\right\} \ ,
\label{gamm regu general}
\end{equation}%
where $\mathrm{A}\left( \mathcal{X}\right) $ is the space of all affine
continuous real-valued functions on\ $\mathcal{X}$. Since the $\Gamma $%
-regularization $\Gamma \left( f\right) $ is a supremum over continuous
functions, $\Gamma \left( f\right) $ is a convex and lower semicontinuous
function on $K$. In fact, it is the largest weak$^{\ast }$-lower
semicontinuous convex function below $f$, by \cite[Corollary 3.2]%
{BruPedraconvex}. See, e.g., \cite{BruPedraconvex} or \cite[Section 10.5]%
{BruPedra2} for a short review on $\Gamma $-regularizations.

\begin{theorem}[Minimization of real-valued functions]
\label{theorem sympa}\mbox{ }\newline
For any function $f:K\rightarrow \lbrack \mathrm{k},\infty ]$ with $\mathrm{k%
}\in \mathbb{R}$, the following assertions hold true:\newline
\emph{(i)} 
\begin{equation*}
\inf \,f\left( K\right) =\inf \,\Gamma \left( f\right) \left( K\right) \ .
\end{equation*}%
\emph{(ii) }The set $\mathit{M}$ of minimizers of $\Gamma \left( f\right) $
over $K$ is the closed convex hull of the set $\mathit{\Omega }\left(
f,K\right) $ of generalized minimizers of $f$ over $K$, i.e., 
\begin{equation*}
\mathit{M}=\overline{\mathrm{co}}\left( \mathit{\Omega }\left( f,K\right)
\right) \ .
\end{equation*}%
\emph{(iii) }All extreme points of the compact convex set $\mathit{M}$
belong to the set of generalized minimizers of $f$, i.e., $\mathcal{E}\left( 
\mathit{M}\right) \subseteq \overline{\mathit{\Omega }\left( f,K\right) }$.
\end{theorem}

This theorem correspond to \cite[Theorems 1.4--1.5]{BruPedraconvex} and
implies, among other things, a generalization to locally convex real spaces
of the Lanford III-Robinson theorem, proven for separable real Banach
spaces, which characterizes the subdifferentials of continuous convex
functions. See \cite[Theorem 1.8]{BruPedraconvex}.

\section{Abstract Theory of Bogoliubov Linearizations\label{Abstract Theory}}

The abstract structure of the arguments used above for the nonlinear
thermodynamic formalism of dynamical systems can be extracted, revealing a
much broader spectrum of applications than that considered in Section \ref%
{thermo form}. In fact, all one needs is a generic compact convex Hausdorff
space $K$, instead of the particular case of the metrizable weak$^{\ast }$%
-compact and convex space $\mathcal{P}(T)$ of $T$-invariant measures of
Section \ref{thermo form}, or the metrizable weak$^{\ast }$-compact and
convex space of states on the CAR $C^{\ast }$-algebra used in \cite%
{BruPedra2}. The considered nonlinear functions can also be very general, a
sum of a concave upper semicontinuous and a convex lower semicontinuous
function. Despite this being a very general framework, explicit results can
be obtained via self-consistency when the associated linear variational
problems are well understood. This transition between nonlinear variational
problems and linear ones is what we call here \emph{Bogoliubov linearization}%
, in honour of Bogoliubov's seminal approach used in his famous microscopic
theory of He$^{4}$ superfluidity, proposed in 1947. That is what is
presented in this section, forming a general theoretical framework that we
believe is both useful and elegant.

For the reader's convenience, we begin by recalling basic definitions and
properties of the Legendre-Fenchel transform and subdifferentials (Section %
\ref{Legendre-Fenchel transform section}). Next, we study the purely concave
case (Section \ref{Sect Bogoliubov app tech1}), which is then used as a
springboard to the general case (Section \ref{sect conc conv Bogo}). In
Section \ref{Decision}, this general study leads to decision rules of a
canonical two-person zero-sum game, referred to here as thermodynamic game,
which is the terminology proposed by \cite{BruPedra2}. In Section \ref%
{transport}, we conclude by linking this formalism to Monge-Kantorovich's
optimal transport and duality formulas. This opens up our approach to a
broader field of research.

\subsection{The Legendre-Fenchel transform and subdifferentials\label%
{Legendre-Fenchel transform section}}

Let $\mathcal{X}$ be any real normed space. Its (topological) dual space is
the set of all continuous linear functionals on $\mathcal{X}$ and is denoted 
$\mathcal{X}^{\ast }$, as is usual. The dual space $\mathcal{X}^{\ast }$ is
then endowed with the weak$^{\ast }$ topology. Recall that it is a locally
convex space. By \cite[Theorem 3.10]{Rudin}, $(\mathcal{X},\mathcal{X}^{\ast
})$ forms a so-called dual pair, meaning that, for all $x\in \mathcal{X}$,
the functional $y\mapsto y(x)\in \mathbb{C}$ on $\mathcal{X}^{\ast }$ is
continuous with respect to the weak$^{\ast }$ topology, and all linear
functionals which are continuous with respect to the weak$^{\ast }$ topology
have this form.\ 

Given such a dual pair $(\mathcal{X},\mathcal{X}^{\ast })$, the
Legendre-Fenchel transform\footnote{%
It is also called the conjugate of $g$.} $g^{\ast }$ of a function $g:%
\mathcal{X}\rightarrow (-\infty ,\infty ]$ (that is, $g$ is allowed to take
the value $\infty $) with nonempty domain 
\begin{equation*}
\mathrm{dom}\left( g\right) \doteq \left\{ x\in \mathcal{X}:g\left( x\right)
<\infty \right\} \neq \emptyset
\end{equation*}%
is the convex lower semicontinuous functional from $\mathcal{X}^{\ast }$ to $%
(-\infty ,\infty ]$ defined by 
\begin{equation}
g^{\ast }\left( x\right) \doteq \sup_{y\in \mathcal{X}}\left\{ x\left(
y\right) -g\left( y\right) \right\} \ ,\qquad x\in \mathcal{X}^{\ast }\ .
\label{conjugate definition}
\end{equation}%
Note that the conditions imposed on $g$ means that we only consider here
so-called \emph{proper} functions, as defined in \cite[Definition 1.3.1]%
{Schirotzek}. In this context, the double Legendre-Fenchel transform $%
g^{\ast \ast }$, also called the biconjugate of $g$, is nothing but the
so-called $\Gamma $-regularization of $g$ (see (\ref{gamm regu general})),
that is, the largest lower semicontinuous convex function below $g$. In
particular, if $g$ is a lower semicontinuous convex function then 
\begin{equation}
g\left( x\right) =\sup_{y\in \mathcal{X}^{\ast }}\left\{ y\left( x\right)
-g^{\ast }\left( y\right) \right\} \ ,\qquad x\in \mathcal{X}\ .
\label{Biconjugate definition}
\end{equation}%
Recall that the Legendre-Fenchel transform is always weak$^{\ast }$-lower
semicontinuous, being the supremum over weak$^{\ast }$-continuous (and
affine) functionals. See \cite[Theorem 2.2.4]{Schirotzek} and \cite[%
Corollary 3.2]{BruPedraconvex}. Observe from (\ref{Biconjugate definition})
that, for lower semicontinuous convex functions $g$, one must have 
\begin{equation}
\mathrm{dom}\left( g^{\ast }\right) \doteq \left\{ x\in \mathcal{X}:g^{\ast
}\left( x\right) <\infty \right\} \neq \emptyset \ ,
\label{sdfsdfsdfsdfsdfs}
\end{equation}%
since the function $g$ never takes the value $-\infty $ (it is even proper,
by assumption). In this case, the nonempty sets $\mathrm{dom}(g)$ and $%
\mathrm{dom}(g^{\ast })$ are always convex, by convexity of Legendre-Fenchel
transforms.

This essential property (that is, (\ref{Biconjugate definition})) is crucial
for allowing us to implement the \emph{Bogoliubov linearization}\ in a
mathematically rigorous manner.

Below, we focus on functions whose Legendre-Fenchel transform grows faster
than linearly. More precisely, we consider functions with large enough
minimal linear growth. In fact, given a dual pair $(\mathcal{X},\mathcal{X}%
^{\ast })$, we say that the Legendre-Fenchel transform $g^{\ast }$ of a
function $g:\mathcal{X}\rightarrow (-\infty ,\infty ]$ with nonempty domain $%
\mathrm{dom}(g)\neq \emptyset $ has \emph{minimal linear growth} $\lambda
\in \mathbb{R}^{+}$ if 
\begin{equation}
\lim_{R\rightarrow \infty }\sup_{y\in \mathcal{X}^{\ast }\backslash
B(0,R)}\left\{ \lambda \left\Vert y\right\Vert _{\mathrm{op}}-g^{\ast
}\left( y\right) \right\} =-\infty \ ,  \label{linear grow}
\end{equation}%
where 
\begin{equation}
B(0,R)\doteq \left\{ y\in \mathcal{X}^{\ast }:\left\Vert y\right\Vert _{%
\mathrm{op}}\leq R\right\} \subseteq \mathcal{X}^{\ast }  \label{ball1}
\end{equation}%
is the closed ball of radius $R\in \mathbb{R}^{+}$ and center $0\in \mathcal{%
X}^{\ast }$, while, as is usual, 
\begin{equation}
\left\Vert y\right\Vert _{\mathrm{op}}\doteq \sup \left\{ \left\vert y\left(
x\right) \right\vert :x\in \mathcal{X},\ \left\Vert x\right\Vert _{\mathcal{X%
}}\leq 1\right\} \ ,\qquad y\in \mathcal{X}^{\ast }\ .  \label{ball2}
\end{equation}%
If such a function $g$ is additionally convex and lower semicontinuous, then
we deduce from (\ref{Biconjugate definition}) and (\ref{linear grow}) the
existence of a strictly positive radius $R\in \mathbb{R}^{+}$ such that, for
any $x\in \mathcal{X}$ satisfying $\left\Vert x\right\Vert _{\mathcal{X}%
}\leq \lambda $,%
\begin{equation}
g\left( x\right) =\sup_{y\in \mathcal{X}^{\ast }}\left\{ y\left( x\right)
-g^{\ast }\left( y\right) \right\} =\sup_{y\in B(0,R)}\left\{ y\left(
x\right) -g^{\ast }\left( y\right) \right\} \ .  \label{inequality grow}
\end{equation}%
This property is important to define below \textquotedblleft equilibrium
states\textquotedblright .

Given again a dual pair $(\mathcal{X},\mathcal{X}^{\ast })$ and a function $%
g:\mathcal{X}\rightarrow (-\infty ,\infty ]$ with nonempty domain $\mathrm{%
dom}(g)\neq \emptyset $, the corresponding Legendre-Fenchel transform is
directly related to the so-called subdifferential of $g$ at $z\in \mathrm{dom%
}(g)$, which is defined by%
\begin{equation}
\partial g\left( z\right) \doteq \left\{ x\in \mathcal{X}^{\ast }:\forall
y\in \mathcal{X},\quad x\left( y-z\right) \leq g\left( y\right) -g\left(
z\right) \right\} \ .  \label{subdifferential}
\end{equation}%
In particular, by \cite[Proposition 4.4.1]{Schirotzek}, when $g$ is also
convex, $z\in \mathrm{dom}(g)$ is solution to the variational problem (\ref%
{conjugate definition}) if and only if $x\in \partial g\left( z\right) $.
When $g$ is a lower semicontinuous convex function, we can also infer from 
\cite[Proposition 4.4.1]{Schirotzek} that, for any $x\in \mathrm{dom}(g)$, $%
z\in \mathcal{X}^{\ast }$ is solution to (\ref{Biconjugate definition}) if
and only if $z\in \partial g\left( x\right) \subseteq \mathcal{X}^{\ast }$.
Note indeed that $z\in \mathcal{X}^{\ast }$ is solution to (\ref{Biconjugate
definition}) if and only if 
\begin{equation*}
g\left( x\right) +g^{\ast }\left( z\right) =z\left( x\right) \ .
\end{equation*}

\subsection{Bogoliubov linearizations for concave interactions\label{Sect
Bogoliubov app tech1}}

\noindent \textbf{General assumptions. }In all this subsection, we consider
the following assumptions:

\begin{itemize}
\item[A1] $\mathcal{X}$ is a real normed space, $K$ is a compact convex
Hausdorff space, $\tau :K\rightarrow \mathcal{X}$ is a continuous affine
transformation and $f:K\rightarrow \mathbb{\{-\infty \}\cup R}$ is an upper
semicontinuous concave function (which is of course not identically equal to 
$\mathbb{-\infty }$).

\item[A2] $g:\mathcal{X}\rightarrow \mathbb{R}$ is a lower semicontinuous
and convex function for which there is a radius $R_{0}\in \mathbb{R}^{+}$
such that $B(0,R)\cap \mathrm{dom}(g^{\ast })\subseteq \mathcal{X}^{\ast }$
is nonempty and weak$^{\ast }$-closed for all $R\in \lbrack R_{0},\infty )$.
Moreover, the Legendre-Fenchel transforms $g^{\ast }$ has minimal\emph{\ }%
linear growth $\Vert \tau \Vert _{\infty }\in \mathbb{R}^{+}$ in the sense
of Equation (\ref{linear grow}).
\end{itemize}

\noindent Above, recall that $B(0,R)\subseteq \mathcal{X}^{\ast }$ is the
weak$^{\ast }$-closed ball of radius $R\in \mathbb{R}^{+}$ and center $0\in 
\mathcal{X}^{\ast }$, see (\ref{ball1})--(\ref{ball2}). We use the standard
notation 
\begin{equation}
\left\Vert \tau \right\Vert _{\infty }\doteq \sup \left\{ \left\Vert \tau
\left( \mu \right) \right\Vert _{\mathcal{X}}:\mu \in K\right\}
\label{sup norm}
\end{equation}%
for the uniform norm (or sup norm) of $\tau $. Recall that any continuous
function on a compact is uniformly bounded. In particular, in Condition A2
one always has that $\Vert \tau \Vert _{\infty }<\infty $. Remark also from
Condition A2 that $g$ has obviously full domain, i.e., $\mathrm{dom}(g)=%
\mathcal{X}$ (it is also proper), and the convex set $\mathrm{dom}(g^{\ast
}) $ is nonempty, thanks to (\ref{sdfsdfsdfsdfsdfs}).

We call the elements of $K$ \emph{(generalized) states}. In this subsection
we study the set 
\begin{equation*}
E_{\mathcal{F}}=\left\{ \mu \in K:\mathcal{F}\left( \mu \right) =\sup 
\mathcal{F}\left( K\right) \right\}
\end{equation*}%
of \emph{nonlinear equilibrium states}, where the function $\mathcal{F}%
:K\rightarrow \mathbb{\{-\infty \}\cup R}$ is defined by\ 
\begin{equation}
\mathcal{F}\doteq f-g\circ \tau \ .  \label{cali F}
\end{equation}%
Note from Conditions A1--A2 that $\mathcal{F}$ is upper semicontinuous and
concave, and so, $E_{\mathcal{F}}$ is a (nonempty) compact convex set.

\begin{remark}
\mbox{ }\newline
Clearly, $\mathcal{F}:K\rightarrow \mathbb{\{-\infty \}\cup R}$ is an upper
semicontinuous concave function like $f$ in Condition A1. The advantage of
splitting $\mathcal{F}$ into two parts $f$ and $-g\circ \tau $ arises when
the Legendre-Fenchel transform of $g^{\ast }$ and the maximization of $%
f-y\circ \tau $, $y\in \mathcal{X}^{\ast }$, can be easily controlled. This
is the raison d'\^{e}tre of Bogoliubov linearizations.
\end{remark}

\noindent \textbf{Bogoliubov linearizations. }We associate to the function (%
\ref{cali F}) the family of functions 
\begin{equation}
\mathcal{G}_{y}\doteq f-y\circ \tau \text{ },\qquad y\in \mathcal{X}^{\ast }%
\text{ }.  \label{def G}
\end{equation}%
This leads in particular to the (nonempty) compact convex sets 
\begin{equation}
E_{\mathcal{G}_{y}}\doteq \left\{ \mu \in K:\mathcal{G}_{y}\left( \mu
\right) =\sup \mathcal{G}_{y}\left( K\right) \right\} \text{ },\qquad y\in 
\mathcal{X}^{\ast }\text{ },  \label{def Gbis}
\end{equation}%
of associated \emph{linear} equilibrium states. We call these new functions $%
\mathcal{G}_{y}$, $y\in \mathcal{X}^{\ast }$, \emph{Bogoliubov linearizations%
}\ of $\mathcal{F}$.

The name \textquotedblleft Bogoliubov linearization\textquotedblright\ used
here is reminiscent to what is known in Quantum Statistical Mechanics as the
Bogoliubov approximation or, sometimes, the approximating Hamiltonian
method. It is a generic method that has been first used in theoretical
physics, starting with Bogoliubov's microscopic superfluidity theory in 1947 
\cite{Bogoliubov1}. It can be made mathematically rigorous, as shown in \cite%
{Bogjunior,AHM-non-poly1,AHM-non-poly2,approx-hamil-method0,approx-hamil-method,approx-hamil-method2,BruPedra2,BruZagrebnov8,LiebSeiringerYngvason3}
for many cases in quantum physics, ranging from Bose gases to quantum-spin
and lattice-fermion systems.

However, apart from \cite{BruPedra2}, all the mathematical results obtained
beyond perturbative arguments (with in particular (often dilute) limits of
infinite particle numbers) on the Bogoliubov approximation concern only the
computation of pressures (or free-energy densities), i.e., the calculation
of infinite-volume limits of the logarithm of partition functions. As far as
we know, the monograph \cite{BruPedra2} is the only study of the Bogoliubov
approximation for equilibrium states. It thus provides results regarding all
correlation functions (or all $N$-body density matrices) and solves in the
case of lattice-fermion systems an old open problem that was first raised by
Ginibre in 1968 \cite{Ginibre}. This approach \cite{BruPedra2} has been
significantly generalized here. Indeed, in the quantum framework, previously
considered in \cite{BruPedra2}, only the specific case of \textquotedblleft
Example \ref{example F I}\footnote{%
Adapted appropriately to the fermionic situation.}\textquotedblright\ is
studied. \bigskip

\noindent \textbf{Exactness of Bogoliubov linearizations at the variational
problem level. }The relationship between the maximization of the function $%
\mathcal{F}:K\rightarrow \mathbb{\{-\infty \}\cup R}$ and its Bogoliubov
linearizations $\mathcal{G}_{y}:K\rightarrow \mathbb{\{-\infty \}\cup R}$ is
established in the sequel using the celebrated von Neumann minimax theorem
(Theorem \ref{theorem minmax von Neumann}). In fact, they are related via
the variational problem 
\begin{equation*}
\mathrm{P}\doteq \inf_{y\in \mathcal{X}^{\ast }}P_{\mathrm{NL}}\left(
y\right)
\end{equation*}%
and its set 
\begin{equation}
\mathrm{M}\doteq \left\{ y\in \mathcal{X}^{\ast }:P_{\mathrm{NL}}\left(
y\right) =\mathrm{P}\right\}  \label{set M}
\end{equation}%
of minimizers, where%
\begin{equation*}
P_{\mathrm{NL}}\left( y\right) \doteq P_{\mathrm{L}}\left( y\right) +g^{\ast
}\left( y\right) ,\qquad P_{\mathrm{L}}\left( y\right) \doteq \sup \mathcal{G%
}_{y}\left( K\right) ,\qquad y\in \mathcal{X}^{\ast }\text{ }.
\end{equation*}%
We call the quantities $P_{\mathrm{L}}(y)$ and $P_{\mathrm{NL}}(y)$, $y\in 
\mathcal{X}^{\ast }$, respectively the linear and nonlinear approximating
pressures associated with the Bogoliubov linearization $\mathcal{G}_{y}$. As 
$P_{\mathrm{L}}(y)\in \mathbb{R}$ for all $y\in \mathcal{X}^{\ast }$, $P_{%
\mathrm{NL}}$ defines a real-valued function on the nonempty convex set $%
\mathrm{dom}(g^{\ast })$ and $P_{\mathrm{NL}}(x)=\infty $ when $x\notin 
\mathrm{dom}(g^{\ast })$.

Note that $\mathrm{P}<\infty $, because $\mathrm{dom}(g^{\ast })\neq
\emptyset $. From Lemma \ref{Lemma bogoluibov4} ($\flat $) (trivially
adapted to the case $g_{+}^{\ast }=0$)\ we can infer that Conditions A1--A2\
imply $\mathrm{P}\in \mathbb{R}$, i.e., $\mathrm{P}>-\infty $. Under the
same conditions, we meanwhile have $\sup \mathcal{F}\left( K\right) \in 
\mathbb{R}$, because $\mathcal{F}$ (\ref{cali F}) is bounded from above, as
it is upper semicontinuous, $\mathcal{F}\left( \mu \right) <\mathbb{\infty }$
for all $\mu \in K$, and $K$ is compact. In fact, the supremum of the
function $\mathcal{F}$ (\ref{cali F}) turns out to be nothing but the
variational problem $\mathrm{P}$:

\begin{proposition}[Exactness of Bogoliubov linearizations]
\label{Proposition importante bogoluibov01}\mbox{ }\newline
Under Conditions A1--A2, 
\begin{equation*}
\sup \mathcal{F}\left( K\right) =\mathrm{P}\doteq \inf_{y\in \mathcal{X}%
^{\ast }}P_{\mathrm{NL}}\left( y\right) \in \mathbb{R}
\end{equation*}%
and there are $x\in \mathrm{M}$ and $\nu \in E_{\mathcal{G}_{x}}$ such that $%
x\in \partial g\left( \tau \left( \nu \right) \right) $, $\nu \in E_{%
\mathcal{F}}$ and%
\begin{equation}
x\circ \tau \left( \nu \right) -g^{\ast }\left( x\right) =\sup_{y\in 
\mathcal{X}^{\ast }}\left\{ y\circ \tau \left( \nu \right) -g^{\ast }\left(
y\right) \right\} \ .  \label{gap eq}
\end{equation}
\end{proposition}

\begin{proof}
We divide the proof into three steps: \medskip

\noindent \underline{Step 1:} Recall that there is a radius $R_{0}\in 
\mathbb{R}^{+}$ such that the (convex) set $B(0,R)\cap \mathrm{dom}(g^{\ast
})\subseteq \mathcal{X}^{\ast }$ is nonempty and weak$^{\ast }$-closed for
all $R\in \lbrack R_{0},\infty )$ (see Condition A2). On the one hand, by
Lemma \ref{Lemma suplin growth} ($\flat $), there is $R_{1}\in \lbrack
R_{0},\infty )$ such that, for all $R\in \lbrack R_{1},\infty )$,%
\begin{equation}
\mathrm{P}=\inf_{y\in B(0,R)\cap \mathrm{dom}\left( g^{\ast }\right) }P_{%
\mathrm{NL}}\left( y\right) <\inf_{y\in \mathcal{X}^{\ast }\backslash
B(0,R)}P_{\mathrm{NL}}\left( y\right) \ .  \label{condition1bis}
\end{equation}%
On the other hand, by Equations (\ref{Biconjugate definition}) and (\ref%
{cali F}) together with Condition A2 (cf. (\ref{inequality grow})) there is
also $R_{2}\in \lbrack R_{0},\infty )$ such that, for all $R\in \lbrack
R_{2},\infty )$, 
\begin{equation}
\sup \mathcal{F}\left( K\right) =\sup_{\mu \in K}\inf_{y\in B(0,R)\cap 
\mathrm{dom}\left( g^{\ast }\right) }\left\{ f\left( \mu \right) -y\circ
\tau \left( \mu \right) +g^{\ast }\left( y\right) \right\} \ .
\label{condition1bisbis}
\end{equation}%
Therefore, there is $R\in \mathbb{R}^{+}$ ($R\geq R_{j}$, $j=0,1,2$) such
that Equations (\ref{condition1bis})--(\ref{condition1bisbis}) hold true and
the (convex) set $B(0,R)\cap \mathrm{dom}(g^{\ast })\subseteq \mathcal{X}%
^{\ast }$ is nonempty and weak$^{\ast }$-closed.\medskip

\noindent \underline{Step 2:} The space $K$\ is by assumption compact,
convex and Hausdorff (see Condition A1). Observe that the set 
\begin{equation*}
N_{0}\doteq \left\{ \mu \in K:f\left( \mu \right) >-\infty \right\} \ 
\end{equation*}%
is weak$^{\ast }$-closed and convex because the function $f$ is by
assumption upper semicontinuous and concave (see again Condition A1). In
particular, $N_{0}$ is a compact convex Hausdorff subspace of $K$.
Meanwhile, $B(0,R)$ is a (Hausdorff) weak$^{\ast }$-compact convex subset
dual vector space $\mathcal{X}^{\ast }$, thanks to the Banach-Alaoglu
theorem \cite[Theorem 3.15]{Rudin}. By taking a radius $R\in \mathbb{R}^{+}$
as in Step 1 we ensure that (\ref{condition1bis})--(\ref{condition1bisbis})
hold true and the nonempty\ convex set 
\begin{equation*}
M_{0}\doteq B(0,R)\cap \mathrm{dom}\left( g^{\ast }\right)
\end{equation*}%
is weak$^{\ast }$-closed, and therefore weak$^{\ast }$-compact. \medskip

\noindent \underline{Step 3:} Define the real-valued function $\mathfrak{f}%
:M_{0}\times N_{0}\rightarrow \mathbb{R}$ by 
\begin{equation*}
\mathfrak{f}\left( y,\mu \right) \doteq f\left( \mu \right) -y\circ \tau
\left( \mu \right) +g^{\ast }\left( y\right) \in \mathbb{R}\ ,\qquad y\in
M_{0},\ \mu \in N_{0}\ .
\end{equation*}%
For all $\mu \in N_{0}$, the mapping $y\mapsto \mathfrak{f}(y,\mu )$ is
convex and weak$^{\ast }$-lower semicontinuous, whereas, for all $y\in M_{0}$%
, the mapping $\mu \mapsto \mathfrak{f}(y,\mu )$ is concave and upper
semicontinuous, since $f$ is by assumption concave and upper semicontinuous
and $\tau :K\rightarrow \mathcal{X}$ is continuous and affine (see Condition
A1). By applying the von Neumann minimax theorem (Theorem \ref{theorem
minmax von Neumann}), we conclude the existence of a saddle point $(x,\nu
)\in M_{0}\times N_{0}$ of the function $\mathfrak{f}$. Since the radius $%
R\in \mathbb{R}^{+}$ is taken so that (\ref{condition1bis})--(\ref%
{condition1bisbis}) hold true, we infer from (\ref{condition1bisbis}) that 
\begin{equation*}
\sup \mathcal{F}\left( K\right) =\sup_{\mu \in N_{0}}\inf_{y\in M_{0}}%
\mathfrak{f}\left( y,\mu \right) =\inf_{y\in M_{0}}\sup_{\mu \in N_{0}}%
\mathfrak{f}\left( y,\mu \right) =\inf_{y\in B(0,R)}P_{\mathrm{NL}}\left(
y\right) \ .
\end{equation*}%
By the definition of a saddle point, the pair $(x,\nu )$ satisfies $\nu \in
E_{\mathcal{G}_{x}}\cap E_{\mathcal{F}}$, $x\in \mathrm{M}$ and 
\begin{equation*}
g\circ \tau \left( \nu \right) =\sup_{y\in \mathcal{X}^{\ast }}\left\{
y\circ \tau \left( \nu \right) -g^{\ast }\left( y\right) \right\} =x\circ
\tau \left( \nu \right) -g^{\ast }\left( x\right) \ .
\end{equation*}%
In particular, since the function $g$ never takes an infinite value (see
Condition A2), we deduce from \cite[Proposition 4.4.1]{Schirotzek} that $%
x\in \partial g\left( \tau \left( \nu \right) \right) \subseteq \mathcal{X}%
^{\ast }$. \bigskip
\end{proof}

\noindent \textbf{Non-exactness of Bogoliubov linearizations at the level of
equilibrium states. }Under Conditions A1--A2, the function $\mathcal{G}_{y}$
is upper semicontinuous for all $y\in \mathcal{X}^{\ast }$ and therefore,
there is $\nu _{y}\in K$ such that 
\begin{equation}
\sup \mathcal{G}_{y}\left( K\right) =\mathcal{G}_{y}\left( \nu _{y}\right) \
.  \label{varsfgsgs}
\end{equation}%
In other words, $E_{\mathcal{G}_{y}}\neq \emptyset $ for all $y\in \mathcal{X%
}^{\ast }$. In particular, each minimizer $x\in \mathrm{M}$ (see (\ref{set M}%
)) leads to a nonempty set $E_{\mathcal{G}_{x}}$ of linear equilibrium
states of the corresponding Bogoliubov linearization. These sets are
expected to provide good candidates for the nonlinear equilibrium states of
the function $\mathcal{F}$ (\ref{cali F}). However, as shown in \cite[%
Section 9.2]{BruPedra2} for fermion-lattice systems where $K$ is the
(metrizable) weak$^{\ast }$-compact convex space of translation-invariant
states on the CAR $C^{\ast }$-algebra of the infinite lattice, elements of
the sets $E_{\mathcal{G}_{x}}$, $x\in \mathrm{M}$, are not necessarily
nonlinear equilibrium states of the function (\ref{cali F}), i.e., elements
of $E_{\mathcal{F}}$. In particular, we cannot generally expect the
inclusion $E_{\mathcal{G}_{x}}\subseteq E_{\mathcal{F}}$ for $x\in \mathrm{M}
$. Instead, we only have the opposite inclusion, as the following
proposition demonstrates:

\begin{proposition}[Bogoliubov linearizations -- Equilibrium states]
\label{Proposition importante bogoluibov01 copy(1)}\mbox{ }\newline
Under Conditions A1--A2, 
\begin{equation*}
E_{\mathcal{F}}\subseteq \bigcap\limits_{x\in \mathrm{M}}E_{\mathcal{G}%
_{x}}\qquad \text{and}\qquad \emptyset \neq \mathrm{M}\subseteq
\bigcap\limits_{\nu \in E_{\mathcal{F}}}\partial g\left( \tau \left( \nu
\right) \right) .
\end{equation*}
\end{proposition}

\begin{proof}
We divide the proof into two steps: \medskip

\noindent \underline{Step 1:} By Proposition \ref{Proposition importante
bogoluibov01}, there is a continuous linear functional $x\in \mathrm{M}$,
along with an equilibrium state $\nu \in E_{\mathcal{G}_{x}}\cap E_{\mathcal{%
F}}$, satisfying (\ref{gap eq}). Using (\ref{Biconjugate definition}) we
thus observe that 
\begin{eqnarray}
\sup \mathcal{G}_{x}\left( K\right) &=&f\left( \nu \right) -x\circ \tau
\left( \nu \right)  \label{equality super bis} \\
&=&f\left( \nu \right) +\inf_{y\in \mathcal{X}^{\ast }}\left\{ -y\circ \tau
\left( \nu \right) +g^{\ast }\left( y\right) \right\} -g^{\ast }\left(
x\right)  \notag \\
&=&f\left( \nu \right) -g\circ \tau \left( \nu \right) -g^{\ast }\left(
x\right)  \notag \\
&=&\sup \mathcal{F}\left( K\right) -g^{\ast }\left( x\right) \ .
\label{equality super bisequality super bis}
\end{eqnarray}%
Take now any nonlinear equilibrium state $\tilde{\nu}\in E_{\mathcal{F}}$.
Then, going backwards from (\ref{equality super bisequality super bis}) to (%
\ref{equality super bis}) we obtain that 
\begin{equation*}
\mathcal{F}\left( \tilde{\nu}\right) -g^{\ast }\left( x\right) =\sup 
\mathcal{G}_{x}\left( K\right) \geq \mathcal{G}_{x}\left( \tilde{\nu}\right)
\ .
\end{equation*}%
Keeping in mind (\ref{Biconjugate definition}), (\ref{cali F}) and (\ref{def
G}), we deduce from the last inequality that 
\begin{equation*}
g\circ \tau \left( \tilde{\nu}\right) =\sup_{y\in \mathcal{X}^{\ast
}}\left\{ y\circ \tau \left( \tilde{\nu}\right) -g^{\ast }\left( y\right)
\right\} \leq x\circ \tau \left( \tilde{\nu}\right) -g^{\ast }\left(
x\right) \ .
\end{equation*}%
It follows that the continuous linear functional $x\in \mathrm{M}$ is always
solution to 
\begin{equation}
g\circ \tau \left( \tilde{\nu}\right) =\sup_{y\in \mathcal{X}^{\ast
}}\left\{ y\circ \tau \left( \tilde{\nu}\right) -g^{\ast }\left( y\right)
\right\} =x\circ \tau \left( \tilde{\nu}\right) -g^{\ast }\left( x\right)
\label{sdsdsddsd}
\end{equation}%
for all nonlinear equilibrium states $\tilde{\nu}\in E_{\mathcal{F}}$. By
going again backwards from (\ref{equality super bisequality super bis}) to (%
\ref{equality super bis}) with $\tilde{\nu}$ in place of $\nu $, along with (%
\ref{sdsdsddsd}), we conclude that, for any $\tilde{\nu}\in E_{\mathcal{F}}$%
, 
\begin{equation*}
\sup \mathcal{G}_{x}\left( K\right) =f\left( \tilde{\nu}\right) -x\circ \tau
\left( \tilde{\nu}\right)
\end{equation*}%
and, hence, $E_{\mathcal{F}}\subseteq E_{\mathcal{G}_{x}}$. \medskip

\noindent \underline{Step 2:} Take now any continuous linear functional $%
\tilde{x}\in \mathrm{M}$. By Proposition \ref{Proposition importante
bogoluibov01}, for any nonlinear equilibrium state $\tilde{\nu}\in E_{%
\mathcal{F}}$, 
\begin{equation*}
\mathcal{F}\left( \tilde{\nu}\right) =\sup \mathcal{F}\left( K\right) =\sup 
\mathcal{G}_{\tilde{x}}\left( K\right) +g^{\ast }\left( \tilde{x}\right)
\geq \mathcal{G}_{\tilde{x}}\left( \tilde{\nu}\right) +g^{\ast }\left( 
\tilde{x}\right) \ .
\end{equation*}%
Via\ Equations (\ref{Biconjugate definition}), (\ref{cali F}) and (\ref{def
G}), we thus get that, for all $\tilde{\nu}\in E_{\mathcal{F}}$, 
\begin{equation*}
g\circ \tau \left( \tilde{\nu}\right) =\sup_{y\in \mathcal{X}^{\ast
}}\left\{ y\circ \tau \left( \tilde{\nu}\right) -g^{\ast }\left( y\right)
\right\} =\tilde{x}\circ \tau \left( \tilde{\nu}\right) -g^{\ast }\left( 
\tilde{x}\right) \ .
\end{equation*}%
Since the function $g$ never takes an infinite value (cf. Condition A2), we
deduce from \cite[Proposition 4.4.1]{Schirotzek} and the last equation that $%
\mathrm{M}\subseteq \partial g\left( \tau \left( \tilde{\nu}\right) \right) $
for all $\tilde{\nu}\in E_{\mathcal{F}}$. From the above equalities in this
step, we also get that, for all $\tilde{\nu}\in E_{\mathcal{F}}$,%
\begin{equation*}
\mathcal{F}\left( \tilde{\nu}\right) =\sup \mathcal{G}_{\tilde{x}}\left(
K\right) +g^{\ast }\left( \tilde{x}\right) =\sup \mathcal{G}_{\tilde{x}%
}\left( K\right) -g\circ \tau \left( \tilde{\nu}\right) +\tilde{x}\circ \tau
\left( \tilde{\nu}\right) \ ,
\end{equation*}%
from which we conclude that%
\begin{equation*}
\sup \mathcal{G}_{\tilde{x}}\left( K\right) =\mathcal{F}\left( \tilde{\nu}%
\right) +g\circ \tau \left( \tilde{\nu}\right) -\tilde{x}\circ \tau \left( 
\tilde{\nu}\right) =f(\tilde{\nu})-\tilde{x}\circ \tau \left( \tilde{\nu}%
\right) \ .
\end{equation*}%
In other words, $E_{\mathcal{F}}\subseteq E_{\mathcal{G}_{\tilde{x}}}$ for
all $\tilde{x}\in \mathrm{M}$.
\end{proof}

The next corollary shows that the elements of $E_{\mathcal{F}}$ are in fact
elements of the sets $E_{\mathcal{G}_{x}}$, $x\in \mathrm{M}$, that satisfy
a self-consistency condition:

\begin{corollary}[Bogoliubov linearizations and self-consistency]
\label{corollary importante bogoluibov01 copy(1)}\mbox{ }\newline
Under Conditions A1--A2, 
\begin{equation*}
E_{\mathcal{F}}=\left\{ \nu \in E_{\mathcal{G}_{x}}:x\in \partial g\left(
\tau \left( \nu \right) \right) \right\} \doteq E_{\mathcal{G}_{x}}^{\mathrm{%
sc}}\ ,\qquad x\in \mathrm{M}\ .
\end{equation*}%
In particular, $E_{\mathcal{G}_{x}}^{\mathrm{sc}}\equiv E_{\mathcal{G}%
_{\cdot }}^{\mathrm{sc}}$ is a nonempty, compact, convex set which does not
depend upon the choice of $x\in \mathrm{M}$. (For this reason the notation $%
E_{\mathcal{G}_{\cdot }}^{\mathrm{sc}}$ is used.)
\end{corollary}

\begin{proof}
Under Condition A1, recall that the function $\mathcal{G}_{y}$ is upper
semicontinuous for all $y\in \mathcal{X}^{\ast }$. So, for all $x\in \mathrm{%
M}$, there is a solution to the variational problem (\ref{varsfgsgs}) for $%
y=x$. So take an arbitrary solution $\nu _{x}\in E_{\mathcal{G}_{x}}$ to (%
\ref{varsfgsgs}). By Proposition \ref{Proposition importante bogoluibov01}
and Equation (\ref{Biconjugate definition}) it follows that, for all $x\in 
\mathrm{M}$, 
\begin{multline*}
f\left( \nu _{x}\right) -x\circ \tau \left( \nu _{x}\right) +g^{\ast }\left(
x\right) =\sup \mathcal{F}\left( K\right) \\
\geq f\left( \nu _{x}\right) -g\circ \tau \left( \nu _{x}\right) =f\left(
\nu _{x}\right) +\inf_{y\in \mathcal{X}^{\ast }}\left\{ -y\circ \tau \left(
\nu _{x}\right) +g^{\ast }\left( y\right) \right\} \ .
\end{multline*}%
The above inequality is satisfied with equality iff 
\begin{equation*}
x\circ \tau \left( \nu _{x}\right) -g^{\ast }\left( x\right) =\sup_{y\in 
\mathcal{X}^{\ast }}\left\{ y\circ \tau \left( \nu _{x}\right) -g^{\ast
}\left( y\right) \right\} <\infty \ ,
\end{equation*}%
meaning that $\tau (\nu _{x})\in \mathrm{dom}(g)$ and\ $x\in \partial
g\left( \tau \left( \nu _{x}\right) \right) $, thanks to \cite[Proposition
4.4.1]{Schirotzek}. Note that $g^{\ast }\left( x\right) $ is a real number
because $x\in \mathrm{M}$. Thus, for all $x\in \mathrm{M}$, $E_{\mathcal{G}%
_{x}}^{\mathrm{sc}}\subseteq E_{\mathcal{F}}$. Now, by Proposition \ref%
{Proposition importante bogoluibov01 copy(1)}, for any $\nu \in E_{\mathcal{F%
}}$ and $x\in \mathrm{M}$,\ one has that $\nu \in E_{\mathcal{G}_{x}}$ and,
thus, thanks again to Proposition \ref{Proposition importante bogoluibov01}
and Equation (\ref{Biconjugate definition}), 
\begin{equation*}
f\left( \nu \right) -x\circ \tau \left( \nu \right) +g^{\ast }\left(
x\right) =f\left( \nu \right) -g\circ \tau \left( \nu \right) =f\left( \nu
\right) +\inf_{y\in \mathcal{X}^{\ast }}\left\{ -y\circ \tau \left( \nu
\right) +g^{\ast }\left( y\right) \right\} \ .
\end{equation*}%
Hence,%
\begin{equation*}
x\circ \tau \left( \nu \right) -g^{\ast }\left( x\right) =\sup_{y\in 
\mathcal{X}^{\ast }}\left\{ y\circ \tau \left( \nu \right) -g^{\ast }\left(
y\right) \right\} <\infty \ ,
\end{equation*}%
and, by \cite[Proposition 4.4.1]{Schirotzek}, $x\in \partial g\left( \tau
\left( \nu \right) \right) $. Note in particular that, for all $x\in \mathrm{%
M}$, $E_{\mathcal{G}_{x}}^{\mathrm{sc}}$ is nonempty, as $E_{\mathcal{F}}$
is nonempty and $E_{\mathcal{G}_{x}}^{\mathrm{sc}}=E_{\mathcal{F}}$.
\end{proof}

The (nonempty)\ compact convex set $E_{\mathcal{G}_{\cdot }}^{\mathrm{sc}}$
is called here the set of \emph{self-consistent equilibrium states}\ of
Bogoliubov linearizations $\mathcal{G}_{x}$, $x\in \mathrm{M}$. This
terminology refers to the paradigmatic example $K=\mathcal{P}\left( T\right) 
$. See the discussions following Theorem \ref{Theorem-main1}.

\subsection{Bogoliubov linearizations for interactions with concave and
convex components\label{sect conc conv Bogo}}

\noindent \textbf{General assumptions. }In all this subsection, we consider
the following assumptions:

\begin{itemize}
\item[B1] $\mathcal{X}_{\pm }$ are two real normed spaces, $K$ is a compact
convex Hausdorff space, $\tau _{\pm }:K\rightarrow \mathcal{X}_{\pm }$ are
two continuous affine transformations and $f:K\rightarrow \mathbb{\{-\infty
\}\cup R}$ is an upper semicontinuous concave function (which is of course
not identically equal to $\mathbb{-\infty }$).

\item[B2] $g_{-}:\mathcal{X}_{-}\rightarrow \mathbb{R}$ is a lower
semicontinuous and convex function for which there is a positive radius $%
R_{0}\in \mathbb{R}^{+}$ such that $B_{-}(0,R)\cap \mathrm{dom}\left(
g_{-}^{\ast }\right) $ is nonempty and weak$^{\ast }$-closed for all $R\in
\lbrack R_{0},\infty )$. Moreover, $g_{-}^{\ast }$ has minimal\emph{\ }%
linear growth $\Vert \tau _{-}\Vert _{\infty }$ in the sense of Equation (%
\ref{linear grow}).

\item[B3] $g_{+}:\mathcal{X}_{+}\rightarrow \mathbb{R}$ is a lower
semicontinuous, bounded and convex function for which the Legendre-Fenchel
transform $g_{+}^{\ast }$ has full domain, i.e., $\mathrm{dom}(g_{+}^{\ast
})=\mathcal{X}_{+}^{\ast }$, and has minimal\emph{\ }linear growth $\Vert
\tau _{+}\Vert _{\infty }$ in the sense of Equation (\ref{linear grow}).
\end{itemize}

\noindent Above, $B_{\pm }(0,R)\subseteq \mathcal{X}_{\pm }^{\ast }$ are the
weak$^{\ast }$-closed ball of radius $R\in \mathbb{R}^{+}$ and center $0\in 
\mathcal{X}_{\pm }^{\ast }$. Note that the rather technical condition $%
\mathrm{dom}(g_{+}^{\ast })=\mathcal{X}_{+}^{\ast }$ in\ Condition B3 is
taken to ensure that the difference $g_{-}^{\ast }\circ \tau _{-}\left( \mu
\right) -g_{+}^{\ast }\circ \tau _{+}\left( \mu \right) $, $\mu \in K$,
which is an important object here, is always well-defined as an element of $%
\mathbb{R}\cup \{\infty \}$. It essentially says that $g_{+}(x)$ grows
faster than the norm $\Vert x\Vert _{\mathcal{X}}$, as $\Vert x\Vert _{%
\mathcal{X}}\rightarrow \infty $.

Again, we call the elements of $K$ \emph{(generalized) states}. In this
subsection we study the sets 
\begin{equation}
M_{\mathbb{F}}\doteq \left\{ \mu \in K:\mathbb{F}\left( \mu \right) =\sup 
\mathbb{F}\left( K\right) \right\}  \label{MF}
\end{equation}%
and 
\begin{equation}
E_{\mathbb{F}}\doteq \left\{ \mu \in K:\exists (\mu _{j})_{j\in J}\subseteq K%
\mathrm{\ }\text{with }\lim_{J}\mu _{j}=\mu \text{ and\ }\lim_{J}\mathbb{F}%
(\mu _{j})=\sup \mathbb{F}(K)\right\} \ ,  \label{EF}
\end{equation}%
of \emph{nonlinear equilibrium states} of the (Borel-measurable) function $%
\mathbb{F}:K\rightarrow \mathbb{\{-\infty \}\cup R}$ defined by\ 
\begin{equation}
\mathbb{F}\doteq f-g_{-}\circ \tau _{-}+g_{+}\circ \tau _{+}\ .
\label{blacb F}
\end{equation}%
Note from Conditions B1--B3 that $\mathbb{F}$ is bounded from above but 
\textbf{neither upper semicontinuous nor concave}, and so, in contrast to
the (purely) concave case ($E_{\mathcal{F}}$), the set $M_{\mathbb{F}}$ is
not necessarily convex anymore. Furthermore, the set $M_{\mathbb{F}}$ could
a priori be empty. By contrast, the set $E_{\mathbb{F}}\supseteq M_{\mathbb{F%
}}$ is necessarily nonempty.

All elements of $E_{\mathbb{F}}$, including those of $M_{\mathbb{F}}$ if
they exist, are called \emph{nonlinear equilibrium states}. In fact, as we
prove below, the two sets turn out to be equal. This may seem very
surprising, given that we are maximizing a function $\mathbb{F}$ that is not
upper semicontinuous.

\begin{remark}[Decomposition as a difference of semicontinuous functions]
\label{Remark-condition1}\mbox{ }\newline
The decomposition property of a function as the difference of lower
semicontinuous ones refers to the set of point-wise limits of a sequence $%
(f_{n})_{n\in \mathbb{N}}$ of continuous functions $K\rightarrow \mathbb{R}$
such that 
\begin{equation*}
\sum_{n\in \mathbb{N}}\left\vert f_{n+1}\left( \mu \right) -f_{n}\left( \mu
\right) \right\vert <\infty \ ,\qquad \mu \in K\ .
\end{equation*}%
If this inequality is uniform with respect to $\mu \in K$, the resulting
lower semicontinuous functions are bounded. Such classes of functions,
defined as differences of semicontinuous ones, are for example discussed in 
\cite{Haydon} for compact metric spaces $K$. However, this set of functions
does not generally exhaust all point-wise limits of continuous functions $%
K\rightarrow \mathbb{R}$, which, for separable and metrizable $K$, form the
Baire-1 class, by the Lebesgue-Hausdorff-Banach theorem \cite[(24.10)]%
{Kechris}; see more generally \cite[Chapter 24]{Kechris}, in particular \cite%
[(24.3)]{Kechris} for (possibly non separable) metrizable $K$.
\end{remark}

\begin{remark}[Decomposition as a difference of convex functions]
\label{Remark-condition2}\mbox{ }\newline
In various situations, very general functions can be represented as the
difference of convex functions. For example, any Gateaux-differentiable
function with domain in a Hilbert space, whose gradient mapping is
Lipschitz-continuous, can be written\footnote{%
For such a function $f$, write $f(x)=(f(x)+C\Vert x\Vert ^{2})-C\Vert x\Vert
^{2}$ for a sufficiently large constant $C\in \mathbb{R}^{+}$.} as a
difference of two convex Gateaux-differentiable functions. This is in
particular true for $C^{1}$-functions on any compact subset of $\mathbb{R}%
^{N}$ ($N\in \mathbb{N}$). Compare with \cite{TF4} and Section \ref%
{Buzzy-Kloeckner-Leplaideur's Approach}.
\end{remark}

\noindent \textbf{Bogoliubov linearizations. }Similar to the concave case,
we associate to the function (\ref{blacb F}) the family of functions 
\begin{equation}
\mathcal{G}_{y_{+},y_{-}}\doteq f-y_{-}\circ \tau _{-}+y_{+}\circ \tau
_{+},\qquad y_{\pm }\in \mathcal{X}_{\pm }^{\ast }\text{ }.
\label{def G mix}
\end{equation}%
This object generalizes the functions $\mathcal{G}_{y}$ defined by Equation (%
\ref{def G}) and leads to the (nonempty) compact convex set 
\begin{equation*}
E_{\mathcal{G}_{y_{+},y_{-}}}\doteq \left\{ \mu \in K:\mathcal{G}%
_{y_{+},y_{-}}\left( \mu \right) =\sup \mathcal{G}_{y_{+},y_{-}}\left(
K\right) \right\}
\end{equation*}%
of \emph{linear} equilibrium states of $\mathcal{G}_{y_{+},y_{-}}$ for all $%
y_{\pm }\in \mathcal{X}_{\pm }^{\ast }$. The functions $\mathcal{G}%
_{y_{+},y_{-}}$, $y_{\pm }\in \mathcal{X}_{\pm }^{\ast }$, are again called 
\emph{Bogoliubov linearizations} of $\mathbb{F}$. \bigskip

\noindent \textbf{Exactness of Bogoliubov linearizations. }Similar to the
concave case, the maximization of the function $\mathbb{F}$ is related to
Bogoliubov linearizations via the variational problem%
\begin{equation*}
\mathrm{P}^{\flat }\doteq \sup_{y_{+}\in \mathcal{X}_{+}^{\ast
}}\inf_{y_{-}\in \mathcal{X}_{-}^{\ast }}P_{\mathrm{NL}}\left(
y_{+},y_{-}\right) \ ,
\end{equation*}%
where, for any continuous linear functionals $y_{\pm }\in \mathcal{X}_{\pm
}^{\ast }$,%
\begin{eqnarray}
P_{\mathrm{NL}}\left( y_{+},y_{-}\right) &\doteq &P_{\mathrm{L}}\left(
y_{+},y_{-}\right) +g_{-}^{\ast }\left( y_{-}\right) -g_{+}^{\ast }\left(
y_{+}\right) \ ,  \label{pression approxbis} \\
P_{\mathrm{L}}\left( y_{+},y_{-}\right) &\doteq &\sup \mathcal{G}%
_{y_{+},y_{-}}\left( K\right) \ .  \label{pression approxbis nl}
\end{eqnarray}%
As in the previous case, given $y_{\pm }\in \mathcal{X}_{\pm }^{\ast }$ we
call the quantities $P_{\mathrm{L}}(y_{+},y_{-})$ and $P_{\mathrm{NL}%
}(y_{+},y_{-})$ respectively the linear and nonlinear approximating
pressures associated with the Bogoliubov linearization $\mathcal{G}%
_{y_{+},y_{-}}$. Conditions B1--B3 imply that $\mathcal{X}_{+}^{\ast }\times 
\mathrm{dom}(g_{-}^{\ast })$ is a nonempty convex subset of $\mathcal{X}%
_{+}^{\ast }\times \mathcal{X}_{-}^{\ast }$ and 
\begin{equation*}
P_{\mathrm{NL}}\left( \mathcal{X}_{+}^{\ast }\times \mathrm{dom}(g_{-}^{\ast
})\right) \subseteq \mathbb{R}\ ,
\end{equation*}%
while $P_{\mathrm{NL}}\left( y_{+},y_{-}\right) =\infty $ when $y_{-}\notin 
\mathrm{dom}(g_{-}^{\ast })$.

It is convenient to define the function $P^{\flat }:\mathcal{X}_{+}^{\ast
}\rightarrow \{-\infty \}\cup \mathbb{R}$ by 
\begin{equation}
P^{\flat }\left( y_{+}\right) \doteq \inf P_{\mathrm{NL}}\left( y_{+},%
\mathcal{X}_{-}^{\ast }\right) \ ,\qquad y_{+}\in \mathcal{X}_{+}^{\ast }\ ,
\label{P bemol plus}
\end{equation}%
along with the following sets of optimizers: 
\begin{equation}
\mathrm{M}^{\flat }\doteq \left\{ x_{+}\in \mathcal{X}_{+}^{\ast }:P^{\flat
}\left( x_{+}\right) =\sup P^{\flat }\left( \mathcal{X}_{+}^{\ast }\right)
\doteq \mathrm{P}^{\flat }\right\} \subseteq \mathcal{X}_{+}^{\ast }
\label{P bemol plus+0}
\end{equation}%
and, for all fixed continuous linear functionals $y_{+}\in \mathcal{X}%
_{+}^{\ast }$, 
\begin{equation}
M^{\flat }\left( y_{+}\right) \doteq \left\{ x_{-}\in \mathcal{X}_{-}^{\ast
}:P_{\mathrm{NL}}\left( y_{+},x_{-}\right) =P^{\flat }\left( y_{+}\right)
\right\} \subseteq \mathcal{X}_{-}^{\ast }\ ,  \label{P bemol plus+1}
\end{equation}%
which are associated with the solutions to the variational problems $\mathrm{%
P}^{\flat }$ and $P^{\flat }(y_{+})$, respectively. The study of variational
problems $P^{\flat }(y_{+})$, $y_{+}\in \mathcal{X}_{+}^{\ast }$, and $%
\mathrm{P}^{\flat }$, including the sets $M^{\flat }(y_{+})$, $y_{+}\in 
\mathcal{X}_{+}^{\ast }$, and $\mathrm{M}^{\flat }$ of optimizers is
postponed to Section \ref{Nonlinear pressure} in order to reduce the
technical arguments and focus on the main result of this section. We
therefore only give the essential properties of these variational problems
in the next proposition:

\begin{proposition}[Properties of variational problems ($\flat $)]
\label{Proposition importante bogoluibov01 copy(3)}\mbox{ }\newline
Under Conditions B1--B3 the following assertions hold: \newline
\emph{(i)} For all $y_{+}\in \mathcal{X}_{+}^{\ast }$, $P^{\flat }(y_{+})\in 
\mathbb{R}$ and $M^{\flat }(y_{+})$ is nonempty, convex and weak$^{\ast }$%
-compact. There is $R\in \lbrack R_{0},\infty )$ such that 
\begin{equation*}
M^{\flat }(y_{+})\subseteq B_{-}(0,R)\cap \mathrm{dom}(g_{-}^{\ast })\
,\qquad y_{+}\in \mathcal{X}_{+}^{\ast }\ .
\end{equation*}%
\emph{(ii)} $\mathrm{P}^{\flat }\in \mathbb{R}$ and $\mathrm{M}^{\flat
}\subseteq \mathcal{X}_{+}^{\ast }$ is nonempty, norm-bounded and weak$%
^{\ast }$-compact.
\end{proposition}

\begin{proof}
Assertion (i) refers to Lemmata \ref{Lemma bogoluibov4} ($\flat $) and \ref%
{Lemma bogoluibov5} ($\flat $). Assertion (ii) corresponds to Lemma \ref%
{Lemma bogoluibov6} ($\flat $).
\end{proof}

We are now in a position to examine the exactness of Bogoliubov
linearizations both at the level of the variational problem and at the level
of the equilibrium states. This brings us to our next theorem, which is the
central result of the present paper and can be applied to very general
situations, such as quantum lattice systems or classical dynamical systems:

\begin{theorem}[Exactness of Bogoliubov linearizations]
\label{Proposition importante bogoluibov01 copy(2)}\mbox{ }\newline
Under Conditions B1--B3 the following assertions hold: \newline
\emph{(i)} Nonlinear pressure:%
\begin{equation*}
\sup \mathbb{F}\left( K\right) =\mathrm{P}^{\flat }\doteq \sup_{y_{+}\in 
\mathcal{X}_{+}^{\ast }}\inf_{y_{-}\in \mathcal{X}_{-}^{\ast }}P_{\mathrm{NL}%
}\left( y_{+},y_{-}\right) \in \mathbb{R}\ .
\end{equation*}%
\emph{(ii)} Self-consistency conditions: For any $x_{+}\in \mathrm{M}^{\flat
}$ and $x_{-}\in M^{\flat }\left( x_{+}\right) $, the set 
\begin{equation*}
E_{\mathcal{G}_{x_{+},x_{-}}}^{\mathrm{sc}}\doteq \left\{ \nu \in E_{%
\mathcal{G}_{x_{+},x_{-}}}:x_{-}\in \partial g_{-}\left( \tau _{-}\left( \nu
\right) \right) \right\} \equiv E_{\mathcal{G}_{x_{+},\cdot }}^{\mathrm{sc}}
\end{equation*}%
of self-consistent equilibrium states of the Bogoliubov linearization $%
\mathcal{G}_{x_{+},x_{-}}$ is nonempty, convex and compact, and does not
depend upon the choice of $x_{-}\in M^{\flat }\left( x_{+}\right) $. (That
is the raison for the notation $E_{\mathcal{G}_{x_{+},\cdot }}^{\mathrm{sc}}$%
.) Furthermore, for all $x_{+}\in \mathrm{M}^{\flat }$,%
\begin{equation*}
E_{\mathcal{G}_{x_{+},\cdot }}^{\mathrm{sc}}\subseteq \left\{ \nu \in
K:x_{+}\in \partial g_{+}\left( \tau _{+}\left( \nu \right) \right) \right\}
\ .
\end{equation*}%
\emph{(iii)} Nonlinear equilibrium states: $E_{\mathbb{F}}=M_{\mathbb{F}}$
is compact and the union of all above sets of self-consistent equilibrium
states, that is, 
\begin{equation*}
E_{\mathbb{F}}=M_{\mathbb{F}}=\bigcup_{x_{+}\in \mathrm{M}^{\flat }}E_{%
\mathcal{G}_{x_{+},\cdot }}^{\mathrm{sc}}\ .
\end{equation*}%
If the function $g_{+}:\mathcal{X}_{+}\rightarrow \mathbb{R}$ is
additionally Gateaux%
%TCIMACRO{\TeXButton{\-}{\-}}%
%BeginExpansion
\-%
%EndExpansion
-differentiable, then the above union is disjoint.\medskip \newline
\emph{(iv)} If the (concave) function $f:K\rightarrow \mathbb{R}$ defining $%
\mathbb{F}$ (see Equation (\ref{blacb F})) is affine and all Bogoliubov
linearizations $\mathcal{G}_{x_{+},x_{-}}$ have one single linear
equilibrium state, then $E_{\mathbb{F}}=M_{\mathbb{F}}\subseteq K$ is a
subset of extreme points of the compact convex Hausdorff space $K$.
\end{theorem}

\begin{proof}
Many assertions need to be proven. So, we divide our arguments into several
steps:\medskip

\noindent \underline{Step 1:} By Equation (\ref{Biconjugate definition}),
observe that%
\begin{eqnarray}
\sup \mathbb{F}\left( K\right) &=&\sup_{\mu \in K}\left\{ f\left( \mu
\right) -g_{-}\circ \tau _{-}\left( \mu \right) +\sup_{y_{+}\in \mathcal{X}%
_{+}^{\ast }}\left\{ y_{+}\circ \tau _{+}\left( \mu \right) -g_{+}^{\ast
}\left( y_{+}\right) \right\} \right\}  \notag \\
&=&\sup_{y_{+}\in \mathcal{X}_{+}^{\ast }}\left\{ \sup_{\mu \in K}\left\{
f\left( \mu \right) +y_{+}\circ \tau _{+}\left( \mu \right) -g_{-}\circ \tau
_{-}\left( \mu \right) \right\} -g_{+}^{\ast }\left( y_{+}\right) \right\} \
,  \label{fdfgdfgd}
\end{eqnarray}%
and a solution in $\mathcal{X}_{+}^{\ast }\times K$ to the variational
problems (if it exists) does not depend on the order in which the suprema
are taken. Note that $f+y_{+}\circ \tau _{+}$ is a concave and upper
semicontinuous function on $K$ for any $y_{+}\in \mathcal{X}_{+}^{\ast }$.
Hence, we can apply Proposition \ref{Proposition importante bogoluibov01}
for $f+y_{+}\circ \tau _{+}$ in place of $f$ to deduce from (\ref{fdfgdfgd})
Assertion (i). \medskip

\noindent \underline{Step 2:} Similarly, we can apply Corollary \ref%
{corollary importante bogoluibov01 copy(1)} to arrive at Assertion (ii),
except for the property $x_{+}\in \partial g_{+}\left( \tau _{+}\left( \nu
\right) \right) $ for all $x_{+}\in \mathrm{M}^{\flat }$ and $\nu \in E_{%
\mathcal{G}_{x_{+},\cdot }}^{\mathrm{sc}}$, which is proven as follows: Take
any $x_{+}\in \mathrm{M}^{\flat }$ and $\nu \in E_{\mathcal{G}_{x_{+},\cdot
}}^{\mathrm{sc}}$. Then, using Assertion (i), along with Corollary \ref%
{corollary importante bogoluibov01 copy(1)} for $f+x_{+}\circ \tau _{+}$, we
deduce that $\nu \in E_{\mathbb{F}}$ and 
\begin{equation*}
f\left( \nu \right) -g_{-}\circ \tau _{-}\left( \nu \right) +g_{+}\circ \tau
_{+}\left( \nu \right) =f\left( \nu \right) -g_{-}\circ \tau _{-}\left( \nu
\right) +x_{+}\circ \tau _{+}\left( \nu \right) -g_{+}^{\ast }\left(
x_{+}\right) \ .
\end{equation*}%
In other words, by Equation (\ref{Biconjugate definition}), one arrives at
the equalities 
\begin{equation*}
g_{+}\circ \tau _{+}\left( \nu \right) =\sup_{y_{+}\in \mathcal{X}_{+}^{\ast
}}\left\{ y_{+}\circ \tau _{+}\left( \nu \right) -g_{+}^{\ast }\left(
y_{+}\right) \right\} =x_{+}\circ \tau _{+}\left( \nu \right) -g_{+}^{\ast
}\left( x_{+}\right) \ .
\end{equation*}%
Using again \cite[Proposition 4.4.1]{Schirotzek} we conclude that $x_{+}\in
\partial g_{+}\left( \tau _{+}\left( \nu \right) \right) $. \medskip

\noindent \underline{Step 3:} By Proposition \ref{Proposition importante
bogoluibov01 copy(3)} (ii), the supremum in (\ref{fdfgdfgd}) with respect to 
$y_{+}\in \mathcal{X}_{+}^{\ast }$ can be restricted to a closed ball $%
B_{+}(0,R)$, which is always weak$^{\ast }$-compact (cf. the Banach-Alaoglu
theorem \cite[Theorem 3.15]{Rudin}). Then, by Conditions B1--B3, observe
that the mapping%
\begin{equation*}
(y_{+},\mu )\mapsto f\left( \mu \right) +y_{+}\circ \tau _{+}\left( \mu
\right) -g_{-}\circ \tau _{-}\left( \mu \right) -g_{+}^{\ast }\left(
y_{+}\right)
\end{equation*}%
defined on $B_{+}(0,R)\times K$ is upper semicontinuous with respect to the
product topology when $B_{+}(0,R)$ is equipped with the weak$^{\ast }$
topology. Therefore, there are solutions in $B_{+}(0,R)\times K$ to the
variational problem (\ref{fdfgdfgd}) and the set of all such solutions is
compact, because $K$ and $B_{+}(0,R)$ are both compact. In particular, $M_{%
\mathbb{F}}$ is nonempty. Furthermore, for any $\mu \in M_{\mathbb{F}}$,
there is $x_{+}\in B_{+}(0,R)$ such that $(x_{+},\mu )$ is solution to (\ref%
{fdfgdfgd}) and all solutions to (\ref{fdfgdfgd}) have this form. As the
projection $(x_{+},\mu )\mapsto \mu $ is continuous, we conclude that $M_{%
\mathbb{F}}$ is compact. \medskip

\noindent \underline{Step 4:} Take any net $(\mu _{j})_{j\in J}\subseteq K%
\mathrm{\ }$converging to $\mu \in E_{\mathbb{F}}$ such that%
\begin{equation*}
\lim_{J}\mathbb{F}\left( \mu _{j}\right) =\sup \mathbb{F}\left( K\right) \ .
\end{equation*}%
Because of Condition B3, there is some fixed radius $R\in \mathbb{R}^{+}$
such that, for any $j\in J$, we have a solution $x_{j,+}\in B_{+}(0,R)$ in
this ball such that 
\begin{equation*}
\sup_{y_{+}\in \mathcal{X}_{+}^{\ast }}\left\{ y_{+}\circ \tau _{+}\left(
\mu _{j}\right) -g_{+}^{\ast }\left( y_{+}\right) \right\} =x_{j,+}\circ
\tau _{+}\left( \mu _{j}\right) -g_{+}^{\ast }\left( x_{j,+}\right) \ .
\end{equation*}%
By the weak$^{\ast }$ compactness of $B_{+}(0,R)$, we can assume that $%
(x_{j,+})_{j\in J}$ converges to some $x_{+}\in B_{+}(0,R)$. In particular,
via Equation (\ref{fdfgdfgd}), Conditions B1--B2 and the weak$^{\ast }$%
-lower semicontinuity of $g_{+}^{\ast }$, we conclude that 
\begin{equation*}
\sup \mathbb{F}\left( K\right) =\lim_{J}\mathbb{F}\left( \mu _{j}\right)
\leq f\left( \mu \right) -g_{-}\circ \tau _{-}\left( \mu \right) +x_{+}\circ
\tau _{+}\left( \mu \right) -g_{+}^{\ast }\left( x_{+}\right)
\end{equation*}%
and $(x_{+},\mu )$ is a solution to (\ref{fdfgdfgd}), which yields $\mu \in
M_{\mathbb{F}}$. In other words, $E_{\mathbb{F}}\subseteq M_{\mathbb{F}%
}\subseteq E_{\mathbb{F}}$. \medskip

\noindent \underline{Step 5:} We already prove in Steps 3 and 4 that $E_{%
\mathbb{F}}=M_{\mathbb{F}}$ is compact. Therefore, having (\ref{fdfgdfgd})
in mind (cf. Steps 1, 3 and 4), we can again apply Corollary \ref{corollary
importante bogoluibov01 copy(1)} to arrive at Assertion (iii), except for
the disjointness of the sets $E_{\mathcal{G}_{x_{+},\cdot }}^{\mathrm{sc}}$, 
$x_{+}\in \mathrm{M}^{\flat }$, when the convex function $g_{+}$ is
additionally Gateaux%
%TCIMACRO{\TeXButton{\-}{\-}}%
%BeginExpansion
\-%
%EndExpansion
-differentiable. This last property is proven as follows: If $g_{+}$ is
Gateaux%
%TCIMACRO{\TeXButton{\-}{\-}}%
%BeginExpansion
\-%
%EndExpansion
-differentiable then the subdifferential $\partial g_{+}\left( \tau
_{+}\left( \nu \right) \right) $, $\nu \in E_{\mathbb{F}}$, contains at most
one point. Hence, in this case, for $x_{+},x_{+}^{\prime }\in \mathrm{M}%
^{\flat }$, $x_{+}\neq x_{+}^{\prime }$, and $\nu \in E_{\mathcal{G}%
_{x_{+},\cdot }}^{\mathrm{sc}}$, $\nu ^{\prime }\in E_{\mathcal{G}%
_{x_{+}^{\prime },\cdot }}^{\mathrm{sc}}$, one necessarily has that $\nu
\neq \nu ^{\prime }$, that is, $E_{\mathcal{G}_{x_{+},\cdot }}^{\mathrm{sc}}$
and $E_{\mathcal{G}_{x_{+}^{\prime },\cdot }}^{\mathrm{sc}}$ are disjoint.
\medskip

\noindent \underline{Step 6:} If $f$ is affine then all Bogoliubov
linearizations $\mathcal{G}_{x_{+},x_{-}}$ are affine, see Equation (\ref%
{def G mix}). So, if they all have a single linear equilibrium state, then
these states have to be extreme points of $K$. As $E_{\mathcal{G}%
_{x_{+},\cdot }}^{\mathrm{sc}}\subseteq E_{\mathcal{G}_{x_{+},x_{-}}}$, $%
x_{-}\in M^{\flat }\left( x_{+}\right) $, it follows in this case that $E_{%
\mathbb{F}}\subseteq K$ is a subset of extreme points of the convex set $K$.
In other words, we obtain Assertion (iv).
\end{proof}

\noindent The above assumptions on $g_{+}$ and $g_{-}$ in Theorem \ref%
{Proposition importante bogoluibov01 copy(2)} exclude the cases where one of
these functions is zero. However, if one of them is zero, or both, then
Theorem \ref{Proposition importante bogoluibov01 copy(2)} still holds true,
mutatis mutandis:

\begin{itemize}
\item If both $g_{-}$ and $g_{+}$ are zero then $\mathbb{F}=f$, that is, $E_{%
\mathbb{F}}$ is nothing but the set of maximizers of $f$ and Theorem \ref%
{Proposition importante bogoluibov01 copy(2)} is useless in this case.

\item If $g_{+}=0$ and $g_{-}\neq 0$ then one has the (purely) concave case
and, therefore, $E_{\mathbb{F}}=M_{\mathbb{F}}=E_{\mathcal{F}}$ for $g=g_{-}$
and $\tau =\tau _{-}$ in (\ref{cali F}), which defines the function $%
\mathcal{F}$. Thus, in this case, Theorem \ref{Proposition importante
bogoluibov01 copy(2)} (i) and (ii)--(iii) correspond to Proposition \ref%
{Proposition importante bogoluibov01} and Corollary \ref{corollary
importante bogoluibov01 copy(1)}, respectively. If $f:K\rightarrow \mathbb{R}
$ is affine and $E_{\mathcal{G}_{x}}$, $x\in \mathrm{M}$, (see (\ref{def
Gbis}) and (\ref{set M})) are singletons, then $E_{\mathbb{F}}=M_{\mathbb{F}%
}=E_{\mathcal{F}}$ is a set of extreme points of $K$.

\item If $g_{+}\neq 0$ and $g_{-}=0$, there is no infimum, only the two
suprema, in Theorem \ref{Proposition importante bogoluibov01 copy(2)} (i),
that is,%
\begin{equation*}
\sup \mathbb{F}\left( K\right) =\mathrm{P}^{\flat }\doteq \sup_{y_{+}\in 
\mathcal{X}_{+}^{\ast }}\left\{ P_{\mathrm{L}}\left( y_{+},0\right)
-g_{+}^{\ast }\left( y_{+}\right) \right\} \in \mathbb{R}\ .
\end{equation*}%
In this case, Condition B2 is irrelevant. As for (ii) one defines the
function $P^{\flat }:\mathcal{X}_{+}^{\ast }\rightarrow \{-\infty \}\cup 
\mathbb{R}$ now by 
\begin{equation*}
P^{\flat }\left( y_{+}\right) \doteq \sup \mathcal{G}_{y_{+},0}\left(
K\right) -g_{+}^{\ast }\left( y_{+}\right) \ ,\qquad y_{+}\in \mathcal{X}%
_{+}^{\ast }\ ,
\end{equation*}%
along with the corresponding set of maximizers: 
\begin{equation*}
\mathrm{M}^{\flat }\doteq \left\{ x\in \mathcal{X}_{+}^{\ast }:P^{\flat
}\left( x\right) =\sup P^{\flat }\left( \mathcal{X}_{+}^{\ast }\right)
\doteq \mathrm{Q}^{\flat }\right\} \subseteq \mathcal{X}_{+}^{\ast }\ .
\end{equation*}%
Then, for all $x_{+}\in \mathrm{M}^{\flat }$ and $\nu \in E_{\mathcal{G}%
_{x_{+},0}}$, $x_{+}\in \partial g_{+}\left( \tau _{+}\left( \nu \right)
\right) $. In other words, \emph{all} linear equilibrium states of the
Bogoliubov linearizations $\mathcal{G}_{x_{+},0}$, $x_{+}\in \mathrm{M}%
^{\flat }$, are self-consistent in this case. Again, $E_{\mathbb{F}}=M_{%
\mathbb{F}}$ is compact and 
\begin{equation*}
E_{\mathbb{F}}=M_{\mathbb{F}}=\bigcup_{x_{+}\in \mathrm{M}^{\flat }}E_{%
\mathcal{G}_{x_{+},0}}\ .
\end{equation*}%
For Gateaux%
%TCIMACRO{\TeXButton{\-}{\-}}%
%BeginExpansion
\-%
%EndExpansion
-differentiable $g_{+}$, the above union is disjoint. If $f:K\rightarrow 
\mathbb{R}$ is affine and $E_{\mathcal{G}_{x_{+},0}}$, $x_{+}\in \mathrm{M}%
^{\flat }$, are singletons, then $E_{\mathbb{F}}=M_{\mathbb{F}}$ is a set of
extreme points of $K$.
\end{itemize}

\subsection{Decision rules of the thermodynamic game\label{Decision}}

The nonlinear approximating pressure $P_{\mathrm{NL}}$ of Equation (\ref%
{pression approxbis}) can be seen as the payoff function of a two-person
zero-sum game. The conservative values of such a game are then the real
quantities%
\begin{equation}
\mathrm{P}^{\flat }\doteq \sup_{y_{+}\in \mathcal{X}_{+}^{\ast
}}\inf_{y_{-}\in \mathcal{X}_{-}^{\ast }}P_{\mathrm{NL}}\left(
y_{+},y_{-}\right) \quad \text{and}\quad \mathrm{P}^{\sharp }\doteq
\inf_{y_{-}\in \mathcal{X}_{-}^{\ast }}\sup_{y_{+}\in \mathcal{X}_{+}^{\ast
}}P_{\mathrm{NL}}\left( y_{+},y_{-}\right) \ .  \label{conservative values}
\end{equation}%
This game is named here the \emph{thermodynamic game}\ associated with $f$, $%
g_{\pm }$ and $\tau _{\pm }$. In all this subsection, we only consider
non-zero functions $g_{\pm }:\mathcal{X}_{\pm }\rightarrow \mathbb{R}$. If
one of them is zero, we have essentially the same results (mutatis
mutandis). See, e.g., the discussions at the end of Section \ref{sect conc
conv Bogo}.

Having in mind the thermodynamic game, in particular the conservative values
of Equation (\ref{conservative values}), it is natural to consider the
function 
\begin{equation}
P^{\sharp }\left( y_{-}\right) \doteq \sup P_{\mathrm{NL}}\left( \mathcal{X}%
_{+}^{\ast },y_{-}\right) \ ,\qquad y_{-}\in \mathcal{X}_{-}^{\ast }\ ,
\label{def diese bemol1}
\end{equation}%
along with $P^{\flat }:\mathcal{X}_{+}^{\ast }\rightarrow \{-\infty \}\cup 
\mathbb{R}$ already defined by (\ref{P bemol plus}), that is,%
\begin{equation}
P^{\flat }\left( y_{+}\right) \doteq \inf P_{\mathrm{NL}}\left( y_{+},%
\mathcal{X}_{-}^{\ast }\right) \ ,\qquad y_{+}\in \mathcal{X}_{+}^{\ast }\ .
\label{def diese bemol2}
\end{equation}
The subsets of solutions to the variational problems $P^{\flat }(y_{+})$ and 
$P^{\sharp }(y_{-})$ for all $y_{\pm }\in \mathcal{X}_{\pm }^{\ast }$, i.e.,%
\begin{eqnarray}
M^{\flat }\left( y_{+}\right) &\doteq &\left\{ x_{-}\in \mathcal{X}%
_{-}^{\ast }:P^{\flat }\left( y_{+}\right) =P_{\mathrm{NL}}\left(
y_{+},x_{-}\right) \right\} \subseteq \mathcal{X}_{-}^{\ast }\ ,
\label{M-y1} \\
M^{\sharp }\left( y_{-}\right) &\doteq &\left\{ x_{+}\in \mathcal{X}%
_{+}^{\ast }:P^{\sharp }\left( y_{-}\right) =P_{\mathrm{NL}}\left(
x_{+},y_{-}\right) \right\} \subseteq \mathcal{X}_{+}^{\ast }\ ,
\label{M-y2}
\end{eqnarray}%
are important, along with the sets 
\begin{eqnarray}
\mathrm{M}^{\flat } &\doteq &\left\{ x_{+}\in \mathcal{X}_{+}^{\ast }:%
\mathrm{P}^{\flat }=P^{\flat }\left( x_{+}\right) \right\} \subseteq 
\mathcal{X}_{+}^{\ast }\ ,  \label{M1} \\
\mathrm{M}^{\sharp } &\doteq &\left\{ x_{-}\in \mathcal{X}_{-}^{\ast }:%
\mathrm{P}^{\sharp }\doteq P^{\sharp }\left( x_{-}\right) \right\} \subseteq 
\mathcal{X}_{-}^{\ast }\ ,  \label{M2}
\end{eqnarray}%
which refer to optimal strategies for the thermodynamic game. The sets $%
\mathrm{M}^{\flat }$ and $M^{\flat }\left( y_{+}\right) $, $y_{+}\in 
\mathcal{X}_{+}^{\ast }$, are already defined in (\ref{P bemol plus+0})--(%
\ref{P bemol plus+1}) and their precise definition are recalled here for
convenience.

The study of the variational problems $P^{\flat }(y_{+})$ and $P^{\sharp
}(y_{-})$ ($y_{\pm }\in \mathcal{X}_{\pm }^{\ast }$), as well as $\mathrm{P}%
^{\flat }$ and $\mathrm{P}^{\sharp }$, including also the sets $M^{\flat
}(y_{+})$, $M^{\sharp }(y_{-})$, $\mathrm{M}^{\flat }$ and $\mathrm{M}%
^{\sharp }$ of optimizers, is carried out in Section \ref{Nonlinear pressure}%
. The essential properties of case ($\flat $) are summarized in Proposition %
\ref{Proposition importante bogoluibov01 copy(3)} above: $P^{\flat
}(y_{+})\in \mathbb{R}$ and $M^{\flat }(y_{+})$ is nonempty, convex and weak$%
^{\ast }$-compact and uniformly norm bounded for all $y_{+}\in \mathcal{X}%
_{+}^{\ast }$; $\mathrm{P}^{\flat }\in \mathbb{R}$ and $\mathrm{M}^{\flat
}\subseteq \mathcal{X}_{+}^{\ast }$ is also nonempty, norm-bounded and weak$%
^{\ast }$-compact. We now give a similar statement for the case ($\sharp $):

\begin{proposition}[Properties of variational problems ($\sharp $)]
\label{Proposition importante bogoluibov01 copy(4)}\mbox{ }\newline
Under Conditions B1--B3 the following assertions hold: \newline
\emph{(i)} For all $y_{-}\in \mathrm{dom}(g_{-}^{\ast })\subseteq \mathcal{X}%
_{-}^{\ast }$, $P^{\sharp }(y_{-})\in \mathbb{R}$ and the set $M^{\sharp
}(y_{-})$ is nonempty and weak$^{\ast }$-compact. There is $R\in \mathbb{R}%
^{+}$ such that 
\begin{equation*}
M^{\sharp }(y_{-})\subseteq B_{+}(0,R)\ ,\qquad y_{-}\in \mathrm{dom}%
(g_{-}^{\ast })\in \mathcal{X}_{-}^{\ast }\ .
\end{equation*}%
\emph{(ii)} $\mathrm{P}^{\sharp }\in \mathbb{R}$ and the set $\mathrm{M}%
^{\sharp }\subseteq \mathcal{X}_{-}^{\ast }$ is a nonempty, convex,
norm-bounded and weak$^{\ast }$-compact subset of $\mathrm{dom}(g_{-}^{\ast
})$.
\end{proposition}

\begin{proof}
Assertion (i) refers to Lemmata \ref{Lemma bogoluibov4} ($\sharp $) and \ref%
{Lemma bogoluibov5} ($\sharp $). Assertion (ii) corresponds to Lemma \ref%
{Lemma bogoluibov6} ($\sharp $).
\end{proof}

We are now able to define the decision rules for the thermodynamic game,
which is the main topic of this subsection.

\begin{definition}[Decision rules of the thermodynamic game]
\label{decision rules}\mbox{ }\newline
We call all weak$^{\ast }$%
%TCIMACRO{\TeXButton{\-}{\-}}%
%BeginExpansion
\-%
%EndExpansion
-to-weak$^{\ast }$\footnote{%
That is, continuity refers to the weak$^{\ast }$ topology for both the
domain and codomain of the mapping.} continuous mappings 
\begin{equation*}
\mathfrak{d}^{\flat }:\mathrm{M}^{\flat }\rightarrow \mathcal{X}_{-}^{\ast
}\qquad \text{and}\qquad \mathfrak{d}^{\sharp }:\mathrm{M}^{\sharp
}\rightarrow \mathcal{X}_{+}^{\ast }
\end{equation*}%
such that $\mathfrak{d}^{\flat }(x_{+})\in M^{\flat }(x_{+})$ and $\mathfrak{%
d}^{\sharp }(x_{-})\in M^{\sharp }(x_{-})$ $\flat $-decision rules\ and $%
\sharp $-decision rules, respectively. The two sets of $\flat $-decision and 
$\sharp $-decision rules are respectively denoted by $\mathfrak{D}^{\flat }$
and $\mathfrak{D}^{\sharp }$.
\end{definition}

\noindent As $M^{\flat }(x_{+})$ is convex for all $x_{+}\in \mathrm{M}%
^{\flat }$ (see Proposition \ref{Proposition importante bogoluibov01 copy(3)}%
), $\mathfrak{D}^{\flat }$ is a convex subset the space $C_{\mathrm{w}^{\ast
}}(\mathrm{M}^{\flat };\mathcal{X}_{-}^{\ast })$ of all weak$^{\ast }$%
-to-weak$^{\ast }$ continuous functions $\mathrm{M}^{\flat }\rightarrow 
\mathcal{X}_{-}^{\ast }$, with the usual point-wise vector space operations.

Observe that 
\begin{equation*}
\inf_{\mathfrak{f}\in C_{\mathrm{w}^{\ast }}(\mathrm{M}^{\flat };\mathcal{X}%
_{-}^{\ast })}\sup_{y_{+}\in \mathrm{M}^{\flat }}P_{\mathrm{NL}}\left( y_{+},%
\mathfrak{f}\left( y_{+}\right) \right) \geq \sup_{y_{+}\in \mathrm{M}%
^{\flat }}\inf_{\mathfrak{f}\in C_{\mathrm{w}^{\ast }}(\mathrm{M}^{\flat };%
\mathcal{X}_{-}^{\ast })}P_{\mathrm{NL}}\left( y_{+},\mathfrak{f}\left(
y_{+}\right) \right) \geq \mathrm{P}^{\flat }\ .
\end{equation*}%
Therefore, if the set $\mathfrak{D}^{\flat }$ of $\flat $-decision rules is 
\textbf{nonempty} then, for all $\mathfrak{d}^{\flat }\in \mathfrak{D}%
^{\flat }$ and $x_{+}\in \mathrm{M}^{\flat }$, one clearly has the following
equalities:%
\begin{eqnarray*}
\mathrm{P}^{\flat } &=&\inf_{\mathfrak{f}\in C_{\mathrm{w}^{\ast }}(\mathrm{M%
}^{\flat };\mathcal{X}_{-}^{\ast })}\sup_{y_{+}\in \mathrm{M}^{\flat }}P_{%
\mathrm{NL}}\left( y_{+},\mathfrak{f}\left( y_{+}\right) \right)
=\sup_{y_{+}\in \mathrm{M}^{\flat }}\inf_{\mathfrak{f}\in C_{\mathrm{w}%
^{\ast }}(\mathrm{M}^{\flat };\mathcal{X}_{-}^{\ast })}P_{\mathrm{NL}}\left(
y_{+},\mathfrak{f}\left( y_{+}\right) \right) \\
&=&\inf_{\mathfrak{f}\in C_{\mathrm{w}^{\ast }}(\mathrm{M}^{\flat };\mathcal{%
X}_{-}^{\ast })}P_{\mathrm{NL}}\left( x_{+},\mathfrak{f}\left( x_{+}\right)
\right) =\sup_{y_{+}\in \mathrm{M}^{\flat }}P_{\mathrm{NL}}\left( y_{+},%
\mathfrak{d}^{\flat }\left( y_{+}\right) \right) =P_{\mathrm{NL}}\left(
x_{+},\mathfrak{d}^{\flat }\left( x_{+}\right) \right) \ .
\end{eqnarray*}%
In particular the two-person zero-sum game on $\mathrm{M}^{\flat }\times C_{%
\mathrm{w}^{\ast }}(\mathrm{M}^{\flat };\mathcal{X}_{-}^{\ast })$ whose
payoff function is%
\begin{equation*}
\left( y_{+},\mathfrak{f}\right) \mapsto P_{\mathrm{NL}}\left( y_{+},%
\mathfrak{f}\left( y_{+}\right) \right)
\end{equation*}%
has a non-cooperative equilibrium (i.e., a saddle-point) and its unique
conservative value is nothing but $\mathrm{P}^{\flat }$, which is equal to $%
\sup \mathbb{F}(K)$, thanks to Theorem \ref{Proposition importante
bogoluibov01 copy(2)} (i). This new two-person zero-sum game is known as
\textquotedblleft extended game with information transfer\textquotedblright
. Such a game is defined for instance in \cite[Ch. 7, Section 7.2]{Aubinbis}%
. This is related to a very general result of game theory: Lasry's theorem 
\cite[Theorem 8.4]{Aubin}.

The same is true for $\mathrm{P}^{\sharp }$ and $\mathfrak{D}^{\sharp }$,
mutatis mutandis. Note however that $\mathfrak{D}^{\sharp }$ is not studied
in further detail, as it is a subject of lesser interest, even though this
study would be no more difficult than the one of $\mathfrak{D}^{\flat }$. By
contrast, as discussed above, $\mathfrak{D}^{\flat }$ is related to the
variational problem $\sup \mathbb{F}(K)$ on the state space, which is our
main object of study.

We give in the sequel some general condition ensuring that $\mathfrak{D}%
^{\flat }$ is nonempty. We start with additional conditions on the
Legendre-Fenchel transforms $g_{-}^{\ast }$ and $g_{+}^{\ast }$. It turns
out that the strict convexity and the finiteness of $g_{-}^{\ast }$ together
with the continuity of $g_{+}^{\ast }$ imply that $\mathfrak{D}^{\flat }$ is
a singleton. In particular, it is nonempty. We then show that the uniqueness
of the linear equilibrium states of the Bogoliubov linearization also
implies the existence of a $\flat $-decision rule (but not necessarily its
uniqueness). This last condition is very useful in our paradigmatic example,
because it is fulfilled when the functions $(\theta _{-},\theta _{+})$ are
of H\"{o}lder type (Definition \ref{sect Holder type theta}).

\begin{proposition}[Existence of $\flat $-decision rules for strictly convex 
$g_{-}^{\ast }$ and continuous $g_{+}^{\ast }$]
\label{decision rules prop}\mbox{ }\newline
Assume Conditions B1--B3. If, additionally, $g_{+}^{\ast }$ is weak$^{\ast }$%
-continuous, $\mathrm{dom}(g_{-}^{\ast })=\mathcal{X}_{-}^{\ast }$ and $%
g_{-}^{\ast }:\mathcal{X}_{-}^{\ast }\rightarrow \mathbb{R}$\ is strictly
convex, then $\mathfrak{D}^{\flat }$\ is a singleton.
\end{proposition}

\begin{proof}
If $g_{-}^{\ast }$ is strictly convex then, for any fixed $y_{+}\in \mathcal{%
X}_{+}^{\ast }$, the mapping $y_{-}\mapsto P_{\mathrm{NL}}(y_{+},y_{-})$,
defined by (\ref{pression approxbis})--(\ref{pression approxbis nl}), is
strictly convex (as $y_{-}\mapsto P_{\mathrm{L}}(y_{+},y_{-})$ is convex, by
Lemma \ref{Lemma bogoluibov1}) and thus, $M^{\flat }(y_{+})=\{y_{-}(y_{+})\}$
is a singleton. Therefore, $\mathfrak{D}^{\flat }$ has at most one element.
We now show that the mapping $x_{+}\mapsto x_{-}(x_{+})$ from $\mathrm{M}%
^{\flat }$ to $\mathcal{X}_{-}^{\ast }$ is weak$^{\ast }$-to-weak$^{\ast }$
continuous. Take any net $(x_{+}^{(j)})_{j\in J}$ in $\mathrm{M}^{\flat }$
converging in the weak$^{\ast }$ topology to some $x_{+}\in \mathrm{M}%
^{\flat }$. By Proposition \ref{Proposition importante bogoluibov01 copy(3)}
(i) or Lemma \ref{Lemma bogoluibov5} ($\flat $), for some $R\in \mathbb{R}%
^{+}$, $M^{\flat }(x_{+})\subseteq B_{-}(0,R)$ for all $x_{+}\in \mathrm{M}%
^{\flat }$, where $B_{-}(0,R)\subseteq \mathcal{X}_{-}^{\ast }$ is the
closed ball of radius $R\in \mathbb{R}^{+}$ and center $0\in \mathcal{X}%
_{-}^{\ast }$. In particular, $x_{-}(x_{+}^{(j)})\in B_{-}(0,R)$ for all $%
j\in J$. Thus, as $B_{-}(0,R)$\ is weak$^{\ast }$%
%TCIMACRO{\TeXButton{\-}{\-}}%
%BeginExpansion
\-%
%EndExpansion
-compact (cf. the Banach-Alaoglu theorem \cite[Theorem 3.15]{Rudin}), we can
assume without loss of generality that $(x_{-}(x_{+}^{(j)}))_{j\in J}$
converges to some $x_{-}\in B_{-}(0,R)$. For all $j\in J$ and $y_{-}\in 
\mathcal{X}_{-}^{\ast }$,%
\begin{equation*}
P_{\mathrm{NL}}(x_{+}^{(j)},x_{-}(x_{+}^{(j)}))=P_{\mathrm{L}%
}(x_{+}^{(j)},x_{-}(x_{+}^{(j)}))+g_{-}^{\ast
}(x_{-}(x_{+}^{(j)}))-g_{+}^{\ast }(x_{+}^{(j)})\leq P_{\mathrm{NL}%
}(x_{+}^{(j)},y_{-})\ .
\end{equation*}%
As $\mathrm{M}^{\flat }$ is norm bounded (Proposition \ref{Proposition
importante bogoluibov01 copy(3)} (ii) or Lemma \ref{Lemma bogoluibov6} ($%
\flat $)), $g_{-}^{\ast }$ is always weak$^{\ast }$-lower semicontinuous and 
$g_{+}^{\ast }$ is by assumption weak$^{\ast }$-continuous, by taking the $j$%
-limit and using Lemma \ref{Lemma bogoluibov2}, we conclude that 
\begin{equation*}
P_{\mathrm{NL}}\left( x_{+},x_{-}\right) =P_{\mathrm{L}%
}(x_{+},x_{-})+g_{-}^{\ast }(x_{-})-g_{+}^{\ast }(x_{+})\leq P_{\mathrm{NL}%
}\left( x_{+},y_{-}\right) \ ,\qquad y_{-}\in \mathcal{X}_{-}^{\ast }\ .
\end{equation*}%
In other words, $x_{-}=x_{-}(x_{+})$ and so, the net $%
(x_{-}(x_{+}^{(j)}))_{j\in J}$ has to converge to $x_{-}(x_{+})$ in the weak$%
^{\ast }$ topology, from which it follows that the mapping $x_{+}\mapsto
x_{-}(x_{+})$ is weak$^{\ast }$%
%TCIMACRO{\TeXButton{\-}{\-}}%
%BeginExpansion
\-%
%EndExpansion
-to-weak$^{\ast }$ continuous.
\end{proof}

We now provide a method for constructing decision rules, based on
equilibrium states. To this end, we recall some general notions of convex
analysis, more specifically the notion of \textquotedblleft
selection\textquotedblright\ \cite[Definition 2.7]{Phe2}: Given a real
normed space $\mathcal{X}$, a subset $\Omega \subseteq \mathcal{X}$ and a
continuous convex function $g:\mathcal{X}\rightarrow \mathbb{R}$, a function 
$s:\Omega \rightarrow \mathcal{X}^{\ast }$ is a \emph{selection}\ for $g$ if 
$s(x)\in \partial g(x)$ for all $x\in \Omega $. Interestingly, for Banach
spaces $\mathcal{X}$, $g$ is Gateaux%
%TCIMACRO{\TeXButton{\-}{\-}}%
%BeginExpansion
\-%
%EndExpansion
-differentiable (Fr\'{e}chet%
%TCIMACRO{\TeXButton{\-}{\-}}%
%BeginExpansion
\-%
%EndExpansion
-differentiable) at $x\in \mathcal{X}$ iff there is a selection $\xi
_{x}:\Omega _{x}\rightarrow \mathcal{X}^{\ast }$ for $g$ that is norm-to-weak%
$^{\ast }$ continuous\footnote{%
That is, the continuity is considered with respect to the norm topology in
the domain and the weak$^{\ast }$ topology in the codomain of the mapping.}
(norm-to-norm continuous\footnote{%
That is, the continuity is considered with respect to the norm topology for
both the domain and codomain of the mapping.}) at $x$, where $\Omega _{x}$
is some neighborhood of $x$. See \cite[Proposition 2.8]{Phe2}. In
particular, if $g$ is Gateaux%
%TCIMACRO{\TeXButton{\-}{\-}}%
%BeginExpansion
\-%
%EndExpansion
-differentiable (Fr\'{e}chet%
%TCIMACRO{\TeXButton{\-}{\-}}%
%BeginExpansion
\-%
%EndExpansion
-differentiable) on the whole space $\mathcal{X}$ then the unique selection $%
s_{g}:\mathcal{X}\rightarrow \mathcal{X}^{\ast }$ for $g$ is norm-to-weak$%
^{\ast }$ continuous (norm-to-norm continuous).

Inspired by the general concept of \emph{selection}, we introduce a similar
notion in relation to equilibrium states:

\begin{definition}[Equilibrium state selection]
\mbox{ }\newline
A mapping $\xi :\mathrm{M}^{\flat }\rightarrow E_{\mathbb{F}}\subseteq K$ is
a equilibrium state selection\ if it is weak$^{\ast }$ continuous and $\xi
(x_{+})\in E_{\mathcal{G}_{x_{+},\cdot }}^{\mathrm{sc}}$ for all $x_{+}\in 
\mathrm{M}^{\flat }$. See Theorem \ref{Proposition importante bogoluibov01
copy(2)} for the definition of $E_{\mathcal{G}_{x_{+},\cdot }}^{\mathrm{sc}%
}\subseteq E_{\mathbb{F}}$.
\end{definition}

\noindent Because of \cite[Proposition 2.8]{Phe2} and the self-consistency
of equilibrium states, if the function $g_{-}:\mathcal{X}_{-}\rightarrow 
\mathbb{R}$ is Gateaux%
%TCIMACRO{\TeXButton{\-}{\-}}%
%BeginExpansion
\-%
%EndExpansion
-differentiable then any equilibrium state selection naturally induces a $%
\flat $-decision rule $\mathfrak{d}_{\xi }^{\flat }:\mathrm{M}^{\flat
}\rightarrow \mathcal{X}_{-}^{\ast }$:

\begin{proposition}[Construction of decision rules from equilibrium state
selections]
\mbox{ }\newline
Assume Conditions B1--B3. If the function $g_{-}:\mathcal{X}_{-}\rightarrow 
\mathbb{R}$ is Gateaux%
%TCIMACRO{\TeXButton{\-}{\-}}%
%BeginExpansion
\-%
%EndExpansion
-differentiable and $\xi :\mathrm{M}^{\flat }\rightarrow E_{\mathbb{F}}$ is
an equilibrium state selection, then the mapping 
\begin{equation*}
\mathfrak{d}_{\xi }^{\flat }\doteq s_{g_{-}}\circ \tau _{-}\circ \xi :%
\mathrm{M}^{\flat }\rightarrow \mathcal{X}_{-}^{\ast }
\end{equation*}%
is a decision rule, where $s_{g_{-}}:\mathcal{X}_{-}\rightarrow \mathcal{X}%
_{-}^{\ast }$ is the unique selection for $g_{-}$. It is weak$^{\ast }$%
-to-norm\footnote{%
That is, the continuity is considered with respect to the weak$^{\ast }$
topology in the domain and the norm topology in the codomain of the mapping.}
continuous when $g_{-}$\ is Fr\'{e}chet%
%TCIMACRO{\TeXButton{\-}{\-}}%
%BeginExpansion
\-%
%EndExpansion
-differentiable.
\end{proposition}

\begin{proof}
By \cite[Proposition 2.8]{Phe2}, $\mathfrak{d}_{\xi }^{\flat }$ is weak$%
^{\ast }$-to-weak$^{\ast }$ continuous and it is weak$^{\ast }$-to-norm
continuous when $g_{-}$\ is Fr\'{e}chet%
%TCIMACRO{\TeXButton{\-}{\-}}%
%BeginExpansion
\-%
%EndExpansion
-differentiable. Take now any optimizer $x_{+}\in \mathrm{M}^{\flat }$.
Then, as $\xi $ is an equilibrium state selection, $\xi (x_{+})\in E_{%
\mathcal{G}_{x_{+},\cdot }}^{\mathrm{sc}}$. By Theorem \ref{Proposition
importante bogoluibov01 copy(2)} (ii), for any $x_{-}\in M^{\flat }(x_{+})$, 
$x_{-}\in \partial g_{-}(\tau _{-}\circ \xi (x_{+}))$. Since%
\begin{equation*}
\partial g_{-}\left( \tau _{-}\circ \xi \left( x_{+}\right) \right)
=\{s_{g_{-}}\circ \tau _{-}\circ \xi \left( x_{+}\right) \}
\end{equation*}%
is a singleton ($g_{-}$ being Gateaux-differentiable), it follows that $%
\mathfrak{d}_{\xi }^{\flat }(x_{+})\in M^{\flat }(x_{+})$ for all $x_{+}\in 
\mathrm{M}^{\flat }$. In other words, $\mathfrak{d}_{\xi }^{\flat }$ is a $%
\flat $-decision rule.
\end{proof}

We now conclude our study of decision rules by giving sufficient conditions
for the existence of an equilibrium state selection as well as its
properties. All of this is summarized in the following proposition, which
concludes our subsection:

\begin{proposition}[Properties of equilibrium state selections]
\label{equilibrium state selections prop}\mbox{ }\newline
Under Conditions B1--B3 the following assertions hold: \newline
\emph{(i)} If the function $g_{+}:\mathcal{X}_{+}\rightarrow \mathbb{R}$ is
Gateaux%
%TCIMACRO{\TeXButton{\-}{\-}}%
%BeginExpansion
\-%
%EndExpansion
-differentiable then any equilibrium state selection is injective. \newline
\emph{(ii)} If, for all $x_{+}\in \mathrm{M}^{\flat }$, the set $E_{\mathcal{%
G}_{x_{+},\cdot }}^{\mathrm{sc}}$ of self-consistent equilibrium states is a
singleton, then there is a unique equilibrium state selection $\mathrm{M}%
^{\flat }\rightarrow E_{\mathbb{F}}$. Additionally, in this case the unique
equilibrium state selection is surjective. This occurs in particular when $%
g_{-}\circ \tau _{-}:K\rightarrow \mathbb{R}$ is strictly convex, or if the
Bogoliubov linearization $\mathcal{G}_{y_{+},y_{-}}$ has a unique linear
equilibrium state for all $y_{\pm }\in \mathcal{X}_{\pm }^{\ast }$. \newline
\emph{(iii)} If $g_{+}:\mathcal{X}_{+}^{\ast }\rightarrow \mathbb{R}$ is
Gateaux%
%TCIMACRO{\TeXButton{\-}{\-}}%
%BeginExpansion
\-%
%EndExpansion
-differentiable and, for all $x_{+}\in \mathrm{M}^{\flat }$, the set $E_{%
\mathcal{G}_{x_{+},\cdot }}^{\mathrm{sc}}$ is a singleton, then the unique
equilibrium state selection $\mathrm{M}^{\flat }\rightarrow E_{\mathbb{F}}$
is a homeomorphism with respect to the weak$^{\ast }$ topology of $\mathrm{M}%
^{\flat }\subseteq \mathcal{X}_{+}^{\ast }$.
\end{proposition}

\begin{proof}
If the function $g_{+}:\mathcal{X}_{+}^{\ast }\rightarrow \mathbb{R}$ is
Gateaux%
%TCIMACRO{\TeXButton{\-}{\-}}%
%BeginExpansion
\-%
%EndExpansion
-differentiable then, for any $x_{+},x_{+}^{\prime }\in \mathrm{M}^{\flat }$
such that $x_{+}\neq x_{+}^{\prime }$, we know from Theorem \ref{Proposition
importante bogoluibov01 copy(2)} (iii) that $E_{\mathcal{G}_{x_{+},\cdot }^{%
\mathrm{sc}}}\cap E_{\mathcal{G}_{x_{+}^{\prime },\cdot }^{\mathrm{sc}%
}}=\emptyset $. Thus, any equilibrium state selection is necessarily
injective in this case. By the proof of Theorem \ref{Proposition importante
bogoluibov01 copy(2)}, for any $x_{+}\in \mathrm{M}^{\flat }$,%
\begin{equation}
E_{\mathcal{G}_{x_{+},\cdot }^{\mathrm{sc}}}=E_{\mathcal{F}%
_{x_{+}}}\subseteq E_{\mathcal{G}_{x_{+},x_{-}}}\ ,  \label{sdsdsds}
\end{equation}%
where $\mathcal{F}_{x_{+}}:K\rightarrow \mathbb{\{-\infty \}\cup R}$ is the
concave and upper semicontinuous function defined by%
\begin{equation*}
\mathcal{F}_{x_{+}}(\mu )\doteq f(\mu )-g_{-}\circ \tau _{-}(\mu
)+x_{+}\circ \tau _{+}(\mu )\ ,\qquad \mu \in K\ .
\end{equation*}%
Clearly, 
\begin{equation}
E_{\mathcal{G}_{x_{+},\cdot }^{\mathrm{sc}}}=\{\mu (x_{+})\}=E_{\mathcal{F}%
_{x_{+}}}\ ,\qquad x_{+}\in \mathrm{M}^{\flat }\ ,  \label{sdsdsds22}
\end{equation}%
is a singleton under the assumptions of Assertion (ii). Thus, the mapping $%
\xi :\mathrm{M}^{\flat }\rightarrow E_{\mathbb{F}}$ defined by $\xi
(x_{+})\doteq \mu (x_{+})$ is the unique possible equilibrium state
selection in this case. It remains to prove that it is weak$^{\ast }$
continuous. With this aim, take any net $(x_{+}^{(j)})_{j\in J}$ in $\mathrm{%
M}^{\flat }$ that converges in the weak$^{\ast }$ topology to some $x_{+}\in 
\mathrm{M}^{\flat }$. By the weak$^{\ast }$-compactness of $E_{\mathbb{F}}$
and Theorem \ref{Proposition importante bogoluibov01 copy(2)} (iii), we can
assume without loss of generality that $\mu (x_{+}^{(j)})\subseteq E_{%
\mathbb{F}}$ converges to some $\mu \in E_{\mathbb{F}}$ in the weak$^{\ast }$
topology. By Equation (\ref{sdsdsds22}), for all $j\in J$ and $\nu \in K$, 
\begin{equation*}
f(\nu )-g_{-}\circ \tau _{-}(\nu )+x_{+}^{(j)}\circ \tau _{+}(\nu )\leq
f(\mu (x_{+}^{(j)}))-g_{-}\circ \tau _{-}(\mu
(x_{+}^{(j)}))+x_{+}^{(j)}\circ \tau _{+}(\mu (x_{+}^{(j)}))\ .
\end{equation*}%
Taking the $j$-limit at fixed $\nu \in K$, we arrive at%
\begin{equation*}
f(\nu )-g_{-}\circ \tau _{-}(\nu )+x_{+}\circ \tau _{+}(\nu )\leq f(\mu
)-g_{-}\circ \tau _{-}(\mu )+x_{+}\circ \tau _{+}(\mu )\ ,\qquad \nu \in K\ ,
\end{equation*}%
using that $\mathrm{M}^{\flat }$ is norm bounded (Proposition \ref%
{Proposition importante bogoluibov01 copy(3)} (ii) or Lemma \ref{Lemma
bogoluibov6} ($\flat $)). In other words, $\mu \in E_{\mathcal{F}%
_{x_{+}}}=E_{\mathcal{G}_{x_{+},\cdot }^{\mathrm{sc}}}$. As $E_{\mathcal{G}%
_{x_{+},\cdot }^{\mathrm{sc}}}=\{\mu (x_{+})\}$, it follows that $\mu
(x_{+}^{(j)})$ converges to $\mu (x_{+})$ in the weak$^{\ast }$ topology and 
$\xi $ is thus weak$^{\ast }$ continuous. If $E_{\mathcal{G}_{\cdot ,x_{+}}^{%
\mathrm{sc}}}$ is a singleton for all $x_{+}\in \mathrm{M}^{\flat }$ then $%
\xi $ is surjective, because of Theorem \ref{Proposition importante
bogoluibov01 copy(2)} (iii). Finally, under the assumption of Assertion
(iii), the unique equilibrium state selection $\xi :\mathrm{M}^{\flat
}\rightarrow E_{\mathbb{F}}$ is weak$^{\ast }$ continuous and bijective,
thanks to Assertions (i)--(ii). Since $\mathrm{M}^{\flat }$ is weak$^{\ast }$%
-compact (Proposition \ref{Proposition importante bogoluibov01 copy(3)} (ii)
or Lemma \ref{Lemma bogoluibov6} ($\flat $)) and $E_{\mathbb{F}}\subseteq K$
with $K$ being a Hausdorff space, $\xi $ is in this case a homeomorphism.
\end{proof}

\subsection{Generalized equilibria and optimal transport\label{transport}}

\noindent \textbf{Generalized equilibria and order parameters. }Physically,
the set of all equilibrium states of a system is expected to be a convex
set, in order to allow for phase mixtures. By contrast, the set $E_{\mathbb{F%
}}=M_{\mathbb{F}}$ of nonlinear equilibrium states defined in Section \ref%
{sect conc conv Bogo} (Equations (\ref{MF})--(\ref{EF})) is not necessarily
convex, because the function $\mathbb{F}$ is generally not concave. That is
why, in the scope of our paradigmatic example, we introduced the notion of
generalized nonlinear equilibrium measures in Section \ref{Generalized
Equilibrium Measures} (Definition \ref{Gneralized equilibrium measures}) by
considering the closed convex hull $G_{P}\doteq \overline{\mathrm{co}}%
(E_{P}) $ of the set $E_{P}$\ of nonlinear equilibrium measures. Exactly the
same can be done in our abstract version of the theory of nonlinear
equilibria.

We thus define the compact convex set%
\begin{equation}
G_{\mathbb{F}}\doteq \overline{\mathrm{co}}(E_{\mathbb{F}})\subseteq K\text{ 
}.  \label{eq gen eq states}
\end{equation}%
We call its elements \emph{generalized nonlinear equilibrium states}.
Physical equilibria should be seen as elements of this set. If $K$ is a
compact convex set in some topological vector space (which is the case in
virtually all important cases, like our paradigmatic example, the nonlinear
thermodynamic formalism of classical dynamical systems, or quantum lattice
systems), note from the Milman theorem \cite[Proposition 1.5]{Phe} that the
extreme points of the convex set $G_{\mathbb{F}}$ are all in $E_{\mathbb{F}}$%
. Thus, the previous set $E_{\mathbb{F}}$ of nonlinear equilibrium states is
rather related to what are called the \emph{pure equilibrium states}, or 
\emph{pure phases}, in Physics.

One also expects that equilibrium states of a physical system are exactly
those that maximize a pressure function associated with the system under
consideration. It turns out that $G_{\mathbb{F}}$ can also be characterized
in this way. It is the set of maximizers of the \emph{upper}\footnote{%
It corresponds to the function $-\Gamma (-\mathbb{F})$ with $\Gamma (f)$
defined by Equation (\ref{gamm regu general}).} $\Gamma $-regularization\ of 
$\mathbb{F}$, which is the smallest concave and upper semicontinuous
function above $\mathbb{F}$. For more details, see Section \ref{Minimization}%
, in particular Theorem \ref{theorem sympa}.

However, an upper $\Gamma $-regularization is usually only concave, which is
not as good as being affine, a more desirable property for a pressure
function from a physical point of view. In fact, in Section \ref{Section
game + measusre} we show for our paradigmatic example that generalized
equilibrium states (measures in this case) are indeed the limits of (weak$%
^{\ast }$-)convergent sequences of approximating maximizers of an affine
pressure function. See Equation (\ref{equilibriun statebis}). This function
turns out to be not semicontinuous in general (it is only Borel-measurable),
but it has the same upper $\Gamma $-regularization as the original nonlinear
pressure. See Theorem \ref{theorem sympa} in this context. Moreover, it is
maximized by any equilibrium state $\rho \in E_{\mathbb{F}}\cap \mathcal{E}%
(K)$ that is extreme in the convex set $K$. As the function is affine, it is
also maximized by any convex combination $\rho \in \mathrm{co}(E_{\mathbb{F}%
}\cap \mathcal{E}(K))\subseteq \overline{\mathrm{co}}(E_{\mathbb{F}})\doteq
G_{\mathbb{F}}$ of such equilibrium states, but not necessarily by any
generalized equilibrium state $\rho \in G_{\mathbb{F}}$, as it is not
necessarily upper semicontinuous. However, in the H\"{o}lder case (which is
the relevant one in this section) the generalized equilibrium states turn
out to be precisely the strict maximizers of the affine pressure function.
See Corollary \ref{Theorem-main1 copy(3)}. In fact, this is not a particular
property of our paradigmatic example, i.e., the thermodynamic formalism of
(classical) dynamical systems. In \cite{BruPedra2} we made the same
observation in the case of quantum lattice systems (quantum spins or
fermionic systems on lattices) when the functions $g_{\pm }$ defining the
nonlinear pressure function $\mathbb{F}$\ are quadratic.

Remark from \cite[Proposition 1.2]{Phe} that if $A$ is a compact subset of a
locally convex (Hausdorff) space $\mathcal{Y}$, then $y\in \overline{\mathrm{%
co}}(A)\subseteq \mathcal{Y}$ iff $y$ is the barycenter of a probability
measure $\mathfrak{m}_{y}$ supported in $A$ (see Definition \ref{Barycenters
of a measure}). If $\overline{\mathrm{co}}(A)$ is a compact Choquet simplex
and $A$ only contains extreme points of $\overline{\mathrm{co}}(A)$ then the
probability measure $\mathfrak{m}_{y}$\ is always unique and the mapping $%
y\mapsto \mathfrak{m}_{y}$\ from $\overline{\mathrm{co}}(A)$ to the set of
probability measures on $A$ is an affine one-to-one correspondence. For more
details see \cite{Phe}.

By Theorem \ref{Proposition importante bogoluibov01 copy(2)} (iii)--(iv), $%
E_{\mathbb{F}}$ is compact and, if the equilibrium states of the associated
Bogoliubov linearizations are unique and $f$ is affine, then $E_{\mathbb{F}}$
is even a compact set of extreme points of $K$ and $G_{\mathbb{F}}\doteq 
\overline{\mathrm{co}}(E_{\mathbb{F}})\subseteq K$ is, in particular, a face
of the convex set $K$. So, in this case, provided $K$ is a compact convex
subset of a locally convex space, $\mu \in G_{\mathbb{F}}$ iff $\mu $ is the
barycenter of a probability measure $\mathfrak{m}_{\mu }$ supported in $E_{%
\mathbb{F}}$ (see Definition \ref{Barycenters of a measure}). As illustrated
in the nonlinear thermodynamic formalism (see Proposition \ref{prop Choquet
inv meas}) or for quantum lattice systems \cite[Theorem 1.9]{BruPedra2}, $G_{%
\mathbb{F}}$ is usually even a compact Choquet simplex and $\mathfrak{m}%
_{\mu }$ is uniquely defined.

If the functions $g_{\pm }:\mathcal{X}_{\pm }\rightarrow \mathbb{R}$ are
Gateaux%
%TCIMACRO{\TeXButton{\-}{\-}}%
%BeginExpansion
\-%
%EndExpansion
-differentiable then the corresponding gradient mappings 
\begin{equation*}
\nabla g_{\pm }:\mathcal{X}_{\pm }\rightarrow \mathcal{X}_{\pm }^{\ast }
\end{equation*}%
are norm-to-weak$^{\ast }$ continuous. They are even norm-to-norm continuous
when the functions $g_{\pm }$ are Fr\'{e}chet%
%TCIMACRO{\TeXButton{\-}{\-}}%
%BeginExpansion
\-%
%EndExpansion
-differentiable. See \cite[Proposition 2.8]{Phe2}. Recall that the
subdifferentials $\partial g_{\pm }(x_{\pm })=\{\nabla g_{\pm }(x_{\pm })\}$
are singletons for any $x_{\pm }\in \mathcal{X}_{\pm }$ , $g_{\pm }$ being
Gateaux-differentiable. In particular, in this case, to any $\mu \in G_{%
\mathbb{F}}$ we naturally associate measures $\mathfrak{y}_{\pm }(\mu )$ on $%
\mathcal{X}_{\pm }^{\ast }$, by considering the pushforward of a probability
measure $\mathfrak{m}_{\mu }$ representing $\mu \in G_{\mathbb{F}}$ (in $G_{%
\mathbb{F}}$) through the mapping 
\begin{equation*}
\nabla g_{\pm }\circ \tau _{\pm }:K\rightarrow \mathcal{X}_{\pm }^{\ast }\ .
\end{equation*}%
As $\nabla g_{\pm }$ is always norm-to-weak$^{\ast }$ continuous and $\tau
_{\pm }:K\rightarrow \mathcal{X}_{\pm }$ are by assumption continuous
(Condition B2), the measures $\mathfrak{y}_{\pm }(\mu )$ are supported on
the weak$^{\ast }$-compact sets $\nabla g_{\pm }\circ \tau _{\pm }(E_{%
\mathbb{F}})\subseteq \mathcal{X}_{\pm }^{\ast }$. In fact, by the
self-consistency of nonlinear equilibrium states (Theorem \ref{Proposition
importante bogoluibov01 copy(2)} (ii)--(iii)), the probability measure $%
\mathfrak{y}_{+}(\mu )$ is supported in the weak$^{\ast }$-compact set $%
\mathrm{M}^{\flat }\subseteq \mathcal{X}_{+}^{\ast }$, while $\mathfrak{y}%
_{-}(\mu )$ is supported in the union%
\begin{equation}
\mathrm{M}_{-}^{\flat }\doteq \bigcup_{x_{+}\in \mathrm{M}^{\flat }}M^{\flat
}\left( x_{+}\right) \subseteq \mathcal{X}_{-}^{\ast }\ .  \label{M_-}
\end{equation}%
If $g_{\pm }$ are Fr\'{e}chet%
%TCIMACRO{\TeXButton{\-}{\-}}%
%BeginExpansion
\-%
%EndExpansion
-differentiable then the sets $\nabla g_{\pm }\circ \tau _{\pm }(E_{\mathbb{F%
}})$\ are even norm-compact, as the gradient mapping is norm-to-norm
continuous in this case, thanks again to \cite[Proposition 2.8]{Phe2}.

Physically, the elements of the dual spaces $\mathcal{X}_{\pm }^{\ast }$
refer to so-called \emph{order parameters} of the system under
consideration. For example, if we consider a spin system, one typical order
parameter of interest is the macroscopic magnetization density. Similarly,
if one considers models for superconductors, then one typical order
parameter would be the macroscopic density of Cooper pairs. Thus, given a
generalized nonlinear equilibrium state $\mu \in G_{\mathbb{F}}$, the
measures $\mathfrak{y}_{\pm }(\mu )$ give the corresponding distributions of
order parameters. The order parameter distribution is well-defined in
virtually all cases of interest, because $G_{\mathbb{F}}$ is usually a
Choquet simplex in these cases. In a so-called pure phase ($\mu \in E_{%
\mathbb{F}}$) such distributions are delta distributions, that is, they are
concentrated in single points and thus have zero variance. If these
distributions have non-vanishing variance then we have a so-called phase
mixture, which is the general situation at a (first order) phase transition.
Of course, the order parameter distributions $\mathfrak{y}_{\pm }(\mu )$
also make sense for states $\mu \in K$ that are more general, that is, not
necessarily elements of $G_{\mathbb{F}}$, as far as the compact convex space 
$K$ is a Choquet simplex (and $g_{\pm }:\mathcal{X}_{\pm }\rightarrow 
\mathbb{R}$ stay Gateaux%
%TCIMACRO{\TeXButton{\-}{\-}}%
%BeginExpansion
\-%
%EndExpansion
-differentiable of course). For simplicity we restrict ourselves to the
equilibrium case.

One question that arises is how to calculate, from possible experiments,
including numerical ones, the order parameter distributions $\mathfrak{y}%
_{\pm }(\mu )$ at least for generalized nonlinear equilibrium measures $\mu
\in G_{\mathbb{F}}$. In our paradigmatic example (Section \ref{thermo form})
we can give a more precise empirical interpretation to these order parameter
distributions, which can thus be recovered from possible experiments:\medskip

\noindent \textbf{Paradigmatic example -- Nonlinear Thermodynamic Formalism
of Dynamical Systems.} In this example, $K=\mathcal{P}(T)$ is the set (\ref%
{set T-inv}) of $T$-invariant probability measures on the Banach space $%
C(\Sigma )^{\ast }$ of continuous functions $\Sigma \rightarrow \mathbb{R}$, 
$f=h$ is the entropy (Definition \ref{uod}) and $\tau _{\pm }:\mathcal{P}%
(T)\rightarrow \mathcal{X}_{\pm }$ are defined by $\tau _{\pm }(\mu )\doteq
\theta _{\pm }(\mu _{S})$ for any $\mu \in \mathcal{P}(T)$. In this
particular case $\mathbb{F}=P$ is defined by Equation (\ref{eq 21}) and we
assume in\ Section \ref{Bogoliubov linearizations TF} Conditions TF1--TF3,
which are nothing more than Conditions B1--B3 of Section \ref{sect conc conv
Bogo} applied to our paradigmatic example, as explained in the proof of
Proposition \ref{Proposition importante bogoluibov01 copy(5)}.

Observe first that $K=\mathcal{P}(T)$ is a weak$^{\ast }$-compact Choquet
simplex (Proposition \ref{prop Choquet inv meas}). It is a subset of the
dual space $C(\Sigma )^{\ast }$, which is a locally convex (Hausdorff) space
when endowed with the weak$^{\ast }$ topology. Since the entropy $f=h$ is
affine (Proposition \ref{Affinity of the entropy copy(1)}), $G_{P}$ is also
a Choquet simplex and one can naturally identify generalized nonlinear
equilibrium states $\mu \in G_{P}$ with probability measures $\mathfrak{m}%
_{\mu }$ on the weak$^{\ast }$-compact set $E_{P}$ of (simple) nonlinear
equilibrium states.

Recall that $S$ denotes the unit closed ball in $C(\Sigma )$ and $\mathcal{M}%
(S)$ is the Banach space of bounded, real-valued Borel-measurable functions $%
S\rightarrow \mathbb{R}$ with the supremum norm. See Section \ref{math
framework}, in particular Equation (\ref{norm1bis}). For any $\sigma \in
\Sigma $ and $n\in \mathbb{N}$, let $\mathbb{E}_{n,\mathbb{\sigma }%
}:S\rightarrow \mathbb{R}$ be the continuous function defined by Equation (%
\ref{eq def Ensigbis}), that is,%
\begin{equation*}
\mathbb{E}_{n,\mathbb{\sigma }}\left[ \varphi \right] \doteq \mathbb{E}_{n}%
\left[ \varphi \right] (\mathbb{\sigma })\doteq n^{-1}\left( \varphi (%
\mathbb{\sigma })+\varphi \circ T(\mathbb{\sigma })+\cdots +\varphi \circ
T^{n-1}(\mathbb{\sigma })\right) \ ,\qquad \varphi \in S\ .
\end{equation*}%
Observe that, at any fixed $n\in \mathbb{N}$, $\mathbb{E}_{n,\mathbb{\sigma }%
^{\prime }}$ converges point-wise to $\mathbb{E}_{n,\mathbb{\sigma }}$, as $%
\sigma ^{\prime }\rightarrow \sigma $. Thus, if the linear mappings $\theta
_{\pm }:\mathcal{M}(S)\rightarrow \mathcal{X}_{\pm }$ are $\mathbb{\sigma }$%
-normal (see Condition TF1 of Section \ref{Bogoliubov linearizations TF})
then $\theta _{\pm }(\mathbb{E}_{n,\mathbb{\sigma }^{\prime }})$ converges
in norm to $\theta _{\pm }(\mathbb{E}_{n,\mathbb{\sigma }^{\prime }})$, as $%
\sigma ^{\prime }\rightarrow \sigma $. That is, for any $n\in \mathbb{N}$,
the mappings $\sigma \mapsto \theta _{\pm }(\mathbb{E}_{n,\mathbb{\sigma }})$
from $\Sigma $ to $\mathcal{X}_{\pm }$ are continuous. Recall that $g_{\pm }:%
\mathcal{X}_{\pm }\rightarrow \mathbb{R}$ are assumed here to be Gateaux%
%TCIMACRO{\TeXButton{\-}{\-}}%
%BeginExpansion
\-%
%EndExpansion
-differentiable and by continuity of the gradient mappings $\nabla g_{\pm }$
the functions $Y_{\pm ,n}:\Sigma \rightarrow \mathcal{X}_{\pm }^{\ast }$
defined by 
\begin{equation}
\sigma \mapsto \nabla g_{\pm }\circ \theta _{\pm }(\mathbb{E}_{n,\mathbb{%
\sigma }})\text{ },  \label{func}
\end{equation}%
are also continuous. In particular, they are Borel-measurable and thus
define real-valued random variables 
\begin{equation*}
\begin{array}{lllll}
Y_{\pm ,n}\left( x_{\pm }\right) & : & \Sigma & \rightarrow & \mathbb{R} \\ 
&  & \sigma & \mapsto & \nabla g_{\pm }\circ \theta _{\pm }\left( \mathbb{E}%
_{n,\mathbb{\sigma }}\right) \left( x_{\pm }\right)%
\end{array}%
\end{equation*}%
for any $x_{\pm }\in \mathcal{X}_{\pm }$, the distribution law of which is
thus the pushforward of some generalized nonlinear equilibrium measure $\mu
\in G_{P}$ through the functions (\ref{func}), whose values are linear
functionals, applied to vectors $x_{\pm }\in \mathcal{X}_{\pm }$. ($Y_{\pm
,n}$ are also random variables taking values in $\mathcal{X}_{\pm }^{\ast }$%
.) The distribution law of $Y_{\pm ,n}\left( x_{\pm }\right) $ is denoted
below by $\mathfrak{y}_{\pm }^{(n)}(\mu )$.

If the pair $(\theta _{-},\theta _{+})$ is a H\"{o}lder-type linear
functions (Definition \ref{sect Holder type theta}) then all the Bogoliubov
linearizations have \emph{unique} equilibrium measures, which are all
ergodic, thanks to Corollary \ref{Theorem-main1 copy(2)}. Then, $E_{P}$ is
in this case a compact set of extreme points of $\mathcal{P}(T)$ and $%
G_{P}\doteq \overline{\mathrm{co}}(E_{P})\subseteq \mathcal{P}(T)$ is, in
particular, a face of the convex set $\mathcal{P}(T)$ and a weak$^{\ast }$%
-compact Choquet simplex. Therefore, as already explained, $\mu \in G_{P}$
iff $\mu $ is the barycenter of a \emph{unique} probability measure $%
\mathfrak{m}_{\mu }$ supported in $E_{P}\subseteq \mathcal{P}_{\mathrm{erg}%
}(T)$ (see Definition \ref{Barycenters of a measure}). Furthermore, in the H%
\"{o}lder case, the mapping $\mu \mapsto \mathfrak{m}_{\mu }$\ from $G_{P}$
to the set of probability measures on $E_{P}$ is an affine one-to-one
correspondence. For more details see again \cite{Phe}. In this situation, we
can prove that, given a (generalized nonlinear equilibrium) $T$-invariant
probability measure $\mu \in G_{P}$, the characteristic functions of $%
\mathfrak{y}_{\pm }^{(n)}(\mu )$ point-wise converge to the characteristic
functions of the order parameter distributions $\mathfrak{y}_{\pm }(\mu )$,
as $n\rightarrow \infty $:

\begin{proposition}[Order parameter distributions from Birkhoff sums -- H%
\"{o}lder case]
\label{prop order}\mbox{ }\newline
Assume Conditions TF1--TF3 of Section \ref{Bogoliubov linearizations TF} and
that the pair $(\theta _{-},\theta _{+})$ of linear functions is H\"{o}%
lder-type and $g_{\pm }:\mathcal{X}_{\pm }\rightarrow \mathbb{R}$ are Gateaux%
%TCIMACRO{\TeXButton{\-}{\-}}%
%BeginExpansion
\-%
%EndExpansion
-differentiable. Let $\mathfrak{y}_{\pm }\equiv \mathfrak{y}_{\pm }(\mu )$
for some $\mu \in G_{P}$. Then, for all $x_{\pm }\in \mathcal{X}_{\pm }$,%
\begin{equation*}
\lim_{n\rightarrow \infty }\int_{\mathcal{X}_{\pm }^{\ast }}\exp \left(
iy_{\pm }\left( x_{\pm }\right) \right) \mathfrak{y}_{\pm }^{(n)}\left( 
\mathrm{d}y_{\pm }\right) =\int_{\mathcal{X}_{\pm }^{\ast }}\exp \left(
iy_{\pm }\left( x_{\pm }\right) \right) \mathfrak{y}_{\pm }\left( \mathrm{d}%
y_{\pm }\right) \text{ }.
\end{equation*}
\end{proposition}

\begin{proof}
Fix $\mu \in G_{P}$. For all $n\in \mathbb{N}$ and $x_{\pm }\in \mathcal{X}%
_{\pm }$,%
\begin{equation*}
\int_{\mathcal{X}_{\pm }^{\ast }}\exp \left( iy_{\pm }\left( x_{\pm }\right)
\right) \mathfrak{y}_{\pm }^{(n)}\left( \mathrm{d}y_{\pm }\right)
=\int_{\Sigma }\exp \left( i\left[ \nabla g_{\pm }\circ \theta _{\pm }\left( 
\mathbb{E}_{n,\mathbb{\sigma }}\right) \right] \left( x_{\pm }\right)
\right) \mu \left( \mathrm{d}\mathbb{\sigma }\right) \text{ }.
\end{equation*}%
Note that the mappings 
\begin{equation*}
\sigma \mapsto \left[ \nabla g_{\pm }\circ \theta _{\pm }\left( \mathbb{E}%
_{n,\mathbb{\sigma }}\right) \right] \left( x_{\pm }\right)
\end{equation*}%
from $\Sigma $ to $\mathbb{R}$ are continuous. Further, since $\mathcal{P}%
(T) $ is a weak$^{\ast }$-compact Choquet simplex, $\mu $ is the barycenter
of a (unique) probability measure $\mathfrak{m}_{\mu }$ on $E_{P}$ and
therefore, 
\begin{equation*}
\int_{\Sigma }\exp \left( i\left[ \nabla g_{\pm }\circ \theta _{\pm }\left( 
\mathbb{E}_{n,\mathbb{\sigma }}\right) \right] \left( x_{\pm }\right)
\right) \mu \left( \mathrm{d}\mathbb{\sigma }\right) =\int_{K}\left\{
\int_{\Sigma }\exp \left( i\left[ \nabla g_{\pm }\circ \theta _{\pm }\left( 
\mathbb{E}_{n,\mathbb{\sigma }}\right) \right] \left( x_{\pm }\right)
\right) \nu \left( \mathrm{d}\mathbb{\sigma }\right) \right\} \mathfrak{m}%
_{\mu }\left( \mathrm{d}\nu \right) \text{ }.
\end{equation*}%
By Corollary \ref{Theorem-main1 copy(2)} (i), all $\nu \in E_{P}$ are
ergodic measures and we can infer from Lemma \ref{lemma In conv pw}\ and
Lebesgue's dominated convergence theorem that, for all $\nu \in E_{P}$, 
\begin{equation*}
\lim_{n\rightarrow \infty }\int_{\Sigma }\exp \left( i\left[ \nabla g_{\pm
}\circ \theta _{\pm }\left( \mathbb{E}_{n,\mathbb{\sigma }}\right) \right]
\left( x_{\pm }\right) \right) \nu \left( \mathrm{d}\mathbb{\sigma }\right)
=\exp \left( i\left[ \nabla g_{\pm }\circ \theta _{\pm }\left( \nu
_{S}\right) \right] \left( x_{\pm }\right) \right) \text{ }.
\end{equation*}%
Hence, again by Lebesgue's dominated convergence theorem,%
\begin{eqnarray*}
\lim_{n\rightarrow \infty }\int_{\Sigma }\exp \left( i\left[ \nabla g_{\pm
}\circ \theta _{\pm }\left( \mathbb{E}_{n,\mathbb{\sigma }}\right) \right]
\left( x_{\pm }\right) \right) \mu \left( \mathrm{d}\mathbb{\sigma }\right)
&=&\int_{K}\exp \left( i\left[ \nabla g_{\pm }\circ \theta _{\pm }\left( \nu
_{S}\right) \right] \left( x_{\pm }\right) \right) \mathfrak{m}_{\mu }\left( 
\mathrm{d}\nu \right) \\
&=&\int_{\mathcal{X}_{\pm }^{\ast }}\exp \left( iy_{\pm }\left( x_{\pm
}\right) \right) \mathfrak{y}_{\pm }\left( \mathrm{d}y_{\pm }\right) \text{ }%
.
\end{eqnarray*}%
Indeed, recall that the measures $\mathfrak{y}_{\pm }(\mu )$ on $\mathcal{X}%
_{\pm }^{\ast }$ are the pushforward of $\mathfrak{m}_{\mu }$ through the
mappings $\nabla g_{\pm }\circ \tau _{\pm }$, where $\tau _{\pm }(\nu
)\doteq \theta _{\pm }(\nu _{S})$ for any $\nu \in \mathcal{P}(T)$. Using
well-known properties of the characteristic functions of real-valued random
variables (cf. L\'{e}vy's continuity theorem \cite[Theorem 5.3]{Levy}), we
conclude from the above proposition that, for all $x_{\pm }\in \mathcal{X}%
_{\pm }$, the distribution laws $\mathfrak{y}_{\pm }^{(n)}(\mu )$ of the
real-valued random variables $Y_{\pm ,n}(x_{\pm })$ weak$^{\ast }$ converge,
as $n\rightarrow \infty $, to the order parameter distributions $\mathfrak{y}%
_{\pm }(\mu )$.
\end{proof}

It means that the order parameter distributions can be recovered from
experiments, possibly numerical ones. The order parameter distributions are
also related to optimal transport \cite{Vilani}, for which various
high-performance numerical tools are available. This link is explained below
for the general (abstract) case. \medskip

\noindent \textbf{Optimal transport and distributions of order parameters.}
We show here that, for any given order parameter distribution at
equilibrium, that is, $\mathfrak{y}_{\pm }\equiv \mathfrak{y}_{\pm }(\mu )$
for some $\mu \in G_{\mathbb{F}}$, which is not necessarily explicitly
known, the nonlinear pressure $\sup \mathbb{F}(K)$ can be exactly recovered
from an optimal-transport problem associated with these order parameter
distributions, the cost function of which is nothing but the nonlinear
approximating pressure $P_{\mathrm{NL}}$ (\ref{pression approxbis}) (similar
to the thermodynamic game).

Recall from\ Theorem \ref{Proposition importante bogoluibov01 copy(2)} (i)
that the nonlinear pressure $\sup \mathbb{F}(K)$ satisfies%
\begin{equation*}
\sup \mathbb{F}\left( K\right) =\mathrm{P}^{\flat }\doteq \sup_{y_{+}\in 
\mathcal{X}_{+}^{\ast }}\inf_{y_{-}\in \mathcal{X}_{-}^{\ast }}P_{\mathrm{NL}%
}\left( y_{+},y_{-}\right)
\end{equation*}%
and consider the continuous mapping 
\begin{equation*}
\begin{array}{lll}
E_{\mathbb{F}} & \rightarrow & \mathcal{X}_{+}^{\ast }\times \mathcal{X}%
_{-}^{\ast } \\ 
\nu & \mapsto & \left( \nabla g_{+}\circ \tau _{+}\left( \nu \right) ,\nabla
g_{-}\circ \tau _{-}\left( \nu \right) \right)%
\end{array}%
\end{equation*}%
for Gateaux%
%TCIMACRO{\TeXButton{\-}{\-}}%
%BeginExpansion
\-%
%EndExpansion
-differentiable functions $g_{\pm }:\mathcal{X}_{\pm }\rightarrow \mathbb{R}$%
. Fix now any $\mu \in G_{\mathbb{F}}$ and let $\mathfrak{n}_{\mu }$ be the
pushforward of a probability measure $\mathfrak{m}_{\mu }$ representing $\mu 
$ in $G_{\mathbb{F}}$ through the above continuous mapping. By construction,
the marginal distributions on $\mathcal{X}_{+}^{\ast }\ $and $\mathcal{X}%
_{-}^{\ast }$ are $\mathfrak{y}_{+}(\mu )$ and $\mathfrak{y}_{-}(\mu )$,
respectively. Moreover, by Theorem \ref{Proposition importante bogoluibov01
copy(2)},%
\begin{eqnarray}
\int_{\mathcal{X}_{+}^{\ast }\times \mathcal{X}_{-}^{\ast }}P_{\mathrm{NL}%
}\left( y_{+},y_{-}\right) \mathfrak{n}_{\mu }\left( \mathrm{d}\left(
y_{+},y_{-}\right) \right) &=&\int_{K}P_{\mathrm{NL}}\left( \nabla
g_{+}\circ \tau _{+}\left( \nu \right) ,\nabla g_{-}\circ \tau _{-}\left(
\nu \right) \right) \mathfrak{m}_{\mu }\left( \mathrm{d}\nu \right)  \notag
\\
&=&\int_{K}\sup \mathbb{F}\left( K\right) \mathfrak{m}_{\mu }\left( \mathrm{d%
}\nu \right) =\sup \mathbb{F}\left( K\right) \ .  \label{sdsssds1}
\end{eqnarray}%
On the other hand, for any probability measure $\mathfrak{n}$ on $\mathcal{X}%
_{+}^{\ast }\times \mathcal{X}_{-}^{\ast }$ whose marginal distributions are 
$\mathfrak{y}_{+}\equiv \mathfrak{y}_{+}(\mu )$ and $\mathfrak{y}_{-}\equiv 
\mathfrak{y}_{-}(\mu )$, one has%
\begin{equation}
\int_{\mathcal{X}_{+}^{\ast }\times \mathcal{X}_{-}^{\ast }}P_{\mathrm{NL}%
}\left( y_{+},y_{-}\right) \mathfrak{n}\left( \mathrm{d}\left(
y_{+},y_{-}\right) \right) \geq \int_{\mathcal{X}_{+}^{\ast }\times \mathcal{%
X}_{-}^{\ast }}P^{\flat }\left( y_{+}\right) \mathfrak{n}\left( \mathrm{d}%
\left( y_{+},y_{-}\right) \right) =\int_{\mathcal{X}_{+}^{\ast }}P^{\flat
}\left( y_{+}\right) \mathfrak{y}_{+}\left( \mathrm{d}y_{+}\right) \ ,
\label{sdsssds20}
\end{equation}%
where we recall that $P^{\flat }(y_{+})$ is defined by (\ref{P bemol plus}),
i.e.,%
\begin{equation*}
P^{\flat }\left( y_{+}\right) \doteq \inf P_{\mathrm{NL}}\left( y_{+},%
\mathcal{X}_{-}^{\ast }\right) \ ,\qquad y_{+}\in \mathcal{X}_{+}^{\ast }\ .
\end{equation*}%
Note that the function $y_{+}\mapsto P^{\flat }\left( y_{+}\right) $ is
Borel-measurable, because it is upper semicontinuous, by Lemma \ref{Lemma
bogoluibov4} ($\flat $). \ As $\mathfrak{y}_{+}$ is supported in $\mathrm{M}%
^{\flat }$ (see Theorem \ref{Proposition importante bogoluibov01 copy(2)}),
one conclude that 
\begin{equation}
\int_{\mathcal{X}_{+}^{\ast }\times \mathcal{X}_{-}^{\ast }}P_{\mathrm{NL}%
}\left( y_{+},y_{-}\right) \mathfrak{n}\left( \mathrm{d}\left(
y_{+},y_{-}\right) \right) \geq \int_{\mathcal{X}_{+}^{\ast }}\sup \mathbb{F}%
\left( K\right) \mathfrak{y}_{+}\left( \mathrm{d}y_{+}\right) =\sup \mathbb{F%
}\left( K\right) \text{ }.  \label{sdsssds2}
\end{equation}%
In other words, by (\ref{sdsssds1}) and (\ref{sdsssds2}), the nonlinear
pressure can be recovered from the following \emph{Monge-Kantorovich
(transportation) problem} \cite[page 10]{Vilani}:%
\begin{equation}
\sup \mathbb{F}\left( K\right) =\inf_{\mathfrak{n}\in \mathfrak{\Gamma }%
\left( \mathfrak{\mathfrak{y}_{+},\mathfrak{y}_{-}}\right) }\int_{\mathcal{X}%
_{+}^{\ast }\times \mathcal{X}_{-}^{\ast }}P_{\mathrm{NL}}\left(
y_{+},y_{-}\right) \mathfrak{n}\left( \mathrm{d}\left( y_{+},y_{-}\right)
\right) \text{ },  \label{Monge-Kantorovich}
\end{equation}%
where $\mathfrak{\Gamma }(\mathfrak{\mathfrak{y}_{+},\mathfrak{y}_{-}})$ is
the set of all probability measures on $\mathcal{X}_{+}^{\ast }\times 
\mathcal{X}_{-}^{\ast }$ whose marginal distributions are respectively $%
\mathfrak{y}_{+}\equiv \mathfrak{y}_{+}(\mu )$ and $\mathfrak{y}_{-}\equiv 
\mathfrak{y}_{-}(\mu )$ for arbitrary fixed equilibrium order parameter
distributions associated with some $\mu \in G_{\mathbb{F}}$. As a
consequence, the dual Kantorovich problem \cite[Equation (5.3)]{Vilani},
along with appropriate assumptions, leads to the following representation of
the nonlinear pressure:

\begin{theorem}[Nonlinear pressure and the dual Kantorovich problem]
\label{prop order copy(1)}\mbox{ }\newline
Assume Conditions B1--B3. Suppose additionally that $\mathcal{X}_{\pm }$ are
separable normed spaces, $g_{\pm }:\mathcal{X}_{\pm }\rightarrow \mathbb{R}$
are Gateaux%
%TCIMACRO{\TeXButton{\-}{\-}}%
%BeginExpansion
\-%
%EndExpansion
-differentiable, $g_{\pm }^{\ast }:\mathcal{X}_{\pm }^{\ast }\rightarrow 
\mathbb{R}$ are continuous with $g_{-}^{\ast }$\ being strictly convex. Let $%
\mathfrak{y}_{\pm }\equiv \mathfrak{y}_{\pm }(\mu )$ for some $\mu \in G_{%
\mathbb{F}}$. Then, 
\begin{equation*}
\sup \mathbb{F}\left( K\right) =\sup \left\{ \int_{\mathcal{X}_{+}^{\ast
}}P_{\mathrm{NL}}^{+}\left( y_{+}\right) \mathfrak{y}_{+}\left( \mathrm{d}%
y_{+}\right) -\int_{\mathcal{X}_{-}^{\ast }}P_{\mathrm{NL}}^{-}\left(
y_{-}\right) \mathfrak{y}_{-}\left( \mathrm{d}y_{-}\right) \right\} \text{ },
\end{equation*}%
where the supremum is taken with respect to all possible choices of
continuous functions $P_{\mathrm{NL}}^{+}:\mathcal{X}_{+}^{\ast }\rightarrow 
\mathbb{R}$ and $P_{\mathrm{NL}}^{-}:\mathcal{X}_{-}^{\ast }\rightarrow 
\mathbb{R}$ such that 
\begin{equation*}
P_{\mathrm{NL}}^{+}\left( y_{+}\right) -P_{\mathrm{NL}}^{-}\left(
y_{-}\right) \leq P_{\mathrm{NL}}\left( y_{+},y_{-}\right) \text{ },\qquad
y_{\pm }\in \mathcal{X}_{\pm }^{\ast }\text{ }.
\end{equation*}
\end{theorem}

\begin{proof}
On the one hand, observe from Theorem \ref{Proposition importante
bogoluibov01 copy(2)} that $\mathfrak{y}_{+}$ is always supported in $%
\mathrm{M}^{\flat }\equiv \mathrm{M}_{+}^{\flat }$, while $\mathfrak{y}_{-}$
is always supported on the set $\mathrm{M}_{-}^{\flat }$ defined\ by (\ref%
{M_-}). As a consequence, Equation (\ref{Monge-Kantorovich}) can be
rewritten as 
\begin{equation}
\sup \mathbb{F}\left( K\right) =\inf_{\mathfrak{n}\in \mathfrak{\Gamma }%
\left( \mathfrak{\mathfrak{y}_{+},\mathfrak{y}_{-}}\right) }\int_{\mathrm{M}%
_{+}^{\flat }\times \mathrm{M}_{-}^{\flat }}P_{\mathrm{NL}}\left(
y_{+},y_{-}\right) \mathfrak{n}\left( \mathrm{d}\left( y_{+},y_{-}\right)
\right) \text{ }.  \label{Monge-Kantorovich2}
\end{equation}%
By Proposition \ref{Proposition importante bogoluibov01 copy(3)} (ii) (or
Lemma \ref{Lemma bogoluibov6} ($\flat $)), $\mathrm{M}^{\flat }\equiv 
\mathrm{M}_{+}^{\flat }$ is a (nonempty) weak$^{\ast }$-compact subset of $%
\mathcal{X}_{+}^{\ast }$. Since, by assumption, $g_{+}^{\ast }$ is
continuous, $\mathrm{dom}(g_{-}^{\ast })=\mathcal{X}_{-}^{\ast }$ and $%
g_{-}^{\ast }:\mathcal{X}_{-}^{\ast }\rightarrow \mathbb{R}$\ is strictly
convex, we can deduce from Propositions \ref{Proposition importante
bogoluibov01 copy(3)} and \ref{decision rules prop} that the nonempty set $%
\mathrm{M}_{-}^{\flat }$ is also a norm-bounded weak$^{\ast }$-compact
subset of $\mathcal{X}_{-}^{\ast }$. Since $\mathcal{X}_{\pm }$ are assumed
to be separable normed spaces, the weak$^{\ast }$ topology is metrizable on
the weak$^{\ast }$-compacts sets $\mathrm{M}_{\pm }^{\flat }\subseteq 
\mathcal{X}_{\pm }^{\ast }$, which can thus be considered as metric spaces.
Because compact metric spaces are separable and complete, it follows that $(%
\mathrm{M}_{\pm }^{\flat },\mathfrak{y}_{\pm })$ are Polish probability
spaces. As $\mathrm{M}_{\pm }^{\flat }$ are norm-bounded, by Lemma \ref%
{Lemma bogoluibov2} and continuity of $g_{\pm }^{\ast }:\mathcal{X}_{\pm
}^{\ast }\rightarrow \mathbb{R}$, the nonlinear approximating pressure 
\begin{equation*}
P_{\mathrm{NL}}:\mathrm{M}_{+}^{\flat }\times \mathrm{M}_{-}^{\flat
}\rightarrow \mathbb{R}
\end{equation*}%
is weak$^{\ast }$ continuous. Moreover, for any $y_{\pm }\in \mathcal{X}%
_{\pm }^{\ast }$, 
\begin{equation*}
P_{\mathrm{NL}}\left( y_{+},y_{-}\right) \doteq P_{\mathrm{L}}\left(
y_{+},y_{-}\right) +g_{-}^{\ast }\left( y_{-}\right) -g_{+}^{\ast }\left(
y_{+}\right) \geq P^{\flat }\left( y_{+}\right) +g_{-}^{\ast }\left(
y_{-}\right) -g_{+}^{\ast }\left( y_{+}\right)
\end{equation*}%
with $P^{\flat }(y_{+})$ defined by (\ref{P bemol plus}). Since $g_{\pm
}^{\ast }:\mathcal{X}_{\pm }^{\ast }\rightarrow \mathbb{R}$ are continuous
and $\mathrm{M}_{\pm }^{\flat }$ are weak$^{\ast }$-compact, 
\begin{equation*}
\int_{\mathrm{M}_{\pm }^{\flat }}g_{\pm }^{\ast }\left( y_{\pm }\right) 
\mathfrak{y}_{\pm }\left( \mathrm{d}y_{\pm }\right) <\infty \text{\qquad
and\qquad }\int_{\mathrm{M}_{+}^{\flat }}P^{\flat }\left( y_{+}\right) 
\mathfrak{y}_{+}\left( \mathrm{d}y_{+}\right) =\sup \mathbb{F}\left(
K\right) <\infty \ ,
\end{equation*}%
as already explained in (\ref{sdsssds20})--(\ref{sdsssds2}). Observe again
that the function $y_{+}\mapsto P^{\flat }\left( y_{+}\right) $ is upper
semicontinuous, by Lemma \ref{Lemma bogoluibov4} ($\flat $). All the
conditions of the Kantorovich duality theorem \cite[Theorem 5.10]{Vilani}
are thus satisfied. So, the assertion follows from Equation (\ref%
{Monge-Kantorovich2}) combined with \cite[Theorem 5.10 (i)]{Vilani}.
\end{proof}

Solutions (or at least approximate solutions $\pm P_{\mathrm{NL}}^{\pm }$)
to the dual Kantorovich problem expressed in Theorem \ref{prop order copy(1)}
should be physically interpreted as being effective pressure functions on
the spaces of order parameters.

Note that the assumptions of Theorem \ref{prop order copy(1)} are trivially
satisfied for physically relevant choices of convex functions $g_{\pm }$.
Take for example $\beta _{\pm }\in \mathbb{R}^{+}$ and 
\begin{equation*}
g_{\pm }(x)=\beta _{\pm }x^{2}/2\qquad \text{or}\qquad g_{\pm }(x)=2\cosh x
\end{equation*}%
with $\mathcal{X}_{\pm }=\mathbb{R}$ and\ $\mathcal{X}_{\pm }^{\ast }\equiv 
\mathbb{R}$, and remark that the corresponding Legendre-Fenchel transforms, 
\begin{equation*}
g_{\pm }^{\ast }(s)=s^{2}/(2\beta _{\pm })\qquad \text{or}\qquad g_{\pm
}^{\ast }(s)=s\text{ }\mathrm{\sinh }^{-1}(s/2)-\sqrt{4+s^{4}}\text{ },
\end{equation*}%
are smooth and strictly convex. See also Remarks \ref{Remark-condition1} and %
\ref{Remark-condition2}. Observe additionally that the strict convexity of
the function $g_{-}^{\ast }$ assumed in Theorem \ref{prop order copy(1)} is
a sufficient condition only used for simplicity in order to ensure the weak$%
^{\ast }$ compactness of the set $\mathrm{M}_{-}^{\flat }$ (\ref{M_-}), but
this should not, of course, be essential.

In fact, in the proof of Theorem \ref{prop order copy(1)}, we only use the
first assertion (i) of the Kantorovich duality theorem \cite[Theorem 5.10]%
{Vilani}. There are many other results in \cite[Theorem 5.10]{Vilani} that
could have been invoked, but we refrain from mentioning them here, as our
goal is not to provide a detailed analysis of the corresponding optimal
transport under the more general conditions, but to establish a fruitful
bridge between the maximization of real-valued functions $\mathbb{F}$ on
compact spaces and the dual Kantorovich problem, thus paving the way for
entirely new mathematical and numerical developments.

\subsection{Technical results on the thermodynamic game\label{Nonlinear
pressure}}

Within the general setting of the previous sections we study here the linear
and nonlinear approximating pressures associated with Bogoliubov
linearizations, that is, the quantities 
\begin{equation}
P_{\mathrm{L}}\left( y_{+},y_{-}\right) \doteq \sup \mathcal{G}%
_{y_{+},y_{-}}\left( K\right) \quad \text{and}\quad P_{\mathrm{NL}}\left(
y_{+},y_{-}\right) \doteq P_{\mathrm{L}}\left( y_{+},y_{-}\right)
+g_{-}^{\ast }\left( y_{-}\right) -g_{+}^{\ast }\left( y_{+}\right)
\label{pression approxbis nlbis}
\end{equation}%
for any $y_{\pm }\in \mathcal{X}_{\pm }^{\ast }$, previously defined by (\ref%
{pression approxbis})--(\ref{pression approxbis nl}). We also analyze in
this context the variational problems associated with the two-person
zero-sum game (Section \ref{Decision}) whose payoff function is the
nonlinear approximating pressure $P_{\mathrm{NL}}$. Note that we always
consider here \emph{non-zero} functions $g_{\pm }$. If one of them is zero,
we have essentially the same results (mutatis mutandis). See, e.g., the
discussions at the end of Section \ref{sect conc conv Bogo}.

Before we start our technical study, we go over a few of the notation: $%
B_{\pm }(0,R)\subseteq \mathcal{X}_{\pm }^{\ast }$ are closed balls (\ref%
{ball1}) of radius $R\in \mathbb{R}^{+}$ and center $0\in \mathcal{X}_{\pm
}^{\ast }$ and $\Vert \cdot \Vert _{\mathrm{op}}$ is the usual operator
norm, defined by (\ref{ball2}) for any continuous linear functional, while $%
\Vert \cdot \Vert _{\infty }$ is the uniform norm (or sup norm), defined by (%
\ref{sup norm}). Recall that any continuous function defined on a compact is
uniformly bounded and the corresponding uniform norm is thus finite.

We start with two simple continuity properties of the linear approximating
pressure $P_{\mathrm{L}}$, seen as a function 
\begin{equation*}
\left( y_{+},y_{-}\right) \mapsto P_{\mathrm{L}}\left( y_{+},y_{-}\right)
\end{equation*}%
from $\mathcal{X}_{+}^{\ast }\times \mathcal{X}_{-}^{\ast }$ to $\mathbb{R}$
and defined by Equation (\ref{pression approxbis nlbis}) (see also (\ref%
{pression approxbis nl})). Some arguments are elementary and are only given
below for completeness. Recall that Condition B1 is stated in Section \ref%
{sect conc conv Bogo}.

\begin{lemma}[Norm continuity of linear approximating pressures]
\label{Lemma bogoluibov1}\mbox{ }\newline
Assume Condition B1. Then, $P_{\mathrm{L}}:\mathcal{X}_{+}^{\ast }\times 
\mathcal{X}_{-}^{\ast }\rightarrow \mathbb{R}$ is convex and Lipschitz-norm
continuous: 
\begin{equation*}
\left\vert P_{\mathrm{L}}\left( x_{+},x_{-}\right) -P_{\mathrm{L}}\left(
y_{+},y_{-}\right) \right\vert \leq \left\Vert \tau _{+}\right\Vert _{\infty
}\left\Vert x_{+}-y_{+}\right\Vert _{\mathrm{op}}+\left\Vert \tau
_{-}\right\Vert _{\infty }\left\Vert x_{-}-y_{-}\right\Vert _{\mathrm{op}}\
,\qquad x_{\pm },y_{\pm }\in \mathcal{X}_{\pm }^{\ast }\ .
\end{equation*}
\end{lemma}

\begin{proof}
Observe that, for any $x_{\pm },y_{\pm }\in \mathcal{X}_{\pm }^{\ast }$ and $%
\nu \in K$, 
\begin{equation*}
\left\vert x_{\pm }\circ \tau _{\pm }\left( \nu \right) -y_{\pm }\circ \tau
_{\pm }\left( \nu \right) \right\vert \leq \left\Vert x_{\pm }-y_{\pm
}\right\Vert _{\mathrm{op}}\left\Vert \tau _{\pm }\right\Vert _{\infty }
\end{equation*}%
and the inequality of the lemma directly follows from the definition 
\begin{equation}
\mathcal{G}_{y_{+},y_{-}}\doteq f-y_{-}\circ \tau _{-}+y_{+}\circ \tau _{+}\
,\qquad y_{\pm }\in \mathcal{X}_{\pm }^{\ast }\ ,  \label{def 121}
\end{equation}
of Bogoliubov linearizations, see Equation (\ref{def G mix}). $P_{\mathrm{L}%
} $ is convex because it is the supremum of affine (real-valued) mappings $%
(y_{+},y_{-})\mapsto \mathcal{G}_{y_{+},y_{-}}(\mu )$, $\mu \in K$.
\end{proof}

\begin{lemma}[Weak$^{\ast }$ continuity of linear approximating pressures]
\label{Lemma bogoluibov2}\mbox{ }\newline
Assume Condition B1. Then, $P_{\mathrm{L}}$ is weak$^{\ast }$ continuous on $%
B_{+}(0,R)\times B_{-}(0,R)$ for any fixed $R\in \mathbb{R}^{+}$. If $%
\mathcal{X}_{\pm }$ are Banach spaces then $P_{\mathrm{L}}$ is weak$^{\ast }$
continuous on the whole space $\mathcal{X}_{+}^{\ast }\times \mathcal{X}%
_{-}^{\ast }$.
\end{lemma}

\begin{proof}
Fix $R\in \mathbb{R}^{+}$. Take any nets $(y_{\pm }^{(j)})_{j\in J}\subseteq
B_{\pm }(0,R)$ weak$^{\ast }$ converging to arbitrary points $y_{\pm }\in
B_{\pm }(0,R)$, respectively. By Condition B1, $K$ is (Hausdorff)\ compact
and $\tau _{\pm }$ are continuous. As a consequence, $\tau _{\pm
}(K)\subseteq \mathcal{X}_{\pm }$ are compact (in the norm topology of $%
\mathcal{X}_{\pm }$). Thus, for any $\varepsilon \in \mathbb{R}^{+}$, there
are finite sets $M_{\pm ,\varepsilon }\subseteq K$ such that%
\begin{equation}
\mathrm{dist}\left( \tau _{\pm }\left( M_{\pm ,\varepsilon }\right) ,\tau
_{\pm }\left( K\right) \right) =\sup_{\mu \in K}\inf_{\nu \in M_{\pm
,\varepsilon }}\left\Vert \tau _{\pm }\left( \mu \right) -\tau _{\pm }\left(
\nu \right) \right\Vert _{\mathcal{X}_{\pm }}\leq \frac{\varepsilon }{4R}\ .
\label{dfgsdfgdfgdf}
\end{equation}%
Remark that the above quantity is the so-called Hausdorff distance between
the compact sets $\tau _{\pm }\left( M_{\pm ,\varepsilon }\right) $ and $%
\tau _{\pm }\left( K\right) $ in $\mathcal{X}_{\pm }$. Then, by the triangle
inequality and Inequality (\ref{dfgsdfgdfgdf}), for any $\mu \in K$ and $\nu
\in M_{\pm ,\varepsilon }$,%
\begin{equation*}
\left\vert y_{\pm }\circ \tau _{\pm }\left( \mu \right) -y_{\pm }^{(j)}\circ
\tau _{\pm }\left( \mu \right) \right\vert \leq \inf_{\nu \in M_{\pm
,\varepsilon }}\left\vert y_{\pm }\circ \tau _{\pm }\left( \nu \right)
-y_{\pm }^{(j)}\circ \tau _{\pm }\left( \nu \right) \right\vert +\frac{%
\varepsilon }{2}\ .
\end{equation*}%
Hence, there is $j_{\varepsilon }\in J$ such that, for all $\mu \in K$ and $%
j\in J$ satisfying $j\geq j_{\varepsilon }$,%
\begin{equation*}
\left\vert y_{\pm }\circ \tau _{\pm }\left( \mu \right) -y_{\pm }^{(j)}\circ
\tau _{\pm }\left( \mu \right) \right\vert \leq \varepsilon \ ,
\end{equation*}%
because $M_{\pm ,\varepsilon }$ is finite and $(y_{\pm }^{(j)})_{j\in J}$
weak$^{\ast }$ converges to $y_{\pm }$. Now, as the last bound is uniform
with respect to $\mu \in K$, from the definition (\ref{def 121}) of $%
\mathcal{G}_{y_{+},y_{-}}$, it follows that the mapping $\left(
y_{+},y_{-}\right) \mapsto P_{\mathrm{L}}\left( y_{+},y_{-}\right) $ from $%
B_{+}(0,R)\times B_{-}(0,R)$ to $\mathbb{R}$ is weak$^{\ast }$ continuous.
If the normed spaces $\mathcal{X}_{\pm }$ are complete, then any weak$^{\ast
}$ convergent net in the dual spaces $\mathcal{X}_{\pm }^{\ast }$ is norm
bounded, thanks to the Uniform Boundedness Principle (the Banach-Steinhaus
theorem \cite[Theorem 2.5]{Rudin}). Consequently, in the Banach case, $P_{%
\mathrm{L}}$ is weak$^{\ast }$ continuous on $\mathcal{X}_{+}^{\ast }\times 
\mathcal{X}_{-}^{\ast }$.
\end{proof}

After briefly studying linear pressures, we now provide a series of lemmata
regarding important properties of the nonlinear approximating pressure $P_{%
\mathrm{NL}}$, seen as a mapping 
\begin{equation*}
\left( y_{+},y_{-}\right) \mapsto P_{\mathrm{NL}}\left( y_{+},y_{-}\right)
\end{equation*}%
from $\mathcal{X}_{+}^{\ast }\times \mathcal{X}_{-}^{\ast }$ to $\mathbb{%
R\cup \{\infty \}}$ and defined by Equation (\ref{pression approxbis nlbis})
(see also (\ref{pression approxbis})).

\begin{lemma}[Properties of the nonlinear approximating pressure]
\label{Lemma bogoluibov3}\mbox{ }\newline
Assume Condition B1. Let $g_{\pm }:\mathcal{X}_{\pm }\rightarrow \mathbb{R}$
be non-zero lower semicontinuous and convex functions with $\mathrm{dom}%
(g_{+}^{\ast })=\mathcal{X}_{+}^{\ast }$. \newline
\emph{(}$+$\emph{)} For $R\in \mathbb{R}^{+}$ and $y_{-}\in \mathrm{dom}%
(g_{-}^{\ast })$, the mapping $y_{+}\mapsto P_{\mathrm{NL}}(y_{+},y_{-})$
from $B_{+}(0,R)$ to $\mathbb{R}$ is weak$^{\ast }$-upper semicontinuous. If 
$\mathcal{X}_{+}$ is a Banach space then it is weak$^{\ast }$-upper
semicontinuous on $\mathcal{X}_{+}^{\ast }$.\smallskip \newline
\emph{(}$-$\emph{)} For $R\in \mathbb{R}^{+}$ and $y_{+}\in \mathcal{X}%
_{+}^{\ast }$, the mapping $y_{-}\mapsto P_{\mathrm{NL}}(y_{+},y_{-})$ from $%
B_{-}(0,R)$ to $\mathbb{R\cup \{\infty \}}$ is convex and weak$^{\ast }$%
%TCIMACRO{\TeXButton{\-}{\-}}%
%BeginExpansion
\-%
%EndExpansion
-lower semicontinuous. If $\mathcal{X}_{-}$ is a Banach space then it is weak%
$^{\ast }$-lower semicontinuous on $\mathcal{X}_{-}^{\ast }$.
\end{lemma}

\begin{proof}
As Legendre-Fenchel transforms, the functions $g_{+}^{\ast }:\mathcal{X}%
_{+}^{\ast }\rightarrow \mathbb{R}$ and $g_{-}^{\ast }:\mathcal{X}_{-}^{\ast
}\rightarrow \mathbb{R\cup \{\infty \}}$ are always weak$^{\ast }$-lower
semicontinuous. Note also that the convex set $\mathrm{dom}(g_{-}^{\ast })$
is nonempty, thanks to (\ref{sdfsdfsdfsdfsdfs}). Therefore, by\ Lemma \ref%
{Lemma bogoluibov2}, for any fixed $R\in \mathbb{R}^{+}$ and $y_{-}\in 
\mathrm{dom}(g_{-}^{\ast })$, the mapping $y_{+}\mapsto P_{\mathrm{NL}%
}(y_{+},y_{-})$ from $B_{+}(0,R)$ to $\mathbb{R}$ is weak$^{\ast }$%
%TCIMACRO{\TeXButton{\-}{\-}}%
%BeginExpansion
\-%
%EndExpansion
-upper semicontinuous. The Banach case is also a direct consequence of Lemma %
\ref{Lemma bogoluibov2}. Mutatis mutandis for ($-$). Notice that 
\begin{equation*}
y_{-}\mapsto P_{\mathrm{NL}}(y_{+},y_{-})\doteq P_{\mathrm{L}%
}(y_{+},y_{-})+g_{-}^{\ast }\left( y_{-}\right) -g_{+}^{\ast }\left(
y_{+}\right)
\end{equation*}%
is convex, because Legendre-Fenchel transforms are always convex and $%
y_{-}\mapsto P_{\mathrm{L}}(y_{+},y_{-})$ is also convex, thanks to Lemma %
\ref{Lemma bogoluibov1}.
\end{proof}

We next undertake the study of the two functions%
\begin{equation*}
P^{\flat }:\mathcal{X}_{+}^{\ast }\rightarrow \{-\infty \}\cup \mathbb{R}%
\qquad \text{and}\qquad P^{\sharp }:\mathcal{X}_{-}^{\ast }\rightarrow 
\mathbb{R}\cup \{\infty \}
\end{equation*}%
defined by (\ref{def diese bemol1})--(\ref{def diese bemol2}), that is, 
\begin{equation}
P^{\flat }\left( y_{+}\right) \doteq \inf P_{\mathrm{NL}}\left( y_{+},%
\mathcal{X}_{-}^{\ast }\right) \qquad \text{and}\qquad P^{\sharp }\left(
y_{-}\right) \doteq \sup P_{\mathrm{NL}}\left( \mathcal{X}_{+}^{\ast
},y_{-}\right)  \label{Pbemol-diese}
\end{equation}%
for any $y_{\pm }\in \mathcal{X}_{\pm }^{\ast }$. We assume that
Legendre-Fenchel transforms $g^{\ast }$ have minimal linear growth $\lambda
\in \mathbb{R}^{+}$ in the sense of (\ref{linear grow}), that is,%
\begin{equation}
\lim_{R\rightarrow \infty }\sup_{y\in \mathcal{X}^{\ast }\backslash
B(0,R)}\left\{ \lambda \left\Vert y\right\Vert _{\mathrm{op}}-g^{\ast
}\left( y\right) \right\} =-\infty \ .  \label{condition1}
\end{equation}%
This allows us to show that these variational problems can be restricted to
closed balls.

\begin{lemma}[Consequences of the minimal\emph{\ }linear growth of $g_{\pm
}^{\ast }$]
\label{Lemma suplin growth}\mbox{ }\newline
Assume Condition B1. Let $g_{\pm }:\mathcal{X}_{\pm }\rightarrow \mathbb{R}$
be non-zero lower semicontinuous and convex functions with $\mathrm{dom}%
(g_{+}^{\ast })=\mathcal{X}_{+}^{\ast }$. \newline
\emph{(}$\flat $\emph{)} If (\ref{condition1}) holds for $g=g_{-}$ and $%
\lambda =\Vert \tau _{-}\Vert _{\infty }$, then there is $R\in \mathbb{R}%
^{+} $ such that, for all $y_{+}\in \mathcal{X}_{+}^{\ast }$,%
\begin{equation*}
P^{\flat }\left( y_{+}\right) =\inf_{y_{-}\in B_{-}(0,R)}P_{\mathrm{NL}%
}\left( y_{+},y_{-}\right) <\inf_{y_{-}\in \mathcal{X}_{-}^{\ast }\backslash
B_{-}(0,R)}P_{\mathrm{NL}}\left( y_{+},y_{-}\right) \ .
\end{equation*}%
\emph{(}$\sharp $\emph{) }If (\ref{condition1}) holds for $g=g_{+}$ and $%
\lambda =\Vert \tau _{+}\Vert _{\infty }$, then there is $R\in \mathbb{R}%
^{+} $ such that, for all $y_{-}\in \mathrm{dom}(g_{-}^{\ast })$,%
\begin{equation*}
P^{\sharp }\left( y_{-}\right) =\sup_{y_{+}\in B_{+}(0,R)}P_{\mathrm{NL}%
}\left( y_{+},y_{-}\right) >\sup_{y_{+}\in \mathcal{X}_{+}^{\ast }\backslash
B_{+}(0,R)}P_{\mathrm{NL}}\left( y_{+},y_{-}\right) \ .
\end{equation*}%
\emph{(}$\sharp \flat $\emph{) }If (\ref{condition1}) holds for $g=g_{\pm }$
and $\lambda =\Vert \tau _{\pm }\Vert _{\infty }$, then there is $R\in 
\mathbb{R}^{+}$ such that%
\begin{equation*}
\sup_{y_{+}\in B_{+}(0,R)}P^{\flat }\left( y_{+}\right) >\sup_{y_{+}\in 
\mathcal{X}_{+}^{\ast }\backslash B_{+}(0,R)}P^{\flat }\left( y_{+}\right) \
.
\end{equation*}%
\newline
\emph{(}$\flat \sharp $\emph{) }If (\ref{condition1}) holds for $g=g_{\pm }$
and $\lambda =\Vert \tau _{\pm }\Vert _{\infty }$, then there is $R\in 
\mathbb{R}^{+}$ such that%
\begin{equation*}
\inf_{y_{-}\in B_{-}(0,R)}P^{\sharp }\left( y_{-}\right) <\inf_{y_{-}\in 
\mathcal{X}_{-}^{\ast }\backslash B_{-}(0,R)}P^{\sharp }\left( y_{-}\right)
\ .
\end{equation*}
\end{lemma}

\begin{proof}
Assertions ($\flat $)\emph{\ }and ($\sharp $) are direct consequences of
Lemma \ref{Lemma bogoluibov1}. Note that the strict inequality of ($\sharp $%
) cannot be satisfied when $y_{-}\notin \mathrm{dom}(g_{-}^{\ast })$,
because, in this case, $P_{\mathrm{NL}}\left( y_{+},y_{-}\right) =\infty $
for all $y_{+}\in \mathcal{X}_{+}^{\ast }$. To prove ($\sharp \flat $) take
any $x_{-}\in \mathrm{dom}(g_{-}^{\ast })\neq \emptyset $ (see (\ref%
{sdfsdfsdfsdfsdfs}))\ and note from Lemma \ref{Lemma bogoluibov1} and
Equation (\ref{condition1}) for $g=g_{+}$ and $\lambda =\Vert \tau _{+}\Vert
_{\infty }$ that%
\begin{multline*}
\lim_{R\rightarrow \infty }\sup_{y_{+}\in \mathcal{X}_{+}^{\ast }\backslash
B_{+}(0,R)}P^{\flat }\left( y_{+}\right) =\lim_{R\rightarrow \infty
}\sup_{y_{+}\in \mathcal{X}_{+}^{\ast }\backslash B_{+}(0,R)}\inf_{y_{-}\in 
\mathcal{X}_{-}^{\ast }}P_{\mathrm{NL}}\left( y_{+},y_{-}\right) \\
\leq \lim_{R\rightarrow \infty }\sup_{y_{+}\in \mathcal{X}_{+}^{\ast
}\backslash B_{+}(0,R)}\left\{ P_{\mathrm{L}}\left( y_{+},x_{-}\right)
+g_{-}^{\ast }\left( x_{-}\right) -g_{+}^{\ast }\left( y_{+}\right) \right\}
=-\infty \ .
\end{multline*}%
We meanwhile infer from ($\flat $) the existence of some strictly positive
radius $\tilde{R}\in \mathbb{R}^{+}$ such that 
\begin{equation*}
P^{\flat }\left( y_{+}\right) =\inf_{y_{-}\in B_{-}(0,\tilde{R})}P_{\mathrm{%
NL}}\left( y_{+},y_{-}\right) \ ,\qquad y_{+}\in \mathcal{X}_{+}^{\ast }\ .
\end{equation*}%
Thus, by the weak$^{\ast }$ compactness of the closed ball $B_{-}(0,\tilde{R}%
)$ (cf. the Banach-Alaoglu theorem \cite[Theorem 3.15]{Rudin}) and Lemma \ref%
{Lemma bogoluibov3} ($-$), 
\begin{equation}
P^{\flat }\left( y_{+}\right) =\min_{y_{-}\in B_{-}(0,\tilde{R})}P_{\mathrm{%
NL}}\left( y_{+},y_{-}\right) =\min_{y_{-}\in B_{-}(0,\tilde{R})\cap \mathrm{%
dom}(g_{-}^{\ast })}P_{\mathrm{NL}}\left( y_{+},y_{-}\right) \in \mathbb{R}
\label{P bemol extra}
\end{equation}%
\emph{\ }for all $y_{+}\in \mathcal{X}_{+}^{\ast }$, keeping in mind that $%
\mathrm{dom}(g_{-}^{\ast })\neq \emptyset $ (see (\ref{sdfsdfsdfsdfsdfs})).
In particular, $P^{\flat }(y_{+})>-\infty $ for all $y_{+}\in \mathcal{X}%
_{+}^{\ast }$ and, as a consequence, for some $R<\infty $, 
\begin{equation*}
\sup_{y_{+}\in B_{+}(0,R)}P^{\flat }\left( y_{+}\right) >\sup_{y_{+}\in 
\mathcal{X}_{+}^{\ast }\backslash B_{+}(0,R)}P^{\flat }\left( y_{+}\right) \
.
\end{equation*}%
The proof of ($\flat \sharp $) is similar: Note from Lemma \ref{Lemma
bogoluibov1} and Equation (\ref{condition1}) for $g=g_{-}$ and $\lambda
=\Vert \tau _{-}\Vert _{\infty }$ that, for any $x_{+}\in \mathcal{X}%
_{+}^{\ast }$, 
\begin{multline*}
\lim_{R\rightarrow \infty }\inf_{y_{-}\in \mathcal{X}_{-}^{\ast }\backslash
B_{-}(0,R)}P^{\sharp }\left( y_{-}\right) =\lim_{R\rightarrow \infty
}\inf_{y_{-}\in \mathcal{X}_{-}^{\ast }\backslash B_{-}(0,R)}\sup_{y_{+}\in 
\mathcal{X}_{+}^{\ast }}P_{\mathrm{NL}}\left( y_{+},y_{-}\right) \\
\geq \lim_{R\rightarrow \infty }\inf_{y_{-}\in \mathcal{X}_{-}^{\ast
}\backslash B_{-}(0,R)}\left\{ P_{\mathrm{L}}\left( y_{+},x_{-}\right)
+g_{-}^{\ast }\left( y_{-}\right) -g_{+}^{\ast }\left( x_{+}\right) \right\}
=\infty \ .
\end{multline*}%
By Assertion ($\sharp $), there is $\tilde{R}\in \mathbb{R}^{+}$ such that%
\begin{equation*}
P^{\sharp }\left( y_{-}\right) =\sup_{y_{+}\in B_{+}(0,\tilde{R})}P_{\mathrm{%
NL}}\left( y_{+},y_{-}\right) \ ,\qquad y_{-}\in \mathrm{dom}(g_{-}^{\ast
})\ ,
\end{equation*}%
which, combined with the weak$^{\ast }$ compactness of the closed ball $%
B_{+}(0,\tilde{R})$ and Lemma \ref{Lemma bogoluibov3} ($+$), yields\emph{\ }%
\begin{equation}
P^{\sharp }\left( y_{-}\right) =\max_{y_{+}\in B_{+}(0,\tilde{R})}P_{\mathrm{%
NL}}\left( y_{+},y_{-}\right) \in \mathbb{R}\ ,\qquad y_{-}\in \mathrm{dom}%
(g_{-}^{\ast })\ .  \label{P diese extra}
\end{equation}%
Since $\mathrm{dom}(g_{-}^{\ast })\neq \emptyset $ (see (\ref%
{sdfsdfsdfsdfsdfs})) and $P^{\sharp }\left( y_{-}\right) =\infty $ when $%
y_{-}\notin \mathrm{dom}(g_{-}^{\ast })$, we conclude that, for some $%
R<\infty $, 
\begin{equation*}
\inf_{y_{-}\in B_{-}(0,R)}P^{\sharp }\left( y_{-}\right) =\inf_{y_{-}\in
B_{-}(0,R)\cap \mathrm{dom}(g_{-}^{\ast })}P^{\sharp }\left( y_{-}\right)
<\inf_{y_{-}\in \mathcal{X}_{-}^{\ast }\backslash B_{-}(0,R)}P^{\sharp
}\left( y_{-}\right) \ .
\end{equation*}
\end{proof}

If the Legendre-Fenchel transforms $g_{\pm }^{\ast }$ have minimal linear
growth $\Vert \tau _{\pm }\Vert _{\infty }$, one infers from Equations (\ref%
{P bemol extra})--(\ref{P diese extra}) that $P^{\flat }$ and $P^{\sharp }$
define real-valued functions on $\mathcal{X}_{+}^{\ast }$ and $\mathrm{dom}%
(g_{-}^{\ast })\subseteq \mathcal{X}_{-}^{\ast }$, respectively.
Additionally, these two functions have the following properties:

\begin{lemma}[Properties of the functions $P^{\sharp }$ and $P^{\flat }$]
\label{Lemma bogoluibov4}\mbox{ }\newline
Assume Condition B1. Let $g_{\pm }:\mathcal{X}_{\pm }\rightarrow \mathbb{R}$
be non-zero lower semicontinuous and convex functions with $\mathrm{dom}%
(g_{+}^{\ast })=\mathcal{X}_{+}^{\ast }$.\newline
\emph{(}$\flat $\emph{)} If (\ref{condition1}) holds for $g=g_{-}$ and $%
\lambda =\Vert \tau _{-}\Vert _{\infty }$, then $P^{\flat }:\mathcal{X}%
_{+}^{\ast }\rightarrow \mathbb{R}$ is weak$^{\ast }$-upper semicontinuous
on $B_{+}(0,R)$ for any $R\in \mathbb{R}^{+}$. If $\mathcal{X}_{+}$ is a
Banach space then $P^{\flat }$ is weak$^{\ast }$-upper semicontinuous on $%
\mathcal{X}_{+}^{\ast }$.\smallskip \newline
\emph{(}$\sharp $\emph{)} If (\ref{condition1}) holds for $g=g_{+}$ and $%
\lambda =\Vert \tau _{+}\Vert _{\infty }$, then $P^{\sharp }:\mathcal{X}%
_{-}^{\ast }\rightarrow \mathbb{R\cup \{\infty \}}$ is convex and weak$%
^{\ast }$-lower semicontinuous on $B_{-}(0,R)$ for any $R\in \mathbb{R}^{+}$%
. Its restriction to (the nonempty convex set) $\mathrm{dom}(g_{-}^{\ast })$
is convex and real-valued. If $\mathcal{X}_{-}$ is a Banach space then $%
P^{\sharp }$ is weak$^{\ast }$-lower semicontinuous on $\mathcal{X}%
_{-}^{\ast }$.
\end{lemma}

\begin{proof}
Fix $R\in \mathbb{R}^{+}$. If (\ref{condition1}) holds for $g=g_{-}$ and $%
\lambda =\Vert \tau _{-}\Vert _{\infty }$, then we deduce from Equation (\ref%
{P bemol extra}) that there is $\tilde{R}\in \mathbb{R}^{+}$ such that the
function $P^{\flat }:B_{+}(0,R)\rightarrow \mathbb{R}$ is the infimum of the
family%
\begin{equation*}
\{y_{+}\mapsto P_{\mathrm{NL}}\left( y_{+},y_{-}\right) \}_{y_{-}\text{$\in $%
}B_{-}(0,\tilde{R})\cap \mathrm{dom}(g_{-}^{\ast })}
\end{equation*}%
of weak$^{\ast }$-upper semicontinuous real-valued functions from $%
B_{+}(0,R) $ to $\mathbb{R}$, thanks to Lemma \ref{Lemma bogoluibov3} ($+$).
The real-valued function $P^{\flat }:B_{+}(0,R)\rightarrow \mathbb{R}$ is
therefore weak$^{\ast }$-upper semicontinuous. The Banach case is proven in
the same way. Mutatis mutandis for ($\sharp $), thanks to Lemma \ref{Lemma
bogoluibov3} ($-$) and Equation (\ref{P diese extra}). To prove the
convexity of $P^{\sharp }$, observe that this mapping is the supremum over $%
y_{+}\in \mathcal{X}_{+}^{\ast }$ of the convex mappings 
\begin{equation*}
y_{-}\mapsto P_{\mathrm{NL}}\left( y_{+},y_{-}\right) \doteq P_{\mathrm{L}%
}\left( y_{+},y_{-}\right) +g_{-}^{\ast }(y_{-})-g_{+}^{\ast }(y_{+})
\end{equation*}%
all defined on the same (nonempty) convex domain $\mathrm{dom}(g_{-}^{\ast
}) $.
\end{proof}

For all $y_{\pm }\in \mathcal{X}_{\pm }^{\ast }$, we study now the subsets (%
\ref{M-y1}) and (\ref{M-y2}) of solutions to the variational problems $%
P^{\flat }(y_{+})$ and $P^{\sharp }(y_{-})$, i.e., 
\begin{eqnarray*}
M^{\flat }\left( y_{+}\right) &\doteq &\left\{ x_{-}\in \mathcal{X}%
_{-}^{\ast }:P^{\flat }\left( y_{+}\right) =P_{\mathrm{NL}}\left(
y_{+},x_{-}\right) \right\} \subseteq \mathcal{X}_{-}^{\ast }\ , \\
M^{\sharp }\left( y_{-}\right) &\doteq &\left\{ x_{+}\in \mathcal{X}%
_{+}^{\ast }:P^{\sharp }\left( y_{-}\right) =P_{\mathrm{NL}}\left(
x_{+},y_{-}\right) \right\} \subseteq \mathcal{X}_{+}^{\ast }\ .
\end{eqnarray*}

\begin{lemma}[Solutions to the variational problems $P^{\flat }(y_{+})$ and $%
P^{\sharp }(y_{-})$]
\label{Lemma bogoluibov5}\mbox{ }\newline
Assume Condition B1. Let $g_{\pm }:\mathcal{X}_{\pm }\rightarrow \mathbb{R}$
be non-zero lower semicontinuous and convex functions with $\mathrm{dom}%
(g_{+}^{\ast })=\mathcal{X}_{+}^{\ast }$. \newline
\emph{(}$\flat $\emph{)} If (\ref{condition1}) holds for $g=g_{-}$ and $%
\lambda =\Vert \tau _{-}\Vert _{\infty }$, then, for all $y_{+}\in \mathcal{X%
}_{+}^{\ast }$, $M^{\flat }(y_{+})$ is nonempty, convex and weak$^{\ast }$%
-compact. There is $R\in \mathbb{R}^{+}$ such that $M^{\flat
}(y_{+})\subseteq B_{-}(0,R)\cap \mathrm{dom}(g_{-}^{\ast })$ for all $%
y_{+}\in \mathcal{X}_{+}^{\ast }$.\smallskip \newline
\emph{(}$\sharp $\emph{)} If (\ref{condition1}) holds for $g=g_{+}$ and $%
\lambda =\Vert \tau _{+}\Vert _{\infty }$, then, for all $y_{-}\in \mathrm{%
dom}(g_{-}^{\ast })$, $M^{\sharp }(y_{-})$ is nonempty and weak$^{\ast }$%
-compact. There is $R\in \mathbb{R}^{+}$ such that $M^{\sharp
}(y_{-})\subseteq B_{+}(0,R)$ for all $y_{-}\in \mathrm{dom}(g_{-}^{\ast })$.
\end{lemma}

\begin{proof}
Recall that $\mathrm{dom}(g_{-}^{\ast })\neq \emptyset $, thanks to (\ref%
{sdfsdfsdfsdfsdfs}). For all $y_{+}\in \mathcal{X}_{+}^{\ast }$, $M^{\flat
}(y_{+})\subseteq \mathrm{dom}(g_{-}^{\ast })$ because $P_{\mathrm{NL}%
}\left( y_{+},y_{-}\right) =\infty $ when $y_{-}\notin \mathrm{dom}%
(g_{-}^{\ast })$.\ If (\ref{condition1}) holds for $g=g_{-}$ and $\lambda
=\Vert \tau _{-}\Vert _{\infty }$ then, by Lemma \ref{Lemma suplin growth} ($%
\flat $)\emph{,} $M^{\flat }(y_{+})\subseteq B_{-}(0,R)\cap \mathrm{dom}%
(g_{-}^{\ast })$ for all $y_{+}\in \mathcal{X}_{+}^{\ast }$. Additionally,
in this case, the mapping $y_{-}\mapsto P_{\mathrm{NL}}\left(
y_{+},y_{-}\right) $ is weak$^{\ast }$-lower semicontinuous, thanks to Lemma %
\ref{Lemma bogoluibov3} ($-$). Bearing in mind again that closed\ balls of
the dual space of a normed space are weak$^{\ast }$-compact (cf. the
Banach-Alaoglu theorem \cite[Theorem 3.15]{Rudin}), we deduce that the sets $%
M^{\flat }(y_{+})\subseteq B_{-}(0,R)$, $y_{+}\in \mathcal{X}_{+}^{\ast }$,
are nonempty and weak$^{\ast }$-compact, as they are always weak$^{\ast }$%
-closed. Their convexity\ is a direct consequence of the convexity of the
function $P_{\mathrm{NL}}(y_{+},\cdot )$ (Lemma \ref{Lemma bogoluibov3} ($-$%
)). The case ($\sharp $) is proven in a similar way, from the weak$^{\ast }$%
-upper semicontinuity of $P_{\mathrm{NL}}(\cdot ,y_{-})$. See Lemma \ref%
{Lemma bogoluibov3} ($+$).
\end{proof}

Now we consider the conservative values%
\begin{eqnarray}
\mathrm{P}^{\flat } &\doteq &\sup_{y_{+}\in \mathcal{X}_{+}^{\ast }}P^{\flat
}\left( y_{+}\right) \doteq \sup_{y_{+}\in \mathcal{X}_{+}^{\ast
}}\inf_{y_{-}\in \mathcal{X}_{-}^{\ast }}P_{\mathrm{NL}}\left(
y_{+},y_{-}\right) \ ,  \label{Q diese bemol2} \\
\mathrm{P}^{\sharp } &\doteq &\inf_{y_{-}\in \mathcal{X}_{-}^{\ast
}}P^{\sharp }\left( y_{-}\right) \doteq \inf_{y_{-}\in \mathcal{X}_{-}^{\ast
}}\sup_{y_{+}\in \mathcal{X}_{+}^{\ast }}P_{\mathrm{NL}}\left(
y_{+},y_{-}\right) \ ,  \label{Q diese bemol1}
\end{eqnarray}%
of the thermodynamic game. See Section \ref{Decision}, in particular
Equation (\ref{conservative values}). We also study the corresponding sets (%
\ref{M1})--(\ref{M2}) of optimizers, i.e., 
\begin{eqnarray*}
\mathrm{M}^{\flat } &\doteq &\left\{ x_{+}\in \mathcal{X}_{+}^{\ast }:%
\mathrm{P}^{\flat }=P^{\flat }\left( x_{+}\right) \right\} \subseteq 
\mathcal{X}_{+}^{\ast }\ , \\
\mathrm{M}^{\sharp } &\doteq &\left\{ x_{-}\in \mathcal{X}_{-}^{\ast }:%
\mathrm{P}^{\sharp }\doteq P^{\sharp }\left( x_{-}\right) \right\} \subseteq 
\mathcal{X}_{-}^{\ast }\ ,
\end{eqnarray*}%
which are nothing but the sets of conservative strategies of the
thermodynamic game.

\begin{lemma}[Sets $\mathrm{M}^{\sharp }$ and $\mathrm{M}^{\flat }$ of
optimizers]
\label{Lemma bogoluibov6}\mbox{ }\newline
Assume Condition B1. Let $g_{\pm }:\mathcal{X}_{\pm }\rightarrow \mathbb{R}$
be non-zero lower semicontinuous and convex functions for which $\mathrm{dom}%
(g_{+}^{\ast })=\mathcal{X}_{+}^{\ast }$ and* (\ref{condition1}) holds for $%
g=g_{\pm }$ and $\lambda =\Vert \tau _{\pm }\Vert _{\infty }$.\newline
\emph{(}$\flat $\emph{)} $\mathrm{P}^{\flat }\in \mathbb{R}$ and the set $%
\mathrm{M}^{\flat }\subseteq \mathcal{X}_{+}^{\ast }$ is nonempty,
norm-bounded and weak$^{\ast }$-compact.\smallskip \newline
\emph{(}$\sharp $\emph{)} $\mathrm{P}^{\sharp }\in \mathbb{R}$ and the set $%
\mathrm{M}^{\sharp }\subseteq \mathcal{X}_{-}^{\ast }$ is a nonempty,
convex, norm-bounded and weak$^{\ast }$-compact subset of $\mathrm{dom}%
(g_{-}^{\ast })$.
\end{lemma}

\begin{proof}
By Lemma \ref{Lemma suplin growth} ($\sharp \flat $)\emph{, }$\mathrm{M}%
^{\flat }\subseteq B_{+}(0,R)\subseteq \mathcal{X}_{+}^{\ast }$ is
norm-bounded. From Lemma \ref{Lemma bogoluibov4} ($\flat $) and the weak$%
^{\ast }$ compactness of $B_{+}(0,R)$ one concludes that $\mathrm{P}^{\flat
}\in \mathbb{R}$ and that $\mathrm{M}^{\flat }$ is nonempty and weak$^{\ast
} $-compact. Mutatis mutandis for that case ($\sharp $), thanks to Lemmata %
\ref{Lemma suplin growth}\ ($\flat \sharp $) and \ref{Lemma bogoluibov4} ($%
\sharp $). $\mathrm{M}^{\sharp }$ is convex, because $P^{\sharp }$ is
convex, thanks to Lemma \ref{Lemma bogoluibov4} ($\sharp $). Note that $%
\mathrm{M}^{\sharp }\subseteq \mathrm{dom}(g_{-}^{\ast })$, because $\mathrm{%
P}^{\sharp }<\infty $ and $P^{\sharp }(y_{-})=\infty $ when $y_{-}\notin 
\mathrm{dom}(g_{-}^{\ast })\neq \emptyset $.
\end{proof}

\subsection{Appendix}

This appendix contains some useful standard results. These could help
non-specialists to understand the present paper more easily. We present the
von Neumann minimax theorem (Section \ref{The von Neumann min-max theorem})
and the Choquet theorem (Section \ref{The Choquet theorem}).

\subsubsection{The von Neumann minimax theorem\label{The von Neumann min-max
theorem}}

This theorem was first proven by John von Neumann in 1928 \cite{minimax thm}%
. He proved it in relation to two-person zero-sum games. This laid the
foundations of game theory, even if the theorem has applications that go far
beyond this field of mathematics. It gives a general criterion for the
existence of \emph{saddle points}.

Recall first that saddle points are defined as follows: Let $M$ and $N$ be
two sets. The element $(x_{0},y_{0})\in M\times N$ is, by definition, a
saddle point of the real-valued function $f:M\times N\rightarrow \mathbb{R}$
if 
\begin{equation*}
\underset{y\in N}{\sup }f\left( x_{0},y\right) =\underset{x\in M}{\inf }%
\underset{y\in N}{\sup }f\left( x,y\right) =\underset{y\in N}{\sup }\underset%
{x\in M}{\inf }f\left( x,y\right) =\underset{x\in M}{\inf }f\left(
x,y_{0}\right) \ .
\end{equation*}%
A saddle point $(x_{0},y_{0})\in M\times N$ in particular satisfies the
equalities%
\begin{equation*}
f(x_{0},y_{0})=\underset{x\in M}{\inf }\underset{y\in N}{\sup }f\left(
x,y\right) =\underset{y\in N}{\sup }\underset{x\in M}{\inf }f\left(
x,y\right) \ .
\end{equation*}%
The von Neumann minimax theorem, which is used here in an essential way,
refers to the following statement about the existence of saddle points of
functions on topological vector spaces:

\begin{theorem}[von Neumann]
\label{theorem minmax von Neumann}\mbox{ }\newline
Let $M$ and $N$ be two (nonempty) compact convex space. Assume that $%
f:M\times N\rightarrow \mathbb{R}$ is a real-valued function such that, for
all $y\in N$, the mapping $x\mapsto f(x,y)$ is convex and lower
semicontinuous, whereas, for all $x\in M$, the mapping $y\mapsto f(x,y)$ is
concave and upper semicontinuous. Then there exists a saddle point $%
(x_{0},y_{0})\in M\times N$ of $f$.
\end{theorem}

\noindent There are many different proofs of this assertion available in the
literature and we recommend, for example, \cite[Chapter 8]{Aubin} for a
concise review on two-person zero-sum games, including a proof of the von
Neumann minimax theorem \cite[Theorem 8.2]{Aubin}.

\subsubsection{The Choquet theorem\label{The Choquet theorem}}

A classical result of Minkowski's states that, in finite dimensions, any
element $x\in K$ in a (nonempty) compact convex subset $K\subseteq \mathcal{X%
}$ can be decomposed into a convex combination of a finite number of extreme
points $\hat{x}_{1},\ldots ,\hat{x}_{k}\in \mathcal{E}(K)$, that is, 
\begin{equation}
x=\overset{k}{\sum\limits_{j=1}}\lambda _{j}\hat{x}_{j}\ ,
\label{barycenter1}
\end{equation}%
where $\lambda _{1},\ldots ,\lambda _{k}\geq 0$ are positive numbers
satisfying $\Sigma _{j=1}^{k}\lambda _{j}=1$. Recall that the extreme points
of a convex set are the elements that cannot be written as (non-trivial)
convex combinations of other elements in $K$. Their existence is ensured for
all compact sets in a locally convex spaces $\mathcal{X}$. In fact, it is
well-known that in this case, any compact convex set $K\subseteq \mathcal{X}$
is the closure of the convex hull of the (nonempty) set $\mathcal{E}(K)$ of
its extreme points, according to the Krein-Milman theorem \cite[Theorems 3.4
(b) and 3.21]{Rudin}.

To the\ decomposition (\ref{barycenter1}) we can naturally associate a
probability measure, i.e., a normalized positive Borel regular measure, $\xi 
$\ on $K$: Take the probability measure $\xi _{x}$\ on $K$ defined by 
\begin{equation*}
\xi _{x}\doteq \overset{k}{\sum\limits_{j=1}}\lambda _{j}\delta _{\hat{x}%
_{j}}
\end{equation*}%
with $\delta _{y}$ being the Dirac (or point) measure\footnote{$\delta _{y}$
is the Borel measure such that for any Borel subset $B\in \mathfrak{B}$ of $%
K $, $\delta _{y}(B)=1$ if $y\in B$ and $\delta _{y}(B)=0$ if $y\notin B$.}
at $y$ and rewrite (\ref{barycenter1}) as the following (weak) integral%
\begin{equation}
x=\int_{K}\hat{x}\;\xi _{x}\left( \mathrm{d}\hat{x}\right) \ .
\label{barycenter2}
\end{equation}%
That is, the point $x$ is the so-called \emph{barycenter} of the probability
measure $\xi _{x}$. This notion is defined in the general case as follows
(cf. \cite[p. 1]{Phe}):

\begin{definition}[Barycenters of a measure]
\label{Barycenters of a measure}\mbox{ }\newline
Let $\mathcal{X}$ be any real topological vector space, $K\subseteq \mathcal{%
X}$ a nonempty compact subset and $\xi $ a (Borel) probability measure on $K$%
. An element $x\in \mathcal{X}$ is a barycenter of $\xi $ if, for any
continuous linear functional $x^{\ast }\in \mathcal{X}^{\ast }$,%
\begin{equation*}
x^{\ast }\left( x\right) =\int_{K}x^{\ast }\left( \hat{x}\right) \xi \left( 
\mathrm{d}\hat{x}\right) \ .
\end{equation*}
\end{definition}

Note that if $\mathcal{X}$ if locally convex then $\xi $ has at most one
barycenter, for $\mathcal{X}^{\ast }$ separates the points of $\mathcal{X}$,
and in this case we refer to \emph{the} barycenter of $\xi $. Its existence
is ensured by the following result \cite[Propositions 1.1--1.2]{Phe}:

\begin{proposition}[Existence of barycenters]
\label{Existence of barycenters}\mbox{ }\newline
Let $\mathcal{X}$ be any locally convex vector space and $K\subseteq 
\mathcal{X}$ a nonempty compact subset. If the closed convex hull $\overline{%
\mathrm{co}}K\subseteq \mathcal{X}$ is also compact, then every normalized
Borel measure on $K$ has a (unique) barycenter. Additionally, for all $x\in 
\mathcal{X}$, $x\in \overline{\mathrm{co}}K$ iff $x$ is the barycenter of
some (not necessarily unique) probability measure on $K$.
\end{proposition}

The Krein-Milman theorem says that any nonempty compact \emph{convex} $%
K\subseteq X$ of a locally convex vector space has extreme points forming a
set $\mathcal{E}(K)$ which satisfies $K=\overline{\mathrm{co}}$ $\mathcal{E}%
(K)$, see \cite[Theorems 3.4 (b) and 3.21]{Rudin}. In particular, by
Proposition \ref{Existence of barycenters}, any element $x\in K$ is the
barycenter of some normalized Borel measure on the closure $\overline{%
\mathcal{E}}(K)$ of the set $\mathcal{E}(K)$ of all extreme points of $K$.
In infinite dimension one generically \cite{article-hypertopology} has that $%
\overline{\mathcal{E}}(K)=K$ and this property is thus useless in such a
situation. If the relative topology of the given compact subset $K\subseteq 
\mathcal{X}$ is metrizable, the Choquet theorem strengthens the Krein-Milman
theorem by stating that the measure representing an arbitrary element $x\in
K $ as its barycenter can always be chosen such that it is supported in $%
\mathcal{E}(K)$ (and not just in the closure $\overline{\mathcal{E}}(K)$).
Indeed, in this case, $\mathcal{E}(K)$ is a Borel set \cite[Proposition 1.3]%
{Phe}:

\begin{lemma}[Extreme points form a $G_{\protect\delta }$-set]
\label{lemma extr gd}\mbox{ }\newline
Let $\mathcal{X}$ be any topological vector space and $K\subseteq \mathcal{X}
$ a compact convex subset. If the relative topology of $K$ is metrizable,
then $\mathcal{E}(K)$ is a $G_{\delta }$-set with respect to the relative
topology of $K$. In particular, it is a Borel set.
\end{lemma}

We are now able to state the Choquet theorem, first proven in 1956 by
Gustave Choquet \cite{Choquet}:

\begin{theorem}[Choquet]
\label{th Choquet}\mbox{ }\newline
Let $\mathcal{X}$ be any locally convex vector space and $K\subseteq 
\mathcal{X}$ a nonempty, compact and convex\ subset, whose relative topology
is metrizable. For any $x\in K$, there is a (not necessarily unique)
probability measure $\xi _{x}$ on $K$, which is supported on $\mathcal{E}%
(K)\subseteq K$ (i.e., $\xi _{x}(K\backslash \mathcal{E}(K))=0$) and whose
barycenter is $x$.
\end{theorem}

\noindent We recommend \cite{Phe} for a concise review on the Choquet
theorem and its generalization to non-metrizable cases. The Choquet theorem
stated above is proven on page 14 of these lecture notes \cite{Phe}.

Such a measure $\xi _{x}$, as given by Theorem \ref{th Choquet}, is called
here a \emph{Choquet measure}\ associated with the element $x\in K$. If any
point $x\in K$ has exactly one associated Choquet measure then the convex
subset $K\subseteq X$ is a so-called \emph{Choquet simplex}.\bigskip

\noindent \textit{Acknowledgments:} This work is supported by FAPESP (grant
2025/12824-5), CNPq (grant 303682/2025-6), the Basque Government through the
BERC 2022-2025 program, as well as the following grant: Grant
PID2024-156184NB-I00 funded by MICIU/AEI/10.13039/501100011033 and cofunded
by the European Union.


\begin{thebibliography}{99}
\bibitem{ACR} D. Aguiar, L. Cioletti and R. Ruviaro, A Variational Principle
for the Specific Entropy for Symbolic Systems with Uncountable Alphabets, 
\textit{Math. Nachr.} \textbf{291}(17-18) (2018) 2506-2515.

\bibitem{Alfsen} E. M. Alfsen, \textit{Compact convex sets and boundary
integrals}. Ergebnisse der Mathematik und ihrer Grenzgebiete -- Band 57.
Springer-Verlag, 1971.

\bibitem{AliBor} Ch. D. Aliprantis, K. C. Border, \textit{Infinite
Dimensional Analysis. A Hitchhiker's Guide}. \textit{3rd Edition}.
Springer-Verlag, 2006.

\bibitem{Aubin} J.-P. Aubin, \textit{Optima and Equilibria: An Introduction
to Nonlinear Analysis, }Berlin-Heidelberg, Springer-Verlag, 1998.

\bibitem{Aubinbis} J.-P. Aubin, \textit{Mathematical Methods of Game and
Economic Theory.} Dover Publications Inc., 2009.

\bibitem{TF-Linear4} V. Baladi, \textit{Positive Transfer Operators and
decay of correlations}, World Scientific, 2000.

\bibitem{TF3} L. Barreira and C. Holanda, Higher-dimensional nonlinear
thermodynamical formalism, \textit{J. Stat. Phys}. \textbf{187} (2022) 18
(article number).

\bibitem{TF-Linear8} A. Bi\'{s}, M. Carvalho, M. Mendes and P. Varandas, A
Convex Analysis Approach to Entropy Functions, Variational Principles and
Equilibrium States, \textit{Commun. Math. Phys.} \textbf{394}(1) (2022)
215-256. Correction: \textit{Commun. Math. Phys.} \textbf{401} (2023)
3335--3342.

\bibitem{Bogoliubov1} N.N. Bogoliubov, On the theory of superfluidity, 
\textit{J. Phys. (USSR)} \textbf{11} (1947) 23-32.

\bibitem{Bogjunior} N.N. Bogoliubov Jr., On model dynamical systems in
statistical mechanics, \textit{Physica} \textbf{32}(5) (1966) 933-944.

\bibitem{approx-hamil-method0} N.N. Bogoliubov Jr., J.G. Brankov, V.A.
Zagrebnov, A.M. Kurbatov and N.S. Tonchev, \textit{Metod
approksimiruyushchego gamil'toniana v statisticheskoi fizike\footnote{%
The Approximating Hamiltonian Method in\ Statistical Physics.}.} Sofia:
Izdat. Bulgar. Akad. Nauk\footnote{%
Publ. House Bulg. Acad. Sci.}, 1981.

\bibitem{approx-hamil-method} N.N. Bogoliubov Jr., J.G. Brankov, V.A.
Zagrebnov, A.M. Kurbatov and N.S. Tonchev, Some classes of exactly soluble
models of problems in Quantum Statistical Mechanics: the method of the
approximating Hamiltonian, \textit{Russ. Math. Surv.} \textbf{39 }(1984)
1-50.

\bibitem{TF-Linear0} R. Bowen, Equilibrium states and the ergodic theory of
Anosov diffeomorphisms, \textit{Lect. Notes Math.} \textbf{470} (1975).

\bibitem{AHM-non-poly1} J.G. Brankov, N.S. Tonchev and V.A. Zagrebnov, A
nonpolynomial generalization of exactly soluble models in statistical
mechanics, \textit{Ann. Phys. (N. Y.)} \textbf{107}(1-2) (1977) 82-94.

\bibitem{AHM-non-poly2} J.G. Brankov, N.S. Tonchev and V.A. Zagrebnov, On a
class of exactly soluble statistical mechanical models with nonpolynomial
interactions, \textit{J. Stat. Phys.} \textbf{20}(3) (1979) 317-330.

\bibitem{approx-hamil-method2} J.G. Brankov, D.M. Danchev and N.S. Tonchev, 
\textit{Theory of Critical Phenomena in Finite--size Systems: Scaling and
Quantum Effects,} Singapore-New Jersey-London-Hong Kong, World Scientific,
2000.

\bibitem{BruPedraconvex} J.-B. Bru, W. de Siqueira Pedra, Remarks on the $%
\Gamma $--regularization of Non-convex and Non-Semi-Continuous Functionals
on Topological Vector Spaces, \textit{J. Convex Anal.} \textbf{19}(3) (2012)
467-483.

\bibitem{BruPedra2} J.-B. Bru, W. de Siqueira Pedra, Non-cooperative
Equilibria of Fermi Systems With Long Range Interactions, \textit{Memoirs of
the AMS} \textbf{224}(1052) (2013).

\bibitem{Bru-pedra-MF-I} J.-B. Bru, W. de Siqueira Pedra, Classical Dynamics
from Self-Consistency Equations in Quantum Mechanics, \textit{J. Math. Phys.}
\textbf{63} (2022) 052101 (33 pages). See also arXiv:2009.04969 [math-ph]
which is an extended version (72 pages).

\bibitem{Bru-Pedra-Lopes1} J.-B. Bru, W. de Siqueira Pedra and A.O. Lopes,
Nonlinear Thermodynamic Formalism: Mean-field Phase Transitions, Large
Deviations and Bogoliubov's Variational Principle, arXiv:2511.06975
[math.DS] (2025).

\bibitem{Kac} J.-B. Bru, W. de Siqueira Pedra and K. Rodrigues Alves, From
Short-Range to Mean-Field Models in Quantum Lattices, \textit{Adv. Theor.
Math. Phys.} \textbf{28}(1) (2024) 69-159.

\bibitem{article-hypertopology} J.-B. Bru and W. de Siqueira Pedra, Weak$%
^{\ast }$ Hypertopologies with Application to Genericity of Convex Sets, 
\textit{J. Convex Anal.} \textbf{29}(1)\ (2021) 13-60.

\bibitem{BruPedra-textbook} J.-B. Bru and de Siqueira Pedra, $C^{\ast }$ 
\textit{-Algebra and Mathematical Foundations of Quantum Statistical
Mechanics}, Latin American Mathematics Series - UFSCar subseries, Springer
Nature Switzerland AG, 2023.

\bibitem{TF-Linear6} H. Bruin, \textit{Topological and Ergodic Theory of
Symbolic Dynamics}, Graduate Studies in Mathematics, Vol. 228, AMS, 2022.

\bibitem{TF4} J. Buzzi, B. Kloeckner and R. Leplaideur, Nonlinear
thermodynamical formalism, \textit{Annales Henri Lebesgue} \textbf{6} (2023)
1429-1477.

\bibitem{Choquet} G. Choquet, Existence et unicit\'{e} des repr\'{e}%
sentations int\'{e}grales au moyen des points extr\'{e}maux dans les c\^{o}%
nes convexes, \textit{Seminaire Bourbaki} \textbf{139} (Dec. 1956) 15 pp.

\bibitem{CLS} L. Cioletti, A. O. Lopes and M. Stadlbauert, Ruelle Operator
for Continuous Potentials and DLR-Gibbs Measures, \textit{Disc and Cont.
Dyn. Syst. A} \textbf{40}(8) (2020) 4625-4652.

\bibitem{Climenhaga} V. Climenhaga, The thermodynamic approach to
multifractal analysis, \textit{Ergodic Theory Dyn. Syst.} \textbf{34}:5
(2014) 1409-1450.

\bibitem{Clime} V. Climenhaga, D. J. Thompson and K. Yamamoto, Large
deviations for systems with non-uniform structure, TAMS, 369 (2017), no. 6,
4167-4192.

\bibitem{Colbrook} M.J. Colbrook, Z. Drma\v{c}, A. Horningar, An
Introductory Guide to Koopman Learning, ArXiv:2510.22002v1 [math.NA] (2025)
(33pp).

\bibitem{Georgii} H.-O. Georgii, \textit{Gibbs Measures and Phase
Transitions }, Berlin-New York, De Gruyter, 2011.

\bibitem{Ginibre} J. Ginibre, On the Asymptotic Exactness of the Bogoliubov
Approximation for many Bosons Systems, \textit{Commun. Math. Phys.} \textbf{%
8 } (1968) 26-51

\bibitem{GKLM} P. Giulietti, B. Kloeckner, A. O. Lopes and D. Marcon, The
calculus of thermodynamical formalism, \textit{J. Eur. Math. Soc.} \textbf{%
20 }(10) (2018) 2357-2412.

\bibitem{Sanz} J. Gonzalez-Conde, D. Lewis, S.S. Bharadwaj, and M. Sanz,
Quantum Carleman linearization efficiency in nonlinear fluid dynamics, 
\textit{Physical Review Research} \textbf{7} (2025) 023254 (16pp).

\bibitem{Haydon} R. Haydon, E. Odell, H. Rosenthal, \textit{On certain
classes of Baire-1 functions with applications to Banach space theory}. In:
Odwell, E.E., Rosenthal, H.P. (eds) Functional Analysis. Lecture Notes in
Mathematics, vol 1470. Springer, Berlin, Heidelberg, 1991.

\bibitem{Levy} O. Kallenberg, \textit{Foundations of Modern Probability},
2nd Edition, New York-Berlin-Heidelberg, Springer-Verlag, 2001.

\bibitem{TF-Linear2} A. Katok and B. Hasselblatt, \textit{Introduction to
the modern theory of dynamical systems}, Cambridge University Press,
Cambridge, 1995.

\bibitem{Katz} J. Katz, G. Muraleedharan and A. Alase, Efficient quantum
algorithm for solving differential equations with Fourier nonlinearity via
Koopman linearization, arXiv:2512.06488 [quant-ph] (2025) (46pp).

\bibitem{Kechris} A.S. Kechris, \textit{Classical Descriptive Set Theory},
Springer-Verlag New York. Inc., 1995.

\bibitem{TF-Linear3} G. Keller, \textit{Equilibrium states in Ergodic Theory}%
, Cambridge Univ. Press, 1998.

\bibitem{HIDETOSHI KOMIYA} H. Komiya, Elementary Proof For Sion's minimax
theorem, \textit{Kodai Math. J.} \textbf{11}(1) (1988) 5-7.

\bibitem{TF5} T. Kucherenko, Nonlinear thermodynamic formalism through the
lens of rotation theory, \textit{Disc. and Cont. Dyn. Systems} \textbf{44}%
(12) (2024) 3760-3773.

\bibitem{LiebSeiringerYngvason3} E.H. Lieb, R. Seiringer and J. Yngvason,
Justification of $c$--Number Substitutions in Bosonic Hamiltonians, \textit{%
\ Phys. Rev. Lett.} \textbf{94} (2005) 080401-1-4.

\bibitem{Lindenstrauss-etal} J. Lindenstrauss, G.H. Olsen and Y. Sternfeld,
The Poulsen simplex, \textit{Ann. Inst. Fourier (Grenoble)} \textbf{28}
(1978) 91-114.

\bibitem{TF1} R. Leplaideur and F. Watbled, Generalized Curie-Weiss model
and quadratic pressure in ergodic theory, \textit{Bull. Soc. Math. France} 
\textbf{147}(2) (2019) 197-219.

\bibitem{TF2} R. Leplaideur and F. Watbled, Curie-Weiss Type Models for
General Spin Spaces and Quadratic Pressure in Ergodic Theory, \textit{J.
Stat. Phys}. \textbf{181} (2020) 263-292.

\bibitem{Lopez-fractal} A.O. Lopes, The Dimension Spectrum of the Maximal
Measure, \textit{SIAM Journal on Mathematical Analysis} \textbf{20}(5)
(1989)\ 1243-1254.

\bibitem{L3} A. O. Lopes, Entropy and Large Deviation, \textit{NonLinearity} 
\textbf{3}(2) (1990) 527-546.

\bibitem{TF-Linear9} A.O. Lopes, Thermodynamic Formalism, Maximizing
Probabilities and Large Deviations, Text under construction, (2024), see
http://mat.ufrgs.br/\symbol{126}alopes/pub3/notesformteherm.pdf.

\bibitem{LMMS} A. O. Lopes, J. K. Mengue, J. Mohr and R. R. Souza, Entropy
and Variational Principle for one-dimensional Lattice Systems with a general
a-priori probability: positive and zero temperature, \textit{Erg. Theory and
Dyn Systems} \textbf{35}(6) (2015) 1925-1961.

\bibitem{LMO} A.O. Lopes, J. K. Mengue and E.R. Oliveira, Idempotent
approach to level-2 variational principles in Thermodynamical Formalism, 
\textit{J. Math. Phys.} \textbf{67} (2026) 042701.

\bibitem{LO} A.O. Lopes and E.R. Oliveira, Level-2 IFS Thermodynamic
Formalism: Gibbs probabilities in the space of probabilities and the
push-forward map, \textit{Dyn. Systems} \textbf{39}, Issue 4 (2024) 823--847.

\bibitem{Parthasarathy} K.R. Parthasarathy, \textit{Probability Measures on
Metric Spaces}, New York and London, Academic Press, 1967.

\bibitem{TF-Linear1} W. Parry and M. Pollicott, Zeta functions and the
periodic orbit structure of hyperbolic dynamics, \textit{Ast\'{e}risque} 
\textbf{187-188} (1990), Chapters 1-4.

\bibitem{Phe} R.R. Phelps, \textit{Lectures on Choquet's Theorem}, 2nd
Edition, Lecture Notes in Mathematics, Vol. 1757, Berlin-Heidelberg,
Springer-Verlag, 2001.

\bibitem{Phe2} R.R. Phelps, \textit{Convex Functions, Monotone Operators and
Differentiability}, 2nd Edition, Lecture Notes in Mathematics, Vol. 1364,
Berlin-Heidelberg, Springer-Verlag 1993.

\bibitem{Rudin} W. Rudin, \textit{Functional Analysis}, McGraw-Hill Science,
1991.

\bibitem{TF-Linear5} D. Ruelle, \textit{Thermodynamic Formalism: The
Mathematical Structures of Classical Equilibrium Statistical Mechanics},
Volumen 5 de Encyclopedia of mathematics and its applications,
Addison-Wesley Publishing Company, Advanced Book Program, 1978; 2nd Edition,
Cambridge University Press, 2004.

\bibitem{Schirotzek} W. Schirotzek, \textit{Nonsmooth Analysis},
Universitext, Berlin-Heidelberg, Springer-Verlag, 2007.

\bibitem{reference} D. Shi and X. Yang, Koopman Spectral Linearization vs.
Carleman Linearization: A Computational Comparison Study, Mathematics 
\textbf{12}(14) (2024) 2156 (16pp).

\bibitem{SouVA} R. R. Souza and V. Vargas, Existence of Gibbs States and
Maximizing Measures on a General One-Dimensional Lattice System with
Markovian Structure, \textit{Qual. Theory Dyn. Syst.} \textbf{21} (2022)
article 5.

\bibitem{Tennie} F. Tennie, S. Laizet, S. Lloyd and Luca Magri, Quantum
computing for nonlinear differential equations and turbulence, \textit{%
Nature Reviews Physics} \textbf{7} (2025) 220-230.

\bibitem{minimax thm} J. Von Neumann, Zur Theorie der Gesellschaftsspiele, 
\textit{Math. Ann.} \textbf{100} (1928) 295-320.

\bibitem{Viana} M. Viana and K. Oliveira, \textit{Foundations of Ergodic
Theory}, Cambridge University Press, 2016.

\bibitem{Vilani} C. Villani, \textit{Optimal Transport, Old and New},
Berlin-Heidelberg, Springer-Verlag, 2009.

\bibitem{Walters} P. Walters, \textit{Introduction to Ergodic Theory},
Graduate Texts in Mathematics, Springer Verlag, 2000.

\bibitem{Wang} T. Wang and W. Wu, Thermodynamic Formalism for Non-uniform
Systems with Controlled Specification and Entropy Expansiveness, \textit{%
Commun. Math. Phys.} \textbf{32} (2026) 407.

\bibitem{Xu} Y. Xu, Z. Kuang, Q. Huang, J. Yang, H. Zahrouni, M.
Potier-Ferry, K. Huang, J.-C. Zhang, H. Fan, H. Hu, A robust quantum
nonlinear solver based on the asymptotic numerical method, arXiv:2412.03939
[quant-ph] (2024) (35pp).

\bibitem{BruZagrebnov8} V.A. Zagrebnov and J.-B. Bru, The Bogoliubov Model
of Weakly Imperfect Bose Gas, \textit{Phys. Rep.} \textbf{350} (2001)
291-434.

\bibitem{CIRM1} \textit{Thermodynamic Formalism, CIRM Jean-Morlet Chair,
Fall 2019}, Lecture Notes in Mathematics, Eds: M. Pollicott, S. Vaienti,
Springer Cham, 2021.
\end{thebibliography}
\end{document}